\documentclass[11pt,reqno]{amsart}

\usepackage{graphics}
\usepackage{cite}
\usepackage{amsmath}
\usepackage{amsfonts}
\usepackage{amssymb}
\usepackage{float}
\usepackage{epic}
\usepackage{eepic}
\usepackage{color}
\usepackage{ae}
\usepackage{mathrsfs,array,graphicx}
\usepackage[all,ps]{xy}
\usepackage{epstopdf}
\usepackage{a4wide}
\usepackage{enumerate}
\usepackage{hyperref} \hypersetup{colorlinks,allcolors=[rgb]{0,0,0}}
\usepackage[margin=1in]{geometry}
\usepackage{stmaryrd}
\usepackage{enumitem}

\theoremstyle{theorem}
\newtheorem{thmx}{Theorem}

\newtheorem{theorem}{Theorem}
\newtheorem{proposition}[theorem]{Proposition}

\newtheorem{corollary}[theorem]{Corollary}
\newtheorem{question}[theorem]{Question}
\newtheorem{lemma}[theorem]{Lemma}
\theoremstyle{definition}
\newtheorem{definition}[theorem]{Definition}
\theoremstyle{remark}

\newtheorem{remark}[theorem]{Remark}

\newtheorem{example}[theorem]{Example}

\numberwithin{equation}{subsection}
\numberwithin{theorem}{subsection}

\let\li\overline
\def\hat{\widehat}

\newcommand{\ab}{\textup{ab}}
\newcommand{\ad}{\textup{ad}}
\newcommand{\Aut}{\textup{Aut}}
\newcommand{\A}{\mathbb A}
\newcommand{\bs}{\backslash}
\newcommand{\cA}{{\mathcal A}}
\newcommand{\cC}{{\mathcal C}}
\newcommand{\cE}{\mathcal E}
\newcommand{\cG}{{\mathcal G}}
\newcommand{\cH}{\mathcal H}
\newcommand{\cI}{{\mathcal I}}
\newcommand{\cL}{\mathcal L}
\newcommand{\cmpt}{\tu{cmpt}}
\newcommand{\cO}{\mathcal O}
\newcommand{\cP}{{\mathcal P}}
\newcommand{\cS}{\mathcal S}
\newcommand{\cusp}{\textup{cusp}}
\newcommand{\C}{\mathbb C}
\newcommand{\der}{\textup{der}}
\newcommand{\disc}{\textup{disc}}
\newcommand{\D}{{\mathbb D}}
\newcommand{\el}{\mathrm{ell}}
\newcommand{\eps}{\varepsilon}

\newcommand{\Exp}{\mathrm{Exp}}

\newcommand{\fka}{{\mathfrak a}}
\newcommand{\fke}{{\mathfrak e}}
\newcommand{\fkp}{\mathfrak{p}}
\newcommand{\fkX}{\mathfrak{X}}

\newcommand{\Fpbar}{\ol{\F}_p}
\newcommand{\Frac}{\tu{Frac}}
\newcommand{\Frob}{\textup{Frob}}
\newcommand{\F}{{\mathbb F}}
\newcommand{\Gal}{{\textup{Gal}}}

\newcommand{\GL}{{\textup {GL}}}
\newcommand{\Groth}{{\textup{Groth}}}
\newcommand{\GSp}{\textup{GSp}}
\newcommand{\G}{\mathbb G}
\newcommand{\Hom}{\textup{Hom}}
\newcommand{\hra}{\hookrightarrow}
\newcommand{\ia}{{\mathfrak a}}
\newcommand{\ie}{\textit{i.e.,} }
\newcommand{\IG}{{\mathfrak{Ig}}}
\newcommand{\Ig}{{\textup{Ig}}}
\newcommand{\Int}{{\textup{Int}}}
\newcommand{\inv}{^{-1}}
\newcommand{\Irr}{\mathrm{Irr}}
\newcommand{\isom}{\stackrel{\sim}{\ra}}

\newcommand{\lhk}{\left(}
\newcommand{\Lie}{\textup{Lie}}
\newcommand{\lisom}{\stackrel{\sim}{\leftarrow}}
\newcommand{\lql}{\overline{\Q}_\ell}
\newcommand{\lqp}{\overline \Q_p}
\newcommand{\mA}{{\mathscr A}}
\newcommand{\mS}{\mathscr{S}}
\newcommand{\nind}{\tu{n-ind}}
\newcommand{\ol}{\overline}
\newcommand{\perf}{\textup{perf}}
\newcommand{\pr}{\textup{pr}}

\newcommand{\qp}{{\mathbb{Q}_p}}
\newcommand{\Qp}{{\mathbb{Q}_p}}
\newcommand{\Qq}{{\mathbb Q}_q}
\newcommand{\Q}{\mathbb Q}
\newcommand{\ra}{\rightarrow}
\newcommand{\reg}{\tu{reg}}
\newcommand{\Rep}{{\tu{Rep}}}
\newcommand{\Res}{\textup{Res}}
\newcommand{\rg}{\rangle}
\newcommand{\rhk}{\right)}

\newcommand{\R}{\mathbb R}
\newcommand{\Sh}{\textup{Sh}}
\newcommand{\SL}{\textup{SL}}

\newcommand{\SO}{{\textup{SO}}}
\newcommand{\Spec}{\tu{Spec}\,}

\newcommand{\ST}{\textup{ST}}

\newcommand{\supp}{{\textup{supp}}}

\newcommand{\Tr}{\textup{Tr}}
\newcommand{\tu}[1]{\textup{#1}}
\newcommand{\uAut}{\underline{\Aut}}
\newcommand{\val}{\textup{val}}
\newcommand{\vierkant}[4]{{\lhk \begin{smallmatrix} #1 & #2 \cr #3 & #4 \end{smallmatrix} \rhk } }
\newcommand{\vol}{\textup{vol}}
\newcommand{\Z}{\mathbb Z}

\renewcommand{\Lie}{\tu{Lie}\,}
\renewcommand{\to}{\rightarrow}
\renewcommand{\Tr}{\tu{Tr}\,}

\begin{document}

\title[$H^0$ of Igusa varieties]{$H^0$ of Igusa varieties via automorphic forms}
\author{Arno Kret, Sug Woo Shin}
\date{\today}

\begin{abstract}
Our main theorem describes the degree 0 cohomology of non-basic Igusa varieties in terms of one-dimensional automorphic representations in the geometry of mod $p$ Hodge-type Shimura varieties with hyperspecial level at $p$. As an application we obtain a completely new approach to two geometric questions. Firstly, we deduce irreducibility of Igusa towers and its generalization to non-basic Igusa varieties in the same generality, extending previous results by Igusa, Ribet, Faltings--Chai, Hida, and others. Secondly, we verify the discrete part of the Hecke orbit conjecture, which amounts to the assertion that the irreducible components of a non-basic central leaf belong to a single prime-to-$p$ Hecke orbit, generalizing preceding works by Chai, Oort, Yu, et al. We also show purely local criteria for irreducibility of central leaves. Our proof is based on a Langlands--Kottwitz type formula for Igusa varieties due to Mack-Crane, an asymptotic study of the trace formula, and an estimate for unitary representations and their Jacquet modules in representation theory of $p$-adic groups due to Howe--Moore and Casselman.
\end{abstract}

\maketitle
\tableofcontents

\section{Introduction}
Igusa varieties were studied by Igusa \cite{Igusa} and Katz--Mazur \cite{KatzMazur} in the case of modular curves. Harris--Taylor and Mantovan \cite{HarrisTaylor,Man05} have generalized the construction to PEL-type Shimura varieties. Recently Caraiani--Scholze \cite{CS17} gave a slightly different definition in the PEL case which gives the same cohomology. Hamacher, Zhang, and Hamacher--Kim went further to define Igusa varieties for Hodge-type Shimura varieties \cite{Ham-almostproduct,ZhangC,HamacherKim}. In the \mbox{($\mu$-)ordinary} setting, Igusa varieties are also referred to as Igusa towers. (Often the definitions differ in a minor way.) There are versions of Igusa varieties as $p$-adic formal schemes or adic spaces over $p$-adic fields, but we concentrate on the characteristic $p$ varieties in this paper. We mention that function-field analogues of Igusa varieties are studied in a forthcoming paper by Sempliner; see also \cite[Ex.~4.7.15]{ZhuCoherent}.

The $\ell$-adic cohomology of Igusa varieties (with $\ell\neq p$) has several arithmetic applications. In \cite{HarrisTaylor,Man05,HamacherKim}, the authors prove a formula computing the cohomology of Shimura varieties in terms of that of Igusa varieties and Rapoport--Zink spaces. This means that, if we understand the cohomology of Igusa varieties well enough, then our knowledge of cohomology can be propagated from Rapoport--Zink spaces to Shimura varieties or the other way around. This is the basic principle underlying \cite{ShinGalois,ShinRZ} on the global Langlands correspondence and the Kottwitz conjecture. For another application, a description of $\ell$-adic cohomology of Igusa varieties was one of the main ingredients in \cite{CS17,CS19} to prove vanishing of cohomology of certain Shimura varieties with $\ell$-torsion coefficients, which in turn supplied a critical input for a recent breakthrough on the Ramanujan and Sato--Tate conjecture for cuspidal automorphic representations of $\GL_2$ of ``weight 2'' over CM fields \cite{10authors}.

Thus an important long-term goal is to compute the $\ell$-adic cohomology of Igusa varieties with a natural group action. A major first step is a Langlands--Kottwitz style trace formula for Igusa varieties, which has been obtained for Shimura varieties of Hodge type at hyperspecial level in \cite{ShinIgusa,ShinIgusaStab,MackCrane} building upon \cite[Ch.~5]{HarrisTaylor} in analogy with \cite{LanglandsRapoport,KottwitzPoints,KSZ}. One wishes to turn that into an expression of the cohomology via automorphic forms, but this requires a solution of various complicated problems; some are tractable but others are out of reach in general, most notably an endoscopic classification and Arthur's multiplicity formula for the relevant groups.

The main objective of this paper is twofold. Firstly, we describe $H^0$ of Igusa varieties via one-dimensional automorphic representations over non-basic Newton strata of Hodge-type Shimura varieties at hyperspecial level.\footnote{In the basic case, Igusa varieties are 0-dimensional, and it can be deduced from \cite{MackCrane} that their $H^0$ is expressed as the space of algebraic automorphic forms on an inner form of $G$. Such a description goes back to Serre \cite{SerreTwoLetters} for modular curves, and Fargues \cite[Ch.~5]{FarguesAst291} in the PEL case.} This mirrors the well-known fact that $H^0$ of complex Shimura varieties is governed by one-dimensional automorphic representations. Secondly, to achieve this, we develop a method and obtain various technical results with a view towards the entire cohomology of Igusa varieties (as an alternating sum over all degrees). Our method, partly inspired by Laumon \cite{Laumon-GSp4} and also by Flicker--Kazhdan \cite{FlickerKazhdan}, should prove useful for studying $\ell$-adic cohomology of Shimura varieties as well.

Our result on $H^0$ not only sets a milestone in its own right, but also reveals deep geometric information. Namely, our theorem readily implies the discrete Hecke orbit conjecture for Shimura varieties and the irreducibility of Igusa varieties in the same generality as above. (The irreducibility means that Igusa varieties are no more reducible than the underlying Shimura varieties in some precise sense.) Our work provides a completely new approach and perspective to these two problems by means of automorphic forms and representation theory.
 
One of our main novelties consists in a careful asymptotic argument via the trace formula to single out $H^0$ (or compactly supported cohomology in the top degree) without reliance on any classification. This is essential for obtaining an unconditional result. Since the ``variable'' for asymptotics is encoded in the test function at $p$, a good amount of local harmonic analysis naturally enters the picture. Another feature of our approach is to allow induction on the semisimple rank of the group; this would make little sense in a purely geometric argument as endoscopy is hard to realize in the geometry of Shimura varieties.

Roughly speaking, cohomology of Igusa varieties is closely related to that of Shimura varieties via the Jacquet module operation at $p$, relative to a \emph{proper} parabolic subgroup in the \emph{non-basic} case. To show that only one-dimensional automorphic representations contribute to $H^0$ of Igusa varieties, the key representation-theoretic input is an estimate for the central action on Jacquet modules due to Casselman and Howe--Moore. Though there is no direct link, it would be interesting to note that a similar situation occurs in the context of beyond endoscopy (e.g.,~\cite[\S5]{FLN}), where the leading term in asymptotics is accounted for by the ``most non-tempered'' (namely one-dimensional) representations. This is also analogous to the spectral gap, which plays a crucial role in the Hecke equidistribution theorems in characteristic zero (cf.~\S\ref{sub:intro-disc-HO} below).
 
\subsection{The main theorem}\label{sub:main-thm}
 
Let $(G,X)$ be a Shimura datum of Hodge type with reflex field $E\subset \C$. Assume that the reductive group $G$ over $\Q$ admits a reductive model over $\Z_p$, and take $K_p:=G(\Z_p)$. (Namely $G$ is unramified at $p$, and $K_p$ is a hyperspecial subgroup.) For simplicity, the adjoint group of $G$ is assumed to be simple over $\Q$ throughout the introduction. (Otherwise the notion of basic elements needs to be modified. See Definition \ref{def:Q-non-basic} below.) We do not assume $p>2$ as the case $p=2$ is covered in \cite{KSZ,MackCrane}.

Fix field maps $\ol \Q\hra \lqp$, $\lqp\simeq \C$, and $\lql \simeq \C$ (which will be mostly implicit). The resulting embedding $E\hra \lqp$ induces a place $\fkp$ of $E$ above $p$. Let $k(\fkp)$ denote the residue field of $E$ at $\fkp$, which embeds into the residue field $\Fpbar$ of $\lqp$. Thereby we identify $\ol{k(\fkp)}\simeq \Fpbar$. Let $\mS_{K_p}$ denote the integral canonical model over $\cO_{E_{\fkp}}$ with a $G(\A^{\infty,p})$-action. In the main text, it is sometimes important to pass to finite level away from $p$ in order to apply a fixed-point formula. However, we will ignore this point and pretend that we are always at infinite level away from $p$ to simplify exposition.

An embedding of $(G,X)$ into a Siegel Shimura datum determines a $G(\A^{\infty,p})$-equivariant map from $\mS_{K_p}$ to a suitable Siegel moduli scheme over $\cO_{E_{\fkp}}$. Via pullback, we obtain a universal abelian scheme $\mA$ over $\mS_{K_p}$, which can be equipped with a family of \'etale and crystalline tensors over geometric points $\ol x\ra  \mS_{K_p,k(\fkp)}$. This assigns to $\ol x$ the $p$-divisible group $\mA_{\ol x}[p^\infty]$ (with $G$-structure, encoded by the family of crystalline tensors).

Let $\mu_p:\G_m\ra G_{\lqp}$ denote the ``Hodge'' cocharacter arising from $(G,X)$ (via $\lqp\simeq \C$). This cuts out a finite subset $B(G_{\Q_p},\mu_p^{-1})$ in the Kottwitz set $B(G_{\Q_p})$ of $G$-isocrystals. Fix $b\in G(\breve \Q_p)$ whose image $[b]$ lies in $B(G_{\Q_p},\mu_p^{-1})$. (The latter condition is equivalent to nonemptiness of the Newton stratum $N_b$ below.) Then $b$ gives rise to a $p$-divisible group $\Sigma_b$ with $G$-structure over $\Fpbar$ via Dieudonn\'e theory and the embedding of Shimura data above. We also obtain a Newton cocharacter $\nu_b$ from $b$. We may arrange that $\nu_b$ is dominant with respect to a suitable Borel subgroup of $G_{\Q_p}$ defined over $\Q_p$. For simplicity, assume that $\Sigma_b$ is defined over $k(\fkp)$ and that $[k(\fkp):\F_p]\nu_b$ is a cocharacter, not just a fractional cocharacter. (In practice, these assumptions are unnecessary since it is sufficient to have a finite extension of $k(\fkp)$ in the last sentence.)

Write $\rho$ for the half sum of all $B$-positive roots. Write $J_b$ for the $\Q_p$-group of self-quasi-isogenies of $\Sigma_b$ (preserving $G$-structure) over $\Fpbar$, and $J_b^{\tu{int}}$ for the subgroup of $J_b(\Q_p)$ consisting of automorphisms. Then $J_b^{\tu{int}}$ is an open compact subgroup of $J_b(\Qp)$. As a general fact, $J_b$ is an inner form of a $\Q_p$-rational Levi subgroup $M_b$ of $G_{\Q_p}$. We say that $b$ is \textbf{basic} if $\nu_b$ is a central in $G_{\Q_p}$, or equivalently if $M_b=G_{\Q_p}$ (namely if $J_b$ is an inner form of $G_{\Q_p}$).

The central leaf $C_b$ (resp. Newton stratum $N_b$) is the locus of $x\in \mS_{K_p}$ on which the geometric fibers of $\mA_{\ol x}[p^\infty]$ are isomorphic (resp.~isogenous) to $\Sigma_b$. By construction, $C_b$ and $N_b$ are stable under the $G(\A^{\infty,p})$-action on $\mS_{K_p}$. We also define the Igusa variety $\mathfrak{Ig}_b$ over $\mS_{K_p,k(\fkp)}$ to be the parameter space of isomorphisms between $\Sigma_b$ and $\mA[p^\infty]$. The obvious action of $J_b^{\tu{int}}$ on $\mathfrak{Ig}_b$ naturally extends to a $J_b(\Q_p)$-action. Below are some basic facts (\S\ref{sub:central-leaves}, \S\ref{sub:inf-level-Igusa}, and \S\ref{sub:completely-slope-divisible}). Put $q:=\#k(\fkp)$.
\begin{itemize}
\item[Fact 1.] $C_b$ is (formally) smooth over $k(\fkp)$ and closed in $N_b$,
\item[Fact 2.] $C_b$ is equidimensional of dimension $\langle 2\rho,\nu_b\rangle$,
\item[Fact 3.] $\mathfrak{Ig}_b$ is a pro-\'etale $J_b^{\tu{int}}$-torsor over the perfection of $C_b$,
\item[Fact 4.] In the completely slope divisible case, the $q^r$-th power Frobenius on $\mathfrak{Ig}_b$ coincides with the action of $\nu_b(q^r)\in Z_{J_b}(\Q_p)$ for sufficiently divisible $r$.
\end{itemize}
In particular $\dim \mathfrak{Ig}_b=\langle 2\rho,\nu_b\rangle$, and every connected component of $\mathfrak{Ig}_{b,\Fpbar}$ (resp.~$C_{b,\Fpbar}$) is irreducible. 

Our main theorem describes the connected components (= irreducible components) of Igusa varieties over $\Fpbar$ with a natural group action. Let us introduce some notation. Write $G(\Q_p)^{\tu{ab}}$ for the abelianization of $G(\Q_p)$ as a topological group. There is a canonical map $\zeta_b:J_b(\Q_p)\twoheadrightarrow G(\Q_p)^{\tu{ab}}$ coming from the fact that $J_b$ is an inner form of a Levi subgroup of $G_{\Q_p}$, cf.~\S\ref{sub:inf-level-Igusa} below. Each one-dimensional smooth representation $\pi_p$ of $G(\Q_p)$ factors through $G(\Q_p)^{\tu{ab}}$, giving rise to a one-dimensional representation $\pi_p\circ \zeta_b$ of $J_b(\Q_p)$.

\begin{thmx}[Theorem \ref{thm:H0(Ig)}]\label{thmx:A}
Assume that $b$ is non-basic with $[b]\in B(G_{\Q_p},\mu_p^{-1})$. Then there is a $G(\A^{\infty,p})\times J_b(\Q_p)$-module isomorphism 
$$ 
H^0(\IG_{b},\lql)\simeq \bigoplus_\pi \pi^{\infty,p}\otimes (\pi_p\circ \zeta_b) , 
$$
where the sum runs over one-dimensional automorphic representations $\pi=\pi^{\infty,p}\otimes \pi_p \otimes \pi_\infty$ of $G(\A)$ such that $\pi_\infty$ is trivial on the preimage of the neutral component $G^{\tu{ad}}(\R)^0$ in $G(\R)$.
\end{thmx}

Before we sketch the idea of proof, let us discuss two geometric applications.

\subsection{Application to irreducibility of Igusa towers and a generalization}\label{sub:intro-irreducibility}

In Hida theory of $p$-adic automorphic forms, an important role is played by Igusa varieties over the \emph{ordinary} Newton stratum, namely when the underlying $p$-divisible group is ordinary. In this case, Igusa varieties (and their natural extension to $p$-adic formal schemes) are usually referred to as Igusa towers. Recently Eischen and Mantovan \cite{EischenMantovan} developed Hida theory in the more general $\mu$-ordinary PEL-type situation, where Howe \cite{HoweCircle} (and its sequel) also shed new light on the role of Igusa varieties (\`a la Caraiani--Scholze). Igusa towers are also featured in Andreatta--Iovita--Pilloni's work \cite{AIPHilbert,AIPHalo} on overconvergent automorphic forms.

A key property of Igusa towers is irreducibility. This property has an application to the $q$-expansion principle for $p$-adic automorphic forms, which is a basic ingredient for the construction of $p$-adic $L$-functions. See p.96 in \cite{HidaPAF} for the relevant remark, and also refer to Thm.~3.3 (Igusa), 4.21 (Ribet), 6.4.3 (Faltings--Chai), and Cor.~8.17 (Hida) therein for the known cases (elliptic modular, Hilbert, Siegel, and PEL type A/C cases, respectively, all over the ordinary stratum) and further references. Irreducibility in the $\mu$-ordinary case of PEL type A was proven in \cite{EischenMantovan}. Such a result was obtained for Igusa varieties of a specific PEL type A by Boyer \cite{BoyerIgusa} without assuming $\mu$-ordinariness.

There are various methods to show the irreducibility as explained in \cite{ChaiMethods} and the introduction of \cite{HidaIrreducibility2}, e.g., by using the automorphism group of the function fields of Shimura varieties in characteristic 0 or by showing that the family of abelian varieties has large monodromy. As an application of Theorem \ref{thmx:A}, we obtain an entirely different representation-theoretic proof and also a natural generalization from the $\mu$-ordinary case to the general non-basic case (and from the PEL case to the case of Hodge type). In the non-$\mu$-ordinary case, Igusa varieties lie over a central leaf rather than an entire Newton stratum, but our method is insensitive to such a distinction.

Write $J_b(\Q_p)':=\ker(\zeta_b:J_b(\Q_p)\ra G(\Q_p)^{\tu{ab}})$. Our result is as follows.

\begin{thmx}\label{thmx:irreducibility} Assume that $b$ is non-basic. The stabilizer subgroup in $J_b(\Q_p)$ of each connected component of $\mathfrak{Ig}_b$ is equal to $J_b(\Q_p)'$. The preimage of each connected component of $C_b$ along $\mathfrak{Ig}_b\ra C_b$ is the disjoint union of $J_b^{\tu{int}}\cap J_b(\Q_p)'$-torsors.
\end{thmx}

Roughly speaking, the stabilizer subgroup cannot be larger than $ J_b(\Q_p)'$, and this should be thought of as saying that Igusa varieties are at least as reducible as Shimura varieties. The point of the theorem is that, conversely, the stabilizer is as large as possible under the given constraint; so Igusa varieties are ``irreducible'' in the sense that they are no more reducible than Shimura varieties, cf.~Corollary \ref{cor:irreducibility} below. The proof is almost immediate from the $J_b(\Q_p)$-action on $H^0$ described in Theorem \ref{thmx:A}. See  \S\ref{sub:main-to-irreducibility} below for further details.

\subsection{Application to the discrete Hecke orbit conjecture}\label{sub:intro-disc-HO}

In 1995 \cite[\S15]{OortNotes} (also see \cite[Problem 18]{EMO01}), Oort proposed the Hecke Orbit (HO) conjecture that the prime-to-$p$ Hecke orbit of a point should be Zariski dense in the central leaf containing it, if the point lies outside the basic Newton stratum. The reader is referred to \cite{ChaiOortHO} for an excellent survey of the HO conjecture with updates. Oort drew analogy with the Andr\'e--Oort conjecture for a Shimura variety in characteristic zero, which asserts that the irreducible components of the Zariski closure of a set of special points are special subvarieties. (See \cite{TsimermanAOconj,PilaShankarTsimerman} and references therein for recent results on the Andr\'e--Oort conjecture.) A common feature is that a set of points with an extraordinary structure (being a prime-to-$p$ Hecke orbit or special points) is Zariski dense in a distinguished class of subvarieties. We can also compare the HO conjecture with the Hecke equidistribution theorems for locally symmetric spaces in characteristic zero \cite{ClozelOhUllmo,EskinOh}, stating roughly that the Hecke orbit of an arbitrary point is equidistributed in the locally symmetric space in a suitable sense. (In particular the Hecke orbit is dense in the entire space even for the analytic topology, to be contrasted with the phenomenon in characteristic $p$.) It is also worth noting works to investigate Hecke orbits for the $p$-adic topology  \cite{GorenKassaei,HerreroMenaresRiveraLetelier}.

Chai and Oort verified the HO conjecture for Siegel modular varieties \cite[Thm.~3.4]{ChaiICM} (details to appear in a monograph), in particular the irreducibility of leaves  \cite{Chai-HO,CO11}. The conjecture is also known for Hilbert modular varieties \cite[Thm.~3.5]{ChaiICM} due to Chai and Yu. (Also see \cite{YuChaiOort}.) The HO conjecture has seen several new results in recent years. Shankar proved the conjecture for Deligne's ``strange models'' (in the sense of \cite[\S6]{Del-Travaux}) in an unpublished preprint. Zhou \cite{ZhouMotivic} settled the HO conjecture in the ordinary locus of some quaternionic Shimura varieties along the way to realize a geometric level raising between Hilbert modular forms. Maulik--Shankar--Tang \cite{MaulikShankarTang} proved the HO conjecture in the ordinary locus of GSpin Shimura varieties. Xiao \cite{XiaoHecke} proved partial results on the HO conjecture in the case of PEL type A and C.

Chai \cite{Chai-HO,ChaiICM} proposed the strategy to divide the HO conjecture into two parts, that is, the discrete part (HO$_{\disc}$) and the continuous part (HO$_{\tu{cont}}$), corresponding to global and local geometry, respectively. In a nutshell, (HO$_{\disc}$) asserts that the prime-to-$p$ Hecke action is transitive on the set of irreducible components of each central leaf. Then (HO$_{\tu{cont}}$) is designed to tell us that the closure of each prime-to-$p$ Hecke orbit has the same dimension as the ambient central leaf, so that (HO$_{\disc}$) and (HO$_{\tu{cont}}$) together imply the HO conjecture. 

We deduce the following result on  \textup{(HO$_{\disc}$)} from Theorem \ref{thmx:A}.
 (See \S\ref{sub:main-to-discHO} for details.)

\begin{thmx}\label{thmx:B}
For Hodge-type Shimura varieties with hyperspecial level at~$p$, \textup{(HO$_{\disc}$)} is true for every central leaf outside the basic Newton stratum.
\end{thmx}

To our knowledge, this is the first general theorem on \textup{(HO$_{\disc}$)}. Let us remark on the proof. Since \textup{(HO$_{\disc}$)} means transitivity of the $G(\A^{\infty,p})$-action on $\pi_0(C_{b})$, it is equivalent to the multiplicity one property of the trivial $G(\A^{\infty,p})$-representation in $H^0(C_{b},\lql)=H^0(\mathfrak{Ig},\lql)^{J_b^{\tu{int}}}$. 
To prove Theorem~\ref{thmx:B}, it is thus enough to observe that if $\pi^{\infty,p}$ is trivial then $\pi_p$ and $\pi_\infty$ must be trivial as well in the formula of Theorem \ref{thmx:A}. This is an easy consequence of the weak approximation that $G(\Q)$ is dense in $G(\Q_p)\times G(\R)$. (The same approximation holds more generally, at least if $G$ splits over an unramified extension.)

We also consider the following strengthening of \textup{(HO$_{\disc}$)}:

\vspace{.1in}

\noindent \textup{(HO$^+_{\disc}$)} The map $\pi_0(C_b)\ra \pi_0(\Sh_{K_p})$ induced by the immersion $C_b\ra \Sh_{K_p}$ is a bijection.

\vspace{.1in}

\noindent This is known as ``irreducibility of central leaves'', as it means that $C_b$ is irreducible in every component of $\Sh_{K_p}$.
Since $G(\A^{\infty,p})$ is known to act transitively on $\pi_0(\Sh_{K_p})$, e.g., by weak approximation, and since $\pi_0(C_b)\ra \pi_0(\Sh_{K_p})$ is $G(\A^{\infty,p})$-equivariant, it is clear that \textup{(HO$^+_{\disc}$)} implies \textup{(HO$_{\disc}$)}. 

We prove purely local criteria for \textup{(HO$^+_{\disc}$)}, either in terms of groups at $p$ or in terms of the stabilizers of points on some affine Deligne--Lusztig varieties (Theorem \ref{thm:HO-criterion}). Using these criteria, we deduce \textup{(HO$^+_{\disc}$)} in the $\mu$-ordinary case, but obtain a counterexample in general with the help of Rong Zhou.See \S\ref{sub:main-to-discHO} below for further details.
(In the earlier version of this paper \texttt{arXiv:2102.10690v1}, we incorrectly asserted that \textup{(HO$^+_{\disc}$)} was true in general. The mistake occurred during the initial reduction in the proof of Lemma 8.1.1, where changing $b$ to a $\sigma$-conjugate element cannot be justified; we thank van Hoften for pointing it out to us.)

\subsection{Some details on the proof of Theorem \ref{thmx:A}}

Changing $\Sigma_b$ by a quasi-isogeny, as this does not affect $\IG_b$ up to isomorphism, we may assume that $\Sigma_b$ is completely slope divisible and defined over a finite field. Then $\IG_b$ can be written, up to perfection, as the projective limit of smooth varieties of finite type defined over $\F_{p^r}$ for a sufficiently divisible $r\in\Z_{>0}$. (In the main text, we use $\Ig_b$ to denote the version without perfection.) This allows us to apply a Lefschetz trace formula technique to compute the cohomology of $\IG_b$ at a finite level. Via Poincar\'e duality, Theorem \ref{thmx:A} may be rephrased in terms of the top degree compact-support cohomology $H_c^{\langle 4\rho,\nu_b\rangle}(\IG_b,\lql)$, which we may access by the Lang--Weil estimate.

Adapting the Langlands--Kottwitz method to Igusa varieties, as worked out in \cite{ShinIgusa} and Mack-Crane's thesis \cite{MackCrane}, one obtains a formula of the form
$$
\Tr(\phi^{\infty,p} \phi_p\times \Frob_{p^r}^j | H_c (\IG_b,\lql))=(\tu{geometric expansion}),\qquad j\in \Z_{\gg 1},
$$
where $\phi^{\infty,p}\phi_p\in \cH(G(\A^{\infty,p})\times J_b(\Q_p))$ and $j\in \Z_{\gg 1}$. In fact, one can show that the $\Frob_{p^r}$-action on $\IG_b$ is represented by the action of a central element of $J_b(\Q_p)$. Thereby $\phi_p\times \Frob_{p^r}^j$ in \eqref{eq:tr(Ig)-intro} may be replaced with a translate $\phi_p^{(j)}\in \cH(J_b(\Q_p))$ of $\phi_p$ by a central element. The geometric expansion is a linear combination of orbital integrals of $\phi^{\infty,p} \phi^{(j)}_p$ on $G(\A^{\infty,p})\times J_b(\Q_p)$ over a certain set of conjugacy classes. The stabilized formula takes the form
\begin{equation}\label{eq:tr(Ig)-intro}
\Tr(\phi^{\infty,p} \phi_p^{(j)} | H_c (\IG_b,\lql))=\sum_{\fke} (\tu{constant})\cdot ST^{\fke}_{\tu{ell}}(f^{\fke,p} f_p^{\fke,(j)}),\qquad j\in \Z_{\gg 1},
\end{equation}
where the sum runs over endoscopic data $\fke$ for $G$ (\S\ref{sub:prelim-endoscopy}), and $f^{\fke,p} f_p^{\fke,(j)}$ is a suitable function on the corresponding endoscopic group $G^\fke$. By $ST^{\fke}_{\tu{ell}}$ we mean the elliptic part of the stable trace formula for $G^\fke$. The most nontrivial point in the stabilization is the ``transfer'' at $p$. Indeed, as $G^\fke$ is not an endoscopic group of $J_b$, this requires a special construction as detailed in \S\ref{sec:Jacquet-endoscopy}.

Ideally we would turn the right hand side of \eqref{eq:tr(Ig)-intro} into a spectral expansion and determine not only $H_c^{\langle 4\rho,\nu_b\rangle}(\IG_b,\lql)$ but $H_c (\IG_b,\lql)$ in the Grothendieck group of $G(\A^{\infty,p})\times J_b(\Q_p)$-representations. This is the long-term goal stated earlier. On the analogous problem for Shimura varieties, a road map has been laid out in \cite{KottwitzAnnArbor}, which can be mimicked for Igusa varieties to some extent. However there are serious obstacles: (1) An endoscopic classification for most reductive groups is out of reach; exactly the same issue occurs for Shimura varieties as well. (2) The geometric side (stable elliptic terms) is very difficult to compare with the spectral side. One could imagine making the comparison more tractable by passing from $H_c$ to intersection cohomology, following the strategy for Shimura varieties to ``fill in'' the stable non-elliptic terms, but no theory of compactification is available for Igusa varieties to allow it. (Franke's formula for $H_c$ of locally symmetric spaces \cite{Franke} suggests that one should expect a similarly complicated answer for $H_c$ of Igusa varieties.)

Our goal is to extract spectral information on $H_c^{\langle 4\rho,\nu_b\rangle}(\IG_b,\lql)$ from the leading terms in \eqref{eq:tr(Ig)-intro} in the variable $j$ via the Lang--Weil estimate. Thus we can get away with less by proving equalities up to error terms of lower order. To bypass (1) and (2), a key is to show that (stable) non-elliptic terms as well as endoscopic (a.k.a.~unstable) terms have slower growth in $j$ than the (stable) elliptic terms. This is the technical heart of our paper taking up \S\ref{sec:asymptotic-analysis}. Let us provide more details.

The basic strategy is an induction on the semisimple rank, based on our observation that some key property of the function $f_p^{\fke,(j)}$ is replicated after taking an endoscopic transfer or a constant term. (For instance, we need to pass along the Newton cocharacter through the inductive steps.) So we want to prove a bound on the trace formula for a quasi-split group over $\Q$, with a test function $f^{p} f_p^{(j)}$ satisfying such a property. The desired bound partly comes from a root-theoretic computation, involving a curious interaction between $p$ and $\infty$ such as ``evaluating'' the Newton cocharacter (coming from $p$) at the infinite place (Lemma \ref{lem:lambda(nu)}). The most interesting component in this part of the argument is

\smallskip

 \centerline{(*)~~a spectral expansion of $T_{\tu{ell}}$, the elliptic part of the trace formula.}

\smallskip

 \noindent
 The problem is actually about $ ST^{\fke}_{\tu{ell}}$ in \eqref{eq:tr(Ig)-intro}, but we can replace $ ST^{\fke}_{\tu{ell}}$ with $T_{\tu{ell}}$ for $G^{\fke}$ once the difference is shown to have lower order of growth. 
 The archimedean test function is stable cuspidal in our setting, so we have Arthur's simple trace formula \cite{ArthurL2} of the following shape:
\begin{equation}\label{eq:IntroExplainStrategy}
T_{\tu{disc}}(f^{\fke,p} f_p^{\fke,(j)}) = T_{\tu{ell}}(f^{\fke,p} f_p^{\fke,(j)}) + (\tu{geometric terms on proper Levi subgroups}).
\end{equation}
The proper Levi terms at finite places look similar to the elliptic part of the trace formula for proper Levi subgroups, but a complicated behavior is seen at the infinite place due to stable discrete series characters along non-elliptic maximal tori of the ambient group. On different open Weyl chambers, we have different character formulas in terms of finite dimensional characters of the Levi subgroup, so this quickly spirals out of control in the induction. Adapting an idea of Laumon \cite{Lau97} from the  non-invariant trace formula, we overcome the difficulty by imposing a regularity condition on the test function at an auxiliary prime $q$ ($\neq p$) and show that the $\Q$-conjugacy classes with nonzero contributions land in a single Weyl chamber. Then a finite dimensional character of a Levi subgroup is itself a stable discrete series character of the same Levi subgroup along elliptic maximal tori of the Levi, so that the inductive argument can continue. (No information is lost by the auxiliary hypothesis at $q$, cf.~\S\ref{sub:completion-of-proof} below.) This technique should prove useful for investigation of compactly supported cohomology of Igusa varieties and Shimura varieties alike.

Returning to our problem, the above argument turns \eqref{eq:tr(Ig)-intro} into
$$ 
\Tr(\phi^{\infty,p} \phi^{(j)}_p | H_c (\IG_b,\lql))=\sum_{\pi^*}m(\pi^*)\Tr(f^{*,p} f_p^{*,(j)}) + (\tu{error terms}),
$$
where $f^{*,p} f_p^{*,(j)}$ is the test function on the quasi-split inner form $G^*$ of $G$ (\ie when $G^\fke=G^*$), and the sum runs over discrete automorphic representations of $G^*(\A)$. At this point, we apply a trace identity. Let $\phi^{*,(j)}_p$ denote a transfer of $\phi^{(j)}_p$ from $J_b$ to its quasi-split inner form $M_b$. For each irreducible smooth representation $\pi^*$ of $G^*(\Q_p)$, we have (Lemma \ref{lem:O-tr-nu-ascent})
$$
\Tr \pi_p^*(f_p^{*,(j)}) = \Tr J(\pi_p^*)(\phi^{*,(j)}_p),
$$
where $J$ is the normalized Jacquet module relative to the parabolic subgroup determined by $\nu_b$ whose Levi component is $M_b$. Since $b$ is non-basic, $M_b$ is a proper Levi subgroup. Moreover the translation $(j)$ is given by a central element satisfying a positivity condition with respect to $\nu_b$. In these circumstances, we make a crucial use of an estimate due to Casselman and Howe--Moore (\S\ref{sub:Jacquet-estimate}), showing that $J(\pi_p^*)(\phi^{*,(j)}_p)$ has the highest growth if and only if $\dim \pi^*_p=1$. A strong approximation argument (\S\ref{sub:one-dim}) promotes this to the condition that $\dim \pi^*=1$, under a group-theoretic condition guaranteed in our setting. Moreover, it is not hard to transfer one-dimensional representations from $M_b(\Q_p)$ to $J_b(\Q_p)$ compatibly with the transfer of functions (\S\ref{sub:transfer-one-dim}). We complete the proof of Theorem \ref{thmx:A} by putting this final piece of the puzzle.

\subsection{A remark on the non-hyperspecial case}

This paper focuses on the case of hyperspecial level at $p$ mainly because the trace formula for Igusa varieties \cite{MackCrane} is available only in that case. Once the trace formula becomes available for Shimura varieties with parahoric level at $p$ (cf.~\S\ref{sub:vHX} below), the methods and results of this paper should extend to that case. To avoid group-theoretic subtleties (e.g., Remark \ref{rem:1-dim_reps} below),  assume that $G$ is quasi-split over $\Q_p$.
Then Theorems \ref{thmx:A} and \ref{thmx:irreducibility} are expected to remain true (with a modified definition of $J_b^{\tu{int}}$). As for Theorem~\ref{thmx:B}, a crucial group-theoretic ingredient is that the diagonal embedding $G(\Q)\ra G(\Q_p)\times G(\R)$ has dense image (weak approximation). If $G$ does not split over an unramified extension of $\Q_p$, then the weak approximation can be false, in which case our argument does not apply. In fact, Oki's example \cite{Oki} suggests that the analogue of Theorem \ref{thmx:B} is false in general, since the prime-to-$p$ Hecke action is not even transitive on the set of connected components of the underlying Shimura variety.
 
\subsection{Work of van Hoften and Xiao}\label{sub:vHX}

Pol van Hoften and Luciena Xiao Xiao \cite{vanHoftenLR,vanHoftenXiao} prove the irreducibility of Igusa varieties (but not Theorems \ref{thmx:A} and \ref{thmx:B} of our paper) and give a counterexample to \textup{(HO$^+_{\disc}$)}.\footnote{The counterexample in \cite[\S6.3]{vanHoftenXiao} is about \textup{(HO$^+_{\disc}$)}, but not \textup{(HO$_{\disc}$)}, cf.~\S\ref{sub:main-to-discHO} of this paper. Note that the maps in Thm.~6.2.1 and Cor.~6.2.2 therein are not asserted to be equivariant for the prime-to-$p$ Hecke actions. In fact, our Theorem \ref{thmx:B} suggesets that those maps should not be equivariant in general.}
Their method is more geometric and totally different from ours in that no use is made of automorphic forms. Further goals in their work and ours are disparate. For instance, \cite{vanHoftenLR} proves new results on the stratification of Shimura varieties and the Langlands--Rapoport conjecture in the parahoric case, whereas our work is a stepping stone for understanding the cohomology of Igusa varieties in all degrees. The two threads could have a future intersection though, as the Langlands--Rapoport conjecture in the parahoric case ought to be an important ingredient for deriving the analogue for Igusa varieties in that case, extending \cite{MackCrane} from the hyperspecial case. 

\subsection{A guide for the reader} The bare-bones structure of our argument is as follows.

{
\footnotesize
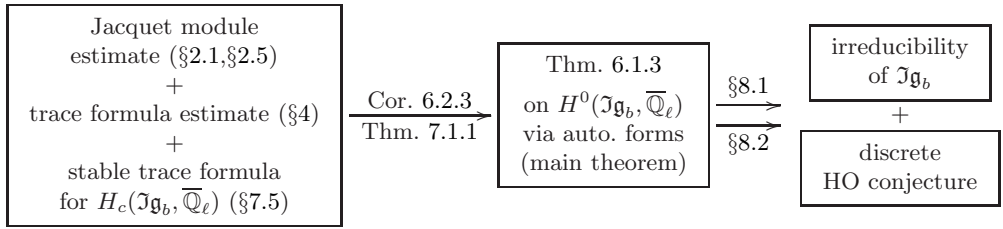
$$\xymatrix{
\boxed{
\begin{array}{c}
\mbox{Jacquet module}\\
\mbox{estimate (\S\ref{sub:Jacquet-estimate},\S\ref{sub:one-dim})}\\
+ \\
\mbox{trace formula estimate (\S\ref{sec:asymptotic-analysis})}\\
+\\
\mbox{stable trace formula}\\
\mbox{for~~$H_c(\IG_b,\lql)$~~(\S\ref{sub:Igusa-STF})}
\end{array}
}
\ar[rr]^-{\mbox{Cor.~\ref{cor:red-to-completely-slope-div}}}_-{\mbox{Thm.~\ref{thm:Hc(Ig)}}} &&
\boxed{\begin{array}{c}
\mbox{Thm.~\ref{thm:H0(Ig)}}\vspace{.05in}\\
\mbox{on}~H^0(\IG_b,\lql)\\
\mbox{via auto.~forms}\\
\mbox{(main theorem)}
\end{array}
}
\ar@<-1ex>[r]_-{\mbox{\S\ref{sub:main-to-discHO}}} \ar@<1ex>[r]^-{\mbox{\S\ref{sub:main-to-irreducibility}}}
&
{
\begin{array}{c}
\boxed{\begin{array}{c}
\mbox{irreducibility}\\
\mbox{of}~\IG_b
\end{array}}
\\
+\\
\boxed{\begin{array}{c}
\mbox{discrete}\\
\mbox{HO conjecture}
\end{array}}
\end{array}
}
}
$$

}

\vspace{.05in}

On a first reading, we suggest that all complexities arising from central characters and $z$-extensions should be skipped, e.g., by assuming that all central character data are trivial. In fact this should be the case in many examples. The central character datum is always trivial on the level of $G$ appearing in the Hodge-type datum, but we allow it to be nontrivial mainly because we do not know whether $\cH$ in the endoscopic datum (\S\ref{sub:prelim-endoscopy}) can always be chosen to be an $L$-group. Another good idea is to start reading in \S\ref{sec:Shimura}, especially if one's main interests lie in geometry, referring to the earlier sections only as needed and taking the results there for granted.

Sections \ref{sec:preliminaries} and \ref{sec:Jacquet-endoscopy} consist of mostly background materials in local harmonic analysis and representation theory. Though we claim little originality, there may be some novelty in the way we organize and present them. Some statements would be of independent interest. Section \ref{sec:asymptotic-analysis} is perhaps the most technical as this is where the main trace formula estimates are obtained. As such, most readers may want to take the results in \S\ref{sub:main-estiamte} on faith and proceed, returning to them as needed.

Sections \ref{sec:Shimura} and \ref{sec:Igusa} introduce the main geometric players, namely Shimura varieties, central leaves, and Igusa varieties. Except for \S\ref{sub:components-char0}, we are always in the Hodge-type case with hyperspecial level at $p$. Our main theorem on Igusa varieties is stated in \S\ref{sub:inf-level-Igusa}. After reduction steps in \S\ref{sub:completely-slope-divisible}--\S\ref{sub:Hc} and some recollection of the trace formula setup up to \S\ref{sub:Igusa-STF}, the proof of the theorem is completed in \S\ref{sub:completion-of-proof}. Lastly Section \ref{sec:apps-to-geometry} is devoted to the main geometric applications on irreducibility of Igusa varieties and a local criterion for the discrete Hecke orbit conjecture.

\subsection{Notation}\label{sub:Notation}

\begin{itemize}
\item The trivial character (of the group that is clear from the context) is denoted by $\mathbf{1}$. 
\item If $T$ is a torus over a field $k$ with algebraic closure $\ol k$, $X_*(T):=\Hom_{\ol k}(T,\G_m)$ and  $X^*(T):=\Hom_{\ol k}(\G_m,T)$. When $R$ is a $\Z$-algebra, we write $X_*(T)_R:=X_*(T)\otimes_{\Z} R$ and $X^*(T)_R:=X^*(T)\otimes_{\Z} R$. 
\item $\D:=\varprojlim \G_m$ is the protorus (over an arbitrary base), where the transition maps are the $n$-th power maps. 
\item $\breve \Z_p:=W(\Fpbar)$, $\breve \Q_p:=\Frac\, \breve \Z_p$, and $\sigma\in \Aut(\breve \Q_p)$ is the arithmetic Frobenius. By $\Z_p^{\tu{ur}}$ (resp.~$\Q_p^{\tu{ur}}$) we mean the subring of elements in $\breve \Z_p$ (resp.~$\breve \Q_p$) which are algebraic over $\Q_p$.
\item $\mathscr P(S)$ is the power set of a set $S$.
\item If $H$ is an algebraic group over a field $k$, we write $H^0 \subset H$
for its neutral component.
\end{itemize}

Let $G$ be a connected reductive group over a field $k$ of characteristic 0.
\begin{itemize}
\item If $k$ is a finite extension of $k_0$, then $\Res_{k/k_0}G$ denotes the restriction of scalars group.
\item If $k'$ is an extension field of $k$ then $G_{k'}:=G\times_{\Spec k} \Spec k'$.
\item  $G_{\der}$ is the derived subgroup, $\varrho:G_{\tu{sc}}\ra G_{\der}\subset G$ the simply connected cover, $Z_{G}$ the center (we also write $Z(G)$), $G^{\ad}:=G/Z_{G}$   the adjoint group, and $G^{\ab}:=G/G_{\der}$ the maximal commutative quotient. Write $A_G\subset Z_{G}$ for the maximal split subtorus over $k$.
\item $G(k)_{?}$ is the set of semisimple (resp. regular semisimple, resp. strongly regular) elements in $G(k)$ for $?=\tu{ss}$ (resp. $\tu{reg}$, resp. $\tu{sr}$). We put $T(k)_{?}:=T(k)\cap G(k)_{?}$ for $?\in \{\tu{reg},\tu{sr}\}$.
\item If $k$ is a local field and $G$ a reductive group over $k$, write $\cI(G(k))$ and $\cS(G(k))$ for the spaces of invariant and stable distributions on $G(k)$. (For more details, see \S~\ref{sub:local-Hecke}). By $\Irr(G(k))$ we mean the set of isomorphism classes of irreducible admissible representations of $G(k)$.
\item When $k=\Q_p$, two elements $\delta, \delta' \in G(\breve \Q_p)$ are \emph{$(G(\breve\Q_p),\sigma)$-conjugate} (resp. $(G(\breve\Z_p),\sigma)$\textit{-conjugate}) if there exists a $g \in G(\breve \Q_p)$ (resp. $g \in G(\breve \Z_p)$) such that $\delta' = \sigma(g) \delta g\inv$.
\end{itemize}

Let $T$ (resp. $S$) be a maximal torus (resp. maximal split torus) of $G$ over $k$ with $T \supset S$. Let $M_0$ be a minimal $k$-rational Levi subgroup containing $T$.
 
\begin{itemize}
\item $\Phi(T,G)$ is the set of absolute roots, $\Phi(S,G)=\Phi_k(S,G)$ the set of $k$-rational roots.
\item $\overline \Omega^G = \Omega(T,G)$ for the Weyl group over $\ol k$, and $\Omega^G_k = \Omega(S,G)$ for the $k$-rational Weyl group. We often omit $k$ from $\Phi_k(S,G)$ and $\Omega^G_k$ when it is clear from the context.
\item $\cL(G)$ or $\cL_k(G)$ is the set of all $k$-rational Levi subgroups of $G$ containing $M_0$. Write $\cL^<(G):=\cL(G)\backslash \{G\}$.
\end{itemize}

\begin{lemma}\label{lem:Levi-of-simply-connected}
If $G_{\der}$ is simply connected then every $k$-rational Levi subgroup of $G$ has simply connected derived subgroup.
\end{lemma}

\begin{proof}
This can be checked after base change to $\ol k$, so assume $k=\ol k$. For every maximal torus $T \subset G$, the cocharacter lattice $X_*(T)$ modulo the coroot lattice is torsion free by hypothesis. Thus $X_*(T)$ modulo the lattice generated by an arbitrary subset of simple coroots is torsion free, implying that every Levi subgroup of $G$ has simply connected derived subgroup.  
\end{proof}

\subsection*{Acknowledgments}

AK is partially supported by a NWO VENI grant and a VIDI grant. SWS is partially supported by NSF grant DMS-1802039, NSF RTG grant DMS-1646385, and a Miller Professorship. AK and SWS are grateful to Erez Lapid, Gordan Savin, and Maarten Solleveld for pointing them in the right direction regarding \S\ref{sub:Jacquet-estimate}. We thank Xuhua He, Pol van Hoften, and Rong Zhou for discussions about \S\ref{sub:main-to-discHO}, and especially Zhou for providing us with Example \ref{ex:HO+false} below.

\section{Preliminaries in representation theory and endoscopy}\label{sec:preliminaries}

\subsection{Estimates for Jacquet modules of unitary representations}\label{sub:Jacquet-estimate}

Here we recall some facts from work of Howe--Moore \cite{HoweMoore} and Casselman \cite{CasselmanNotes} in order to bound the absolute value of central characters in the Jacquet modules of unitary representations of $p$-adic reductive groups.

We consider the following setup and notation.
\begin{itemize}
\item Let $F$ be a non-archimedean local field of characteristic $0$. We write $\val_F$, $\cO_F$, $k$, $q$, $\varpi_F$ respectively for the normalized valuation of $F$, the ring of integers of $F$, the residue field of $F$, the cardinality of $k$, and an uniformizer of $F$ so that $\val_F(\varpi_F) = 1$,
\item $G$ is a connected reductive group over $F$ with center $Z=Z_G$,
\item $\Rep(G)$ is the category of smooth representations of $G(F)$,
\item $P = MN$ is a Levi decomposition of an $F$-rational proper parabolic subgroup of $G$,
\item $A_M$ is the maximal $F$-split torus in the center of $M$,
\item $\Delta$ is the set of roots of $A_M$ in $N$,
\item $A_P^-:=\{ x\in A_M(F): |\alpha(x)|\le 1 ,~\forall \alpha\in \Delta\}$,
\item $A_P^{--}:=\{ x\in A_M(F): |\alpha(x)|< 1 ,~\forall \alpha\in \Delta \}$,
\item $\delta_P: M(F)\ra \R^\times_{>0}$ is the modulus character given by $\delta_P(m):=|\det(\textup{Ad}(m)|\Lie N(F))|$.
\item $J_P:\Rep(G)\ra \Rep(M)$ is the normalized Jacquet module functor, so $J_P(\pi) = \pi_N \otimes \delta_P^{-1/2}$ with $\pi_N$ denoting the $N(F)$-coinvariants of $\pi$,
\item $I_P^G:\Rep(M)\ra \Rep(G)$ is the normalized parabolic induction functor, sending $\pi_M$ to the smooth induction of $\pi_M\otimes \delta_P^{1/2}$ from $P(F)$ to $G(F)$.
\item When $R\in \Rep(M)$ has finite length, write $\Exp(R)$ for the set of $A_M(F)$-characters appearing as central characters of irreducible subquotients of $R$.
\end{itemize}

\begin{lemma}\label{lem:normal-in-sc-groups}
If $G$ is simply connected, $F$-simple, and $F$-isotropic, then every normal subgroup of $G(F)$ is either $G(F)$ itself or contained in $Z(F)$.
\end{lemma}

\begin{proof}
A normal subgroup $N$ of $G(F)$ not contained in $Z(F)$ is open of finite index in $G(F)$ by \cite[Prop.~3.17]{PlatonovRapinchuk} since $G$ is $F$-simple. Since $G(F)$ is $F$-isotropic and simply connected, $G(F)$ is generated by the $F$-points of the unipotent radicals of $F$-rational parabolic subgroups. Thus, by Tit's theorem proven in \cite{PrasadTits}, every open proper subgroup of $G(F)$ is compact. On the other hand, $N$ is easily seen to be non-compact by considering the adjoint action of a maximal $F$-split torus on a root subgroup.\footnote{For instance, see the proof of Proposition 3.9 in \url{http://virtualmath1.stanford.edu/~conrad/JLseminar/Notes/L2.pdf} for details.} Therefore $N=G(F)$.
\end{proof}

\begin{proposition}[Howe--Moore]\label{prop:Howe-Moore}
Assume that $G_{\tu{sc}}$ is $F$-simple. Let $\pi$ be an infinite dimensional irreducible unitary representation of $G(F)$. Then there exists an integer $2\le k<\infty$ such that every matrix coefficient of $\pi$ belongs to $L^k(G(F)/Z(F))$.
\end{proposition}

\begin{proof}
This follows from the explanation on pp.74--75 of \cite{HoweMoore} below Theorem 6.1, once we verify the following claim: if $\pi(g)$ is a scalar operator for $g\in G(F)$ then $g\in Z(F)$. Taking a $z$-extension of $G$, we reduce to the case when $G_{\tu{der}}$ is simply connected. Pulling back $\pi$ via the multiplication map $Z(F)\times G_{\tu{der}}(F)\ra G(F)$ and passing to one of the finitely many constituents (cf.~\cite[Lem.~6.2]{Xu-lifting}) which is infinite-dimensional, we may assume that $G$ is itself $F$-simple and simply connected. Now $Z'$ be the group of $g\in G(F)$ such that $\pi(g)$ is a scalar. Then $Z'$ is a normal subgroup of $G(F)$, and $Z'\neq G(F)$ since $\dim \pi=\infty$. Therefore $Z'\subset Z(F)$ by Lemma \ref{lem:normal-in-sc-groups}, proving the claim.
\end{proof}

\begin{proposition}[Casselman]\label{prop:BoundExp}
Let $\pi$ be an irreducible unitary representation of $G(F)$. For every $\omega \in \textup{Exp}(\pi_N)$ and every $a\in  A_P^{-}$, we have the inequality
\begin{equation}\label{eq:CasselmanBound}
|\omega(a)| \leq 1.
\end{equation}
If $G_{\tu{sc}}$ is $F$-simple and $a\in A_P^{--}$, then the equality holds if and only if $\dim \pi<\infty$.
\end{proposition}

\begin{proof}
The inequality \eqref{eq:CasselmanBound} follows from the obvious extension of \cite[\S4.4]{CasselmanNotes} (where $p<\infty$ is assumed) to cover the case $p=\infty$. (For instance, \cite[Lem.~4.4.3, Prop.~4.4.4]{CasselmanNotes} have the analogues for $p=\infty$, with ``bounded'' in place of ``summable'' and ``$|\chi(x)|\le 1$'' in place of ``$|\chi(x)|<1$''.)

As for the last assertion, suppose that $\dim \pi=\infty$. In the notation of \cite[\S2.5]{CasselmanNotes}, Proposition \ref{prop:Howe-Moore} tells us that the matrix coefficient is $L^k$, \ie $|c_{v,\tilde v}|^k$ is integrable modulo center for some $2\le k<\infty$. Applying \cite[Cor.~4.4.5]{CasselmanNotes} to $p=k$, $F=c_{v,\tilde v}$ and $a\in A_P^{--}$, we obtain that $|\omega(a)\delta_P^{-1/k}(a)|<1$. Therefore $|\omega(a)|<1$. For the converse, suppose that $\dim\pi<\infty$. Then $\ker \pi$ is an open subgroup of $G(F)$. As the open subgroup $N(F)\cap \ker \pi$ of the unipotent subgroup $N(F)$ acts trivially on $\pi$, we see that $N(F)$ itself acts trivially on $\pi$. (Use conjugation by $A_M(F)$.) Therefore $\textup{Exp}(\pi_N)$ consists of the central character $\omega$ of $\pi$ (restricted to $M(F)$) only, which is unitary. In particular $|\omega(a)|=1$ for all $a\in A_P^{--}$
\end{proof}

\begin{remark}
Proposition~\ref{prop:BoundExp} is sharp in general. (Some strengthening is possible under a hypothesis, cf.~\cite{OhDuke}.) For example, consider $G = \GL_2(F)$ with $P$ (resp.~$N$) consisting of upper triangular (resp.~upper triangular unipotent) matrices. The complementary series representations $\pi_\epsilon = I_P^G(|\cdot|^{\epsilon}, |\cdot|^{- \epsilon})$ with $\epsilon \in \R$ with $0 < \epsilon < 1/2$ are irreducible and unitary. We have
\begin{align*}
(\pi_\epsilon)_N = J_N(\pi_\epsilon)\otimes\delta_P^{1/2}  &=  \delta_P^{1/2} \otimes \lhk ( |\cdot|^{\epsilon}, |\cdot|^{-\epsilon} ) \oplus  ( |\cdot|^{-\epsilon}, |\cdot|^{\epsilon} ) \rhk  \cr
&=  (|\cdot|^{\epsilon+\frac{1}{2}}, |\cdot|^{-\epsilon-\frac{1}{2}}) \oplus (|\cdot|^{-\epsilon+\frac{1}{2}}, |\cdot|^{\epsilon-\frac{1}{2}} ).
\end{align*}
So in this case, $\Exp((\pi_\epsilon)_N)$ contains the character $\omega=(|\cdot|^{-\epsilon+\frac{1}{2}}, |\cdot|^{\epsilon-\frac{1}{2}})$ of $\Q_p^\times\times \Q_p^\times$. Then $a=\vierkant p001\in A_P^{--}$. We get $\omega(a) = p^{\epsilon-\frac{1}{2}}$ which gets arbitrarily close to $1$ as $\epsilon$ tends to $1/2$. 
\end{remark}

\begin{lemma}\label{lem:fin-implies-one}
Assume that $G^{\textup{ad}}$ has no $F$-anisotropic factor. Then every irreducible smooth representation of $G(F)$ is either one-dimensional or infinite-dimensional.
\end{lemma}

\begin{proof}
We may assume that $G_{\tu{der}}$ is simply connected via $z$-extensions. Suppose that $\pi$ is a finite-dimensional irreducible smooth representation of $G(F)$. Then the normal subgroup $\ker\pi \cap G_{\tu{der}}(F)$ of $G_{\tu{der}}(F)$ is open. Lemma \ref{lem:normal-in-sc-groups} implies that $\ker\pi \cap G_{\tu{der}}(F)= G_{\tu{der}}(F)$, thus $\pi$ factors through the abelian quotient $G(F)/G_{\tu{der}}(F)$. Therefore $\dim \pi=1$, completing the proof.
\end{proof}

\subsection{Local Hecke algebras and their variants}\label{sub:local-Hecke}

We retain the notation from the preceding section but allow the local field $F$ to be either non-archimedean or archimedean. A basic setup of local Hecke algebras will be introduced, partly following \cite[\S1]{ArthurLocalCharacter}.

Fix a Haar measure on $G(F)$ and a maximal compact subgroup $K\subset G(F)$. Let $G(F)_{\tu{sr}}$ denote the subset of strongly regular elements $g\in G(F)$, namely the semisimple elements whose centralizers in $G$ are (maximal) tori. By \cite[2.15]{SteinbergRegular}, $G(F)_{\tu{sr}}$ is open and dense in $G(F)$ (for both the Zariski and non-archimedean topologies). Write $R(G)$ for the space of finite $\C$-linear combinations of irreducible characters of $G(F)$, which is a subspace in the space of functions on $G(F)_{\tu{sr}}$. We also identify $R(G)$ with the Grothendieck group of smooth finite-length representations of $G(F)$ with $\C$-coefficients. Let $\cH(G)=\cH(G(F))$ denote the space of smooth compactly supported bi-$K$-finite functions on $G(F)$. Let $\cI(G)$ denote the invariant space of functions on $G(F)$, namely the quotient of $\cH(G)$ by the ideal generated by functions of the form $g\mapsto f(g)-f(hgh^{-1})$ with $h\in G(F)$ and $f\in \cH(G)$. From \cite[Thm.~0]{KazhdanCuspidal}, we see that $f\in \cH(G)$ has trivial image in $\cI(G)$ if and only if its orbital integral vanishes on $G(F)_{\tu{sr}}$ if and only if $\Tr \pi(f)=0$ for all irreducible tempered representations of $G(F)$; moreover, the same is true if $G(F)_{\tu{sr}}$ is replaced with $G(F)$ and if the temperedness condition is dropped. By abuse of notation, we frequently write $f\in \cI(G)$ to mean a representative $f\in \cH(G)$ of an element in $\cI(G)$. The trace Paley--Wiener theorem \cite{BDK86} describes  $\cI(G)$ as a subspace of $\C$-linear functionals on $R(G)$ via
\begin{equation}\label{eq:trace-PW}
f\mapsto \bigg(\Theta \mapsto \int_{G(F)_{\tu{sr}}} f(g)\Theta(g) dg\bigg).
\end{equation}
If $R(G)$ is thought of as a Grothendieck group, the above map is simply
$f\mapsto (\pi\mapsto \Tr \pi(f))$.

Denote by $\cS(G)$ the quotient of $\cH(G)$ by the ideal generated by functions each of which has vanishing stable orbital integrals on $G(F)_{\tu{sr}}$. Thus we have natural surjections $\cH(G)\twoheadrightarrow \cI(G)\twoheadrightarrow \cS(G)$. By $R(G)^{\tu{st}}$ we mean the subspace of $R(G)$ consisting of stable linear combinations (\ie~constant on each stable conjugacy class in $G(F)_{\tu{sr}}$). Then $\cS(G)$ is identified with a subspace of functions on $R(G)^{\tu{st}}$ via \eqref{eq:trace-PW} (since $\Theta\in R(G)^{\tu{st}}$ now, the image depends only on the image of $f$ in $\cS(G)$); the subspace is characterized by \cite[Thm.~6.1,~6.2]{ArthurLocalCharacter} in the $p$-adic case, cf.~last paragraph on p.491 of \cite{XuCuspidalSupport}. Via the obvious quotient map $\cI(G)\ra \cS(G)$ and the restriction map from $R(G)$ to $R(G)^{\tu{st}}$, we have a commutative diagram $$\xymatrix{ \cI(G) \ar[r] \ar[d]_-{\Tr} & \cS(G) \ar[d]^-{\Tr} \\
\Hom_{\C\tu{-linear}}(R(G),\C)  \ar[r] & \Hom_{\C\tu{-linear}}(R(G)^{\tu{st}},\C).
}$$

Let us extend the setup so far to allow a fixed central character. By a \textbf{local central character datum} for $G$, we mean a pair $(\fkX,\chi)$, where
\begin{itemize}
\item $\fkX$ is a closed subgroup of $Z(F)$ with a Haar measure $\mu_{\fkX}$ on $\fkX$,
\item $\chi:\fkX\ra \C^\times$ is a smooth character.
\end{itemize}
Let $\cH(G,\chi^{-1})=\cH(G(F),\chi^{-1})$ denote the space of smooth bi-$K$-finite functions $f$ on $G(F)$ which have compact support modulo $\fkX$ and satisfy $f(xg)=\chi^{-1}(x)f(g)$ for $x\in \fkX$ and $g\in G(F)$. The $\chi$-averaging map
$$
\cH(G)\ra \cH(G,\chi^{-1}),\qquad f\mapsto \bigg( g\mapsto \int_{\fkX} f(gz)\chi(z) d\mu_{\fkX} \bigg),
$$
is a surjection. We have the obvious definitions of $\cI(G,\chi^{-1})$ and $\cS(G,\chi^{-1})$, the $\chi$-averaging maps $\cI(G)\ra \cI(G,\chi^{-1})$ and $\cS(G)\ra \cS(G,\chi^{-1})$, as well as the quotient maps
$$
\cH(G,\chi^{-1})\twoheadrightarrow \cI(G,\chi^{-1})\twoheadrightarrow \cS(G,\chi^{-1}).
$$
We can think of $\cI(G,\chi^{-1})$ as a subspace of functions on $R(G,\chi)$, the subspace of $R(G)$ generated by irreducible characters with central character $\chi$. Analogously $\cS(G,\chi^{-1})$ is the subspace of functions on $R(G,\chi)^{\tu{st}}$ defined similarly.

\subsection{Transfer of one-dimensional representations}\label{sub:transfer-one-dim}

Let $G$ and $G^*$ be connected reductive groups over a non-archimedean local field $F$ of characteristic zero, with $G^*$ quasi-split over $F$. Let $\xi:G_{\ol F}\isom G^*_{\ol F}$ be an inner twisting, namely an $\ol F$-isomorphism such that $\xi^{-1} \sigma(\xi)$ is an inner automorphism of $G_{\ol F}$ for every $\sigma\in \Gal(\ol F/F)$. As in \S\ref{sub:Notation}, we have canonical $F$-morphisms $\varrho:G_{\textup{sc}}\ra G$ and $\varrho^*:G^{*}_{\textup{sc}}\ra G^*$. Define an $F$-torus and two topological groups 
$$
G^{\flat}:=G/G_{\der},\qquad  G(F)^{\flat}:=\textup{cok}(G_{\textup{sc}}(F)\stackrel{\varrho}{\ra} G(F)),\qquad G(F)^{\ab}:=G(F)/G(F)_{\der},
$$
where $G(F)_{\der}$ is the commutator subgroup of $G(F)$ as an abstract group, which is closed in $G(F)$. (This is clear if $G$ is a torus. If not, $G(F)_{\der}$ is not contained in $Z_{G_{\der}}(F)$ so an open subgroup of finite index in $G_{\der}(F)$ by \cite[Prop.~3.17]{PlatonovRapinchuk}, after reducing to the simply connected and $F$-simple case via $z$-extensions.)
Moreover, $G(F)_{\der}$ is contained in $\tu{im}(G(F)_{\tu{sc}}\ra G(F))$ \cite[2.0.2]{DeligneCorvallis}, so there are natural morphisms
\begin{equation}\label{eq:G-flat-ab}
G(F)\twoheadrightarrow G(F)^{\ab} \twoheadrightarrow  G(F)^{\flat} \twoheadrightarrow G(F)/G_{\der}(F) \hra G^{\flat}(F).
\end{equation}
In particular, $ G(F)^{\flat}$ is an abelian group. 
The last three maps in \eqref{eq:G-flat-ab} are isomorphisms if $G_{\der}=G_{\textup{sc}}$ by Kneser's vanishing theorem for $H^1$ of simply connected groups (applicable since $F$ is non-archimedean). The definition and discussion above applies to $G^*$ in the same way.

Let $1\ra Z_1\ra G_1\stackrel{\alpha}{\ra} G\ra 1$ be a $z$-extension of $G$ over $F$. Since $G_1\ra G$ induces $G_1^{\ad}\isom G^{\ad}$, the classifying data for inner twists of $G_1$ and those of $G$ are identified (up to isomorphism). Thus we may assume that there is a $z$-extension $1\ra Z_1\ra G^*_1 \stackrel{\alpha^*}{\ra} G^*\ra 1$ with an inner twisting $\xi_1:G_{1,\ol F} \isom G^*_{1,\ol F}$ such that $\xi_1$ and $\xi$ form a commutative square together with the maps $\alpha$ and~$\alpha^*$. The map $G_{1,\der}\ra G_{\der}$ induced by $\alpha$ is a simply connected cover, allowing an identification $G_{1,\der}=G_{\textup{sc}}$. Likewise we have $G_{1,\der}^{*}=G^{*}_{\textup{sc}}$.

\begin{lemma}\label{lem:Gab=G*ab}
There is a commutative diagram in which rows are exact and vertical maps are isomorphisms:
$$
\xymatrix{
  1 \ar[r] & Z_1(F)/Z_1(F)\cap G_{1,\der}(F) \ar[r] \ar[d]^{\sim} & G_1(F)^{\flat}=G_1^{\flat}(F)   \ar[d]^{\sim} \ar[r] & G(F)^{\flat} \ar[r] \ar[d]^{\sim} & 1\\
  1 \ar[r] & Z_1(F)/Z_1(F)\cap G_{1,\der}^{*}(F) \ar[r]  & G'_1(F)^{\flat}=G_1^{*,\flat}(F)    \ar[r] & G^*(F)^{\flat} \ar[r] &  1.
 } 
$$
Here the second vertical map is given by the isomorphism $G_1^{\flat}\isom G_1^{*,\flat}$ induced by $\xi$, and the first and third vertical maps are induced by the second. Moreover the isomorphism $G(F)^{\flat}\isom G^*(F)^{\flat}$ is canonical, \ie~independent of the choice of $z$-extensions.
\end{lemma}

We will write $\xi^{\flat}: G(F)^{\flat}\isom G^*(F)^{\flat}$ for the canonical isomorphism.

\begin{proof}
The row-exactness in the diagram is straightforward from the definition. The map $\xi_1$ induces an $F$-isomorphism $G_1^{\flat}\isom G_1^{*,\flat}$ and restricts to an $F$-isomorphism from $Z_1$ onto $Z_1$. Thus the first two vertical maps are isomorphisms, which implies that the last vertical map is also.

As for the last assertion, if $1\ra Z_1\ra G_1\stackrel{\alpha_1}\ra G\ra 1$ and $1\ra Z_2\ra G_2\stackrel{\alpha_2}\ra G\ra 1$ are two $z$-extensions then there is a third $z$-extension $1\ra Z_1\times Z_2\ra G_3\ra G\ra 1$ by $G_3=\{(g_1,g_2)\in G_1\times G_2: \alpha_1(g_1)=\alpha_2(g_2)\}$, equipped with projections $G_3 \twoheadrightarrow G_1$ and $G_3 \twoheadrightarrow G_1$. Thus we are reduced to showing that $G_3$ and $G_1$ (and likewise $G_3$ and $G_2$) induce the same isomorphism $G(F)^{\flat}\isom G^*(F)^{\flat}$. This is an elementary compatibility check.
\end{proof}

\begin{lemma}\label{lem:Gab-Gflat}
If $G_{\tu{sc}}$ has no $F$-anisotropic factor, then $G(F)^\flat = G(F)^{\ab}$.
\end{lemma}
\begin{proof}
We may assume that $G$ is not a torus. Via a $z$-extension, we reduce to the case when $G_{\tu{sc}}=G_{\der}$. Then $G(F)_{\der}$ is a noncentral normal subgroup of $G_{\der}(F)$. Applying Lemma \ref{lem:normal-in-sc-groups} to each $F$-simple factor of $G_{\der}$, we deduce that $G(F)_{\der}=G_{\der}(F)$, hence $G(F)^\flat = G(F)^{\ab}$.
\end{proof}

\begin{corollary}\label{cor:1-dim reps}
If $G_{\tu{sc}}$ has no $F$-anisotropic factor, then
the following four groups (under multiplication) are in canonical isomorphisms with each other:
\begin{enumerate}
\item the group of smooth characters $G(F)\ra \C^\times$,
\item the group of smooth characters $G(F)^{\flat}\ra \C^\times$,
\item the group of smooth characters $G^*(F)^{\flat}\ra \C^\times$,
\item the group of smooth characters $G^*(F)\ra \C^\times$,
\end{enumerate}
where the maps from (2) to (1) and from (3) to (4) are given by pullbacks, and the map between (2) and (3) is via the isomorphism of Lemma \ref{lem:Gab=G*ab}.
With no assumption on $G_{\tu{sc}}$, we still have canonical isomorphisms between (2), (3), and (4), and a canonical embedding from (2) to (1).
\end{corollary}
\begin{proof}
Since $G(F)^{\ab}$ is the maximal abelian topological quotient of $G(F)$, we can replace $G(F)$ with $G(F)^{\tu{ab}}$ in (1), and likewise for (4). From \eqref{eq:G-flat-ab} and Lemma \ref{lem:Gab=G*ab}, we have
$$ 
G(F)^{\ab} \twoheadrightarrow G(F)^{\flat} \simeq G^*(F)^{\flat} \twoheadleftarrow G^*(F)^{\ab}.
$$
Lemma \ref{lem:Gab-Gflat} tells us that the last map is always an isomorphism (since $G^*$ has no $F$-anisotropic factor); so is the first map if $G$ has no $F$-anisotropic factor. The corollary follows.
\end{proof}

\begin{remark}\label{rem:1-dim_reps}
The only nontrivial $F$-anisotropic simply connected simple group over $F$ is of the form $\Res_{F'/F}\SL_1(D)$ for a central division algebra $D$ over a finite extension $F'$ of $F$ with $[D:F]=n^2$ and $n\ge 2$. So the condition in the corollary is that $G_{\tu{sc}}$ has no such factor. Two exemplary cases are (i) $G=\GL_1(D)$, $G^*=\GL_n(F)$ and (ii) $G=\SL_1(D)$, $G^*=\SL_n(F)$. In (i), it is standard (e.g., \cite[Intro.]{Riehm}) that $G(F)_{\der}=G_{\der}(F)$, and the four sets are still isomorphic. However, in (ii), $G(F)_{\der}$ is the group of 1-units in the maximal order of $D$ by \cite[\S5,~Cor.]{Riehm}. In particular (1) is a nontrivial group, whereas (2) and (3) are evidently trivial, thus (4) is trivial by the corollary.
\end{remark}

\begin{remark}
One can also construct a natural map from (4) to (1) through the continuous cohomology $H^1(W_F,Z(\hat G))=H^1(W_F,Z(\hat G^*))$ following Langlands. (This works for archimedean local fields $F$ as well.) Indeed, \cite[App.~A]{Xu-lifting} explains the isomorphism between $H^1(W_F,Z(\hat G^*))$ and (4), and a map from $H^1(W_F,Z(\hat G^*))$ to (1).\footnote{The latter map is asserted to be also an isomorphism in  \cite[App.~A]{Xu-lifting}, but this is false for $G=\SL_1(D)$ (in which case $Z(\hat G)=\{1\}$) as explained in Remark \ref{rem:1-dim_reps}. In \emph{loc.~cit.}, for a simply connected group $G'$ over $F$, it is said that all continuous characters $G'(F)\ra \C^\times$ are trivial, but this is not guaranteed unless $G_{\tu{sc}}$ has no $F$-anisotropic factor (e.g., this is okay for $G^*$). This mistake is surprisingly prevalent in the literature.}
\end{remark}

Let  $g\in G(F)_{\tu{ss}}$ and $g^*\in G^*(F)_{\tu{ss}}$. When $G_{\der}=G_{\tu{sc}}$, we say $g$ and $g^*$ are \emph{matching} if their $\ol F$-conjugacy classes correspond via $\xi$. In general, \emph{matching} is defined by lifting $\xi$ to an inner twisting between $z$-extensions of $G$ and $G^*$ as in \cite[pp.799--800]{KottwitzRational} (specialized to the case $E=F$). From \emph{loc.~cit.}~we see that the notion of matching is independent of the choice of $z$-extensions, anddepends only on the $G(\ol F)$-conjugacy class of $\xi$. 

Since $G^*$ is quasi-split, every $g$ admits a matching element in $G^*(F)$.
When $g$ and $g^*$ are matching, we have an inner twisting between the connected centralizers $I_g,I_{g^*}$ in $G,G^*$ by \cite[Lem.~5.8]{KottwitzRational}. Fix Haar measures on the pairs of inner forms $(G(F),G^*(F))$ and $(I_g(F),I_{g^*}(F))$ compatibly in the sense of \cite[p.631]{KottwitzTamagawa} to define (stable) orbital integrals at $g$ and $g^*$, cf.~\cite[pp.637--638]{KottwitzTamagawa}. Write $e(G)\in \{\pm1\}$ for the Kottwitz sign. Now $f\in \cH(G(F))$ and $f^*\in \cH(G^*(F))$ are said to be \emph{matching} if for every $g^*\in G^*(F)_{\tu{sr}}$, we have the identity of stable orbital integrals
\begin{equation}\label{eq:SO=SO}
SO_{g^*}(f^*)=\begin{cases}
SO_g(f), & \mbox{if there exists a matching}~g\in G(F)_{\tu{ss}},\\
 0, & \mbox{if there is no such}~g\in G(F)_{\tu{ss}}.
\end{cases}
\end{equation}

\begin{remark}\label{rem:transfer-factor-sign}
The sign convention in \eqref{eq:SO=SO} is chosen in favor of simplicity. (See also Remark \ref{rem:Igusa-Kottwitz-sign} below.) One could require $SO_{g^*}(f^*)=e(G)SO_g(f)$ instead, so that the Kottwitz sign $e(G)$ plays the role of transfer factor, but that would introduce $e(G)$ in the trace identity of Lemma \ref{lem:transfer-1-dim}.
\end{remark}
A standard fact (cf.~\S\ref{sub:prelim-endoscopy} below) from \cite{WaldspurgerFLimpliesTC} is every $f$ admits a matching $f^*$ as above, called a \emph{(stable) transfer} of $f$. If the Harish-Chandra character $\Theta_{\pi^*}$ of $\pi^*\in \Irr(G^*(F))$ is stable, \ie~ $\Theta_{\pi^*}(g^*_1)=\Theta_{\pi^*}(g^*_2)$ whenever $g^*_1,g^*_2\in G^*(F)_{\tu{sr}}$ are stably conjugate, then the value $\Tr {\pi^*}(f^*)=\int_{G^*(F)_{\tu{sr}}} f^*(g^*)\Theta_{\pi^*}(g^*)dg^*$ is determined by the stable orbital integrals of $f^*$ on $G^*(F)_{\tu{sr}}$. This follows from the stable version of the Weyl integration formula, cf.~\eqref{eq:Weyl-integration-stable} below. This discussion applies to $\pi^*$ with $\dim \pi^*=1$ for example, since the characters of such $\pi^*$ are clearly stable. The analogue holds true with $G$ and $f$ in place of $G^*$ and $f^*$.

\begin{lemma}\label{lem:transfer-1-dim}
Let $f\in \cH(G(F))$ and $f^*\in \cH(G^*(F))$ be matching functions. Let $\pi^*:G^*(F)\ra \C^\times$ be a smooth character. If $\pi:G(F)\ra \C^\times$ is given by $\pi^*$ via Corollary \ref{cor:1-dim reps} then
$$ 
\Tr \pi(f)=\Tr \pi^*(f^*).
$$
\end{lemma}

\begin{proof}
As $\dim\pi=\dim \pi^*=1$, the characters $\Theta_\pi$ and $\Theta_{\pi^*}$ are stable, and $\Theta_\pi(g)=\pi(g)$, $\Theta_{\pi^*}(g^*)=\pi^*(g^*)$ for $g\in G(F),~g^*\in G^*(F)$. For a maximal torus $T$ of $G$ over $F$, write $W(T)$ for the associated Weyl group. By the stable Weyl integration formula,
\begin{equation}\label{eq:Weyl-integration-stable}
\Tr \pi(f)=\int_{G(F)_{\tu{sr}}} f(g) \Theta_\pi(g)dg
= \sum_{ T } \frac{1}{|W(T)|} \int_{T(F)_{\tu{sr}}}  SO_t(f) \Theta_\pi(t) dt,
\end{equation}
where the sum runs over a set of representatives for stable conjugacy classes of maximal tori of $G$ over $F$. The analogous formula holds for $G^*(F)$. 
From here, the proof is an easy exercise using \eqref{eq:SO=SO} and the following fact coming from quasi-splitness of $G^*(F)$: every maximal torus of $G(F)$ is a transfer of that of $G^*(F)$ in the sense of \cite[9.5]{KottwitzCuspidalTempered}.
\end{proof}

\begin{remark}
The correspondence of Lemma \ref{lem:transfer-1-dim} need not be the Jacquet--Langlands correspondence when $G^*=\GL_n$. E.g., if $G=\GL_1(D)$ for a central division algebra $D$ over a $p$-adic field $F$ with $n>1$, then the trivial representation of $D^\times$ corresponds to the Steinberg representation of $\GL_n(F)$ under Jacquet--Langlands, but to the trivial representation of $\GL_n(F)$ in the lemma.
\end{remark}

\subsection{Lefschetz functions on real reductive groups}\label{sub:Lefschetz}

Let $G$ be a connected reductive group over~$\R$ containing an elliptic maximal torus. Fix a maximal compact subgroup $K_\infty\subset G(\R)$.
Denote by $G(\R)_+$ the preimage of the neutral component $G^{\textup{ad}}(\R)^0$ (for the real topology) under the natural map $G(\R)\ra G^{\textup{ad}}(\R)$.

\begin{lemma}\label{lem:G(R)+}
We have $G(\R)_+ = Z(\R)\cdot \varrho(G_{\tu{sc}}(\R))$.
\end{lemma}
\begin{proof}
Since $G_{\tu{sc}}(\R)$ is connected, clearly $\varrho(G_{\tu{sc}}(\R))$ maps into $G^{\textup{ad}}(\R)^0$. Therefore $G(\R)_+ \supset Z(\R)\cdot \varrho(G_{\tu{sc}}(\R))$. We have surjections $G_{\tu{sc}}(\R)^0 \times Z(\R)^0 \twoheadrightarrow G(\R)^0  \twoheadrightarrow G^{\tu{ad}}(\R)^0$ by \cite[Prop.~5.1]{MilneIntroduction}. This implies that $G(\R)_+\subset Z(\R) G(\R)^0 = Z(\R)\cdot \varrho(G_{\tu{sc}}(\R))$.
\end{proof}

Let $\xi$ be an irreducible algebraic representation of $G_\C$, and
$\zeta:G(\R)\ra\C^\times$ be a continuous character. By restriction $\xi$ yields a continuous representation of $G(\R)$ on a complex vector space, which we still denote by $\xi$. Write $\omega_\xi:Z(\R)\ra\C^\times$ for the central character of $\xi$. By $\Pi_\infty(\xi,\zeta)$ we mean the set of isomorphism classes of irreducible discrete series representations whose central and infinitesimal characters are equal to those of the contragredient of $\xi\otimes\zeta$. This is a discrete series $L$-packet by the construction of \cite{LanglandsRealClassification}, which assigns to $\Pi_\infty(\xi,\zeta)$ an $L$-parameter
$$
\varphi_{\xi,\zeta} \colon W_\R\ra {}^L G.
$$
Thus we also write $\Pi_\infty(\varphi_{\xi,\zeta})$ for $\Pi_\infty(\xi,\zeta)$. We have $\xi\otimes \zeta\simeq \xi'\otimes \zeta'$ as representations of $G(\R)$ if and only if there exists an algebraic character $\chi$ of $G_\C$ such that $\xi'=\xi\otimes \chi$ and $\zeta'=\zeta\otimes \chi^{-1}$. In this case $\Pi_\infty(\xi,\zeta)=\Pi_\infty(\xi',\zeta')$, and $\varphi_{\xi,\zeta} \simeq \varphi_{\xi',\zeta'}$. In fact $| \Pi_\infty(\xi,\zeta)|$ is a constant $d(G)\in \Z_{\ge 1}$ depending only on $G$. When $\xi=\mathbf 1$, we also write $\Pi_\infty(\zeta)$ and $f_\zeta$ for $\Pi_\infty(\xi,\zeta)$ and $f_{\xi,\zeta}$.

Write $A_G$ for the maximal split torus in the center of $G$. Let $\chi:A_G(\R)^0\ra \C^\times$ be a continuous character. Let $\tu{Irr}_{\tu{temp}}(G(\R),\chi)$ be the set of (isomorphism classes of) irreducible tempered representations of $G(\R)$ whose central character equals $\chi$ on $A_G(\R)^0$. Write $\cH(G(\R),\chi^{-1})$ for the space of smooth $K_\infty$-finite functions on $G(\R)$ with central character $\chi^{-1}$. Following \cite[\S4]{ArthurL2}, $f\in \cH(G(\R),\chi^{-1})$ is said to be \textbf{stable cuspidal} if $\Tr \pi(f)$ is constant as $\pi$ varies over each discrete series $L$-packet and if $\Tr \pi(f)=0$ for every $\pi\in \tu{Irr}_{\tu{temp}}(G(\R),\chi)$ outside of discrete series.

Fix a Haar measure on $G(\R)$ and the Lebesgue measure on $A_G(\R)^0$, so as to determine a Haar measure on $G(\R)/A_G(\R)^0$. Choose a pseudo-coefficient $f_\pi \in \cH(G(\R),\omega_\xi\zeta)$ for each $\pi\in \Pi_\infty(\xi,\zeta)$ \`a la \cite{ClozelDelorme}. Although it is not unique, the orbital integrals of $f_\pi$ are uniquely determined by the property that $\Tr \pi(f_\pi)=1$ and that $\Tr \pi'(f_\pi)=0$ for $\pi'\in \Irr_{\tu{temp}}(G(\R),(\omega_\xi\zeta)^{-1})$. An \textbf{averaged Lefschetz function} associated with $(\xi,\zeta)$ is defined as
\begin{equation}\label{eq:AveragedLefschetz}
f_{\xi,\zeta}:=| \Pi_\infty(\xi,\zeta)|^{-1} \sum_{\pi\in \Pi_\infty(\xi,\zeta)} f_{\pi} \quad  \in \cH(G(\R),\omega_\xi\zeta).
\end{equation}
By construction, $f_{\xi,\zeta}$ is stable cuspidal in the above sense.

For elliptic $\gamma\in G(\R)$, let $I_\gamma$ denote its connected centralizer in $G(\R)$, and $e(I_\gamma)\in \{\pm 1\}$ its Kottwitz sign. Let $I_\gamma^{\tu{cpt}}$ denote an inner form of $I_\gamma$ over $\R$ that is anisotropic modulo $Z_G$. From \cite[p.659]{KottwitzInventiones}, as our $O_\gamma(f_{\xi,\zeta})$ equals $d(G)^{-1}\SO_{\gamma_\infty}(f_\infty)$ there, we see that
\begin{equation}\label{eq:orbital-integral-Lefschetz}
 O_\gamma(f_{\xi,\zeta})
= \begin{cases}
 d(G)^{-1}\tu{vol}(A_G(\R)^0\backslash I^{\tu{cpt}}_\gamma(\R))^{-1}\zeta(\gamma)e(I_\gamma)\Tr \xi(\gamma), & \gamma:\textup{elliptic},\\
   0, & \gamma:\textup{non-elliptic}.
\end{cases}
\end{equation}
In \eqref{eq:orbital-integral-Lefschetz}, the Haar measure on $I^{\tu{cpt}}_\gamma(\R)$ is chosen to be compatible (in the sense of \cite[p.631]{KottwitzTamagawa}) with the measure on $I_\gamma(\R)$ used in the orbital integral, to compute $\tu{vol}(A_G(\R)^0\backslash I^{\tu{cpt}}_\gamma(\R))$ with respect to the Lebesgue measure on $A_G(\R)^0$. Again by \cite[p.659]{KottwitzInventiones},
\begin{equation}\label{eq:SO-Lefschetz}
SO_\gamma(f_{\xi,\zeta}) = 
\begin{cases}
 \tu{vol}(A_G(\R)^0\backslash I^{\tu{cpt}}_\gamma(\R))^{-1}\zeta(\gamma)\Tr \xi(\gamma),& \gamma:\textup{elliptic},\\
  0, & \gamma:\textup{non-elliptic}.
\end{cases}
\end{equation}

Let $G^*$ be a quasi-split group over $\R$ with inner twisting $G_{\C}\isom G^*_{\C}$, through which $\xi,\zeta$ above are transported to $G^*$. Thereby we obtain an averaged Lefschetz function $f^{*}_{\xi,\zeta}$ on $G^*(\R)$.

\begin{lemma}\label{lem:transfer-of-Lefschetz}
The function $f^{*}_{\xi,\zeta}$ is a transfer of $f_{\xi,\zeta}$ (in the sense of \eqref{eq:SO=SO}).
\end{lemma}
\begin{proof}
 This is immediate from \eqref{eq:SO-Lefschetz}.
\end{proof}

\begin{lemma}\label{lem:tr-of-char-at-infty}
Assume that $\xi=\mathbf 1$. Let $\pi:G(\R)\ra\C^\times$ be a continuous character whose central character equals $\zeta^{-1}$ when restricted to $A_G(\R)^0$. Then $\pi|_{G(\R)_+}=\zeta^{-1}|_{G(\R)_+}$ if and only if $\pi|_{Z(\R)}=\zeta^{-1}|_{Z(\R)}$. If the equivalent conditions hold then $\Tr (f_\zeta | \pi) =  1$ if $\pi=\zeta^{-1}$; otherwise $\Tr (f_\zeta | \pi) =  0$.
\end{lemma}
\begin{proof}
The first assertion is clear from Lemma \ref{lem:G(R)+}. For the second assertion, it follows from \eqref{eq:orbital-integral-Lefschetz} via the Weyl integration formula that $\Tr (f_\zeta | \pi) = \tu{vol}(K)^{-1}\int_{K} \zeta(k)\pi(k)dk$, where $K$ is a maximal compact-modulo-$A_G(\R)^0$ subgroup of $G(\R)$. The integral vanishes unless $\pi=\zeta^{-1}$ on $K$, in which case $\pi=\zeta^{-1}$ on the entire $G(\R)$ (since $K$ meets every component of $G(\R)$) and $\Tr (f_\zeta | \pi)=1$.
\end{proof}

\subsection{One-dimensional automorphic representations}\label{sub:one-dim}

Now let $G$ be a connected reductive group over a \emph{number field} $F$. Let $v$ be a place of $F$ and set $G_v:=G_{F_v}$. We have a finite decomposition of $G_{\tu{sc}}$ into $F$-simple factors
\begin{equation}\label{eq:Gsc-decomposed}
G_{\tu{sc}} = \prod_{i\in I} G_i,\qquad \mbox{with}\quad G_i=\Res_{F_i/F}H_i,
\end{equation}
for a finite extension $F_i/F$ and an absolutely $F$-simple simply connected group $H_i$ over each $F_i$. Accordingly $G^{\tu{ad}}=\prod_{i\in I} G_i^{\tu{ad}}$. Note that we have a natural composite map $G\ra G^{\tu{ad}} \ra G_i^{\tu{ad}}$ for each $i\in I$, where the last arrow is the projection onto the $i$-component. Let $P_v=M_vN_v$ be a Levi decomposition of a parabolic subgroup of $G_v$. We consider the following assumption, where ``nb'' stands for non-basic (cf.~Definition \ref{def:Q-non-basic} and Lemma \ref{lem:Q-non-basic} below).

\smallskip

\noindent
\textbf{($\Q$-nb($P_v$))}  The image of $P_v$ in $(G_i^{\tu{ad}})_v$ is a \emph{proper} parabolic subgroup for every $i\in I$.

\smallskip

The assumption implies that $G^{\textup{ad}}$ has no nontrivial $F$-simple factor that is anisotropic over $F_v$, thus so the embedding $G_{\tu{sc}}(F)\hra G_{\tu{sc}}(\A_F^v)$ has dense image by strong approximation. When $G^{\tu{ad}}$ is itself $F$-simple, ($\Q$-nb($P_v$)) is simply saying that $P_v$ is a \emph{proper} parabolic subgroup of $G_v$.

\begin{lemma}\label{lem:local-1dim-global-1dim}
Assume that $G_{\der}=G_{\tu{sc}}$ and that $G_i$ is isotropic over $F_v$ for every $i\in I$. Let $\pi$ be a discrete automorphic representation of $G(\A_F)$, and $\pi'$ an irreducible $G_{\der}(\A_F)$-subrepresentation of $\pi$. Decompose $\pi'=\otimes_i \pi'_i$ according to $G_{\tu{der}}(\A_F)=\prod_{i\in I} G_i(\A_F)$. Write
$$
G_i(F_v)= H_i(F_i\otimes_F F_v)=\prod_{w|v} H_i(F_{i,w}),
$$
where $w$ runs over the set of places of $F_i$ above $v$, and decompose $\pi'_{i,v}=\otimes_{w|v} \pi'_{i,w}$ accordingly. If for every $i\in I$, there exists $w|v$ such that $\pi'_{i,w}$ is trivial, then $\dim \pi=1$.
\end{lemma}

\begin{proof}
By strong approximation, the embedding $H_i(F_i)\hra H_i(\A_{F_i}^w)$ has dense image for each $i\in I$.
Since the underlying space of $\pi'_i$ consists of automorphic forms which are left-invariant under $H_i(F_i)$, and since $\pi'_{i,w}$ is trivial, the arugment for \cite[Lem.~6.2]{KST2} shows that $\pi'_i$ is trivial on the entire $H_i(\A_{F_i})$. Hence $\pi'$ is trivial. Since $G(\A_F)/G_{\der}(\A_F)$ is abelian, we deduce that $\dim \pi=1$ as $\pi$ is generated by $\pi'$ as a $G(\A_F)$-module.
\end{proof}

\begin{corollary}\label{cor:BoundExp}
Let $\pi$ be an irreducible $G(\A_F)$-subrepresentation of $L^2_{\disc}(G(F)\backslash G(\A_F)/A_{G,\infty})$ and let $\omega_v \in \textup{Exp}(J_{P_v}(\pi_v))$. Then
\begin{equation}\label{eq:TheBound}
\delta_{P_v}^{1/2}(a) \leq  |\omega_v(a)| \leq \delta_{P_v}^{-1/2}(a),\qquad a\in A_{P_v}^{--}.
\end{equation}
Now assume \textup{($\Q$-nb($P_v$))}. Then
the left (resp. right) equality holds for some $a\in A_{P_v}^{--}$ if and only if the left (resp. right) equality holds for all $a\in A_{P_v}^{--}$ if and only if $\dim \pi=1$.
\end{corollary}
\begin{proof}
The assertion about the (in)equality on the left follows from the right counterpart by considering the contragredient of $\pi$. From now on, we concentrate on the right (in)equality.

The right inequality in \eqref{eq:TheBound} is immediate from Proposition \ref{prop:BoundExp} and the normalization $J_{P_v}(\pi_v)=(\pi_v)_{N_v}\otimes \delta_{P_v}^{-1/2}$. It remains to check the three conditions for the equality are equivalent. The only nontriviality is to show that $\dim \pi=1$, assuming that $|\omega_v(a)|=\delta_{P_v}^{-1/2}(a)$ for some $a\in A_{P_v}^{--}$.

We may assume $G_{\tu{der}}=G_{\tu{sc}}$ via $z$-extensions.   We decompose
$$
P_v':=P_v\cap G_{\tu{der}} =\prod_{i\in I}\prod_{w|v} P_{v,i,w} \quad\mbox{according as}\quad (G_{\tu{der}})_v = \prod_{i\in I}\prod_{w|v} (H_{i})_w,
$$
where $w$ runs over places of $F_i$ above $v$. Similarly $A_{P'_v} = \prod_{i,w} A_{P_{v,i,w}}$. Assumption ($\Q$-nb($P_v$)) tells us that for every $i$, there exists $w|v$ such that $P_{v,i,w}$ is a \emph{proper} parabolic subgroup of $(H_i)_w$. In particular, $G_{i,v}$ is isotropic for every $i$. So we can apply Lemma \ref{lem:local-1dim-global-1dim}. Adopting the setup and notation from there, it suffices to show that for every $i$, there exists a place $w|v$ such that $\pi'_{i,w}=\textbf{1}$. In fact, we only need to find $w|v$ such that $\dim \pi'_{i,w}<\infty$ by Lemma \ref{lem:fin-implies-one} and Corollary \ref{cor:1-dim reps}.  

The central isogeny $Z\times G_{\tu{der}}\ra G$ induces a map $A_{G_v}\times A_{P'_v}\ra A_{P_v}$, which has finite kernel and cokernel on the level of $F_v$-points. Replacing $a$ with a finite power, we may assume that $a$ is the image of $(a_0,a')\in A_{G_v}(F_v)\times A_{P'_v}(F_v)$, so that $|\omega(a)|=|\omega(a')|$. (The central character of $\pi'_v$ is unitary on $A_{G_v}(F_v)$, so $|\omega(a_0)|=1$.) Write $a'=(a_{i,w})_{i,w}$ and $\omega_v|_{A_{P'_v}(F_v)}=(\omega_{v,i,w})_{i,w}$ according to the decomposition of $P'_v$ above. We have $|\omega_{v,i,w}(a_{i,w})| \le \delta_{P_{v,i,w}}^{-1/2}(a_i)$ by Proposition \ref{prop:BoundExp}, while $\prod_{i,w} |\omega_{v,i,w}(a_{i,w})| = \prod_{i,w} \delta_{P_{v,i,w}}^{-1/2}(a_{i,w})$ from our running assumption. Therefore
$$ 
|\omega_{v,i,w}(a_{i,w})| = \delta_{P_{v,i,w}}^{-1/2}(a_{i,w}),\quad \forall i\in I.
$$ 
  
Since $J_{P'_v}(\pi'_v)\subset J_{P_v}(\pi_v)$, we see that $\omega_v|_{A_{P'_v}(F_v)}\in \tu{Exp}(J_{P'_v}(\pi'_v))$. Thus we have
$$
\omega_{v,i,w}\in \tu{Exp}(J_{P_{v,i,w}}(\pi'_{v,i,w})).
$$
Finally for each $i$, we apply the equality criterion of Proposition \ref{prop:BoundExp} at a place $w$ where $P_{v,i,w}$ is proper in $(H_i)_w$. Thereby we deduce that $\dim \pi'_{i,w}<\infty$ as desired.
\end{proof}

Let $\xi:G_{\ol F}\isom G^*_{\ol F}$ be an inner twisting, with $G^*$ a connected reductive group over $F$.

\begin{lemma}\label{lem:one-dim-auto-rep}
One-dimensional automorphic representations of $G(\A_F)$ are in a canonical bijection with those of $G^*(\A_F)$, compatibly with the bijection of Corollary \ref{cor:1-dim reps} at every place of $F$.
\end{lemma}

\begin{proof}
Define $G(\A_F)^{\flat}:=\tu{cok}(G_{\tu{sc}}(\A_F)\stackrel{\varrho}{\ra} G(\A_F))$. Similarly we have $G^*(\A_F)^{\flat}$, $G(F)^{\flat}$, and $G^*(F)^{\flat}$. Adapting the arguments of \S\ref{sub:transfer-one-dim} via $z$-extensions, we see that  $G(\A_F)^{\flat}$ is an abelian group and that there exists a canonical isomorphism $G(\A_F)^{\flat} \simeq G^*(\A_F)^{\flat}$ compatible with the isomorphism of Lemma \ref{lem:Gab=G*ab} at every place of $v$ and that the above isomorphism carries $G(F)^{\flat}$ onto $G^*(F)^{\flat}$. 

Again by taking a $z$-extension, we can assume that $G_{\tu{sc}}=G_{\der}$. It suffices to show that the inclusion $G_{\der}(F)G(\A_F)_{\der}\subset G_{\der}(\A_F)$ is an equality so that every one-dimensional automorphic representations of $G(\A_F)$ factors through $G(\A_F)^{\flat}$ (and likewise for $G^*$). Since $G(\A_F)_{\der}$ contains $G(F_v)_{\der}=G_{\der}(F_v)$ whenever $G$ is quasi-split over $F_v$ (Lemma \ref{lem:Gab-Gflat}), the desired equality follows from the strong approximation for $G_{\der}$.
\end{proof}

To state the next lemma, define a \textbf{(global) central character datum} to be a pair $(\fkX,\chi)$ as follows, where $\prod'_v$ means the restricted product over all places of $F$.
\begin{itemize}
\item $\fkX=\prod'_v \fkX_v$ is a closed subgroup of $Z(\A_F)$ such that  $Z(F)\fkX$ is closed in $Z(\A_F)$, and
\item $\chi=\prod_v \chi_v: \fkX\cap Z(F)\backslash\fkX\ra \C^\times$, with $\chi_v:\fkX_v\ra\C^\times$ a continuous character. (Implicitly for each $x=(x_v)\in \fkX$, we have $\chi_v(x_v)=1$ for almost all $v$, so that $\chi$ is well defined on $\fkX$.)
\end{itemize}

\begin{lemma}\label{lem:existence-1-dim}
Let $(\fkX,\chi)$ be a central character datum for $G$. Let $v$ be a finite place of $F$, and $g_v\in G(F_v)$ such that the image of $g_v$ in $G(F_v)^{\ab}$ is not contained in the image of $\fkX_v$. Then there exists a one-dimensional automorphic representation $\pi$ of $G(\A_F)$ with $\pi|_{\fkX}=\chi$ such that $\pi_v(g_v)\neq 1$.
\end{lemma}

\begin{proof}
Replacing $G$ with a $z$-extension and $(\fkX,\chi)$ with its pullback to the $z$-extension, we may assume that $G_{\der}=G_{\tu{sc}}$. Then we may replace $G$ with $G^{\ab}$ as $(\fkX,\chi)$ factors through a central character datum for $G^{\ab}$.
Thus we assume that $G=T$ is a torus. By assumption $g_v\in T(F_v)$ lies outside $\fkX_v$ and thus $g_v\notin T(F)\fkX$ in $T(\A_F)$, in which $g_v$ is an element via the obvious embedding $T(F_v)\hra T(\A_F)$. Thus the proof is complete by the fact (from Pontryagin duality) that, for every non-identity element $x$ in a locally compact Hausdorff abelian group $X$, there exists a unitary character of $X$ whose value is nontrivial at $x$. (Take $X=T(\A_F)/T(F)\fkX$ and $x=g_v$.)
\end{proof}

\subsection{Endoscopy with fixed central character}\label{sub:prelim-endoscopy}

Let $F$ be a local or global field of characteristic~$0$. Let $G$ be a connected reductive group over $F$ with an inner twisting $G_{\ol F}\isom G^*_{\ol F}$ with $G^*$ quasi-split over $F$. Let $\cE(G)$ (resp.~$\cE_{\el}(G)$) denote a set of representatives for isomorphism classes of endoscopic (resp.~elliptic endoscopic) data for $G$ as defined in \cite{LanglandsShelstad,KottwitzShelstad}. A member of $\cE(G)$ is represented by a quadruple $\fke=(G^\fke,\cG^\fke,s^\fke,\eta^\fke)$ consisting of a quasi-split group $G^\fke$, a split extension $\cG^\fke$ of $W_F$ by $\hat G^\fke$, $s^\fke\in Z_{G^{\fke}}$, and $\eta^\fke: \cG^\fke\hra {}^L G$ satisfying the conditions detailed in \emph{loc.~cit.} In particular $\fke^*:=(G^*,{}^L G^*, 1, \tu{id})\in \cE_{\el}(G)$. Write $\cE^<_{\el}(G):=\cE_{\el}(G)\backslash \{\fke^*\}$.

From now on, let $\fke\in \cE(G)$. Set
$$
\iota(G,G^\fke):=\tau(G)\tau(G^\fke)^{-1}\zeta(\fke)^{-1}\in \Q.
$$
Throughout \S\ref{sub:prelim-endoscopy}, we make the following assumption, which will be removed via $z$-extensions in the next subsection. (The assumption is known to be true if $\fke=\fke^*$, when it is evident, or if $G^{\tu{der}}$ is simply connected, by \cite[Prop.~1]{LanglandsStableConjugacy}.)
\begin{itemize}
\item (assumption) $~\cG^\fke={}^L G^\fke$.
\end{itemize}

For now we restrict to the case when $F$ is local. Let $\fke$ be as above. Consider a local central character datum $(\fkX,\chi)$ for $G$ as in \S\ref{sub:local-Hecke}. Let $\fkX^\fke\subset Z_{G^\fke}(F)$ denote the image of $\fkX$ under the canonical embedding $Z_G \hra Z_{G^\fke}$. Thus we can identify $\fkX=\fkX^\fke$. We say a semisimple element $\gamma^\fke\in G^\fke(F)$ is strongly $G$-regular if $\gamma^\fke$ corresponds to (the $G(\ol F)$-conjugacy class of) an element of $G(F)_{\tu{sr}}$ via the correspondence between semisimple conjugacy classes in $G^\fke(\ol F)$ and those in $G(\ol F)$ \cite[1.3]{LanglandsShelstad}. Write $G^\fke(F)_{G\tu{-sr}}\subset G^\fke(F)$ for the subset of strongly $G$-regular elements.

 Thanks to the proof of the transfer conjecture and the fundamental lemma \cite{WaldspurgerChangement,CluckersLoeser,NgoFL}, we know that each $f\in \cH(G(F))$ admits a \emph{transfer} $f^\fke\in \cH(G^\fke(F))$ whose stable orbital integrals on strongly $G$-regular semisimple elements are determined by the following formula, where the sum runs over strongly regular $G(F)$-conjugacy classes, and $\Delta(\cdot,\cdot)$ denotes the transfer factor as in \cite{LanglandsShelstad} (see the remark below on normalization).
\begin{equation}\label{eq:orbital-integral-id}
SO_{\gamma^\fke}(f^\fke)= \sum_{\gamma\in G(F)_{\tu{sr}}/\sim} \Delta(\gamma^\fke,\gamma) O_\gamma(f),\qquad \gamma^\fke\in G^\fke(F)_{G\tu{-sr}}.
\end{equation}
The assignment of $f^\fke$ to $f$ is not unique on the level of Hecke algebras, but \eqref{eq:orbital-integral-id} determines a well-defined map $\tu{LS}^\fke:\cI(G)\ra \cS(G^\fke)$.

The transfer satisfies an equivariance property. For each $z\in Z_{G}(F)\subset Z_{G^\fke}(F)$, define the translates $f_{z},f^\fke_{z}$ of $f,f^\fke$ by $f_{z}(g)=f(zg)$ and $f^\fke_{z}(h)=f^\fke(zh)$. The equivariance of transfer factors under translation by central elements (see \cite[Lem.~4.4.A]{LanglandsShelstad}) implies that $f^\fke_{z}$ is a transfer of $\lambda^\fke(z) f_{z}$ for a smooth character $\lambda^\fke:Z_{G}(F)\ra \C^\times$.
The character $\lambda^\fke$ is independent of $f^\fke$ and $f$, and its restriction $\lambda^\fke|_{Z_{G}^0(F)}$ can be described as follows. Consider the composite map
\begin{equation}\label{eq:Galois-char-of-lambda}
W_F \ra {}^L G^\fke \stackrel{\eta^\fke}{\ra} {}^L G \ra {}^L Z_{G}^0,
\end{equation}
where the last map is dual to the embedding $Z_{G}^0\hra G$. Then $\lambda^\fke|_{Z_{G}^0(F)}$ is the character of $Z_{G}^0(F)$ corresponding to the composite map above by \cite[Lem.~7.4.6]{KSZ}. Define a smooth character $\chi^\fke:\fkX^\fke\ra \C^\times$ by the relation
\begin{equation}\label{eq:endoscopic-central-char}
\chi^\fke(z)=\lambda^\fke(z)^{-1}\chi(z), \qquad z\in \fkX=\fkX^\fke.
\end{equation}
In light of the equivariance property above, the transfer map $\tu{LS}^\fke:\cI(G)\ra \cS(G^\fke)$ descends to
\begin{equation}\label{eq:LS}
\tu{LS}^\fke:\cI(G,\chi^{-1})\ra \cS(G^\fke,\chi^{\fke,-1})
\end{equation}
 via averaging, still denoted by $\tu{LS}^\fke$ for simplicity. The identity \eqref{eq:orbital-integral-id} still holds if $f^\fke=\tu{LS}^\fke(f)$ under \eqref{eq:LS}. In the special case of $\fke=\fke^*$ (so that $\chi^\fke=\chi$), we write $f^*\in \cH(G^*(F),\chi^{-1})$ for a transfer of $f\in \cH(G(F),\chi^{-1})$. If $\fkX=\{1\}$ then $f^*$ here coincides with the one in \S\ref{sub:transfer-one-dim}, noting that $e(G)$ in \eqref{eq:SO=SO} plays the role of transfer factor.

The fundamental lemma tells us the following. Assume that $G$ and $\fke$ are unramified; the latter means that $G^\fke$ is an unramified group and that the $L$-morphism $\eta^\fke$ is inflated from a morphism of $L$-groups with respect to an unramified extension of $F$. We fix pinnings for $G$ and $G^\fke$ defined over $F$, which determine hyperspecial subgroups $K\subset G(F)$ and $K^\fke\subset G^\fke(F)$ as in \cite[\S4.1]{Waldspurger-sitordue}. The Haar meausres on $G(F)$ and $G^\fke(F)$ are normalized to assign volume 1 to $K$ and $K^\fke$.
We also assume that $\chi$ is unramified, \ie~$\chi$ is trivial on $\fkX\cap K$. We normalize the transfer factors canonically as in \cite{LanglandsShelstad} (which is possible as $G$ is quasi-split). Then $\tu{LS}^\fke$ can be realized by a linear map of the unramified Hecke algebras (defined relative to $K$ and $K^\fke$)
$$
\xi^{\fke,*} \colon \cH^{\tu{ur}}(G(F),\chi^{-1}) \ra \cH^{\tu{ur}}(G^\fke(F), \chi^{\fke,-1}).
$$

We turn to the case of global field $F$. Recall that $Z$ is the center of $G$.
Let $(\fkX,\chi)$ be a global central character datum (\S\ref{sub:one-dim}).
As in the local case, we define $\fkX^\fke=\prod_v \fkX^\fke_v$ to be the image of $\fkX$ under the canonical embedding $Z_G(\A_F)\hra Z_{G^\fke}(\A_F)$.
We have $\chi^\fke := \prod_v\chi_{v}^\fke: \fkX^\fke\ra \C^\times$, where $\chi_{v}^\fke$ was given by the local consideration above, so that functions in $\cH(G(\A_F),\chi^{-1})$ transfer to those in $\cH(G^\fke(\A_F),(\chi^\fke)^{-1})$. Denote by $\lambda^\fke=\prod_v \lambda^\fke_v:Z_{G}(F)\backslash Z_{G}(\A_F)\ra \C^\times$ the character with $\lambda^\fke_v$ as in the local context above. (The $Z_{G}(F)$-invariance of $\lambda^\fke$ follows from the equivariance of transfer factors \cite[Lem.~4.4.A]{LanglandsShelstad} and the product formula \cite[Cor.~6.4.B]{LanglandsShelstad}.)  The restriction of $\lambda$ to $Z_{G}^0(\A_F)$ corresponds to the composite map \eqref{eq:Galois-char-of-lambda} (with $F$ now global). There is an equality $\chi^\fke=\lambda^{\fke,-1} \chi$ as characters on $\fkX=\fkX^\fke$ as in \eqref{eq:endoscopic-central-char} since this holds at every place of $F$. In particular $\chi^\fke$ is trivial on $Z_{G}(F)\cap \fkX^\fke$, and $(\fkX^\fke,\chi^\fke)$ is a central character datum for $G^\fke$.

\begin{remark}\label{rem:transfer-factors}
The local transfer factors are well defined only up to a nonzero scalar (unless $G$ is quasi-split or $G^\fke=G^*$, if no further choices are made), so we always choose a normalization implicitly, for instance throughout \S\ref{sec:Jacquet-endoscopy}. Scaling the transfer factor results in scaling the transfer map \eqref{eq:LS}. However, according to \cite[\S6.4]{LanglandsShelstad}, we may and will choose a normalization at each place such that the product of local transfer factors over all places \emph{is} the canonical global transfer factor. This will not introduce ambiguity in our main argument as it takes place in the global context.
\end{remark}

It simplifies some later arguments if $\fke$ is chosen to enjoy a boundedness property. We say that a subgroup of $^L G=\hat G\rtimes W_F$ is \emph{bounded} if its projection to $\hat G \rtimes \Gal(E/F)$ is contained in a compact subgroup for some (thus every) finite Galois extension $E/F$ containing the splitting field of $G$.

\begin{lemma}\label{lem:bounded-representative}
In either local or global case, we can choose the representative $\fke=(G^\fke,\cG^\fke,s^\fke,\eta^\fke)$ in its isomorphism class to satisfy the following condition: $\eta^\fke(W_F)$ is a bounded subgroup of $^L G$. (We restrict $\eta^\fke$ to $W_F$ via the splitting $W_F\ra \cG^\fke$ built into the data.)
\end{lemma}

\begin{remark}
Bergeron--Clozel \cite[Lem.~3.7]{BC17} proved a similar lemma when $F=\R$.
\end{remark}

\begin{proof}
Since $\eta^\fke|_{\hat G^\fke}$ will be fixed throughout, we use it to identify $\hat G^\fke$ with a subgroup of $\hat G$. We take the convention that all cocycles/cohomology below are continuous cocycles/cohomology.

It suffices to show that there exists an $L$-morphism $\eta^\fke_0: \cG^\fke \ra {}^L G$ extending $\eta^\fke|_{\hat G^\fke}$ such that $\eta^\fke_0(W_F)$ is bounded. Indeed, $\fke_0=(G^\fke,\cG^\fke,s^\fke,\eta^\fke_0)$ is then the desired representative.

To prove the existence of $\eta^\fke_0$ as above, we reduce to the case that $G_{\der}=G_{\tu{sc}}$ and that $\cG^\fke={}^L G^\fke$ via a $z$-extension. (In the notation of \S\ref{sub:endoscopy-z-ext} below, the idea is to multiply $\eta_1^\fke$ by a suitable 1-cocycle $c:W_F \ra Z(\hat G_1)$ to make the image of $(c\cdot \eta_1^\fke)_{W_F}$ contained in $^L G$ and still bounded.) 
In the case that $G_{\der}=G_{\tu{sc}}$ and $\cG^\fke={}^L G^\fke$, our approach is to refine the proof of \cite[Prop.~1]{LanglandsStableConjugacy}, where Langlands shows that $\eta^\fke|_{\hat G^\fke}$ extends to an $L$-morphism $\eta^\fke_0:{}^L G^\fke \ra {}^L G$ under the hypothesis but without guaranteeing boundedness of image. 
To construct $\eta^\fke_0$ (denoted $\xi$ therein), Langlands reduces to the elliptic endoscopic case, chooses a sufficiently large finite extension $K/F$, and then constructs $\xi':W_{K/F}\ra {}^L G$ such that $\eta^\fke_0(g\rtimes w):=\eta^\fke(g)\xi'(w)$ gives the desired $L$-morphism. (In the current proof, we follow Langlands to use the Weil group $W_{K/F}$ to form the $L$-group, \ie $^L G=\hat G\rtimes W_{K/F}$.)
It is enough to arrange that $\xi'$ has bounded image in Langlands's construction. 

Write $\hat N$ for the normalizer of $\hat T$ (which is $^L T^0$ in \emph{loc.~cit.}) in $\hat G$. Let $\hat N_c$ (resp. $Z(\hat G^\fke)_c$) denote the maximal compact subgroup of $\hat N$ (resp. $Z(\hat G^\fke)$). The starting point is a set-theoretic map $\xi':W_{K/F}\ra {}^L G$ satisfying the second displayed formula on p.709 therein. Such a $\xi'$ is chosen using the Langlands--Shelstad representative of each Weyl group element $\omega$, denoted by $n(\omega)\in \hat N$ in \cite[\S2.1]{LanglandsShelstad}. (The point is that the $\sigma$-action $\omega_{T/G}(\sigma)$ on $^L T^0$ and the action $\omega^1(\sigma)$ differ by the Weyl action $\omega^2(\sigma)$ in his notation. See the seventh displayed formula on p.703.) In fact $n(\omega)\in \hat N_c$ since it is a product of finite-order elements in $\hat N$. Thereby $\xi'$ has image in $\hat N_c \rtimes W_{K/F}$ (thus bounded). It follows that the 2-cocycle of $W_{K/F}$ given by
$$
a_{w_1,w_2}=\xi'(w_1)\xi'(w_2)\xi'(w_1w_2)^{-1}
$$
has values in $Z(\hat G^\fke)_c$ (not just $Z(\hat G^\fke)$ as in \cite[p.709]{LanglandsStableConjugacy}). We need to verify the claim that the 2-cocycle is trivial in $H^2(W_{K/F},Z(\hat G^\fke)_c)$; then $\xi'$ can be made a homomorphism after multiplying a $Z(\hat G^\fke)_c$-valued 1-cocycle, keeping its image bounded, so we will be done. In fact, thanks to Lemma~2 therein (stated for $Z(\hat G^\fke)$ but also applicable to $Z(\hat G^\fke)_c$ since both groups have the same group of connected components), we may assume that $a_{w_1,w_2}\in (Z(\hat G^\fke)_c)^0$. Then the claim follows from a variant of Lemma 4 therein, with $S$ replaced by the maximal compact subtorus in the statement and proof. (In particular, the map (1) on p.719 is still surjective if $S_1$ and $S_2$ are replaced with their maximal compact subtori, by considering unitary characters.) 
\end{proof}

\subsection{Endoscopy and $z$-extensions}\label{sub:endoscopy-z-ext}

Here we explain a general endoscopic transfer with fixed central character by removing the assumption that $\cG^\fke={}^L G^\fke$ in \S\ref{sub:TF-fixed-central} via $z$-extensions. For the time being, let the base field $F$ of $G$ be either local or global. Fix a $z$-extension over $F$
$$
1\ra Z_1\ra G_1\ra G\ra 1.
$$
Let $\fke=(G^\fke,\cG^\fke,s^\fke,\eta^\fke)\in \cE^<_{\el}(G)$. As explained in \cite[\S4.4]{LanglandsShelstad} (see also \cite[\S7.2.6]{KSZ}), we have a central extension
$$
1\ra Z_1  \ra G^\fke_1 \ra G^\fke \ra 1,
$$
and $\fke$ can be promoted to an endoscopic datum $\fke_1=(G_1^\fke,{}^L G_1^\fke,s_1^\fke,\eta^\fke_1)$ for $G_1$ such that $\eta^\fke_1:{}^L G_1^\fke \hra {} ^L G$ extends $\eta^\fke:\cG^\fke\hra {}^L G$. Moreover, changing $\fke_1$ and $\fke$ in their isomorphism classes if necessary, we may ensure that $\eta^\fke_1(W_F)$ and $\eta^\fke(W_F)$ are bounded subgroups in $^L G_1$ and $^L G$, respectively. Indeed, this is done in the course of proof of Lemma \ref{lem:bounded-representative} in the general case. Write $\fkX_1$ (resp.~$\fkX^\fke_1$) for the preimage of $\fkX$ in $G_1$ (resp.~$G^\fke_1$), and $\chi_1:\fkX_1\ra\C^\times$ for the character pulled back from $\chi$.

To describe endoscopic transfers, it is enough to work locally, so let $F$ be a local field. Applying \S\ref{sub:TF-fixed-central} to $G_1$ and $\fke_1$ in place of $G$ and $\fke$, we obtain an identification $\fkX^\fke_1 = \fkX_1$ under the canonical embedding $Z_{G_1}\hra Z_{G^\fke_1}$ as well as characters $\lambda^\fke_1:Z_{G_1}(F)\ra\C^\times$ and $\chi^\fke_1:\fkX^\fke_1=\fkX_1\ra\C^\times$ such that $\chi^\fke_1=\lambda^{\fke,-1}_1 \chi_1$ as characters on $\fkX^\fke_1 = \fkX_1$. Again $\lambda^\fke_1|_{Z_{G_1}^0(F)}$ corresponds to the parameter \eqref{eq:Galois-char-of-lambda} (with $G^\fke_1,G_1$ replacing $G^\fke,G$). We also have a transfer
$$
\tu{LS}^\fke: \cI(G,\chi^{-1})=\cI(G_1,\chi_1^{-1}) \stackrel{\tu{\eqref{eq:LS}}}{\ra} \cS(G^\fke_1,\chi^{\fke,-1}),
$$
where the equality is induced by $G_1(F)\twoheadrightarrow G(F)$.

\subsection{The trace formula with fixed central character}\label{sub:TF-fixed-central}

In this subsection, $G$ is a connected reductive group over $\Q$. Let $A_G$ denote the maximal $\Q$-split torus in $Z_{G}$, and $A_{G_\R}$ denote the maximal $\R$-split torus in $Z_{G_\R}$. Put
$$
A_{G,\infty}:=A_G(\R)^0, \qquad A_{G_\R,\infty}:=A_{G_{\R}}(\R)^0.
$$
Let $\chi_0:A_{G,\infty}\ra\C^\times$ denote a continuous character. By $L^2_{\disc,\chi_0}(G(\Q)\backslash G(\A))$ we mean the discrete spectrum in the space of square-integrable functions (modulo $A_{G,\infty}$) on $G(\Q)\backslash G(\A)$ which transforms under $A_{G,\infty}$ by $\chi_0$.
Let $(\fkX=\prod_v \fkX_v,\chi=\prod_v \chi_v)$ be a central character datum as in \S\ref{sub:one-dim}. Henceforth we always assume that
$$
A_{G,\infty}\subset \fkX_\infty.
$$
We can define $L^2_{\disc,\chi}(G(\Q)\backslash G(\A))$ in the same way as $L^2_{\disc,\chi_0}(G(\Q)\backslash G(\A))$. Let $\mathcal{A}_{\disc,\chi}(G)$ stand for the set of isomorphism classes of irreducible $G(\A)$-subrepresentations in $L^2_{\disc,\chi}(G(\Q)\bs G(\A))$. The multiplicity of $\pi\in \mathcal{A}_{\disc,\chi}(G)$ in $L^2_{\disc,\chi}(G(\Q)\bs G(\A))$ is denoted $m(\pi)$.

Define $\cH(G(\A),\chi^{-1}):=\otimes'_v \cH(G(\Q_v),\chi_v^{-1})$ as a restricted tensor product. Each $f\in \cH(G(\A),\chi^{-1})$ defines a trace class operator, yielding the discrete part of the trace formula:
\begin{equation}\label{eq:DiscretePartOfTF_chi}
T^G_{\disc,\chi}(f) :=  \Tr \left(f \,|\, L^2_{\disc,\chi}(G(\Q)\bs G(\A))\right)
= \sum_{\pi\in \mathcal{A}_{\disc,\chi}(G)} m(\pi)\Tr (f|\pi).
\end{equation}

Fix a minimal $\Q$-rational Levi subgroup $M_0\subset G$. Write $\cL$ for the set of $\Q$-rational Levi subgroups of $G$ containing $M_0$. Define the subset $\cL_{\tu{cusp}}\subset \cL$ of \emph{relatively cuspidal} Levi subgroups; by definition, $M\in \cL$ belongs to $\cL_{\tu{cusp}}$ exactly when the natural map $A_{M,\infty}/A_{G,\infty} \ra A_{M_\R,\infty}/A_{G_\R,\infty}$ is an isomorphism. Let $M\in \cL$ and $\gamma\in M(\Q)$ be a semisimple element. Write $M_\gamma$ for the centralizer of $\gamma$ in $M$, and $I_\gamma^M:=(M_\gamma)^0$ for the identity component. Write $\iota^M(\gamma)\in \Z_{\ge 1}$ for the number of connected components of $M_\gamma$ containing $\Q$-points. Write $|\Omega^M|$ for the order of the Weyl group of $M$. For $\gamma\in M(\Q)$, let $\tu{Stab}^M_{\fkX}(\gamma)$ denote the set of $x\in \fkX$ such that $\gamma$ and $x\gamma$ are $M(\Q)$-conjugate. Note that $\tu{Stab}^M_{\fkX}(\gamma)$ is necessarily finite (by reducing to the case of general linear groups via a faithful representation). When $M=G$, we often omit $M$, e.g.,~$I_\gamma=I_\gamma^G$ and $\iota(\gamma)=\iota^G(\gamma)$.

Fix Tamagawa measures on $M(\A)$ and $I^M_\gamma(\A)$ for $M\in \cL_{\tu{cusp}}$ and fix their decomposition into Haar measures on $M(\A^\infty)$ and $M(\R)$ (resp. $I^M_\gamma(\A^\infty)$ and $I^M_\gamma(\R)$). This determines a measure on the quotient $I^M_\gamma(\A)\backslash M(\A)$, which is used to define the ad\`elic orbital integral at $\gamma$ in $M$, and similarly over finite-ad\`elic groups. We also fix Haar measures on $\fkX$ and $\fkX_\infty$. We equip $I^M_\gamma(\Q)$ and $\fkX_\Q:=\fkX\cap Z(\Q)$ with the counting measures and $A_G(\R)^0$ with the multiplicative Lebesgue measure. Thereby we have quotient measures on $I^M_\gamma(\Q)\bs I^M_\gamma(\A)/\fkX$, $\fkX_\Q\backslash \fkX/A_G(\R)^0$, and $\fkX_\infty/A_G(\R)^0$.

We define the elliptic part of the trace formula as
\begin{eqnarray}\label{eq:EllipticPartOfTF}
T^G_{\el,\chi}(f):=\sum_{\gamma\in \Gamma_{\el,\fkX}(G)}|\textup{Stab}_{\mathfrak X}^G(\gamma)|^{-1} \iota(\gamma)^{-1} \vol(I_\gamma(\Q)\bs I_\gamma(\A)/\fkX) O_{\gamma}(f), \quad f\in \cH(G(\A),\chi^{-1}),
\end{eqnarray}
where $\Gamma_{\el, \fkX}(G)$ is the set of $\fkX$-orbits of elliptic conjugacy classes of $G$.

Now we assume that $G_{\R}$ contains an elliptic maximal torus. Let $\xi$ be an irreducible algebraic representation of $G_\C$, and $\zeta : G(\R)\ra \C^\times$ a continuous character. Let $M\in \cL_{\tu{cusp}}$ and $T_\infty$ an $\R$-elliptic torus in $M$. Arthur introduced the function $\Phi_M(\gamma,\xi)$ in $\gamma\in T_\infty(\R)$ in \cite[(4.4), Lem.~4.2]{ArthurL2}. (See Lemma \ref{lem:TF-technical-lemma} below for a concrete description.) 

Let $\gamma\in M(\Q)$ and suppose that $\gamma$ is elliptic in $M(\R)$. Denote by $I_\gamma^{M,\tu{cpt}}$ a compact-mod-center inner form of $(I_\gamma^M)_{\R}$. We choose a Haar measure on $I_\gamma^{M,\tu{cpt}}(\R)$ compatibly with that on $I_\gamma^M(\R)$. Write $q(I_\gamma)\in \Z_{\ge 0}$ for the real dimension of the symmetric space associated with the adjoint group of $(I_\gamma^M)_{\R}$. Following \cite[(6.3)]{ArthurL2}, define
\begin{equation}\label{eq:chi-I-gamma-M}
\chi(I_\gamma^M):=(-1)^{q(I_\gamma)}\tau(I_\gamma^M)\tu{vol}(A_{I_\gamma^M,\infty}\backslash I_\gamma^{M,\tu{cpt}}(\R))^{-1}d(I_\gamma^M).
\end{equation}
For $f^\infty\in \cH(G(\A^\infty),(\chi^{\infty})^{-1})$, let $f^\infty_M\in \cH(M(\A^\infty),(\chi^{\infty})^{-1})$ denote the constant term, cf.~\S\ref{sub:nu-constant-terms} and \S\ref{sub:fixed-central-character} below.
Dalal extended Arthur's Lefschetz number formula \cite[Thm.~6.1]{ArthurL2} to the setting with fixed central characters. The condition on $\fkX$ is imposed below because it is also in \cite{Dalal}. It is a harmless condition that is satisfied in our main setup, but we expect it to be superfluous.

\begin{proposition}\label{prop:DalalsFormula}
Assume that $\fkX=Y(\A) A_{G,\infty}$ for a central torus $Y\subset Z_G$ over $\Q$. Let $\xi,\zeta$ be as above. Then for each $f^\infty\in \cH(G(\A^\infty),(\chi^{\infty})^{-1})$, 
$$
T^G_{\tu{disc},\chi}(f_{\xi,\zeta} f^\infty)=
\frac{1}{d(G)}
\sum_{M\in \mathcal L_{\tu{cusp}}} 
\frac{(-1)^{\dim (A_M/A_G)} }{\tu{vol}(\fkX_\Q\backslash \fkX/A_{G,\infty})}
\frac{|\Omega^M|}{|\Omega^G|} 
\sum_\gamma\frac{\chi(I^M_\gamma)\zeta(\gamma)
\Phi_M(\gamma,\xi)O^M_\gamma(f^\infty_M)}{\iota^M(\gamma)\cdot |\tu{Stab}^M_\fkX(\gamma)|} ,
$$
where the second sum runs over the $\fkX$-orbits on the set of $\R$-elliptic conjugacy classes of $M(\Q)$.
\end{proposition}
\begin{proof}
This is \cite[Cor.~6.5.1]{Dalal}.
\end{proof}

\subsection{The stable trace formula}\label{sub:STF}
Let $H$ be a quasi-split group over $\Q$. Let $(\fkX_H,\chi_H)$ be a central character datum for $H$. Write $\Sigma_{\tu{ell},\chi_H}(H)$ for the set of stable elliptic conjugacy classes in $H(\Q)$ modulo $ \fkX_H$, namely two stable conjugacy classes are equivalent if one is mapped to the other by mulitpying an element $x\in \fkX_H$. Following \cite[\S8.3.7]{KSZ}, define
$$
ST^H_{\tu{ell},\chi_H}(h):=\tau_{\fkX_H}(H) \sum_{\gamma_H\in \Sigma_{\tu{ell},\chi_H}(H) } |\tu{Stab}_{\fkX_H}(\gamma_H)|^{-1} SO^{H(\A)}_{\gamma_H}(h),\qquad  h\in \cH(H(\A),\chi_H^{-1}).
$$

Consider a central character datum $(\fkX,\chi)$ for $G$ as well as $f=\otimes'_v f_v\in \cH(G(\A),\chi^{-1})$. For each $\fke\in \cE^<_{\tu{ell}}(G)$, we have $\fke_1$ and a central character datum $(\fkX^\fke_1,\chi^\fke_1)$ (whose $v$-components are given as in the preceding subsection). Write $f^\fke_{1,v}\in \cH(G^\fke_1(\A),(\chi^\fke_1)^{-1})$ for a transfer of $f_v$ at each $v$. Put $f^\fke_1:=\otimes'_v f^\fke_{1,v}$. For $\fke=\fke^*$, we transfer $f$ to $f^*\in \cH(G^*(\A),\chi^{-1})$ as in \S\ref{sub:prelim-endoscopy}.

\begin{proposition}\label{prop:stabilization-elliptic}
 Let $f=\otimes'_v f_v\in \cH(G(\A),\chi^{-1})$. Assume that there exists a finite place $q$ such that $O_g(f_q)=0$ for every $g\in G(\Q_q)_{\tu{ss}}$ that is not regular. With $f^*$ and $f^\fke_1$ as above, we have
$$
T^G_{\el,\chi}(f) = ST^{G^*}_{\tu{ell},\chi}(f^*)+ \sum_{\fke \in \cE^<_{\tu{ell}}(G)} \iota(G,G^{\fke}) ST^{G^\fke_1(\A)}_{\tu{ell},\chi^\fke_1}(f_1^\fke).$$
\end{proposition}

\begin{proof}
By hypothesis, the stable orbital integral of $f_{1,q}^\fke$ vanishes outside $G$-regular semisimple conjugacy classes. When the central character datum is trivial, the stabilization of regular elliptic terms is due to Langlands \cite{LanglandsDebuts}, cf.~\cite[Thm.~9.6]{KottwitzEllipticSingular}, \cite[\S7.4]{KottwitzShelstad}, \cite{LabesseStabilization}.
For general central character data, the argument is essentially the same if one uses the Langlands--Shelstad transfer with fixed central character as in \S\ref{sub:prelim-endoscopy}. 
\end{proof}

The following finiteness result is going to tell us that the sum in Theorem \ref{thm:stable-Igusa} (and a similar sum in Theorem \ref{thm:Hc(Ig)} below) is finite for each choice of $\phi^{\infty,p}$.

\begin{lemma}\label{lem:finiteness} The following are true.
\begin{enumerate}
\item Let $v$ be a rational prime such that $G_{\Q_v}$ and $\chi_v$ are unramified. Let $f_v\in \cH^{\tu{ur}}(G(\Q_v),\chi_v^{-1})$. Then $f_v$ transfers to the zero function on $G_1^{\fke}(\Q_v)$ for each $\fke=(G^\fke,\mathcal G^\fke,s^\fke,\eta^\fke)\in \cE^<_{\el}(G)$ if $G^\fke$ is ramified over $\Q_v$.
\item Let $S$ be a finite set of rational primes. The set of  $\fke\in \cE^<_{\el}(G)$ such that $G^\fke_{\Q_v}$ is unramified at every rational prime $v\notin S$ is finite.
\end{enumerate}
\end{lemma}

\begin{proof}
Part (1) follows from \cite[Prop.~7.5]{KottwitzEllipticSingular}. Part (2) is well known; see \cite[Lem.~8.12]{LanglandsDebuts}.
\end{proof}
\section{Jacquet modules, regular functions, and endoscopy}\label{sec:Jacquet-endoscopy}

Throughout this section, let $F$ be a finite extension of $\Q_p$ with a uniformizer $\varpi$ and residue field cardinality $q$. The valuation on $F$ is normalized such that $|\varpi|=q^{-1}$. Let $G$ be a connected reductive group over $F$. We study how certain maps of invariant or stable distributions between $G$ and its Levi subgroups interact with Jacquet modules and endoscopy, based on \cite{ShinIgusaStab,XuCuspidalSupport}.

\subsection{$\nu$-ascent and Jacquet modules}\label{sub:nu-Jacquet}

Let $\nu:\G_m\ra G$ be a cocharacter defined over $F$. Let $M_\nu$ denote the centralizer of $\nu$ in $G$, which is an $F$-rational Levi subgroup. The maximal $F$-split torus in the center of $M_\nu$ is denoted by $A_{M_\nu}$.

Write $P_\nu$ (resp.~$P_\nu^{\tu{op}}$) for the $F$-rational parabolic subgroup of $G$ which contains $M_\nu$ as a Levi component and such that $\langle \alpha,\nu\rangle<0$ (resp. $\langle \alpha,\nu\rangle>0$) for every root $\alpha$ of $A_{M_\nu}$ in $P_\nu$ (resp.~$P_\nu^{\tu{op}}$). The set of $\alpha$ as such is denoted by $\Phi^+(P_{\nu})$ (resp.~$\Phi^+(P^{\tu{op}}_{\nu})$). Let $N_\nu,N_\nu^{\tu{op}}$ denote the unipotent radical of $P_\nu,P_\nu^{\tu{op}}$. For every $\alpha\in \Phi^+(P^{\tu{op}}_{\nu})$, we have $|\alpha(\nu(\varpi))|=q^{-\langle \alpha,\nu\rangle}<1$. Therefore $\nu(\varpi)\in  A_{P^{\tu{op}}_\nu}^{--}$. The following definition is a rephrase of \cite[Def.~3.1]{ShinIgusaStab}. 

\begin{definition}\label{def:acceptable}
We say that $\gamma\in M_\nu(\ol F)$ is \textbf{acceptable} (with respect to $\nu$) if the action of $\tu{Ad}(\gamma)$ on $(\Lie N^{\tu{op}}_\nu)_{\ol F}$ is contracting, \ie~all its eigenvalues $\lambda\in \ol F$ have the property that $|\lambda|<1$.
\end{definition}

By definition, $a \in A_{M_\nu}(F)$ is acceptable if and only if $a \in A_{P^{\tu{op}}_\nu}^{--}$. The subset of acceptable elements is nonempty, open, and stable under $M_\nu(\ol F)$-conjugacy. Define $\cH_{\tu{acc}}(M_\nu)\subset \cH(M_\nu)$ as the subspace of functions supported on acceptable elements. We also write $\cH_{\nu\tu{-acc}}(M_\nu)$ to emphasize the dependence on $\nu$. As in \S\ref{sub:local-Hecke} we often omit $F$ for simplicity.

\begin{lemma}\label{lem:O-tr-nu-ascent}
Let $\phi\in \cH_{\tu{acc}}(M_\nu)$. There exists $f \in \cH(G)$ with the following properties.
\begin{enumerate}
\item For every $g\in G(F)_{\tu{ss}}$,
$$O^{G}_g (f) = \delta_{P_\nu}(m)^{-1/2} O^{M_\nu}_m(\phi)$$
if there exists an acceptable $m\in M_\nu(F)$ which is conjugate to $g$ in $G(F)$
(in which case $m$ is unique up to $M_\nu(F)$-conjugacy, and the Haar measures are chosen compatibly on the connected centralizers of $m$ and $g$), and $O^{G}_g (f) =0$ otherwise.
\item $\Tr (f|\pi)=\Tr \left( \phi | J_{P_\nu^{\tu{op}}}(\pi) \right)$ for $ \pi\in \Irr(G(F))$.
\end{enumerate}
\end{lemma}
\begin{proof}
This is \cite[Lem.~3.9]{ShinIgusaStab} except that we corrected typos in the statement. The same proof still works with two remarks. Firstly, we removed the assumption in \emph{loc.~cit.}~that orbital integrals of $\phi$ vanish on semisimple elements with disconnected centralizers. This is possible by reducing to the case of $G$ with simply connected derived subgroup (then $M_{\nu,\der}$ is also simply connected by Lemma \ref{lem:Levi-of-simply-connected}) so that the centralizers of semisimple elements are connected in both $M_\nu$ and $G$. Secondly, the mistake in \emph{loc.~cit.}~occurs in line 1, p.806, where it should read $\phi^0:=\phi\cdot \delta^{-1/2}_{P_\nu}$.
\end{proof}

\begin{corollary}\label{cor:SO-nu-ascent}
Let $\phi$ and $f$ be as in Lemma \ref{lem:O-tr-nu-ascent}. For every $g\in G(F)_{\tu{ss}}$, 
$$ 
SO^{G}_g (f) = \delta_{P_\nu}(m)^{-1/2} SO^{M_\nu}_m(\phi)
$$
if there exists an acceptable $m\in M_\nu(F)$ which is conjugate to $g$ in $G(F)$. Otherwise, $SO^{G}_g (f) =0$.
\end{corollary}

\begin{proof}
This is clear from the preceding lemma, using \cite[Lem.~3.5]{ShinIgusaStab}.
 \end{proof}

\begin{definition}
In the setup of Lemma \ref{lem:O-tr-nu-ascent}, we say that $f$ is a \textbf{$\nu$-ascent} of $\phi$.
\end{definition}

Recall the definition of $\cI(\cdot)$ and the trace Paley--Wiener theorem from \S\ref{sub:local-Hecke}. According to  \cite[Prop.~3.2]{BDK86}, the Jacquet module induces the map
\begin{equation}\label{eq:J-nu-invariant}
\mathscr{J}_\nu \colon \cI(M_\nu) \ra \cI(G),\qquad \mathcal F \mapsto \left(\pi\mapsto \mathcal F(J_{P_\nu^{\tu{op}}}(\pi))\right).
\end{equation}
Write $\cI_{\tu{acc}}(M_\nu)$ for the image of $\cH_{\tu{acc}}(M_\nu)$ in $\cI(M_\nu)$. Then Lemma \ref{lem:O-tr-nu-ascent} means that, when $\phi\in \cI_{\tu{acc}}(M_\nu)$, a $\nu$-ascent of $\phi$ is well defined as an element of $\cI(G)$, which is nothing but $\mathscr{J}_\nu(\phi)$. The lemma yields extra information on orbital integrals. Xu \cite[Prop.~C.4]{XuCuspidalSupport} shows that \eqref{eq:J-nu-invariant} induces a similar map for the stable analogues, which we denote by the same symbol:
\begin{equation}\label{eq:J-nu-stable}
\mathscr{J}_\nu:\cS(M_\nu) \ra \cS(G).
\end{equation}

Write $X^*_F(G)$ for the group of $F$-rational characters of $G$. Put $X^*_F(G)_{\Q}:=X^*_F(G)\otimes_\Z \Q$ and $\fka_G:=\Hom(X^*_F(G)_{\Q},\R)$.
We have the map
$$
H^G \colon G(F)\ra \fka_G,\qquad g \mapsto (\chi\mapsto \log |\chi(g)|).
$$
It is easy to see that $H^G$ is invariant under $G(\ol F)$-conjugacy. Indeed, if $g_1,g_2$ become conjugate in $G(F')$ for a finite extension $F'/F$, then it boils down to the obvious fact that $H^{G'}(g_1)=H^{G'}(g_2)$, since the map $H^G$ is functorial with respect to $G\hra G':=\Res_{F'/F}G$.

For $f\in \cH(G)$, define the following subsets of $\fka_G$:
\begin{align}\label{eq:supp-def}
\tu{supp}_{\fka_G}(f) &:=\{H^G(x): x\in G(F)_{\tu{ss}}~\mbox{s.t.}~f(x)\neq 0\},\nonumber\\
\tu{supp}^O_{\fka_G}(f)&:=\{H^G(x): x\in G(F)_{\tu{ss}}~\mbox{s.t.}~O_x(f)\neq 0\},\\
\tu{supp}^{SO}_{\fka_G}(f)&:=\{H^G(x): x\in G(F)_{\tu{ss}}~\mbox{s.t.}~SO_x(f)\neq 0\}.\nonumber
\end{align}
Obviously $\tu{supp}^{SO}_{\fka_G}(f)\subset \tu{supp}^O_{\fka_G}(f) \subset \tu{supp}_{\fka_G}(f).$ Writing 
$$
\mathscr P(*):=\mbox{collection of subsets of}~*,
$$
we obtain a map $\tu{supp}_{\fka_G}$ (resp.~$\tu{supp}^O_{\fka_G}$, $\tu{supp}^{SO}_{\fka_G}$) from $\cH(G)$ (resp.~$\cI(G)$, $\cS(G)$) to $\mathscr P(\fka_G)$. 

We define analogous objects for $M_\nu$ in place of $G$. The injective restriction map $X^*_F(G)_{\Q}\ra X^*_F(M_\nu)_{\Q}$ induces a canonical surjection
\begin{equation}\label{eq:pr-G}
\tu{pr}_G:\fka_{M_\nu}\ra \fka_G.
\end{equation}
Set $\fka_{P_\nu}:=\fka_{M_\nu}$ and identify $X_*(A_{M_\nu})_{\R}\simeq \fka_{P_\nu}$ by $\mu\in X_*(A_{M_\nu})\mapsto (\chi \mapsto \langle \chi,\mu\rangle)$. Then it is an easy exercise to describe $\tu{pr}_G$ as the average map along Weyl orbits: if $T$ is a maximal torus of $M_\nu$ (thus also of $G$) over $F$, and if the Weyl group is taken relative to $T$, then
\begin{equation}\label{eq:pr=Weyl-orbit}
\tu{pr}_G(\mu)=|\Omega^G|^{-1}\sum_{\omega\in \Omega^G} \omega(\mu) =|\ol\Omega^G|^{-1}\sum_{\omega\in \ol\Omega^G} \omega(\mu),\qquad \mu\in X_*(A_{M_\nu})_{\R}.
\end{equation}

\begin{lemma}\label{lem:support-open-dense-enough}
The sets $\tu{supp}_{\fka_G}(f)$, $\tu{supp}^O_{\fka_G}(f)$, and $\tu{supp}^{SO}_{\fka_G}(f)$ remain unchanged if we restrict $x$  in the definition \eqref{eq:supp-def} to a subset $D\subset G(F)_{\tu{reg}}$ that is open dense in $G(F)$.
\end{lemma}

\begin{proof}
Since the map $H^G$ is continuous with discrete image, for each $y$ in $\tu{supp}_{\fka_G}(f)$, $\tu{supp}^O_{\fka_G}(f)$, or $\tu{supp}^{SO}_{\fka_G}(f)$, the preimage $(H^G)^{-1}(y)$ is open and closed. If $y\in\tu{supp}_{\fka_G}(f)$ then $\supp_{\fka_G}(f)\cap (H^G)^{-1}(y)$ is nonempty open in $G(F)$ thus intersects $D$. This proves the assertion for $\tu{supp}_{\fka_G}(f)$.

Next let $y\in \tu{supp}^O_{\fka_G}(f)$. Then $(H^G)^{-1}(y)\cap D$ is open dense in $(H^G)^{-1}(y)$. If $O_x(f)=0$ for every $x\in (H^G)^{-1}(y)\cap D$, we claim that 
$$
O_x(f)=0 ,\qquad x\in (H^G)^{-1}(y)\cap G(F)_{\tu{ss}}.
$$
If $x$ is regular, this follows from local constancy of $O_x(f)$ on regular elements. A Shalika germ argument then proves $O_x(f)=0$ for non-regular semisimple $x$. (Compare with the proof of Lemma \ref{lem:represent-by-C-regular} (1) below.) However, the claim contradicts $y\in \tu{supp}^O_{\fka_G}(f)$. The lemma for $\tu{supp}^O_{\fka_G}(f)$ follows. Finally, the stable analogue is proved likewise.
\end{proof}

\begin{lemma}\label{lem:nu-ascent-support}
The following diagrams commute.
$$\xymatrix{
  \cI_{\tu{acc}}(M_\nu)\ar[r]^-{\mathscr{J}_\nu} \ar[d]_-{\tu{supp}^O_{\fka_{M_\nu}}}  & \cI(G) \ar[d]^-{\tu{supp}^O_{\fka_G}} \\
    \mathscr{P}(\fka_{P_\nu}) \ar[r]^-{\tu{pr}_G}   & \mathscr{P}(\fka_G)
  }\qquad
  \xymatrix{
    \cS_{\tu{acc}}(M_\nu)\ar[r]^-{\mathscr{J}_\nu} \ar[d]_-{\tu{supp}^{SO}_{\fka_{M_\nu}}}  & \cS(G) \ar[d]^-{\tu{supp}^{SO}_{\fka_G}} \\
    \mathscr{P}(\fka_{P_\nu}) \ar[r]^-{\tu{pr}_G}   & \mathscr{P}(\fka_G)
  }
$$
\end{lemma}
\begin{proof}
This follows from Lemma \ref{lem:O-tr-nu-ascent} and Corollary \ref{cor:SO-nu-ascent} since, for each $m\in M_{\nu}(F)$, the canonical map $\fka_{M_\nu}\ra \fka_G$ sends $H^{M_\nu}(m)$ to $H^G(m)$.
\end{proof}

Let $k\in \Z$ and $\phi\in \cH(M_\nu)$. Define $\phi^{(k)}(l):=\phi(\nu(\varpi)^{-k} l)$ for $l\in M_\nu(F)$ so that $\phi^{(k)}\in \cH(M_\nu)$. Since $\nu$ is central in $M_\nu$, this induces a map
\begin{equation}\label{eq:central-translate}
(\cdot)^{(k)}\colon \cI(M_\nu)\ra \cI(M_\nu).
\end{equation}

\begin{lemma}\label{lem:nu-twist-support}
If $\phi\in \cH_{\tu{acc}}(M_\nu)$ then $\phi^{(k)}\in \cH_{\tu{acc}}(M_\nu)$ for all $k\ge 0$. Given $\phi\in \cH(M_\nu)$, there exists $k_0=k_0(\phi)$ such that $\phi^{(k)}\in\cH_{\tu{acc}}(M_\nu)$ for all $k\ge k_0$. The analogue holds true with $\cI$ in place of $\cH$. Moreover, letting $f^{(k)}$ denote the $\nu$-ascent of $\phi^{(k)}$ for $k\ge k_0$, we have
$$
\supp^{\star}_{\fka_G}\left(f^{(k)}\right) = \tu{pr}_G(\supp^{\star}_{\fka_{M_\nu}}(\phi^{(k)})) = k \cdot H^G(\nu(\varpi)) + \tu{pr}_G(\supp^{\star}_{\fka_{M_\nu}}(\phi)),\qquad \star\in \{O,SO\},
$$
where $\tu{pr}_G:\fka_{M_\nu}\ra \fka_G$ is the canonical surjection.
\end{lemma}

\begin{proof}
The assertions before ``Moreover'' follow from the facts that $\nu(\varpi)$ is acceptable and that $\phi$ has compact support. As for the last assertion, 
the second equality is obvious, so we check the first equality. By Lemma \ref{lem:support-open-dense-enough} it is enough to verify firstly that if $O_g(f^{(k)})\neq 0$ for $g\in G(F)_{\reg}$ then $H^G(g)\in \tu{pr}_G(\supp^{O}_{\fka_{M_\nu}}(\phi^{(k)}))$, and secondly that if $O_m(\phi^{(k)})\neq 0$ for $m\in M(F)_{\reg}$ then $\tu{pr}_G(H^M(m))\in \supp^{O}_{\fka_{G}}(f^{(k)})$. This follows from Lemma \ref{lem:O-tr-nu-ascent} (1) and Lemma \ref{lem:nu-ascent-support}. The case of stable orbital integrals is analogous.
\end{proof}

Let $\Groth(M_\nu(F))$ denote the Grothendieck group of admissible representations of $M_\nu(F)$.

\begin{lemma}\label{lem:Trace-equal-acceptable}
Let $\pi_1,\pi_2\in \Groth(M_\nu(F))$. Assume that for each $\phi\in \cI(M_\nu)$, there exists $k_0(\phi)\in \Z$ such that $\Tr \pi_1(\phi^{(k)})=\Tr \pi_2(\phi^{(k)})$ for all $k\ge k_0(\phi)$. Then $\pi_1=\pi_2$ in $\Groth(M_\nu(F))$.
\end{lemma}

\begin{proof}
This is proved by the argument of \cite[p.536]{ShinIgusa}.  
\end{proof}

\subsection{$\nu$-ascent and constant terms}\label{sub:nu-constant-terms}

Fix an $F$-rational minimal parabolic subgroup $P_0\subset P^{\tu{op}}_\nu$ of $G$ with a Levi factor $M_0\subset M_\nu$. Let $P$ be another $F$-rational parabolic subgroup of $G$ containing $P_0$, with a Levi factor $M$ containing $M_0$. Henceforth we will often write $L:=M_\nu$.

We have the constant term map (compare with \eqref{eq:J-nu-invariant})
\begin{equation}\label{eq:const-term-def}
\mathscr C^G_M: \cI(G)\ra \cI(M),\qquad \mathcal F \mapsto \left((\pi_M\mapsto \mathcal F(\nind_M^G(\pi_M))\right),
\end{equation}
where $\nind_M^G:\Groth(M(F))\ra \Groth(G(F))$ is the normalized parabolic induction (which does not change if $P$ is replaced with a different parabolic with Levi factor $M$). On the level of functions, when $f\in \cH(G)$, we can define $f_M\in \cH(M)$ by an integral formula (e.g., \cite[(3.5)]{ShinGalois}) so that 
\begin{equation}\label{eq:orb-int-const-term}
\begin{array}{ll}
O^G_g(f)=0,& \forall g\in G(F)_{\tu{reg}} \textup{~not conjugate to elements of~} M(F),\\
O^G_m(f)=D_{G/M}(m)^{1/2}O^M_m(f_M),& \forall ~ G\textup{-regular}~~m\in M(F), 
\end{array}
\end{equation}
where $D_{G/M}:M(F)\ra\R^\times_{>0}$ denotes the Weyl discriminant of $G$ relative to $M$. This identity and parts (i) and (ii) of \cite[Lem.~3.3]{ShinGalois} tell us that $f\mapsto f_M$ descends to the map $\mathscr C^G_M$ above. (Even though $G$ is a general linear group in \emph{loc.~cit.}, everything applies to general reductive groups since that lemma is based on the general results of \cite{Dijk}.)

Since $\nind_M^G$ induces a map $R(M)^{\tu{st}}\ra R(G)^{\tu{st}}$ \cite[Cor.~6.13]{KazhdanVarshavskyInduction}, the map $\mathscr C^G_M$ descends to a map on the stable spaces, still denoted by the same symbol:
$$
\mathscr C^G_M: \cS(G)\ra \cS(M).
$$

Define the following set of representatives for $\Omega^L\backslash \Omega^G/\Omega^M$:
$$\Omega^G_{M,L}:=\{\omega \in \Omega^G: \omega(M\cap P_0)\subset P_0,~ \omega^{-1}(L\cap P_0)\subset P_0\}.$$
For $\omega\in \Omega^G_{M,L}$, write $M_\omega:=M\cap \omega^{-1}(L)$, $P_\omega:=M\cap \omega^{-1}(P_\nu)$, and $L_\omega:=\omega(M)\cap L$. Note that $M_\omega$ (resp.~$L_\omega$) is an $F$-rational Levi subgroup of $M$ (resp.~$L$) and that $\omega$ induces an isomorphism $M_\omega\isom L_\omega$, thus also $\omega: \cI(M_\omega)\stackrel{\sim}{\ra}  \cI(L_\omega)$ by $\phi\mapsto (g\mapsto \phi(\omega^{-1} g))$. Since $\nu$ is central in $L$, its image lies in $L_\omega$. So $\nu_\omega:=\omega^{-1}(\nu)$ is a cocharacter of $M_\omega$. Hence we have a chain of maps 
$$\cI(L)~\stackrel{\mathscr C^L_{L_\omega}}{\longrightarrow} ~\cI(L_\omega)~\stackrel{\omega^{-1}}{\simeq}~ \cI(M_\omega)~ \stackrel{\mathscr J_{\nu_\omega}}{\longrightarrow} ~\cI(M).$$

\begin{lemma}\label{lem:acc-preservation-const-term}
 If $\phi\in \cI_{\nu\tu{-acc}}(L)$ then $\mathscr C^L_{L_{\omega}}(\phi)$ is contained in $\cI_{\nu\tu{-acc}}(L_\omega)$.
\end{lemma}

\begin{proof}
The proof of Lemma \ref{lem:acc-preservation} below works verbatim: just replace stable orbital integrals there with ordinary orbital integrals, and use \eqref{eq:orb-int-const-term}. (Since Lemma \ref{lem:acc-preservation} is more general, we supply a detailed argument only for the latter.)
\end{proof}

\begin{lemma}\label{lem:nu-constant-terms}
We have the following commutative diagram.
$$
\xymatrix{ \cI(L) \ar[r]^-{\mathscr J_\nu} \ar[d]_-{\oplus \mathscr C^L_{L_\omega}} & \cI(G) \ar[r]^-{\mathscr C^G_M} & \cI(M) \\
\bigoplus\limits_{\omega\in \Omega^G_{M,L}}  \cI(L_\omega) \ar[rr]^-{\oplus \omega^{-1}}_{\sim}  & & \bigoplus\limits_{\omega\in \Omega^G_{M,L}}  \cI(M_\omega) \ar[u]_-{\sum_\omega \mathscr J_{\nu_\omega}}
  }
$$
Finally, all this holds true with $\cS(\cdot)$ in place of $\cI(\cdot)$.
\end{lemma}

\begin{proof}
Let $\phi\in \cI(L)$. We check that the images of $\phi$ in $\cI(M)$ given in the two different ways have the same trace against every $\pi_M\in \Irr(M(F))$:
  $$\Tr \left( \mathscr C^L_{L_\omega}( \mathscr J_\nu(\phi)) | \pi_M\right)
  = \Tr \left( \mathscr J_\nu(\phi)|\nind_L^G(\pi_M)\right)
  = \Tr \left( \phi| J_{P_{\nu}}(\nind_L^G(\pi_M)) \right)
  $$
  $$
  = \sum_{\omega\in \Omega^G_{M,L}} \Tr \left( \phi| \nind_{L_\omega}^{L} \left(\omega(J_{P_{\nu_\omega}}(\pi_M))\right) \right)
  = \sum_{\omega\in \Omega^G_{M,L}} \Tr \left( \mathscr J_{\nu_\omega}(\omega^{-1}(\mathscr C^L_{L_\omega}(\phi)))| \pi_M\right),
  $$
where the second last equality comes from Bernstein--Zelevinsky's geometric lemma \cite[2.12]{BernsteinZelevinskyInduced}. The $\cS(\cdot)$-version is immediate from the $\cI(\cdot)$-version proven just now, since each map in the big diagram descends to a map between the stable analogues.
\end{proof}

\subsection{$\nu$-ascent and endoscopic transfer}\label{sub:nu-endoscopy}

In this subsection we assume that $G$ is quasi-split over $F$. Let $\fke=(G^\fke,\cG^\fke,s^\fke,\eta^\fke)$ be an endoscopic datum for $G$ such that $\cG^\fke= {}^L G^\fke$. (The last condition will be removed via $z$-extensions in \S\ref{sub:nu-z-extension}.) Here we fix $\Gamma_F$-pinnings $(\mathcal{B}^\fke,\mathcal{T}^\fke,\{\mathcal{X}_{\alpha^\fke}\})$ and $(\mathcal{B},\mathcal{T},\{\mathcal{X}_{\alpha}\})$ for $\hat G^\fke$ and $\hat G$, respectively. (These choices are implicit in the discussion of \S\ref{sub:prelim-endoscopy}.) Conjugating $\eta^\fke$ we may and will assume that $\eta^\fke(\mathcal{T}^\fke)=\mathcal{T}$ and $\eta^\fke(\mathcal{B}^\fke)\subset \mathcal{B}$.

We have a standard embedding ${}^L P_\nu\hra {}^L G$ and a Levi subgroup ${}^L M_\nu\subset {}^L P_\nu$ as in \cite[3.3,~3.4]{BorelCorvallis}. Choose a subtorus $S\subset \mathcal{T}$ such that $\tu{Cent}(S,{}^L G)={}^L M_\nu$. (This is possible by \cite[Lem.~3.5]{BorelCorvallis}.) Following \cite[\S6]{XuCuspidalSupport}, define
$$
\Omega^G(\fke,\nu):=\{\omega\in \Omega^G\,|\, \tu{Cent}(\omega(S),{}^L G^\fke)\ra W_{F}~\mbox{is~surjective}\}
$$
and $ \Omega_{\fke,\nu}:=\Omega^{G^\fke}\backslash\Omega^G(\fke,M_\nu)/ \Omega^{M_\nu}$. For each $\omega\in \Omega_{\fke,\nu}$, we obtain an endoscopic datum 
$$\fke_\omega=(G^\fke_\omega,{}^L G^\fke_\omega,s^\fke_\omega,\eta^\fke_\omega)\quad \mbox{for}~L=M_\nu$$
as follows. (Henceforth we view $^L G^\fke$ as a subgroup of $^L G$ via $\eta^\fke$.) Pick $g\in \hat G$ such that $\Int(g)$ induces $\omega$ on $S$. Then $g\, {}^L \! P_\nu g^{-1}\cap {}^L G^\fke$ is a parabolic subgroup of $^L G^\fke$ with Levi subgroup $g\, {}^L \! M_\nu g^{-1}$, so there is a corresponding standard parabolic subgroup $P^\fke_\nu = M^\fke_\nu N^\fke_\nu$ such that the standard embedding ${}^L P^\fke_\nu\hra{}^L G^\fke$ (resp.~$^L M^\fke_\nu\hra {}^L G^\fke$) becomes $g\, {}^L \! P_\nu g^{-1}\cap {}^L G^\fke$ (resp.~$g \, {}^L \! M_\nu g^{-1}\cap {}^L G^\fke$) after composing with $\Int(g^\fke)$ for some $g^\fke\in \hat{G}^\fke$. Then there is a unique $L$-embedding $\eta^\fke_\omega:{}^L M^\fke_\nu\hra {}^L M_\nu$ such that $\Int(g)\circ \eta^\fke_\omega = \eta^\fke\circ \Int(g^\fke)$.
Set $G^\fke_\omega:=M^\fke_\nu$, and $s^\fke_\omega:=g^{-1} s g\in \hat M_\nu$. Then it is a routine exercise to check that $(G^\fke_\omega,{}^L G^\fke_\omega,s^\fke_\omega,\eta^\fke_\omega)$ is an endoscopic datum for $M_\nu$.

There is a canonical embedding $A_{M_\nu}\hra A_{M_\nu^\fke}=A_{G_\omega^\fke}$ (just like $Z_{H}\hra Z$ in \S\ref{sub:prelim-endoscopy}). Composing with $\nu: \G_m\ra A_{M_\nu}$, we obtain
$$
\nu_\omega \colon \G_m \ra A_{G_\omega^\fke}.
$$
By construction, $G^\fke_\omega = M^\fke_\nu$ is a Levi subgroup of $G^\fke$ that is the centralizer of $\nu_\omega$. In particular we have a map $\mathscr J_{\nu_\omega}:\mathcal S(G^\fke_\omega) \ra \mathcal S(G^\fke)$ as in \eqref{eq:J-nu-stable}. Consider the following commutative diagram

\begin{equation}\label{eq:big-Levi-diagram}
\xymatrix{
  W_F \ar[r]  & {}^L G^\fke_\omega \ar[r]^-{\eta^\fke_\omega} \ar@{^(->}[d]_-{\Int(g^\fke)} &  {}^L M_\nu \ar[r] \ar@{^(->}[d]^-{\Int(g)} & {}^L Z^0_{M_\nu} \ar[d] \\
W_F \ar[r] &  {}^L G^\fke  \ar[r]^-{\eta^\fke} &  {}^L G \ar[r] & {}^L Z^0_G ,
}
\end{equation}
where the maps out of $W_F$ come from canonical splittings for the $L$-groups, the two horizontal maps on the right are induced by $Z^0_{M_\nu}\subset M_\nu$ and $Z^0_G\subset G$, the first two vertical maps correspond to the Levi embeddings (coming from $G^\fke_\omega \subset G^\fke$ and $M_\nu\subset G$) followed by $\Int(g^\fke)$ and $\Int(g)$ respectively, and finally the rightmost vertical map is induced by $Z^0_G\subset Z^0_{M_\nu}$. The left square in \eqref{eq:big-Levi-diagram} commutes by $\Int(g)\circ \eta^\fke_\omega = \eta^\fke\circ \Int(g^\fke)$ above. The commutativity of the right square is obvious since $\Int(g)$ acts trivially on $^L Z^0_G$. Denote by
$$
\lambda^\fke_\omega:Z^0_{M_\nu}(F)\ra \C^\times\qquad  (\mbox{resp}.~\lambda^\fke:Z^0_G(F)\ra \C^\times)
$$
the smooth character corresponding to the composite morphism from $W_F$ to ${}^L Z^0_{M_\nu}$ (resp. $^L Z^0_G$) in the first (resp. second) row. The character $\lambda^\fke$ is the same as in \S\ref{sub:prelim-endoscopy}. The commutativity of \eqref{eq:big-Levi-diagram} implies that $\lambda^\fke_\omega|_{Z^0_G(F)}=\lambda^\fke$. The canonical splittings from $W_F$ to ${}^L G^\fke_\omega$ and ${}^L G^\fke$ commute with the Levi embedding ${}^L G^\fke_\omega \hra {}^L G^\fke$ without $\Int(g^\fke)$, but the point is that $\Int(g^\fke)$ on ${}^L G^\fke$ is equivariant with the trivial action on ${}^L Z^0_G$ via the horizontal maps in \eqref{eq:big-Levi-diagram}.

\begin{lemma}\label{lem:lambda-is-unitary}
Assume that $\eta^\fke(W_F)$ is a bounded subgroup of $^L G$ in the sense above Lemma \ref{lem:bounded-representative}. (This condition can always be ensured by that lemma.) Then $\lambda^\fke_\omega$ is a unitary character.
\end{lemma}

\begin{proof}
By assumption and commutativity of \eqref{eq:big-Levi-diagram}, $\eta^\fke_\omega(W_F)$ is a bounded subgroup of $^L M_\nu$, whose image in $^L Z^0_{M_\nu}$ is a bounded subgroup accordingly. Therefore $\lambda^\fke_\omega$ is a unitary character via the Langlands correspondence for tori.
\end{proof}

\begin{proposition}\label{prop:nu-transfer}
The following diagram commutes.
$$
\xymatrix{\cI(M_\nu) \ar[r]^-{\mathscr J_\nu} \ar[dr]_-{\oplus \tu{LS}^{\fke,\omega}} & \cI(G) \ar[r]^-{\tu{LS}^\fke} & \cS(G^\fke) \\
  & \bigoplus\limits_{\omega\in \Omega_{\fke,\nu}}  \cS(G^\fke_\omega)  \ar[ur]_-{\sum_\omega \mathscr J_{\nu_\omega}}
  }
$$
Let $\phi\in C^\infty_c(M_\nu(F))$. If $f^{(k)}=\mathscr J_\nu(\phi^{(k)})$ then writing $\phi^{(k)}_\omega:=\tu{LS}^{\fke,\omega}(\phi)^{(k)}$, we have
$$
\phi^{(k)}_\omega=\lambda^\fke_\omega(\nu(\varpi))^{-k} \tu{LS}^{\fke,\omega}(\phi^{(k)}),\qquad
  \tu{LS}^\fke(f^{(k)})=\sum_{\omega\in \Omega_{\fke,\nu}} \lambda^\fke_\omega(\nu(\varpi))^{k} \mathscr J_{\nu_\omega}\left(\phi^{(k)}_\omega\right).
$$
\end{proposition}

\begin{remark}\label{rem:transfer=constantterm}
When $\fke$ is given by a Levi subgroup $M$ as in \S\ref{sub:nu-constant-terms} (so that $G^\fke=M$), we have $\tu{LS}^\fke=\mathscr C^G_M$, $\tu{LS}^{\fke,\omega}=\mathscr C^L_{L_\omega}$, and the meaning of $\nu_\omega$ is consistent between \S\ref{sub:nu-constant-terms} and \S\ref{sub:nu-endoscopy}. We leave it to the interested reader to compare the diagram above with that of Lemma \ref{lem:nu-constant-terms}.
\end{remark}

\begin{proof}
The first equality follows from the equivariance property of transfer as discussed in the paragraph containing \eqref{eq:endoscopic-central-char} (applied to $z=\nu(\varpi)^{-k}$, $G=M_\nu$, $G^\fke=G^\fke_\omega$, and $f=\phi$). The commutative diagram comes from (C.4) in \cite{XuCuspidalSupport} (when $\theta$ is trivial). This, together with the first equality, implies the last equality.
\end{proof}

\begin{corollary}\label{cor:nu-endoscopy}
Let $\phi^{(k)}$, $\phi_\omega^{(k)}$, and $f^{(k)}$ be as in Proposition \ref{prop:nu-transfer}. Then
$$
\supp^{SO}_{\fka_{G^\fke}}\left(\mathscr J_{\nu_\omega}(\phi^{(k)}_\omega)\right) =
k \cdot H^{G^\fke}(\nu_\omega(\varpi)) + \tu{pr}_{G^\fke}\left(\supp^{SO}_{\fka_{L_\omega}}(\tu{LS}^{\fke,\omega}(\phi)) \right),\qquad \omega\in  \Omega_{\fke,\nu} ,
$$
where $\tu{pr}_{G^\fke}:\fka_{G^\fke_\omega}\twoheadrightarrow \fka_{G^\fke}$ is the natural projection.
\end{corollary}

\begin{proof}
By Lemma \ref{lem:nu-twist-support} and Proposition \ref{prop:nu-transfer}, 
$$
\supp^{SO}_{\fka_{G^\fke}}\left(\mathscr J_{\nu_\omega}(\phi^{(k)}_\omega)\right) 
= \tu{pr}_{G^\fke}\left( \supp^{SO}_{\fka_{G^\fke_\omega}}(\phi^{(k)}_\omega)\right)
= \tu{pr}_{G^\fke}\left( \supp^{SO}_{\fka_{G^\fke_\omega}}(\tu{LS}^{\fke,\omega}(\phi^{(k)})).\right)
$$
$$
=\tu{pr}_{G^\fke}\left( k\cdot  H^{G^\fke_\omega}(\nu_\omega(\varpi))+ \supp^{SO}_{\fka_{G^\fke_\omega}}(\tu{LS}^{\fke,\omega}(\phi))\right).
$$
We finish by observing that $\tu{pr}_{G^\fke}(H^{G^\fke_\omega}(\nu_\omega(\varpi)))=  H^{G^\fke}(\nu_\omega(\varpi))$.
\end{proof}

It is useful to know preservation of acceptability in the setting of Proposition \ref{prop:nu-transfer} as this will allow an inductive argument in the proof of Corollary \ref{cor:main-estimate} below.

\begin{lemma}\label{lem:acc-preservation}
 If $\phi\in \cI_{\tu{acc}}(M_\nu)$ then $\phi_\omega:= \tu{LS}^{\fke,\omega}(\phi)$ is contained in $\cS_{\tu{acc}}(G^\fke_\omega)$.
  
\end{lemma}

\begin{proof}
Suppose that $SO_{\gamma_\omega}(\phi_\omega)\neq 0$ for some strongly $M_{\nu}$-regular element $\gamma_\omega\in G^\fke_\omega(F)$. We need to check that $\gamma_\omega$ is $\nu_\omega$-acceptable. (It is enough to consider strongly $M_{\nu}$-regular elements thanks to Lemma \ref{lem:support-open-dense-enough}.)

From the orbital integral identity for $SO_{\gamma_\omega}(\phi_\omega)$ (cf.~\eqref{eq:orbital-integral-id}), we see the existence of $\gamma\in M_{\nu}(F)_{\tu{sr}}$ whose stable conjugacy class matches that of $\gamma_\omega$ such that $O_\gamma(\phi)\neq 0$. The latter implies that $\gamma$ is $\nu$-acceptable. Write $T$, $T_\omega$ for the centralizers of $\gamma$, $\gamma_\omega$ in $M_\nu$, $G^\fke_\omega$, respectively. The matching of conjugacy classes tells us that there is a canonical $F$-isomorphism $i:T\simeq T_\omega$ which carries $\gamma$ to $\gamma_\omega$, cf.~\cite[\S3.1]{KottwitzEllipticSingular}. (A priori $i$ sends the stable conjugacy class of $\gamma$ to that of $\gamma_\omega$ and is canonical up to a Weyl group orbit. But $i$ is determined if required to send $\gamma$ to $\gamma_\omega$.) Since $\nu$ is central in $M_\nu$, the map $i$ necessarily carries $\nu$ to $\nu_\omega$. Regarding $T$ and $T_\omega$ as maximal tori of $G$ and $G^\fke$, respectively, we have an injection $i^*:R(G^\fke_\omega,T_\omega)\hra R(G,T)$ between the sets of roots induced by $i$ (again \cite[\S3.1]{KottwitzEllipticSingular}) such that
\begin{equation}\label{eq:i-star}
\langle \alpha,\nu_\omega\rangle = \langle i^*(\alpha),\nu\rangle,\qquad \alpha\in R(G^\fke_\omega,T_\omega).
\end{equation}

We are ready to show that $\gamma_\omega$ is $\nu_\omega$-acceptable. Let $\alpha\in R(G^\fke_\omega,T_\omega)$ such that $\langle \alpha,\nu_\omega\rangle>0$. We need to verify that $|\alpha(\gamma_\omega)|<1$, cf.~Definition \ref{def:acceptable}. But $\langle i^*(\alpha),\nu\rangle>0$ by \eqref{eq:i-star}, so the $\nu$-acceptability of $\gamma$ implies that $|i^*(\alpha)(\gamma)|<1$. Since $i^*(\alpha)(\gamma)=\alpha(\gamma_\omega)$, the proof is finished. 
\end{proof}

\subsection{$C$-regular functions and constant terms}\label{sub:C-regular}
Assume that $G$ is \emph{split} over $F$ and fix a reductive model over $\cO_F$, still denoted by $G$. Let $T$ be a split maximal torus of $G$ over $\cO_F$.
  Let $C\in \R_{>0}$.

\begin{definition}\label{def:C-reg-cochar}
A cocharacter $\mu:\G_m\ra T$ is $C$-\textbf{regular} if the following two conditions hold.
\begin{enumerate}
\item $|\langle \alpha, \mu \rangle| >C$ for every $\alpha\in \Phi(T,G)$,
\item $|\langle \alpha|_{A_M}, \tu{pr}_M(\omega\mu) \rangle| >C$ for every proper Levi subgroup $M$ of $G$ containing $T$, every $\omega\in \Omega^G$, and every $\alpha\in \Phi(T,G)\backslash \Phi(T,M)$.
\end{enumerate}
Write $X_*(T)_{C\tu{-reg}}$ for the set of $C$-regular cocharacters.
\end{definition}

\begin{lemma}\label{lem:C-reg-cochar} The following are true.
\begin{enumerate}
\item The subset $X_*(T)_{C\tu{-reg}}$ of $X_*(T)$ is nonempty, and stable under both $\Z$-multiples and the $\Omega^G$-action.
\item Let $\mu,\mu_0\in X_*(T)$. If $\mu$ is $C$-regular, then there exists $k_0\in \Z_{>0}$ such that $\mu_0+k\mu$ is $C$-regular for all $k\ge k_0$.
\end{enumerate}
\end{lemma}

\begin{proof}
(1) Let $X_*(T)_{\R,C\tu{-reg}}$ denote the subset of $X_*(T)_{\R}$ defined by the same inequalities as in Definition \ref{def:C-reg-cochar}. We choose an inner product on $X_*(T)_{\R}$ invariant under the Weyl group action.
Clearly $X_*(T)_{C\tu{-reg}}$ and $X_*(T)_{\R,C\tu{-reg}}$ are stable under $\Z$-multiples and the Weyl group action, and the latter is open. It suffices to verify the claim that $X_*(T)_{\R,C\tu{-reg}}$ is nonempty. Indeed, if the claim is true, we choose an open ball $U \subset X_*(T)_{\R,C\tu{-reg}}$. For $k\in \Z_{>0}$ large enough, $k\cdot U$ contains a point of $X_*(T)$, which then also lies in $X_*(T)_{C\tu{-reg}}$.

Let us prove the claim. Identify $X_*(T)_{\R}$ with the standard inner product space $\R^n$ via a linear isomorphism. Say that a measurable subset $A\subset \R^n$ has density 0 if $\tu{vol}(A\cap B(0,r))/\tu{vol}(B(0,r))\ra 0$ as $r\ra\infty$, where $B(0,r)$ denotes the ball of radius $r$ centered at $0$. We will show that the complement of $X_*(T)_{\R,C\tu{-reg}}$ in $X_*(T)_{\R}$ is a density 0 set. Since a finite union of density 0 sets still has density 0, it is enough to check that each of the conditions $|\langle \alpha, \mu \rangle|\le C$ and $|\langle \alpha|_{A_M}, \tu{pr}_M(\omega\mu) \rangle|\le C$ defines a density 0 subset in $X_*(T)_{\R}$. Either condition defines a subset of $\R^n$ of the form 
\begin{equation}\label{eq:strip}
| a_1 x_1 + \cdots + a_n x_n | \le C
\end{equation}
 in the standard coordinates $(x_1,...,x_n)$, with $a_1,...,a_n\in \R$.
 Moreover, not all $a_i$'s are zero, 
since neither $\langle \alpha, \mu \rangle$ nor $\langle \alpha|_{A_M}, \tu{pr}_M(\omega\mu) \rangle$ is identically zero on all $\mu\in X_*(T)_{\R}$.
(In the case of $\langle \alpha|_{A_M}, \tu{pr}_M(\omega\mu) \rangle$, the reason is that $\tu{pr}_M: X_*(T)_{\R}\ra X_*(A_M)$ is surjective, and that $\alpha|_{A_M}\in X^*(A_M)$ is nontrivial since $\alpha\notin \Phi(T,M)$.) Now it is elementary to see that \eqref{eq:strip} determines a density 0 subset. This proves the claim.

(2) Since the pairings in Definition \ref{def:C-reg-cochar} are linear in $\mu$, it is enough to choose $k_0$ such that $(k_0-1)C$ is greater than $|\langle \alpha, \mu \rangle|$ and $|\langle \alpha|_{A_M}, \tu{pr}_M(\omega\mu) \rangle|$ for all $\alpha,M,\omega$ as in that definition.
\end{proof}
 
Define $T(F)_{C\tu{-reg}}$ to be the union of $\mu(\varpi_F)T(\cO_F^\times)$ as $\mu$ runs over the set of $C$-regular cocharacters. For each $M\in \cL(G)$ containing $T$, set
\begin{equation}\label{eq:aM-Creg}
\fka_{M,C\tu{-reg}}:=\{a\in \fka_M: | \langle \alpha, a \rangle | > C |\log |\varpi|| , ~\forall\alpha\in \Phi^+(T,G)\backslash \Phi^+(T,M) \}.
\end{equation}
Here we use the pairing $X^*(T)_{\R}\times X_*(T)_{\R}\ra \R$ to compute $\langle \alpha, a \rangle$, viewing $a$ in $X_*(T)_{\R}$ via $\fka_M=X_*(A_M)_{\R}\subset X_*(T)_{\R}$. Recall that $X_*(A_M)_{\R} \simeq \Hom(X^*(M)_{\R},\R)$ via $a\mapsto (\chi\mapsto \langle \chi,a\rangle)$. Analogously $X_*(A_T)_{\R} \simeq \Hom(X^*(T)_{\R},\R)$. Write $\tu{pr}_M:X_*(A_T)_{\R}\twoheadrightarrow X_*(A_M)_{\R}$ for the map induced by the restriction $X^*_F(M)\ra X^*_F(T)$. (This is the analogue of $\tu{pr}_G$ in  \S\ref{sub:nu-Jacquet}.)
 
\begin{lemma}\label{lem:reg-maps-to-reg}
Let $M\subsetneq G$ be a Levi subgroup containing $T$ over $F$. Then $H^M(T(F)_{C\tu{-reg}})\subset \fka_{M,C\tu{-reg}}$.
\end{lemma}

\begin{proof}
Consider $t:=\mu(\varpi)$ with $\mu\in X_*(T)_{C\tu{-reg}}$. Then $H^M(t)\in \Hom(X^*(M)_{\R},\R)$ is identified with the unique element $a\in X_*(A_M)_{\R}$ such that $$\langle \chi,a\rangle = \log |\chi(\mu(\varpi))| = \langle \chi,\mu\rangle \log |\varpi|,\qquad \chi\in X^*(M)_{\R}.$$ Let $\alpha\in \Phi(T,G)\backslash \Phi(T,M)$. Since the composite of the restriction maps $X^*_F(M)_{\R}\ra X^*_F(T)_{\R}\ra X^*(A_M)_{\R}$ is an isomorphism, we can find $\chi\in X^*_F(M)_{\R}$ such that $\chi|_{A_M}=\alpha|_{A_M}$. Hence 
$$
\langle \alpha,a\rangle = \langle \alpha|_{A_M},a\rangle = \langle \chi|_{A_M},a\rangle = \langle \chi,a\rangle=  \langle \chi,\mu\rangle \log |\varpi| = \langle \chi|_{A_T}, \mu\rangle  \log |\varpi| 
$$
$$
 = \langle \chi, \tu{pr}_M(\mu)\rangle  \log |\varpi|
  = \langle \alpha|_{A_M},\tu{pr}_M(\mu)\rangle  \log |\varpi|.
$$
Since $\mu$ is $C$-regular, $|\langle \alpha|_{A_M},\tu{pr}_M(\mu)\rangle|>C$. Hence $| \langle \alpha, a \rangle | > C |\log |\varpi||$.
\end{proof}

The following definition is motivated by \cite[p.195]{FlickerKazhdan}.

\begin{definition}\label{def:C-regular}  
Let $C>0$. We say $f\in \cH(G)$ is \textbf{$C$-regular} if $\supp(f)$ is contained in the $G(F)$-conjugacy orbit of $T(F)_{C\tu{-reg}}$.\footnote{In practice it seems enough to impose the condition on $\supp^O(f)$. However when producing examples of $C$-regular $f$, often we have this condition satisfied.} Write $\cH(G)_{C\tu{-reg}}$ for the space of $C$-regular functions.
\end{definition}
 
\begin{lemma}\label{lem:represent-by-C-regular}
Let $f'\in \cH(G)$. Assume that every $g\in G(F)_{\tu{reg}}$ such that $O_g(f')\neq 0$ (resp.~$SO_g(f')\neq 0$) is $G(F)$-conjugate (resp.~stably conjugate) to an element of $T(F)_{C\tu{-reg}}$. Then
\begin{enumerate}
\item $O_g(f')=0$ (resp.~$SO_g(f')=0$) if $g\in G(F)_{\tu{ss}}$ is not regular, and
\item there exists  $f\in \cH(G)_{C\tu{-reg}}$ such that $f$ and $f'$ have the same image in $\cI(G)$ (resp.~$\cS(G)$).
\end{enumerate}
\end{lemma}

\begin{proof}
(1) If $g\in G(F)_{\tu{ss}}$ is not regular then no regular element in a sufficiently small neighborhood of $g$ intersects the $G(F)$-orbit of $T(F)_{C\tu{-reg}}$. (Since every $t\in T(F)_{C\tu{-reg}}$ satisfies $|1-\alpha(t)|=1$ for $\alpha\in \Phi(T,G)$, no $\alpha(t)$ approaches $1$.) Thus $f'$ has vanishing regular orbital integrals in a neighborhood of $g$. This implies that $O_g(f')=0$, by an argument as in the proof of \cite[Lem.~2.6]{Rog83} via the Shalika germ expansion around $g$. The case of stable orbital integrals is analogous.

(2) The point is that the $G(F)$-conjugacy orbit of $T(F)_{C\tu{-reg}}$ is open and closed in $G(F)$. (Since $T(F)_{C\tu{-reg}}$ is open and closed in $T(F)$, and the map $G(F)/T(F)\times T(F)\ra G(F)$ induced by $(g,t)\mapsto gtg^{-1}$ is a local isomorphism.)  Thus the product of $f'$ and the characteristic function on the latter orbit is smooth and compactly supported, and thus belongs to $\cH(G)_{C\tu{-reg}}$. Denoting the product by $f$, we see that $f$ and $f'$ have equal orbital integrals (resp.~stable orbital integrals) on regular semisimple elements. Therefore have the same image in $\cI(G)$ (resp.~$\cS(G)$).
\end{proof}

\begin{corollary}\label{cor:ascent-from-T-to-G-at-q}
Fix a $C$-regular cocharacter $\mu:\G_m\ra T$. Let  $\phi\in \cH(T)$.  Then there exists an integer $k_0=k_0(\phi)$ such that for every integer $k\ge k_0$, the $\mu$-ascent of $\phi^{(k)}$ is represented by a $C$-regular function on $G(F)$.
\end{corollary}
\begin{proof}
There is a finite subset $X\subset X_*(T)$ such that $\supp(\phi)\subset \bigcup_{\mu_0\in X} \mu_0(\varpi_F)T(\cO_F)$. Applying Lemma \ref{lem:C-reg-cochar} to each $\mu_0\in X$ and also Lemma \ref{lem:nu-twist-support}, we can find $k_0=k_0(\phi)\in \Z_{\ge 0}$ such that $\phi^{(k)}$ is $\mu$-acceptable and $\supp(\phi^{(k)})\subset T(F)_{C\tu{-reg}}$ for all $k\ge k_0$. Write $f^{(k)}$ for a $\mu$-ascent of $\phi^{(k)}$. By Lemma \ref{lem:represent-by-C-regular} it suffices to check for each $k\ge k_0$ and $g\in G(F)_{\tu{reg}}$ that if $O_g(f^{(k)})\neq 0$ then $g$ is in the $G(F)$-orbit of $T(F)_{C\tu{-reg}}$. This follows from the observed properties of $\phi^{(k)}$ by Lemma \ref{lem:O-tr-nu-ascent}.
\end{proof}

\begin{lemma}\label{lem:C-reg-inherited}
Let $f\in  \cH(G)_{C\tu{-reg}}$, $M\in \cL^<(G)$, and $\fke\in \cE^<(G)$. The following are true.
\begin{enumerate}
\item $\mathscr C^G_M(f)\in \cI(M)$ is represented by a function $f_M\in\cH(M)$ whose support is contained in the $M(F)$-conjugacy orbit of $T(F)_{C\tu{-reg}}$. (In particular $f_M$ is a $C$-regular function on $M(F)$.)
\item $\tu{LS}^\fke(f)\in \cS(G^\fke)$ vanishes unless $G^\fke$ is split over $F$. If $G^\fke$ is split over $F$ then $\tu{LS}^\fke(f)$ is represented by a $C$-regular function on $G^\fke(F)$.
\end{enumerate}
\end{lemma}
\begin{proof}
(1) We keep writing $T(F)_{C\tu{-reg}}$ for the set of $C$-regular elements relative to $G$, which contain $C$-regular elements relative to $M$. Since $T(F)_{C\tu{-reg}}$ is invariant under the Weyl group of $G$, an element $\gamma\in M(F)_{\tu{ss}}$ is conjugate to an element of $T(F)_{C\tu{-reg}}$ in $G(F)$ if and only if it is so in $M(F)$. In light of Lemma \ref{lem:represent-by-C-regular}, it suffices to show the following: if $O_\gamma(\mathscr{C}^G_M(f))\neq 0$ for regular semisimple $\gamma\in M(F)$ then $\gamma$ is $M(F)$-conjugate to an element of $T(F)_{C\tu{-reg}}$.

If $\gamma$ is $G$-regular then we have from \S\ref{sub:nu-constant-terms} that $O_\gamma(f)=D_{G/M}(\gamma)O_\gamma(\mathscr{C}^G_M(f))$, which is nonzero only if $\gamma$ is conjugate to an element of $T(F)_{C\tu{-reg}}$. If $\gamma$ is regular but outside the $M(F)$-orbit of $T(F)_{C\tu{-reg}}$, then a sufficiently small neighborhood $V$ of $\gamma$ does not intersect the $M(F)$-orbit of $T(F)_{C\tu{-reg}}$. On the other hand, $G$-regular elements are dense in $V$. Since an orbital integral is locally constant on the regular semisimple set, it follows that $O_\gamma(\mathscr{C}^G_M(f))=0$.

(2) If $T$ transfers to a maximal torus in $G^\fke$ then $G^\fke$ is split over $F$ since $T$ is a split torus. Thus $\tu{LS}^\fke(f)=0$ unless $G^\fke$ is split over $F$. Now we assume that $T$ transfers to a maximal torus $T^\fke\subset G^\fke$, equipped with an $F$-isomorphism $T\simeq T^\fke$ (canonical up to the Weyl group action). Via the isomorphism we transport $\lambda$ to $\lambda^\fke:\G_m\ra T^\fke$ and identify $\Phi(T^\fke,G^\fke)$ as a subset of $\Phi(T,G)$. By abuse of notation, keep writing $T(F)_{C\tu{-reg}}$ for its image in $T^\fke(F)$. Then $C$-regular elements of $T^\fke(F)$ are contained in $T(F)_{C\tu{-reg}}$.

Now the rest of the proof of (2) similar to that of (1), based on Lemma \ref{lem:represent-by-C-regular}. It suffices to check that if $SO_{\gamma^\fke}(\tu{LS}^\fke(f))\neq 0$ for $G$-regular semisimple $\gamma^\fke\in G^\fke(F)$ then $\gamma^\fke$ is stably conjugate to an element of $T(F)_{C\tu{-reg}}$. This is evident from the transfer of orbital integral identity.
\end{proof}

\begin{corollary}
For $f\in \cH(G)_{C\tu{-reg}}$ and $M\in \cL^<(G)$, we have $\supp^{O}_{\fka_M}(\mathscr{C}^G_M(f))\subset \fka_{M,C\tu{-reg}}$.
\end{corollary}
\begin{proof}
Let $f_M$ be as in the preceding lemma. Then
$$\supp^{O}_{\fka_M}(\mathscr{C}^G_M(f))=\supp^{O}_{\fka_M}(f_M)\subset \supp_{\fka_M}(f_M)\subset H^M(T(F)_{C\tu{-reg}})\subset \fka_{M,C\tu{-reg}},$$
where the last inclusion comes from Lemma \ref{lem:reg-maps-to-reg}.
\end{proof}

The following lemma, to be invoked in \S7.5, sheds light on how much $C$-regular functions detect.

\begin{lemma}\label{lem:C-regular-spectral-implication}
Let $I$ be a finite set and let $C>0$. Let $\pi_i\in \Irr(G(F))$ and $c_i\in \C$ for $i\in I$. If 
$$
\sum_{i\in I} c_i \Tr \pi_i(f)=0 \qquad \forall~f\in  \cH(G)_{C\tu{-reg}}
$$ 
then $\sum_{i\in I} c_i \cdot J_{P_0}(\pi_i)=0$ in $\Groth(G(F))\otimes_{\Z} \C$.
\end{lemma}
\begin{proof}
Fix a regular cocharacter $\mu:\G_m\ra T$ over $F$ such that $P_0=P_\mu^{\tu{op}}$. For each $\phi\in \cH(T)$, we have some integer $k_0$ such that $\phi^{(k)}$ are $\mu$-acceptable for all $k\ge k_0$ and their $\mu$-ascent $f^{(k)}$ are represented by $C$-regular functions by Corollary \ref{cor:ascent-from-T-to-G-at-q}. Thanks to Lemma \ref{lem:O-tr-nu-ascent}, 
$$
0=\sum_{i\in I} c_i \Tr(f^{(k)}|\pi_i)=\sum_{i\in I} c_i \Tr(\phi^{(k)}|J_{P_0}(\pi_i)),\quad
 \forall k\ge k_0.
$$
We conclude by Lemma \ref{lem:Trace-equal-acceptable}.
\end{proof}
  
\subsection{Fixed central character}\label{sub:fixed-central-character}

We explain that the facts thus far in \S\ref{sec:Jacquet-endoscopy} hold in the setup with fixed central character. Let $\nu:\G_m\ra G$ be a cocharacter over $F$ and $(G^\fke,\cG^\fke,s^\fke,\eta^\fke)$ an endoscopic datum for $G$ with $\cG^\fke={}^L G^\fke$. We can view $\fkX$ as a closed subgroup of $M_\nu(F)$, $G^\fke(F)$, and $G^\fke_\omega(F)$ of the preceding sections via the canonical embeddings of $Z(F)$ into their centers.

As before, $\cH_{\tu{acc}}(M_\nu,\chi^{-1})\subset \cH(M_\nu,\chi^{-1})$ is the subspace of functions which are supported on acceptable elements. Taking the image, we also have $\cI_{\tu{acc}}(M_\nu,\chi^{-1})$ and $\cS_{\tu{acc}}(M_\nu,\chi^{-1})$. Since acceptability is invariant under the translation by central elements, the $\chi$-averaging map induces a surjection $\cH_{\tu{acc}}(M_\nu)\ra \cH_{\tu{acc}}(M_\nu,\chi^{-1})$. The analogous surjectivity holds for $\cI_{\tu{acc}}$ and $\cS_{\tu{acc}}$.

The earlier results are valid in the setting of fixed central characters, with the following minor modifications. The proofs are omitted as no new ideas are required.

\smallskip

\paragraph{\textbf{To adapt \S\ref{sub:nu-Jacquet}}} Averaging the $\nu$-ascent map, we obtain 
$$
\mathscr J_\nu:\cI(M_\nu,\chi^{-1})\ra  \cI(G,\chi^{-1}),\qquad \mathscr J_\nu:\cS(M_\nu,\chi^{-1})\ra  \cS(G,\chi^{-1})
$$ 
satisfying the orbital integral and trace identities in Lemma \ref{lem:O-tr-nu-ascent} (with central character of $\pi$ equal to $\chi$) and Corollary \ref{cor:SO-nu-ascent}. The obvious analogues of Lemmas \ref{lem:support-open-dense-enough} and \ref{lem:nu-ascent-support} hold true (with no changes to the bottom rows in the latter lemma). The map $(\cdot)^{(k)}$ in \eqref{eq:central-translate} induces linear automorphisms on $\cI(M_\nu,\chi^{-1})$ and $\cI(G,\chi^{-1})$. With this, Lemmas \ref{lem:nu-twist-support} and \ref{lem:Trace-equal-acceptable} imply their natural analogues, restricting $\pi_1,\pi_2$ in the latter lemma to those with central character $\chi$.

\smallskip

\paragraph{\textbf{To adapt \S\ref{sub:nu-constant-terms}}} Averaging the map $\cH(G)\ra \cH(M)$ given by $f\mapsto f_M$, we obtain a map $\cH(G,\chi^{-1})\ra \cH(M,\chi^{-1})$, which induces $$\mathscr C^G_M: \cI(G,\chi^{-1})\ra \cI(M,\chi^{-1})$$ satisfying the same orbital integral identity as in \S\ref{sub:nu-constant-terms}. We can also describe $\mathscr C^G_M$ by the same formula \eqref{eq:const-term-def} from the space of linear functionals on $R(G,\chi)$ to that on $R(M,\chi)$. Lemmas \ref{lem:nu-constant-terms} and \ref{lem:acc-preservation-const-term} carry over as written, with $\chi^{-1}$-equivariance imposed everywhere.

\smallskip

\paragraph{\textbf{To adapt \S\ref{sub:nu-endoscopy}}} The Langlands--Shelstad transfer with fixed central character was already considered in \S\ref{sub:prelim-endoscopy} by averaging the transfer without fixed central character. With this in mind, we deduce the obvious analogues of Proposition \ref{prop:nu-transfer}, Corollary \ref{cor:nu-endoscopy}, and Lemma \ref{lem:acc-preservation}. In particular, the diagram in that proposition yields the following analogue. 
$$
\xymatrix{\cI(M_\nu,\chi^{-1}) \ar[r]^-{\mathscr J_\nu} \ar[dr]_-{\oplus \tu{LS}^{\fke,\omega}} & \cI(G,\chi^{-1}) \ar[r]^-{\tu{LS}^\fke} & \cS(G^\fke,\chi^{\fke,-1}) \\
& \bigoplus\limits_{\omega\in \Omega_{\fke,\nu}}  \cS(G^\fke_\omega,\chi^{\fke,-1})  \ar[ur]_-{\sum_\omega \mathscr J_{\nu_\omega}}
  }
$$

\smallskip

\paragraph{\textbf{To adapt \S\ref{sub:C-regular}}}
Definition \ref{def:C-regular} extends obviously to $\cH(G,\chi^{-1})$ by the same support condition. A key observation is that the notion of $C$-regularity is invariant under $Z(F)$-translation, so that the latter definition behaves well. More precisely, the $\chi$-averaging map from $\cH(G)\twoheadrightarrow \cH(G,\chi^{-1})$ is still surjective when restricted to the respective subspaces of $C$-regular functions. Using this, we carry over all results in \S\ref{sub:C-regular} to the setup with fixed central character, restricting to $\chi^{-1}$-equivariant functions and representations with central character $\chi$.

\subsection{$z$-extensions}\label{sub:nu-z-extension}
Throughout this section up to now, we assumed $\cG^\fke={}^L G^\fke$ on the endoscopic datum $\fke$. When the assumption is not guaranteed, we pass from $\fke$ and $G$ to $\fke_1$ and $G_1$ via $z$-extensions and pull back the central character datum from $(\fkX,\chi)$ to $(\fkX_1,\chi_1)$ as explained in \S\ref{sub:endoscopy-z-ext}.

Let $\nu_1:\G_m\ra G_1$ be a cocharacter lifting $\nu$. (Such a $\nu_1$ is going to be chosen in practice; see \S\ref{sub:test-function-at-p} below.) By Definition \ref{def:acceptable}, $\gamma_1\in M_{\nu_1}(F)$ is $\nu_1$-acceptable if and only if its image in $M_{\nu}(F)$ is $\nu$-acceptable. Everything in this section goes through with $\fke_1,G_1,(\fkX_1,\chi_1),\nu_1$ playing the roles of $\fke,G,(\fkX,\chi),\nu$. We write $\lambda^\fke_1,\lambda^\fke_{1,\omega}$ for the characters $\lambda^\fke,\lambda^\fke_\omega$ of \S\ref{sub:nu-endoscopy} in the setup for $\fke_1$ and $G_1$.

\section{Asymptotic analysis of the trace formula}\label{sec:asymptotic-analysis}

We prove key trace formula estimates in this section, to be applied to identify leading terms in the trace formula for Igusa varieties in \S\ref{sec:cohomology-Igusa}. The main estimate is Theorem \ref{thm:main-estimate}, whose lengthy proof is presented in \S\ref{sub:proof-main-estimate}. We work in a purely group-theoretic setup, with no reference to Shimura or Igusa varieties in order to enable an inductive argument on $\Q$-semisimple rank. The point is that the trace formula appearing in the intermediate steps need not arise from geometry.  

\subsection{Setup and some basic lemmas}
Throughout Section \ref{sec:asymptotic-analysis}, $G$ is a connected \emph{quasi-split} reductive group over $\Q$ which is cuspidal, \ie~$Z_{G}^0$ has the same $\Q$-rank and $\R$-rank. Let $(\fkX,\chi)$ be a central character datum as in \S\ref{sub:TF-fixed-central}. Let $\xi$ be an irreducible algebraic representation of $G_{\C}$ and $\zeta \colon G(\R)\ra\C^\times$ be a continuous character such that $\xi\otimes\zeta$ has central character $\chi_\infty^{-1}$ on $\fkX_\infty$. The restriction $\chi_\infty|_{A_G(\R)^0}$ via $A_G(\R)^0\subset \fkX_\infty$ can be viewed as an element of $X^*(A_G)_{\C}$, which is again denoted $\chi_\infty$ by abuse of notation. Let $\lambda_{\chi_\infty}$ denote the unique character making the following diagram commute. (The existence is obvious since the composition $A_G(\R)^0\ra \fka_G$ is an isomorphism.)
$$
\xymatrix{
A_G(\R)^0 ~~ \ar@{^(->}[r] \ar@/_2.0pc/[rrr]_-{\chi_\infty}
& G(\R) \ar@{->>}[r]^-{H^G_\infty} & \fka_G=X_*(A_G)_{\R} \ar[r]^-{\lambda_{\chi_\infty}} & \C^\times
}
$$
We have a canonical identification
\begin{equation}\label{eq:aG=}
\fka_G = X_*(A_G)_{\R}=\Hom(X^*(G)_{\Q},\R),\qquad a\mapsto (\chi\mapsto \langle \chi|_{A_G},a\rangle).
\end{equation}
Fix distinct primes $p,q$. Let $\nu:\G_m\ra G_{\Q_p}$ be a cocharacter over $\Q_p$. Let $\ol \nu\in \Hom(X^*_{\Q}(G),\Q)$ denote the image of $\nu\in X_*(A_{M_\nu})_{\Q}=\Hom(X^*_{\Q}(M_\nu),\Q)$ induced by $M_\nu\hra G$.\footnote{In the notation of the preceding section, $\ol \nu=\tu{pr}_G(\nu)$.} By definition, $\ol \nu(\chi)=\nu(\chi)$ for $\chi\in X^*_{\Q}(G)$. Viewing $\ol \nu$ as a member of $\fka_G$, we can compute $\langle \chi_\infty,\ol \nu\rangle\in \C$ via the canonical pairing $X^*(A_G)_{\C} \times X_*(A_G)_{\C}\ra \C$.

\begin{lemma}\label{lem:lambda(nu)}
$\lambda_{\chi_\infty}(H^G_p(\nu(p))) = p^{-\langle \chi_\infty, \ol\nu \rangle}$.
\end{lemma}

\begin{proof}
By definition $H^G_p(\nu(p))$ sends $\chi\in X^*_{\Q}(G)$ to $\log |\chi(\nu(p))|_p$. Similarly for $a\in A_G(\R)^0$, we have $H^G_\infty(a) = (\chi \mapsto \log |\chi(a)|_\infty$). We claim that $H^G_p(\nu(p)) = H^G_\infty( \ol \nu(p)^{-1})$. To show this, choose $r\in \Z_{\ge 1}$ such that $r\ol\nu\in X_*(A_G)$. Since $\fka_G$ is torsion-free it suffices to check that $H^G_p((r\nu)(p)) = H^G_\infty( (r\ol \nu)(p)^{-1})$, or equivalently that 
$$
|\chi((r\nu)(p))|_p =  |\chi((r\ol \nu)(p))|^{-1}_\infty,\qquad \chi\in X^*_{\Q}(G).
$$
Since $\chi((r\ol \nu)(p))\in \Q$ is an integral power of $p$ (as both $\chi$ and $r\ol \nu$ are algebraic), we have $|\chi((r\ol \nu)(p))|_\infty = |\chi((r\ol \nu)(p))|^{-1}_p = |\chi((r\nu)(p))|^{-1}_p$. This proves the claim. Now the claim implies that
$$
\lambda_{\chi_\infty}(H^G_p(\nu(p)))=\lambda_{\chi_\infty}(H^G_\infty( \ol \nu(p)^{-1})) = \chi_\infty(\ol \nu(p)^{-1}) = p^{-\langle \chi_\infty,\ol \nu \rangle}.
$$
\end{proof}

If $G$ is a connected reductive group over $\Q$ and $S$ is a set of $\Q$-places, we write
$$ 
H_S^G(\gamma) := \sum_{v \in S} H_v^G(\gamma) \in \ia_G.
$$
If $S^c$ is the complement of $S$, we write $H^{G, S^c} := H^G_S$.

\begin{lemma}\label{lem:HM-product-formula}
Let $G$ be a connected reductive group over $\Q$. 
\begin{enumerate}[label=\textit{(\roman*)}]
\item Let $S$ be a set of $\Q$-places. Let $\gamma, \gamma' \in G(\Q)$. If $\gamma$ and $\gamma'$ are conjugate in $G(\li \Q)$ then we have $H_S^G(\gamma) = H_S^G(\gamma') \in \ia_G$.
\item Let $S$ be the set of all $\Q$-places. Let $\gamma \in G(\Q)$, then $H_S^G(\gamma) = 0 \in \ia_G$.
\end{enumerate}
\end{lemma}
\begin{proof}
(\textit{i}) Let $F/\Q$ be a finite extension such that $\gamma$ and $\gamma'$ are conjugate in $G(F)$. Set $G':=\Res_{F/\Q}M$. The natural embedding $i:G\ra G'$ allows to view $\gamma,\gamma'$ as elements of $G'(\Q)$, and induces an injection $\ia_G\hra \ia_{G'}$. Thus it suffices to prove that $H_S^{G'}(\gamma) = H_S^{G'}(\gamma')$, since the map $H^G_S \colon G(\A_S)\to \ia_G$ is functorial with respect to $i$. By the reduction in the preceding paragraph, we may assume that $\gamma$ and $\gamma'$ are conjugate in $G(\Q)$. Then the proof is trivial since $H^G_S$ is a homomorphism into an abelian group.

(\textit{ii}). Using the functoriality for $G \to Z_G'$ from step 1, we may replace $G$ by its cocenter $Z_G'$, then $Z_G'$ by the maximally split torus $A$ inside $Z_G'$, and finally we may replace $A$ by $\G_m$, in which case the statement boils down to the usual product formula.
\end{proof}

If $M \in \cL_{\tu{cusp}}(G)$, we write $\Gamma_{\R\tu{-ell},\fkX}(M)$ for the set of $\gamma \in M(\Q)$ such that $\gamma \in M(\R)$ is elliptic, and $\gamma$ is taken up to $M(\Q)$-conjugacy. The following will be useful when studying Levi terms in the geometric side of the trace formula. 

\begin{lemma}\label{lem:RecognizeEllipticPart}
Let $M \in \cL_{\tu{cusp}}(G)$ and let $\gamma \in \Gamma_{\R\tu{-ell}, \fkX}(M)$ be a regular element. Let $P \subset G$ be a parabolic subgroup with Levi component $M$. Let $\xi$ be an irreducible representation of $M_\C$, and $\zeta \colon M(\R) \to \C^\times$ a continuous character. Write $f^M_{\zeta, \xi}$ for the function on $M(\R)$ given by \eqref{eq:AveragedLefschetz}. Then we have
\begin{align*}
&\vol(\fkX_{\Q} \backslash \fkX / A_{G, \infty})\inv \chi(I^M_\gamma)\zeta(\gamma) \Tr(\gamma; \xi ) 
\cr
= ~~ & d(M)\vol(I_\gamma^M(\Q) A_{I^M_\gamma, \infty} \backslash I_\gamma^M(\A)/\fkX) O_\gamma^M(f_{\xi, \zeta}^M).
\end{align*}
\end{lemma}
\begin{proof}
By Equation \eqref{eq:chi-I-gamma-M} we have
$$
\chi(I_\gamma^M) = (-1)^{q(I^M_\gamma)} \tau(I_\gamma^M) \vol(A_{I^M_{\gamma, \infty}} \backslash I_\gamma^{M, \cmpt}(\R))\inv d(I_\gamma^M) = 1.
$$
As $\gamma$ is regular, $I^M_\gamma$ is a torus and $d(I_\gamma^M) = 1$. Additionally we have $q(I_\gamma^M) = 0$ (as $\gamma$ is elliptic) and
$$
\tau(I_\gamma^M) = \vol(I_\gamma^M(\Q) A_{I_\gamma^M, \infty} \backslash I_\gamma^M(\A)).
$$
Thus we obtain
$$
\chi(I_\gamma^M) = \vol(I_\gamma^M(\Q) A_{I_\gamma^M, \infty} \backslash I_\gamma^M(\A))\vol(A_{I^M_{\gamma, \infty}} \backslash I_\gamma^{M, \cmpt}(\R))\inv.
$$
Since $\gamma$ is $\R$-elliptic, we obtain from \eqref{eq:orbital-integral-Lefschetz} that 
$$
 \zeta(\gamma) \Tr(\gamma; \xi ) = d(M) \vol(A_{M, \infty}\backslash I_\gamma^{M,\cmpt}(\R))  O_\gamma^M(f_{\xi, \zeta}^M).
$$
(we also used $e(I_\gamma^M) = 1$; recall $A_{M, \infty} := A_M(\R)^0$).
We obtain
\begin{align*}
\vol(\fkX_{\Q} &\backslash \fkX / A_{G, \infty})\inv \chi(I^M_\gamma)\zeta(\gamma) \Tr(\gamma; \xi )  = \cr
& = d(M)\frac {\vol(I_\gamma^M(\Q) A_{I_\gamma^M, \infty} \backslash I_\gamma^M(\A))} 
{\vol(\fkX_{\Q} \backslash \fkX / A_{G, \infty})} \frac {\vol(A_{M, \infty} \backslash I_\gamma^{M,\cmpt}(\R))  } {\vol(A_{I^M_{\gamma}, \infty} \backslash I_\gamma^{M, \cmpt}(\R))} O_\gamma^M(f_{\xi, \zeta }^M) \cr
 &= d(M) \vol(I_\gamma^M(\Q) A_{I^M_\gamma, \infty} \backslash I_\gamma^M(\A)/\fkX)  O_\gamma^M(f_{\xi, \zeta }^M),
 \end{align*}
where we used that $A_{M}$ equals $A_{I^M_\gamma}$
 because $\gamma$ is elliptic.
\end{proof}

\subsection{The main estimate and its consequences}\label{sub:main-estiamte}

We prove the following bounds for elliptic endoscopic groups and Levi subgroups of $G$, to be applied in \S\ref{sec:cohomology-Igusa}.

The notation $O(f(k))$ (resp. $o(k)$) for a nonzero $\C$-valued function $f(k)$ on $k\in \Z$ means that the quantity divided by $|f(k)|$ has bounded absolute value (resp. tends to $0$) as $k\ra {+\infty}$. In fact we only take $f(k)$ to be complex powers of $p$ (but not necessarily real powers; that is why we take absolute values). In our argument, every instance of $o(f(k))$ turns out to represent a power-saving, namely it is bounded by a power of $p$ with (the real part of) exponent strictly smaller than the exponent for $f(k)$.

Let us fix a $\Q$-rational Borel subgroup $B$ with Levi component $T \subset B$ (which is a maximal torus in $G$). We fix a Levi decomposition $B = T N_0$. As before we write $A_T \subset T$ for the maximal $\Q$-split subtorus. Additionally we write $S_p \subset T_{\qp}$ for the maximal $\Q_p$-split subtorus. 

Part of our setup is a cocharacter $\nu \colon \G_m\ra G$ over $\Q_p$. By conjugating $\nu$ if necessary, we may and do assume that $\nu$ has image in $T$ and that $\nu$ is $B$-dominant. Write $\rho\in X^*(T)_{\Q}$ for the half sum of all $B$-positive roots of $T$ in $G$ over $\li \Q_p$. Thus we have $\langle \rho,\nu \rangle \in \tfrac{1}{2} \Z_{>0}$. We transport various data over $\Q_p$ or $\li \Q_p$ to ones over $\C$ via a fixed isomorphism $\iota_p:\ol \Q_p\simeq \C$. (In the context of geometry, $\iota_p$ is fixed in \S\ref{sub:integral-canonical-models}.)

We also fix a prime $q$ such that $G_{\Q_q}$ is a split group. (For the existence, choose a number field $F$ over which $G$ splits. Then any prime $q$ that splits completely in $F$ will do.) If needed for endoscopy, an auxiliary $z$-extension $G_1$ of $G$ over $\Q$ is always chosen to be split over $\Q_q$; this is possible because $G_{\Q_q}$ is split. Thus the contents of \S\ref{sub:C-regular} and their adaptation to $z$-extensions apply to $G$ and $G_1$ over $\Q_q$. Since all endoscopic groups appearing in the argument will be split over $\Q_q$ (to be ensured by Lemma \ref{lem:C-reg-inherited} in the proof of Corollary \ref{cor:main-estimate}), whenever choosing their $z$-extensions, we take them to be also split over $\Q_q$ without further comments.

\begin{proposition}\label{prop:spectral-estimate}
Let $f^{\infty,p}=\prod_{v\neq \infty,p} f_v\in \cH(G(\A^{\infty,p}),(\chi^{\infty,p})^{-1})$ and $\phi_p\in \cH_{\tu{acc}}(M_\nu(\Q_p),\chi_p^{-1})$. For $k\in \Z$, write $f_p^{(k)}\in \cH(G(\Q_p),\chi_p^{-1})$ for a $\nu$-ascent of $\phi_p^{(k)}$ as in \S\ref{sub:nu-constant-terms}. Then
$$
T_{\disc,\chi}^G( f_{p}^{(k)}f^{\infty,p}f_{\xi,\zeta} ) = O \left( p^{k(\langle \rho,\nu \rangle+\langle\chi_\infty,\ol\nu\rangle)} \right).
$$
\end{proposition}

\begin{proof}
The left hand side equals
$$
\sum_{\pi\in \mathcal{A}_{\disc,\chi}(G)} m(\pi) \Tr (f_{p}^{(k)} |\pi_p) \Tr(f^{\infty,p}|\pi^p) \Tr(f_{\xi,\zeta}|\pi_\infty).
$$
Write $J_{P^{\tu{op}}_\nu}(\pi_p)=\sum_i c_i \tau_i$ in $\Groth(M_\nu(\Q_p))$ with $\tau_i\in \Irr(M_{\nu}(\Q_p))$. Let $\omega_{\tau_i}$ denote the central character of $\tau_i$. Then
$$
\Tr (f_{p}^{(k)} |\pi_p)=\Tr \left(\phi_p^{(k)}| J_{P^{\tu{op}}_\nu}(\pi_p)\right)=\sum_i c_i \Tr (\phi_p^{(k)}| \tau_i)
= \sum_i c_i \omega_{\tau_i}(\nu(p))^k \Tr (\phi_p|\tau_i).
$$

We define a character $\lambda_{\A}:G(\Q)\backslash G(\A)\ra \R_{>0}^\times$ as the composite
$$
\lambda_{\A} \colon G(\Q)\backslash G(\A)\stackrel{H^G}{\ra} \fka_G \stackrel{\lambda_{\chi_\infty}}{\ra} \R^\times_{>0}.
$$
Write $\lambda_v$ for the restriction of $\lambda_{\A}$ to $G(\Q_v)$ for a place $v$ of $\Q$. For each $\pi\in \mathcal{A}_{\disc,\chi}(G)$ contributing to the sum, we see that $\pi\otimes \lambda^{-1}_{\A}$ is a unitary automorphic representation of $G(\A)$ since $\pi_\infty\otimes \lambda^{-1}_\infty$ is unitary (by construction, $\pi_\infty\otimes \lambda_\infty^{-1}$ has trivial central character on $A_G(\R)^0$). Thus $\pi_p\otimes \lambda_p^{-1}$ is unitary. Applying Corollary \ref{cor:BoundExp} to $\pi\otimes \lambda^{-1}$ at $p$, we have
$$
\left|\omega_{\tau_i}(\nu(p)) \lambda^{-1}_p(\nu(p))\right| \le \delta_{P_\nu^{\tu{op}}}^{-1/2}(\nu(p))=p^{\langle \rho,\nu\rangle},
$$
noting that $\nu(p)\in A_{P_\nu}^{--}$. We deduce via Lemma \ref{lem:lambda(nu)} that
$$
\left|\omega_{\tau_i}(\nu(p))\right| \le  p^{\langle \rho,\nu\rangle} |\lambda_p(\nu(p))|_p = p^{\langle \rho,\nu\rangle} | p^{-\langle\chi_\infty,\ol\nu\rangle}|_p
= p^{\langle \rho,\nu\rangle+\langle\chi_\infty,\ol\nu\rangle}.
$$
By \cite[Cor.~2.13]{BernsteinZelevinskyInduced} the length of $J_{P^{\tu{op}}_\nu}(\pi_p)$, namely $\sum_i c_i$, can be bounded only in terms of $G$. This completes the proof.
\end{proof}

We state the main trace formula estimate of this paper. The proof will be given in \S\ref{sub:proof-main-estimate} below.

\begin{theorem}[main estimate]\label{thm:main-estimate}
Let $G,(\fkX,\chi),p,q,\xi,\zeta,\nu$ be as defined in the beginning of Section \ref{sec:asymptotic-analysis}, and additionally assume that $\fkX=Y(\A) A_{G,\infty}$ for a central torus $Y\subset Z_G$ over $\Q$. Let
\begin{itemize}
\item $f^{\infty,p,q}=\prod_{v\neq \infty,p,q} f_v\in \cH(G(\A^{\infty,p,q}),(\chi^{\infty,p,q})^{-1})$,
\item $\phi_p\in \cH_{\tu{acc}}(M_\nu(\Q_p),\chi_p^{-1})$, and
\item $f_p^{(k)}\in \cH(G(\Q_p),\chi_p^{-1})$ be a $\nu$-ascent of $\phi_p^{(k)}$, for $k\in \Z_{\ge 0}$.
\end{itemize}
Then there exists a constant $C=C(f^{\infty,q},\phi_p)\in \R_{>0}$ such that for each $f_q\in \cH(G(\Q_q),\chi_q^{-1})_{C\tu{-reg}}$,
$$
T_{\el,\chi}^G(f_{p}^{(k)}f_q f^{\infty,p,q}f_{\xi,\zeta} )= T_{\disc,\chi}^G(f_{p}^{(k)}f_q f^{\infty,p,q}f_{\xi,\zeta}) + o \left(p^{k(\langle \rho,\nu \rangle + \langle\chi_\infty,\ol\nu\rangle)} \right).
$$
\end{theorem}

As a corollary, we derive the stable analogue of Theorem \ref{thm:main-estimate}. We keep the setup of Theorem \ref{thm:main-estimate} and let $f^{\infty,p,q},\phi_p,f_p^{(k)}$ be as in that theorem. For each $\fke\in \cE^<_{\el}(G)$, we have $\fke_1=(G_1^\fke,{}^L G_1^\fke,s_1^\fke,\eta^\fke_1)$ and a central character datum $(\fkX_1^\fke,\chi_1^\fke)$ as in \S\ref{sub:endoscopy-z-ext}. Moreover we choose the representatives $\fke,\fke_1$ such that $\eta^\fke(W_F)$ and $\eta^\fke_1(W_F)$ have bounded images, as explained in Lemma \ref{lem:bounded-representative} and \S\ref{sub:endoscopy-z-ext}. Let
$$
f_1^{(k),\fke}=\prod_v f_{1,v}^{(k),\fke}\in \cH(G_1^\fke(\A),(\chi_1^\fke)^{-1})
$$
be a transfer of $f_{p}^{(k)}f_q f^{\infty,p,q}f_{\xi,\zeta}$. Then we have the following bound.

\begin{corollary}\label{cor:main-estimate}
In the setup of Theorem \ref{thm:main-estimate}, there exists a constant $C=C(f^{\infty,q},\phi_p,\xi,\zeta)\in \R_{>0}$ such that for every $f_q\in \cH(G(\Q_q))_{C\tu{-reg}}$, firstly
$$
ST_{\el,\chi}^G(f_{p}^{(k)}f_q f^{\infty,p,q} f_{\xi,\zeta}) = \begin{cases}
T_{\disc,\chi}^G( f_{p}^{(k)}f^{\infty,p,q}f_q f_{\xi,\zeta} )+ o\left( p^{k(\langle \rho,\nu \rangle + \langle\chi_\infty,\ol\nu\rangle)}\right),\\
O\left( p^{k(\langle \rho,\nu \rangle + \langle\chi_\infty,\ol\nu\rangle)}\right),
\end{cases}
$$
and secondly for each $\fke\in \cE^<_{\el}(G)$ (note that $f^{(k),\fke}_{1,q}$ inherits $C$-regularity from $f_q$), 
$$
\ST^{G_1^\fke}_{\el,\chi^\fke_1} \left(f^{(k),\fke}_1\right) = o\left( p^{k(\langle \rho,\nu \rangle + \langle\chi_\infty,\ol\nu\rangle)}\right).
$$
\end{corollary}

\begin{remark}\label{rem:cor-main-estimate}
In the inductive proof of the last bound, we only use the fact that its $q$-component is $C$-regular, $\infty$-component is a Lefschetz function, and most importantly the $p$-component is an ascent for a suitable cocharacter. We do not rely on the fact that $f^{(k),\fke}_1$ is a transfer of a function on $G(\A)$.
\end{remark}

\begin{proof}
The second estimate is immediate from the first via Proposition \ref{prop:spectral-estimate}. Let us prove the first and third asymptotic formulas, by reducing the former to the latter.

We induct on the semisimple rank of $G$. (For each $G$, we prove the corollary for all central character data and all $\nu$.) The estimate is trivial when $G$ is a torus, in which case $\ST_{\el,\chi}^G=T_{\el,\chi}^G=T_{\disc,\chi}^G$. We assume that $G$ is not a torus and that Corollary \ref{cor:main-estimate} is true for all groups which have lower semisimple rank than $G$. Write $f^{(k)}:=f_{p}^{(k)}f_q f^{\infty,p,q} f_{\xi,\zeta}$. The stabilization (Proposition \ref{prop:stabilization-elliptic}) tells us that
$$ 
\ST^{G}_{\el,\chi} (f^{(k)})=T^{G}_{\el,\chi} (f)
- \sum_{\fke\in \cE^{<}_{\el}(G)} \iota(G,G^\fke) \ST^{G_1^\fke}_{\el,\chi^\fke_1} \left(f^{(k),\fke}_1\right).
$$
In light of Theorem \ref{thm:main-estimate}, since the summand is nonzero only for a finite set of $\fke$ by Lemma \ref{lem:finiteness} (depending only on the finite set of primes $v$ where either $G_{\Q_v}$ or $f_v$ is ramified), it suffices to establish the last bound of the corollary. This task takes up the rest of the proof.
  
If $G^\fke_{\R}$ contains no elliptic maximal torus or if $A_{G^\fke}\neq A_G$ (equivalently if $A_{G^\fke_1}\neq A_{G_1}$), then $f_{1,\infty}^\fke$  is trivial as observed in \cite[p.182,~p.189]{KottwitzAnnArbor} so the desired estimate is trivially true. Henceforth, suppose that $G^\fke_{\R}$ contains an elliptic maximal torus. Then $f^{(k),\fke}_{1,\infty}$ is a finite linear combination of $f_{\eta^\fke_1,\zeta^\fke_1}$ over the set of $(\eta^\fke_1,\zeta^\fke_1)$ such that $\eta^\fke_1\circ \varpi_{\eta^\fke_1,\zeta^\fke_1} \simeq  \varpi_{\xi,\zeta}$. Proposition \ref{prop:nu-transfer} and its adaptation to $z$-extensions according to \S\ref{sub:fixed-central-character} and \S\ref{sub:nu-z-extension} tell us that
$$
f^{(k),\fke}_{1,p}=\sum_{\omega} \lambda^{\fke}_{1,\omega}(\nu_{1}(p))^{k} \mathscr J_{\nu_{1,\omega}}\left(\phi^{(k),\fke}_{1,p,\omega}\right)
= f^{(k),\fke}_{1,p}=\sum_{\omega} \lambda^{\fke}_{1,\omega}(\nu_{1}(p))^{k} f^{(k),\fke}_{1,p,\omega},
$$
where we have put $f^{(k),\fke}_{1,p,\omega}:=\mathscr J_{\nu_{1,\omega}}\left(\phi^{(k),\fke}_{1,p,\omega}\right)$ for a $\nu_{1,\omega}$-ascent of $\phi^{(k),\fke}_{1,p,\omega}\in \cH_{\tu{acc}}(G^\fke_{1,\nu}(\Q_p),(\chi^\fke_{1,p})^{-1})$. Here we applied Lemma \ref{lem:acc-preservation} (keeping \S\ref{sub:fixed-central-character} and \S\ref{sub:nu-z-extension} in mind) to have the transfer $\phi^{(k),\fke}_{1,p,\omega}$ of $\phi^{(k)}_{1,p}$ supported on $\nu_{1,\omega}$-acceptable elements. 

Recalling that $\eta^\fke_1(W_F)\subset {}^L G_1$ is a bounded subgroup, we see from Lemma \ref{lem:lambda-is-unitary} that $ \lambda^{\fke}_{1,\omega}$ is a unitary character. Thus we are reduced to showing the existence of some $C_\fke>0$ such that the following estimate holds for $\omega$ and $(\eta^\fke_1,\zeta^\fke_1)$ as above whenever $f^{\fke}_{1,q}$ is $C_\fke$-regular:
\begin{equation}\label{eq:find-Ce}
\ST^{G_1^\fke}_{\el,\chi^\fke_1} \left( f_1^{\fke,\infty,q,p} f^{\fke}_{1,q} f^{(k),\fke}_{1,p,\omega} f_{\eta^\fke_1,\zeta^\fke_1} \right) \stackrel{?}{=} o\left( p^{k(\langle \rho,\nu \rangle + \langle\chi_\infty,\ol\nu\rangle)}\right),\qquad k\in \Z_{\ge 0}.
\end{equation}
Indeed, take $C$ to be the maximum of all $C_\fke$ over the finite set of $\fke$ contributing to the sum. Then for each $C$-regular $f_q$, Lemma \ref{lem:C-reg-inherited} tells us either that $G_1^\fke$ is split over $\Q_q$ and $f_{1,q}^\fke$ is $C$-regular (thus also $C_\fke$-regular), or that $G_1^\fke$ is non-split over $\Q_q$ and $f_{1,q}^\fke$ vanishes. Thus the bound \eqref{eq:find-Ce} applies, and we will be done.

By the induction hypothesis, there exists $C_\fke>0$ such that whenever $f^{\fke}_{1,q}$ is $C_\fke$-regular, the left hand side of \eqref{eq:find-Ce} is $O\left( p^{k(\langle \rho^\fke_1,\nu_{1,\omega} \rangle + \langle \chi^\fke_{1,\infty},\ol \nu_{1,\omega}\rangle)}\right)$, with $\ol \nu_{1,\omega}\in X_*(A_{G^\fke_1})$ defined from $\nu_{1,\omega}$ in the same way $\ol \nu$ from $\nu$, and where $\rho^\fke_1$ is the half sum of positive roots of $G^\fke_1$ for which $\nu_{1,\omega}$ is a dominant cocharacter. (In other words, $\rho^\fke_1$ is to $\nu_{1,\omega}$ as $\rho$ is to $\nu$.) Therefore it is enough to check that
\begin{enumerate}
\item[(a)] $\langle \rho^\fke_1,\nu_{1,\omega} \rangle< \langle \rho,\nu \rangle$ (in $\Q$).
\item[(b)] $\tu{Re}\langle\chi^\fke_\infty,\ol\nu\rangle = \tu{Re}\langle \chi_{1,\infty},\ol \nu_{1,\omega}\rangle$,
\end{enumerate}

Let us begin with (a). Since $\langle \rho,\nu \rangle=\langle \rho_1,\nu_1 \rangle$, with $\rho_1$ defined for $G_1$ as $\rho$ is for $G$ (recall that $\nu_1:\G_m \ra G_1$ is a lift of $\nu$), the proof of (a) is reduced to the case when $G_1=G$ and $\nu_1=\nu$. We have an embedding $\hat G^\fke \hra \hat G$ coming from $\eta^\fke$, which restricts to $\hat G^\fke_\omega \hra \hat M_\nu$. Here we have chosen $\Gamma_F$-invariant pinnings for the dual groups such that the restriction works as stated. We may and will arrange that the Borel subgroup of $\hat G$ restricts to that of $\hat G^\fke$. Fix a maximal torus $\hat T\subset \hat G^\fke_\omega$ that is part of the pinning for $\hat G^\fke_\omega$. Viewing $\hat T$ also as a maximal torus in each of $\hat G^\fke$ and $\hat G$, we write $\Phi^\vee(\hat T, \hat G)$ and $\Phi^\vee(\hat T, \hat G^\fke)$ for the corresponding sets of coroots. Write $\hat \nu\in X^*(\hat T)$ for the dominant member in the Weyl orbit of characters determined by $\nu$. Then
\begin{equation}\label{eq:EndoscopicFinalRootArg}
\langle \rho^\fke,\nu_{\omega} \rangle = \sum_{\alpha^\vee\in \Phi^\vee(\hat T, \hat G^\fke_\omega)\atop  \langle \alpha^\vee,\nu \rangle>0} \langle \alpha^\vee,\nu \rangle, \qquad
\langle \rho,\nu \rangle = \sum_{\alpha^\vee\in \Phi^\vee(\hat T, \hat G)\atop \langle \alpha^\vee,\nu \rangle>0}  \langle \alpha^\vee,\nu \rangle.
\end{equation}
Thus it suffices to verify that there exists a coroot $\alpha^\vee\in \Phi^\vee(\hat T, \hat G)$ outside $\hat G^\fke$ such that $\langle \alpha^\vee,\nu \rangle>0$. The centralizer of $\hat \nu$ in $\hat G$ is identified with the dual group $\hat{M_\nu}$ (namely $\langle \alpha^\vee,\nu \rangle=0$ if and only if $\alpha^\vee$ is a coroot of $\hat{M_\nu}$), so we will be done if $\Lie \hat{M_\nu} + \Lie \hat G^\fke$ is a proper subspace of $\Lie \hat G$. This is exactly proved in \cite[Lem.~4.5 (ii)]{KST2} applied to $\textsf{G}=\hat G$, $\textsf{M}=\hat M_\nu$, and $\delta=s^\fke$. (The proof of \emph{loc.~cit.}~greatly simplifies. One reduces to the case when the Dynkin diagram of $G$ is connected as in the first paragraph in the proof of that lemma. Then argue as in the fourth paragraph of that lemma, with $X_n=0$ and with the role of $X_{ss}$ played by the semisimple element $s^\fke$.) 

Now we prove (b). Since $\langle\chi_\infty,\ol\nu\rangle= \langle\chi_{1,\infty},\ol \nu_1\rangle$, we reduce to showing (b) when $G_1=G$ and $\fke_1=\fke$ (with possibly nontrivial central character data). Thus we drop the 1's from the subscripts and check that
$$  
\tu{Re}\langle\chi^\fke_\infty,\ol\nu\rangle =  \tu{Re}\langle \chi_{\infty},\ol \nu_{\omega}\rangle.
$$  
We claim that $\ol\nu=\ol \nu_{\omega}$ in $X_*(A_G)_{\R}=X_*(A_{G^\fke})_{\R}$. In the diagram below, the triangle on the right commutes, and we want the triangle on the left commutes as well. 
$$
\xymatrix{ &  \G_m \ar@/_1.0pc/[dl]_-{\ol \nu} \ar[d]_-{\ol \nu_\omega} \ar[dr]_-{\nu} \ar[drr]^-{\nu_\omega}  \\
A_G \ar@{=}[r]  & A_{G^\fke}  
&  A_{M_\nu}  \ar@{^(->}[r]  &  A_{G^\fke_\omega}
}
$$
We choose maximal tori $T\subset M_\nu\subset G$ and $T^\fke\subset G^\fke_\omega\subset G^\fke$ with an isomorphism $T_{\ol F}\simeq T^\fke_{\ol F}$ to identify the absolute Weyl group $\ol\Omega^{G^\fke}$ as a subgroup of $\ol\Omega^{G}$. (This is done as in \cite[\S3]{KottwitzEllipticSingular}.) The isomorphism also identifies $\nu=\nu_\omega$. By \eqref{eq:pr=Weyl-orbit}, we have the equalities
$$
\ol \nu = |\ol\Omega^G|^{-1}\sum_{\omega\in \ol\Omega^G} \omega(\nu),\qquad \ol \nu_\omega = |\ol\Omega^{G^\fke}|^{-1}\sum_{\omega\in \ol\Omega^{G^\fke}} \omega(\nu).
$$
Hence $\ol\nu=|\ol\Omega^G\!\!/\ol\Omega^{G^\fke}|^{-1}\sum_{\omega\in \ol\Omega^G \!\!/\ol\Omega^{G^\fke}}   \omega(\ol \nu_\omega) = \ol \nu_\omega$. Indeed, the last equality follows since $\ol \nu_\omega\in X_*(A_{G^\fke})_{\R}=X_*(A_G)_{\R}$, which tells us that $\omega(\ol \nu_\omega)=\ol \nu_\omega$ for $\omega\in \ol\Omega^G$.

Applying \eqref{eq:endoscopic-central-char} at the archimedean place, we have $\chi_\infty=\lambda^\fke_\infty \chi^\fke_\infty$ as characters of $A_G(\R)$. Since $\lambda^\fke_\infty$ is unitary, $|\chi_\infty|=|\chi^\fke_\infty|$. Since $\ol \nu\in X_*(A_G)_{\R}$ (not just in $X_*(A_G)_{\C}$), we conclude that $\tu{Re}\langle\chi^\fke_\infty,\ol\nu\rangle =  \tu{Re}\langle \chi_{\infty},\ol \nu\rangle$ as desired. This verifies (b).
\end{proof}

\subsection{Some facts and notation on Weyl groups and Weyl chambers}

In this subsection we fix some additional notation on Weyl groups, Weyl chambers, which will be needed in the proof of the main estimate in the next subsection.

Let $P = MN \subset G$ be a parabolic subgroup such that $B \subset P$ and $T \subset M$. Write $Z_M^0$ for the identity component of the center of $M$ and $S_{M,p}$ for te maximal $\qp$-split subtorus in $Z_M^0$. If $M = T$, we will write more simply $S_p := S_{M,p}$. 

We then have
$A_{T, \qp} \subset S_p \subset T_{\qp}$. We will write
$$
\Omega^G \subset \Omega_p^G \subset \ol \Omega^G
$$
for the Weyl groups of $A_T$, $S_p$, and $T$ in $G$. Similar notation will be used for other objects related to Weyl groups, for instance we write $\Omega^G_{M, p} \subset \Omega_p^G$ for the set of Kostant representatives for $\Omega^G_p / \Omega^M_p$.

Write $\Phi^G_M = \Phi^G_M(A_M; G)$ for the set of roots of $A_M$ in $\tu{Lie}(G)$. Recall that an element $x \in \ia_M$ is called \emph{regular} if $\langle \alpha, x \rangle \neq 0$ for all $\alpha \in \Phi^G_M$. We write $\ia_M^{\tu{reg}}\subset  \ia_M$ for the subset of all regular elements. The connected components of $\ia_M^{\tu{reg}}$ are said to be the (open) \emph{Weyl chambers} of $\ia_M$. The subset
$$
\cC^{+}_M := \{x \in \ia_M^{\tu{reg}} \,|\, \forall \alpha \in \Phi^G_M:
\langle \alpha, x \rangle > 0\} \subset \ia_M^{\tu{reg}},
$$
is the \emph{dominant Weyl chamber}. Let $\Omega_M^G \subset \Omega^G$ be the set of Kostant representatives for the quotient $\Omega^G /\Omega^M$. The Weyl chambers $\cC \in \pi_0(\ia_M^{\tu{reg}})$ are parametrized via the bijection
$$
\Omega^G_M \to \pi_0(\ia_M^{\tu{reg}}), \quad \omega \mapsto \cC_\omega := \omega\inv(\cC^+_M)\in \pi_0(\ia_M^{\tu{reg}})
$$
If $\cC \subset \ia_M^{\tu{reg}}$ is a Weyl chamber, we write $\cC^\vee \subset \ia_M^*$ for the \emph{dual chamber}, \ie the set of $t \in \ia_M^*$ such that $t(x) > 0$ for all $x \in \cC$.

Write $\Phi(A_M; B)$, $\Phi(S_{M,p}; B_{\qp})$, $\Phi(Z_M; B)$ for the sets of positive roots attached to $A_M$, $S_{M,p}$ and $Z_M$.

\begin{lemma}\label{lem:RootSpacesCompare}
 The following statements are true:
\begin{enumerate}
\item The inclusions $A_{M, \lqp} \subset S_{M,p, \lqp} \subset Z_{M,\lqp}^0$ induce (by restriction) a sequence of maps
$$
\Phi(Z_M^0; B)  \to \Phi(S_{M,p}; B)  \to \Phi(A_M; B)
$$
which are all surjective.
\item The following 3 subsets of $\ia_M$ are equal:
\begin{enumerate}
\item The set of $x \in \ia_M$ such that for all $\alpha \in \Phi(Z_M^0; B)$ we have $\langle \alpha, x \rangle > 0$;
\item The set of $x \in \ia_M$ such that for all $\alpha \in \Phi(S_{M,p}; B)$ we have $\langle \alpha, x \rangle > 0$;
\item The set of $x \in \ia_M$ such that for all $\alpha \in \Phi(A_M; B)$ we have $\langle \alpha, x \rangle > 0$.
\end{enumerate}
\item The natural maps $\pi_0(\ia_M^{\tu{reg}}) \to \pi_0( X_*(S_{M,p})_{\R}^{\tu{reg}}) \to \pi_0(X_*(Z_M^0)_{\R}^{\tu{reg}})$ are injections.
\end{enumerate}
\end{lemma}
\begin{proof}
We have inclusions of the centralizer groups
$$M_{\lqp}=\tu{Cent}(A_{M, \lqp},G_{\lqp}) \supset \tu{Cent}(S_{M,p, \lqp},G_{\lqp}) \supset \tu{Cent}(Z_{M,\lqp}^0,G_{\lqp})\supset M_{\lqp},$$
where the first equality is well known \cite[Prop.~20.6(i)]{BorelGroups}. Hence the equality holds everywhere. So $\Phi(A_M; B)$, $\Phi(S_{M,p}; B)$, and $\Phi(Z^0_M; B)$ consist of eigen-characters for the adjoint actions of $A_{M, \lqp} \subset S_{M,p, \lqp} \subset Z_{M,\lqp}^0$ on the same space $\Lie(B)/\Lie(B\cap M)$, respectively. Therefore the maps in (1) are surjections.
Statements (2) and (3) are directly deduced from (1). 
\end{proof}

\subsection{Proof of Theorem \ref{thm:main-estimate}}\label{sub:proof-main-estimate}

The rest of this section is devoted to establishing the main estimate in Theorem \ref{thm:main-estimate} over several pages. The technical lemma (Lemma \ref{lem:TF-technical-lemma}) could be taken for granted at a first reading. Before diving into the details we recommend the reader to review the outline that we sketched below \eqref{eq:IntroExplainStrategy} in the introduction.

\begin{proof}[Proof of \ref{thm:main-estimate}]
We argue by induction on the $\Q$-semisimple rank $r_G$ of $G$. If $r_G = 0$, then we have $T^G_{\el, \chi} = T^G_{\disc,\chi}$, and the statement follows. Assume now that the theorem is established for all groups of lower $\Q$-semisimple rank and all accompanying data (\ie $(\fkX,\chi),p,q,\xi,\zeta$ and $\nu$).

We introduce a constant to control regularity at $q$:
\begin{equation}\label{eq:ConstantC}
C = C(f^{\infty, q},\phi_p) := \frac1{\log q} \cdot  \max_{M, x^{p,q,\infty}, \eps_p, \alpha} |\langle \alpha , x^{p,q,\infty} + \eps_p \rangle |,
\end{equation}
where $M \in \cL_\cusp(G)$, $x^{p,q,\infty}\in \tu{supp}_{\ia_M}^O(f^{\infty, p,q}_M )$, $\eps_p \in \tu{pr}_M \tu{supp}_{\ia_{\omega_p(M) \cap M_\nu}}^O(\phi_{p, \omega_p(M) \cap M_\nu})$, and $\alpha \in \Phi_M^G$.
Define the constants
\begin{align*}\label{eq:Constant-c_M}
 v_{\fkX} := \tu{vol}(\fkX_\Q\backslash \fkX/A_{G, \infty}) \quad\textup{and} \quad c_M := (-1)^{\dim (A_M/A_G)} \frac{|\Omega^M|}{|\Omega^G|},\quad M \in \cL_\cusp(G).
\end{align*}
Write $f^{\infty, (k)} := f^{\infty, p,q} f_p^{(k)} f_q$, to indicate the dependence on $k$ at $p$. The running hypothesis on $f_q$ is that it is $C$-regular for \eqref{eq:ConstantC}. By Proposition \ref{prop:DalalsFormula} we have
\begin{equation}\label{eq:DalalRepeat1}
T^G_{\tu{disc},\chi}(f_{\xi,\zeta} f^{\infty, (k)}) =
d(G)\inv \sum_{M\in \mathcal L_{\tu{cusp}}} c_Mv_{\fkX}\inv \sum_{\gamma \in \Gamma_{\R\tu{-ell}, \fkX}(M)}
\frac{\chi(I^M_\gamma)\zeta(\gamma) \Phi_M(\gamma,\xi)O^M_\gamma(f^{\infty,(k)}_M)}{|\iota^M(\gamma)| |\tu{Stab}^M_\fkX(\gamma)|},
\end{equation}
where we used that $\fkX$ is of the form $Y(\A) A_{G,\infty}$ for a central torus $Y\subset Z_G$ over $\Q$. 

We first compare the term corresponding to $M = G \in \cL_\cusp$ on the right hand side of Equation~\eqref{eq:DalalRepeat1} with $T_{\el, \chi}^G$. 
Only regular $\R$-elliptic conjugacy classes contribute to \eqref{eq:DalalRepeat1}: For $\gamma \in \Gamma_{\el, \fkX}(G)$ non-regular, we have $O_\gamma(f_q) = 0$ since $f_q$ is $C$-regular. The orbital integrals $O_{\gamma}^{G(\R)}(f_{\xi, \zeta})$ vanish for non $\R$-elliptic $\gamma \in G(\R)$. In Lemma~\ref{lem:RecognizeEllipticPart} we checked that for $\gamma \in \Gamma_{\R\tu{-ell}, \fkX}(G)$ we have 
\begin{equation}\label{eq:IsolateMainTerm}
d(G)\inv c_Gv_{\fkX}\inv 
\frac{\chi(I_\gamma^G) \zeta(\gamma) \Phi_G(\gamma, \xi) O_\gamma^G( f^{\infty, (k)})}{|\iota^G(\gamma)|  |\tu{Stab}_{\fkX}^G(\gamma)|} = \frac{ \vol(I_\gamma(\Q) \backslash I_\gamma(\A)/\fkX) O^G_\gamma(f_{\xi,\zeta} f^{\infty, (k)})}{\iota(\gamma)\inv |\tu{Stab}^G_{\fkX}(\gamma)|\inv}
\end{equation}
(this uses $\Phi_G(\gamma,\xi)=\Tr(\gamma,\xi)$, cf. \cite[below eq.~(4.4)]{ArthurL2} 
).  Therefore $T^G_{\el,\chi}(f_{\xi,\zeta}f^{\infty, (k)})$ appears on the right hand side of \eqref{eq:DalalRepeat1} as the summand for $M = G$ (see also \eqref{eq:EllipticPartOfTF}). Thus \eqref{eq:DalalRepeat1} can be rearranged as
\begin{equation}\label{eq:DalalRepeat2}
T^G_{\el, \chi}(f_{\xi, \zeta}f^{\infty, (k)}) = T^G_{\tu{disc},\chi}(f_{\xi,\zeta} f^{\infty, (k)}) -d(G)\inv
\sum_{M\in \cL_{\cusp}^{<}} c_Mv_{\fkX}\inv \sum_{\gamma \in \Gamma_{\el, \fkX}(M)}
\frac{\chi(I^M_\gamma)\zeta(\gamma) \Phi_M(\gamma,\xi)O^M_\gamma(f^{\infty, (k)}_M)}{|\iota^M(\gamma)| |\tu{Stab}^M_\fkX(\gamma)|}.
\end{equation}
As $f_{p}^{(k)}$ is a $\nu$-ascent of $\phi_p^{(k)}$, we have by Lemma~\ref{lem:nu-constant-terms} 
\begin{equation}\label{eq:function-f-p-decomposition}
f_{p, M}^{(k)} = \sum_{\omega_p \in \Omega^{G_{\qp}}_{M, M_\nu}} f_{p, M, \omega_p}^{(k)} \in \cH(M(\qp),\chi_p\inv),  
\end{equation}
where
\begin{equation}\label{eq:function-f-p-decomposition-2}
f_{p, M, \omega_p}^{(k)} := {\mathscr J}_{ \nu_{\omega_p}}(\omega_p\inv \phi^{(k)}_{p,M_{\omega_p}}), \quad\quad M_{\omega_p} = \omega_p(M) \cap M_\nu.
\end{equation}

At the prime $q$, we may arrange by Lemma~\ref{lem:C-reg-inherited}(1) and Lemma \ref{lem:RootSpacesCompare} that the constant term $f_{q, M}$ is supported on $C$-regular elements. Thus $f_{q,M}$ is decomposed according to the various chambers $\cC$ of $\ia_M^{\tu{reg}}$:
\begin{equation}\label{eq:f_qM-decomposition}
f_{q, M} = \sum_{\omega_q \in \Omega^G_M} f_{q, M, \omega_q} \in \cH(M(\Qq),\chi_q\inv)_{C\tu{-reg}},
\end{equation}
where $f_{q, M, \omega_q}$ satisfies $\tu{supp}_{\ia_M}^O(f_{q, M, \omega_q}) \subset \cC_{\omega_q}$. 

We define
$$
f^{\infty,(k)}_{M, \omega_p, \omega_q} := f^{\infty, p, q}_M f_{p, M, \omega_p}^{(k)} f_{q, M, \omega_q} \in \cH(M(\A),\chi^{-1}),
$$
so that 
$$
f^{\infty, (k)}_{M} = \sum_{\omega_p \in \Omega^{G_{\qp}}_{M, M_\nu}, \omega_q \in \Omega_M^G} f^{\infty,(k)}_{M, \omega_p, \omega_q} \in \cH(M(\A),\chi^{-1}).
$$
Changing the order of summation (each sum is finite), Equation~\eqref{eq:DalalRepeat2} becomes
\begin{equation}\label{eq:DalalRepeat3}
T^G_{\el, \chi}(f_{\xi, \zeta}f^{\infty, (k)}) = T^G_{\tu{disc},\chi}(f_{\xi,\zeta} f^{\infty, (k)}) -
d(G)\inv\sum_{M, \omega_p, \omega_q} c_{M} v_{\fkX}\inv \sum_{\gamma \in \Gamma_{\el, \fkX}(M)}
\frac{\chi(I^M_\gamma)\zeta(\gamma) \Phi_M(\gamma,\xi)O^M_\gamma(f^{\infty,(k)}_{M,\omega_p,\omega_q})}{|\iota^M(\gamma)| |\tu{Stab}^M_\fkX(\gamma)|},
\end{equation}
where the sum is over $M \in \cL_{\cusp}^<$, $\omega_p \in \Omega^{G_{\qp}}_{M, M_\nu}$, $\omega_q \in \Omega_M^G$.

To state the next lemma, we define a constant
\begin{equation}\label{eq:k_1-def}
k_1 = k_1(f^{\infty, q}, \phi_p) :=
\max_{M, \omega_p, \omega_q, \alpha, \eps_p, x^{p,\infty}} \left|  \frac {\langle \alpha, \eps_p + x^{p,\infty} \rangle} {\log(p) \langle \alpha, \tu{pr}_M(\omega_p\inv \nu) \rangle} \right| 
 \in \R_{>0},
\end{equation}
where the maximum is taken over $M\in \cL_\cusp^{<}(G)$, $\omega_p\in \Omega^{G_{\qp}}_{M, M_\nu}$, $\omega_q\in \Omega^G_M$, $\eps_p \in \tu{pr}_M\supp_{\ia_{M_{\omega_p}}}^O(\omega_p\inv \phi_{p, M_{\omega_p}})$, $x^{p,\infty} \in \tu{supp}_{\ia_M}^O(f_M^{\infty, p})$, and $\alpha$ ranges over those $\alpha \in \Phi^G_M$ such that $\langle \alpha, \tu{pr}_M(\omega_p\inv \nu)) \rangle \neq 0$. 

We have fixed a maximal torus $T$ in $G_{\C}$ (we have $G_{\C} \simeq G_{\lqp}$ via $\iota_p$), along with a Borel subgroup $B$. We write $\li \rho = \li \rho_G$ for the half sum of the $B$-positive roots of $T$ in $\tu{Lie}(B)$. Note that we have $\li \rho|_{A_T} = \rho$. We use similar definitions for $\li \rho_M$ and $\rho_M$ if $M \subset G$ is a Levi subgroup. Let $\lambda = \lambda_B,\lambda_B^* \in X^*(T)$ denote the highest weight of $\xi$ and its dual representation $\xi^*$, respectively, relative to $B$. 
 
For each $M \in \cL_\cusp(G)$ we introduce the following notation. Denote by $\cP(M)$ the set of parabolic subgroups $P$ of $G$ of which $M$ is a Levi component.
For each $\lambda_0 \in X^*(T)^+$ we write $\xi^M_{\lambda_0}$ for the irreducible $M_{\C}$-representation with highest weight $\lambda_0$. We  define $\omega_\infty\star \lambda_0:=\omega_\infty(\lambda_0+\li \rho)-\li \rho$ for each $\omega_\infty\in \ol \Omega_M^G$ and $\lambda_0 \in X^*(T)^+$. 
Let $\omega_\infty \in \ol \Omega^G_M$. Write $\omega_0^M \in \li \Omega^M$ for the longest Weyl group element, and 
$$
\lambda_B(\omega_\infty) := -\omega_0^M(\omega_\infty\star \lambda_B^*) =
\omega_0^M \omega_\infty \omega_0^M \lambda_B - \omega_0^M \omega_\infty \li \rho - \omega_0^M \li \rho
,
$$
so that we have
$$
\xi^M_{\lambda_B(\omega_\infty)} = (\xi^M_{\omega_\infty\star \lambda_B^*})^*.
$$

\begin{lemma}\label{lem:TF-technical-lemma}
Assume that $k > k_1$. Consider $M, \omega_p, \omega_q, \gamma$ as in \eqref{eq:DalalRepeat3}, and assume that
\begin{equation}\label{eq:OrbitalNonzeroAssumption}
O_\gamma^M(f^{\infty, (k)}_{M, \omega_p, \omega_q}) \neq 0.
\end{equation}
Let $x_\infty := H^M_\infty(\gamma) \in \ia_M$. Then the following are true.
\begin{enumerate}[label=(\roman*)]
\item The element $x_\infty \in \ia_M$ is regular and  lies in the chamber $\cC_0 = \cC_0(M, \omega_p, \omega_q) \subset \ia_M^{\tu{reg}}$ which has the following set of positive roots
\begin{equation}\label{eq:Chamber-of-x}
\{ \alpha \in \Phi^G_M\ |\ \langle \alpha, \tu{pr}_M(\omega_p\inv\nu) \rangle < 0\} \cup \{\alpha \in \Phi^G_M\ |\ \langle \alpha, \tu{pr}_M(\omega_p\inv\nu) \rangle = 0 \textup{ and } \alpha \in -\cC_{\omega_q}^\vee\}.
\end{equation}
\item There exists an explicit subset $\li \Omega_M^{G\diamond} = \li \Omega_M^{G\diamond}(M, \omega_p, \omega_q) \subset \li \Omega^G_M$ (see \eqref{eq:WG-double}) and an explicit sign $\eps^\diamond = \eps^\diamond(M, \omega_p, \omega_q)$ (see \eqref{eq:eps-diamond}) such that we have
$$
\Phi_M(\gamma, \xi) = \eps^\diamond \sum_{P \in \cP(M)} \delta_P^{-1/2}(\gamma) \sum_{\omega_\infty \in \li \Omega^{G\diamond}_{M}} \eps(\omega_\infty) \Tr(\gamma
 ; \xi^M_{\lambda_B(\omega_\infty)}),
$$
where $\eps(w_\infty)\in \{\pm1\}$ denotes the sign as an element of the Weyl group $\li \Omega^G$.
\end{enumerate}
\end{lemma}
\begin{proof}
(\textit{i})
If $S$ is a set of places of $\Q$, we write in this proof
$$
x_S := H^M_S(\gamma) \in \ia_M, \quad x^S := H^{M, S}(\gamma) \in \ia_M.
$$
We check that $\langle \alpha, x_\infty \rangle \neq 0$ for all $\alpha \in \Phi^G_M$ (\textit{\ie} $x$  is regular). By the product formula in Lemma~\ref{lem:HM-product-formula} we have
\begin{equation} \label{eq:RegularCheck1}
-\langle \alpha, x_\infty \rangle = \langle \alpha, x^{p,\infty} \rangle + \langle \alpha, x_p \rangle.
\end{equation}

At $p$ the non-vanishing of $O_\gamma(f_{p, M, \omega_p}^{(k)})$ implies $x_p \in \supp_{\ia_M}^O(f^{(k)}_{p, M, \omega_p})$. By Lemma~\ref{lem:nu-twist-support} (and \eqref{eq:function-f-p-decomposition-2}) 
$$
\supp_{\ia_M}^O(f_{p, M, \omega_p}^{(k)}) = k \cdot H^M_p(\omega_p\inv \nu(p)) + \tu{pr}_M(\supp_{\ia_{M_{\omega_p}}}^O(\omega_p\inv \phi_{p, M_{\omega_p}})).
$$
Therefore
\begin{equation} \label{eq:x_p-form}
x_p = k \cdot H^M_p(\omega_p\inv \nu(p)) + \eps_p
\end{equation}
for some $\eps_p \in \tu{pr}_M\supp_{\ia_{M_{\omega_p}}}^O(\omega_p\inv \phi_{p, M_{\omega_p}})$. Thus 
\begin{equation}\label{eq:RegularCheck1ab}
\langle \alpha, x_p \rangle = k \cdot \langle \alpha, H^M_p(\omega_p\inv \nu(p)) \rangle + \langle \alpha, \eps_p \rangle =
- k (\log p) \cdot \langle \alpha, \tu{pr}_M(\omega_p\inv \nu)) \rangle + \langle \alpha,
\eps_p \rangle.
\end{equation}
We now distinguish cases. First consider $\alpha \in \Phi^G_M$ such that $\langle \alpha, \tu{pr}_M(\omega_p\inv \nu) \rangle \neq 0$. By \eqref{eq:RegularCheck1} and \eqref{eq:x_p-form},
\begin{equation}\label{eq:alpha_x_infty}
-\langle \alpha, x_\infty \rangle = \langle \alpha, x^{p,\infty} \rangle + \langle \alpha,
\eps_p \rangle  - k(\log p) \cdot \langle \alpha, \tu{pr}_M(\omega_p\inv \nu) \rangle.
\end{equation}
As $k > k_1$ (see \eqref{eq:k_1-def}) we have  
\begin{equation*}
k (\log p) \cdot |\langle \alpha, \tu{pr}_M(\omega_p\inv \nu) \rangle | > |\langle \alpha, \eps_p + x^{p,\infty} \rangle |.
\end{equation*}
Thus from \eqref{eq:alpha_x_infty} we get $\langle \alpha, x_\infty \rangle \neq 0$.

The second case is when $\alpha \in \Phi^G_M$ is such that $\langle \alpha, \tu{pr}_M(\omega_p\inv \nu) \rangle = 0$. Again by \eqref{eq:RegularCheck1} and \eqref{eq:x_p-form} we have
\begin{equation}\label{eq:RegularCheck1c}
- \langle \alpha, x_\infty \rangle = \langle \alpha, x^{p,\infty} \rangle
+ \langle \alpha, \eps_p \rangle.
\end{equation}
As $f_q$ is $C$-regular,  we have from~\eqref{eq:aM-Creg}, \eqref{eq:ConstantC}, and Lemma \ref{lem:RootSpacesCompare}  that
\begin{equation}\label{eq:alpha_x_q_dominate}
|\langle \alpha, x_q \rangle | > C \log q \geq | \langle \alpha, x^{p,q,\infty} + \eps_p \rangle|.
\end{equation}
for all $\alpha \in \Phi^G_M$. In particular
$$
\langle \alpha, x_q \rangle + \langle \alpha, x^{p,q,\infty} + \eps_p \rangle \neq 0.
$$
Therefore each side of \eqref{eq:RegularCheck1c} does not vanish. Hence $\langle \alpha, x_\infty \rangle \neq 0$ for all $\alpha \in \Phi_M^G$.

We now determine for which $\alpha \in \Phi_M^G$ we have $\langle \alpha, x_\infty \rangle > 0$. If $\langle \alpha, \tu{pr}_M(\omega_p\inv \nu) \rangle \neq 0$, then
$$
\tu{sign}(\langle \alpha, x_\infty \rangle) = -\tu{sign}(\langle \alpha, \tu{pr}_M(\omega_p\inv \nu) \rangle)
$$
by the arguments following \eqref{eq:RegularCheck1ab}. If $\langle \alpha, \tu{pr}_M(\omega_p\inv \nu) \rangle = 0$, then
$$
\tu{sign}(\langle \alpha, x_\infty \rangle) = -\tu{sign}(\langle \alpha, x_q \rangle)
$$
by $C$-regularity (see~\eqref{eq:alpha_x_q_dominate}). We have $x_q \in \tu{supp}_{\ia_M}(f_{q, M, \omega_q})$. Statement (\textit{i}) follows.

(\textit{ii}) Let us start by recalling a result of Goresky, Kottwitz and MacPherson in \cite{GKM97}. We write $\tu{pr}_M^* \colon X^*(T)_\R \to X^*(A_M)_\R$ for the restriction map. Let $P = MN \in \cP(M)$. Write $\rho_N$ (resp. $\li \rho_N$)  for the half sum of the positive roots of $A_M$ (resp. $\li T$) that occur in the Lie algebra of the unipotent radical $N$ of $P$.
Write $\omega_\xi$ for the central character of $\xi$. Write $\alpha_1, \ldots, \alpha_n \in \ia_M^*$ for the simple roots of $A_M$ in $\Lie(N)$, which form a basis of $(\ia_M/\ia_G)^*$. This determines the dual basis consisting of $t_1, \ldots, t_n \in \ia_M/\ia_G$. Put $I := \{1, 2, \ldots, n\}$.
Define the following subsets of $I$ (cf.~\cite[p.534]{GKM97})
\begin{align}\label{eq:I-sets}
I(\gamma) & := \{i \in I\ |\ \langle  \alpha_i, x \rangle < 0 \}, \cr
I(\omega_\infty) & := \{i \in I\ |\ \langle \tu{pr}_M^*(-\omega_\infty\star\lambda_B^*) - \li \rho_N - \omega_\xi, t_i \rangle > 0\}.
\end{align}
By the discussion above Thm.~7.14.B in \cite{GKM97} we have
\begin{equation}\label{eq:varphi_P}
\varphi_P(-x_\infty, \pr_M^*(\omega_\infty\star \lambda_B^*) + \li \rho_N + \omega_\xi) =
\begin{cases}
(-1)^{\dim(A_G)} (-1)^{\dim(A_M/A_G) - |I(\gamma)|}, & \textup{ if }
I(\omega_\infty) = I(\gamma), \cr
0, & \textup{otherwise}.
\end{cases}
\end{equation}

We define  $L_M(\gamma)\in \C$ following \cite[p. 511]{GKM97},\footnote{We write $L_M(\gamma)$ where the authors of \cite{GKM97} write $L_M^\nu(\gamma)$. This is because we only need to use the ``middle weight profile", so there is no need to distinguish in our notation. In their formula the symbol $\nu_P$ appears, as they allow more general weight profiles. Since we use the middle weight profile, we have $\nu_P = - \li \rho_N - \omega_\xi$. } when $x_\infty$ is regular:
\begin{align}\label{eq:LM-nu}
L_M(\gamma) &:= (-1)^{\dim(A_G)} \sum_{P \in \cP(M)} \delta_P^{-1/2}(\gamma) \sum_{\omega_\infty \in \li \Omega^G_{M}}  \eps(\omega_\infty) \Tr(\gamma\inv ; \xi^M_{\omega_\infty\star \lambda_B^*}) \cdot \cr
& \cdot \varphi_P(-x_\infty, \pr p_M^*(\omega_\infty\star \lambda_B^*) + \li \rho_N + \omega_\xi).
\end{align}

As $\gamma$ is regular, Theorems 5.1 and 5.2 of \cite{GKM97} imply the following identity\footnote{In \cite{GKM97}, they write $\Phi_M(\gamma, \Theta^{\xi^*})$ for $\Phi_M(\gamma, \xi^*)$. Their $E$ corresponds to our $\xi^*$.}
\begin{equation}\label{eq:GKM-Phi-M-formula}
\Phi_M(\gamma, \xi) = L_M(\gamma).
\end{equation}
By (\textit{i}), $x_\infty$ is regular, so the right hand side of \eqref{eq:LM-nu} is an expression for $\Phi_M(\gamma,\xi)$.

We assumed that $O_\gamma^M(f^{\infty, (k)}_{M, \omega_p, \omega_q}) \neq 0$ in \eqref{eq:OrbitalNonzeroAssumption},  so $x_\infty=H^M_\infty(\gamma)$ lies in the chamber $\cC_{0}(M, \omega_p, \omega_q)$ by statement (i) of this lemma. Thus the set $I(\gamma)$ does not depend on $x_\infty$. Write $\cI_0 = \cI_0(M, \omega_p, \omega_q)$ for $I(\gamma)$, and
\begin{equation}\label{eq:WG-double}
\li \Omega^{G\diamond}_{M} = \li \Omega^{G\diamond}_{M}(M, \omega_p, \omega_q) := \{\omega_\infty \in
\li \Omega^G_{M} \ |\ I(\omega_\infty) = \cI_0\},
\end{equation}
in terms of \eqref{eq:I-sets}. Then \eqref{eq:LM-nu} simplifies thanks to \eqref{eq:varphi_P}:
\begin{equation}\label{eq:WG-double2}
L_M(\gamma) = (-1)^{\dim(A_M/A_G) - |\cI_{0}|} \sum_{P \in \cP(M)} \delta_P^{-1/2}(\gamma) \sum_{\omega_\infty \in \li \Omega^{G\diamond}_{M}} \eps(\omega_\infty) \Tr(\gamma\inv ; \xi^M_{\omega_\infty\star\lambda_B^*}).
\end{equation}
We obtain (\textit{ii}) by using $\Tr(\gamma\inv; \xi^M_{\omega_\infty \star \lambda_B^*}) =
\Tr(\gamma, \xi^M_{\lambda_B(\omega_\infty)})$ and taking 
\begin{equation}\label{eq:eps-diamond}
\eps^\diamond := (-1)^{\dim(A_M/A_G) - |\cI_{0}|}.
\end{equation}
\end{proof}

We keep on assuming $k > k_1$ and write $c_M' := \eps^\diamond  \eps(\omega_\infty) c_M d(G)\inv$ from now. We apply Lemma~\ref{lem:TF-technical-lemma}\,(\textit{ii}) to Equation~\eqref{eq:DalalRepeat3} and change the order of
summation to obtain
\begin{align}\label{eq:DalalRepeat4}
& T^G_{\el, \chi}(f_{\xi, \zeta}f^{\infty, (k)}) =  T^G_{\tu{disc},\chi}(f_{\xi,\zeta} f^{\infty, (k)})\, +
 \cr
& \sum_{M, P, \omega_p, \omega_q,\omega_\infty} c_{M}' v_{\fkX}\inv \sum_{\gamma \in \Gamma_{\el, \fkX}(M)}
\frac{\chi(I^M_\gamma) \Tr(\gamma ; \xi^M_{\lambda_B(\omega_\infty)}
\otimes \zeta \delta_P^{-1/2}) O^M_\gamma(f^{\infty,(k)}_{M,\omega_p,\omega_q})}{|\iota^M(\gamma)| |\tu{Stab}^M_\fkX(\gamma)|},
\end{align}
where $M,\omega_p,\omega_q$ run over the same sets as before and $P,\omega_\infty$ range over $\cP(M),\li \Omega^{G\diamond}_{M}$, respectively. We apply Lemma \ref{lem:RecognizeEllipticPart} to equalize
\begin{align}\label{eq:DalalRepeat5}
 v_{\fkX}\inv & \sum_{\gamma \in \Gamma_{\el, \fkX}(M)} \frac{\chi(I^M_\gamma)  \Tr(\gamma ; \xi^M_{\lambda_B(\omega_\infty)} \otimes \zeta \delta_P^{-1/2}) O^M_\gamma(f^{\infty,(k)}_{M,\omega_p,\omega_q})}{|\iota^M(\gamma)| |\tu{Stab}^M_\fkX(\gamma)|} \cr
 =\,  & d(M) \sum_{\gamma \in \Gamma_{\el, \fkX}(M)}
\frac{\vol(I_\gamma^M(\Q) A_{I^M_\gamma, \infty} \backslash I_\gamma^M(\A)/\fkX) O_\gamma^M(f_{\lambda_B(\omega_\infty), \zeta \delta_P^{-1/2}}f^{\infty, (k)}_{M, \omega_p, \omega_q})}{|\iota^M(\gamma)| |\tu{Stab}^M_\fkX(\gamma)|},
\end{align}
using that every $\gamma$ with $O_\gamma^M(f^{\infty, (k)}_{M, \omega_p, \omega_q}) \neq 0$ in \eqref{eq:DalalRepeat5} is regular since $f_{q,M,\omega_q}$ is supported on regular elements. Define
$$
\fkX_M := \fkX \cdot A_{M, \infty}, \qquad v_M := \vol(\fkX_\Q \backslash \fkX / A_{G, \infty})\inv \vol(\fkX_{M, \Q} \backslash
\fkX_M / A_{M, \infty}).
$$
The restriction of the central character of $\xi^M_{\lambda_B(\omega_\infty)} \otimes \zeta \delta_P^{-1/2}$ to $A_{M, \infty}$
is denoted by
$$
z_{\omega_\infty}^M \colon A_{M, \infty} \to \C^\times.
$$
Since the central character of $\xi^M_{\lambda_B(\omega_\infty)}$ restricts to the central character of $\xi$ on $Z_G$, we have
$$
z_{\omega_\infty}^M|_{A_{G,\infty}} = \chi_0\inv.
$$
On the other hand, $Z_G(\R) \cap A_{M, \infty} = A_{G, \infty} \subset Z_M(\R)$.
Therefore
$$
\fkX \cap A_{M, \infty} = A_{G, \infty}.
$$
Consequently, there exists a unique character
$$
\chi^M_{\omega_\infty} \colon \fkX_M \to \C^\times
$$
such that
$$
\chi^M_{\omega_\infty}|_{A_{M, \infty}} = (z^M_{\omega_\infty})\inv \quad
\tu{and} \quad \chi^M_{\omega_\infty}|_{\fkX} = \chi.
$$

The pair $(\fkX_M, \chi^M_{\omega_\infty})$ is a central character datum for $M$ as in \S\ref{sub:TF-fixed-central}. Observe that $\fkX_M$ is of the form $\fkX_M = Y(\A) A_{M,\infty}$ as required by the statement of Theorem~\ref{thm:main-estimate}. Moreover,
$$
f_{\lambda_B(\omega_\infty), \zeta \delta_P^{-1/2}}f^{\infty, (k)}_{M, \omega_p, \omega_q} \in \cH(M(\A), \chi^{M,-1}_{\omega_\infty}).
$$
The expression in \eqref{eq:DalalRepeat5} can be rewritten as
\begin{align}\label{eq:DalalRepeat5b}
 d(M)\sum_{\gamma \in \Gamma_{\el, \fkX_M}(M)} 
 & \frac{ v_M \cdot\vol(I_\gamma^M(\Q) A_{I^M_\gamma, \infty} \backslash I_\gamma^M(\A)/\fkX_M) O_\gamma^M(f_{\lambda_B(\omega_\infty), \zeta \delta_P^{-1/2}}f^{\infty, (k)}_{M, \omega_p, \omega_q})}{|\iota^M(\gamma)| |\tu{Stab}^M_{\fkX_M}(\gamma)|} \cr
=\, & d(M) v_M \cdot T^M_{\el, \chi^{M}_{\omega_\infty}} (f_{\lambda_B(\omega_\infty)}, \zeta \delta_P^{-1/2} f^{\infty, (k)}_{M, \omega_p, \omega_q}).
\end{align}
Put $c_M'' := c_M' v_M d(M)$. Combining \eqref{eq:DalalRepeat4} and \eqref{eq:DalalRepeat5b}, we obtain
\begin{align}\label{eq:DalalRepeat6}
& T^G_{\el, \chi}(f_{\xi, \zeta}f^{\infty, (k)}) =  T^G_{\tu{disc},\chi}(f_{\xi,\zeta} f^{\infty, (k)}) +
  \sum_{M, P, \omega_p, \omega_q, \omega_\infty} c_{M}'' \cdot T_{\el,\chi^{M}_{\omega_\infty} }^M(
f_{\lambda_B(\omega_\infty), \zeta \delta_P^{-1/2}} f_{M, \omega_p, \omega_q}^{\infty,(k)}).
\end{align}

Let $\omega_p$ be as in the sum. Since $\omega_p \in \Omega^{G_{\qp}}_{M, M_\nu}$, we have $\omega_p(M\cap B)\subset B$ and $\omega^{-1}_p(M_\nu \cap B)\subset B$. In particular, for each root $\alpha$ in $\tu{Lie}(M \cap N_0)$, the root $\omega_p\alpha$ also appears in $\tu{Lie}(M \cap N_0)$. So
$$
\langle \alpha, w_p\inv \nu \rangle = \langle w_p \alpha , \nu \rangle \geq 0.
$$
Hence $w_p\inv \nu$ is dominant for $M \cap B$. (See the paragraph above Proposition~\ref{prop:spectral-estimate} for dominance of $\nu$ relative to $B$.) By Proposition~\ref{prop:spectral-estimate} and the induction hypothesis for $M\in \cL^<_{\tu{cusp}}$, we have
\begin{align}\label{eq:Spectral-Est1}
T_{\el, \chi^{M}_{\omega_\infty}}^M(
f_{\lambda_{\omega_\infty}, \zeta \delta_P^{1/2}} f_{M, \omega_p, \omega_q}^{\infty,(k)}) =
O\left( p^{k(\langle \rho_M, \omega_p\inv \nu \rangle + \langle
(\chi_{\omega_\infty}^M)_\infty
, \tu{pr}_M(\omega_p\inv
\nu)\rangle )}
\right).
\end{align}
(To apply the induction hypothesis, we need to ensure that the setup of Theorem \ref{thm:main-estimate} applies to the left hand side. The point is that the conditions at $p$ and $q$ are satisfied. At $p$, this is a consequence of \eqref{eq:function-f-p-decomposition-2} and Lemma \ref{lem:acc-preservation-const-term}; thus each $f_{p, M, \omega_p}^{(k)}$ is an ascent from an acceptable function. At $q$, this follows from Lemma \ref{lem:C-reg-inherited} (1).). In the special case $M = G$ we obtain
\begin{align}\label{eq:Spectral-Est2}
T_{\disc,\chi}^G( f_{p}^{(k)}f^{\infty,p}f_{\xi,\zeta} ) = O \left( p^{k(\langle \rho,\nu \rangle+\langle \chi_\infty,\tu{pr}_G\nu\rangle)} \right).
\end{align}

Now assume that the datum $(M, P, \omega_p, \omega_q, \omega_\infty)$ contributes to~\eqref{eq:DalalRepeat6}, in particular $M\in \cL^<_{\tu{cusp}}$, and also assume that
\begin{equation}\label{eq:Part2Assumption}
O_\gamma^M(f_{\lambda_{\omega_\infty}, \zeta \delta_P^{1/2}}f^{\infty, (k)}_{M, \omega_p, \omega_q}) \neq 0
\end{equation}
for some $\gamma \in \Gamma_{\tu{ell}, \fkX_M}(M)$. Then we claim that
\begin{equation}\label{eq:Part2Start}
\tu{Re}(\langle \rho, \nu \rangle + \langle \chi_\infty, \tu{pr}_G\nu \rangle) >
\tu{Re}(\langle \rho_M, \omega_p\inv \nu \rangle + \langle
(\chi_{\omega_\infty}^M)_\infty, \tu{pr}_M(\omega_p\inv
\nu)\rangle).
\end{equation}
This claim, together with \eqref{eq:Spectral-Est1} and \eqref{eq:Spectral-Est2}, tells us that the main term for $G$ dominates the proper Levi terms in~\eqref{eq:DalalRepeat6}, thereby implies the theorem. 

It remains to verify the claim \eqref{eq:Part2Start}. Clearly it is sufficient to show that 
\begin{itemize}
\item[(\textit{a})]  $\langle \rho, \nu \rangle > \langle \rho_M, \omega_p\inv\nu \rangle$,
\item[(\textit{b})] $\tu{Re} \langle \chi_\infty, \tu{pr}_G\nu \rangle \geq \tu{Re} \langle (\chi_{\omega_\infty}^M)_\infty, \tu{pr}_M(\omega_p\inv \nu) \rangle$.
\end{itemize}
Moreover, it is enough to prove (a) and (b) for sufficiently large $k$ (note that the set $\Omega^{G \diamond}_M$ and thus $\omega_\infty$ depends on $k$). To prove (\textit{a}), we start from the equality $\langle \rho_M, \omega_p\inv \nu \rangle = \langle \rho_{\omega_p M},  \nu \rangle$. Since $\omega_p$ is a Kostant representative (cf. \eqref{eq:function-f-p-decomposition}),
$$
 \langle \rho_{\omega_p M},  \nu \rangle < \langle \rho,  \nu \rangle.
$$
To check that
$$
\langle \rho_{\omega_p M}, \nu \rangle \neq \langle \rho, \nu \rangle,
$$
we argue as in the paragraph below Equation \eqref{eq:EndoscopicFinalRootArg}: It suffices to find a root $\alpha$ in $\tu{Lie}(G)$ that is not in $\tu{Lie}(\omega_p M)$ such that $\langle \alpha, \nu \rangle \neq 0$. As $\nu$ is not central, the argument for Lemma 4.5(ii) of \cite{KST2} shows that  $\tu{Lie}(M_\nu) + \tu{Lie}(M) \neq \tu{Lie}(G)$. Hence we can find a root  $\alpha$ in $\tu{Lie}(G)$ which does not occur in either $\tu{Lie}(M)$ or $\tu{Lie}(M_\nu)$, \ie $\langle \alpha, \nu \rangle \neq 0$. The proof of (a) is finished.

Now we prove (\textit{b}). We begin with the following temporary assumption:
\begin{equation}\label{eq:TempAssumption}
\langle \pr_M^*(-\omega_\infty \star \lambda_B^*) - \li \rho_N - \omega_\xi, \tu{pr}_M(w_p\inv \nu) \rangle \neq 0.
\end{equation}
From the equality $I(\gamma) = I(\omega_\infty)$ we obtain (cf. \eqref{eq:I-sets})
$$
\langle \pr_M^*(-\omega_\infty \star \lambda_B^*) - \li \rho_N - \omega_\xi, x \rangle \leq 0.
$$
As $O_\gamma(f_{p, M, \omega_p}^{(k)}) \neq 0$, we have $x = -k(\log p) \tu{pr}_M(w_p\inv \nu) + \eps_p$, where $\eps_p \in \tu{pr}_M\supp_{\ia_{M_{\omega_p}}}^O(\omega_p\inv \phi_{p, M_{\omega_p}})$ (cf. \eqref{eq:x_p-form}). Therefore
$$
\langle \pr_M^*(-\omega_\infty \star \lambda_B^*) - \li \rho_N - \omega_\xi, -k(\log p)\tu{pr}_M(w_p\inv \nu) + \eps_p \rangle \leq 0.
$$
We may and do assume that  $k$ is sufficiently large so that
$$
k|\langle \pr_M^*(-\omega_\infty \star \lambda_B^*) - \li \rho_N - \omega_\xi, -(\log p)\tu{pr}_M(w_p\inv \nu)\rangle| > |\langle \pr_M^*(-\omega_\infty \star \lambda_B^*) - \li \rho_N - \omega_\xi, \eps_p' \rangle | 
$$
for all $\eps_p' \in \tu{pr}_M\supp_{\ia_{M_{\omega_p}}}^O(\omega_p^{\prime,-1} \phi_{p, M_{\omega_p}})$ and all $\omega_p' \in \Omega^{G_{\qp}}_{M, M_\nu}$. (It is enough to impose a lower bound for $k$ depending only on the initial data.) We deduce from the temporary assumption \eqref{eq:TempAssumption} and $k \gg 1$ in the above sense that
\begin{equation}\label{eq:TempAssumption2}
\langle \pr_M^*(-\omega_\infty \star \lambda_B^*) - \li \rho_N - \omega_\xi, \tu{pr}_M(w_p\inv \nu) \rangle \geq 0.
\end{equation}
Now, if \eqref{eq:TempAssumption} is false then \eqref{eq:TempAssumption2} obviously holds. Hence \eqref{eq:TempAssumption2} is true regardless of the temporary assumption \eqref{eq:TempAssumption}. We now conclude:
\begin{align*}
\langle (\chi_{\omega_\infty}^M)_\infty, \pr_M(w_p\inv \nu) \rangle
& = -\langle z_{\omega_\infty}^M, \pr_M(w_p\inv \nu) \rangle \cr
& = -\langle \pr_M^*(\lambda_{\omega_\infty}), \pr_M(w_p\inv \nu) \rangle
- \langle \zeta \delta_P^{-1/2}, \pr_M(w_p\inv \nu) \rangle
\cr
& = -\langle \pr_M^*(-w_0(\omega_\infty \star \lambda_B^*)), \pr_M(w_p\inv \nu) \rangle
- \langle \zeta - \rho_N , \pr_M(w_p\inv \nu) \rangle
\cr
& = \underset {\leq 0} {\underbrace{\langle \pr_M^*(\omega_\infty \star \lambda_B^*) + \li \rho_N + \omega_\xi, \pr_M(w_p\inv \nu) \rangle}}
 + \underset {= \langle \chi_\infty, \tu{pr}_M(w_p\inv \nu) \rangle }
 {\underbrace{\langle - \omega_\xi - \zeta, \pr_M(w_p\inv \nu) \rangle}}
\end{align*}
\end{proof}\label{EndOfArgMainEst}
 
\section{Shimura varieties of Hodge type}\label{sec:Shimura}

The goal of this section is to set up the scene for the mod $p$ geometry of Shimura varieties and central leaves, paving the way for introducing Igusa varieties in the next section. We pay special attention to the connected components and $H^0$.

\subsection{Connected components in characteristic zero}\label{sub:components-char0} From here on, let $(G,X)$ be a Shimura datum as in \cite{DeligneCorvallis} satisfying axioms (2.1.1.1), (2.1.1.2), and (2.1.1.3) therein. Write $E=E(G,X)$ for the reflex field \cite[2.2.1]{DeligneCorvallis}, which is a finite extension of $\Q$ in $\C$. We have the algebraic closure $\ol E\subset \C$. Let $K$ be a neat open compact subgroup of $G(\A^\infty)$. (See \cite[p.82]{LanPEL} for the definition of neatness in an adelic group following Pink.) We write $\Sh_K=\Sh_K(G,X)$ for the canonical model over $E$, which forms a projective system of quasi-projective varieties with finite \'etale transition maps as $K$ varies. We have the $E$-scheme $\Sh:=\varprojlim_{K} \Sh_{K}$. Put $d:=\dim \Sh_K$ (which does not depend on $K$). Write $G(\Q)_+$ for the preimage of $G(\R)_+$ (defined in \S\ref{sub:Lefschetz}) in $G(\Q)$. The closure of $G(\Q)_+$ in $G(\A^\infty)$ is denoted by $G(\Q)_+^-$.

Recall some facts about connected components from \cite[2.1]{DeligneCorvallis}. We have a bijection
\begin{equation}\label{eq:pi_0-generic}
\pi_0(\Sh_{K,\ol{E}}) \isom G(\Q) \backslash G(\A)/G(\R)_+ K,
\end{equation}
which yields a $G(\A^\infty)$-equivariant bijection $\pi_0(\Sh_{\ol{E}})\isom G(\A)/G(\Q)\varrho(G_{\textup{sc}}(\A)) G(\R)_+$ upon taking limit over all $K$. 
 Note that $G(\A)/G(\Q)\varrho(G_{\textup{sc}}(\A)) G(\R)_+$ is an abelian group quotient of $G(\A)$, and $G(\Q) \backslash G(\A)/G(\R)_+ K$ is a finite abelian group quotient.

 Fix a prime $\ell$ and a field isomorphism $\iota:\lql\simeq \C$. When $V$ is a $\lql$-vector space, write $\iota V:=V\otimes_{\lql,\iota} \C$. By convention, all instances of cohomology in this paper are \'etale cohomology. The description of $\pi_0(\Sh_{\ol{E}})$ translates into a $G(\A^\infty)$-module isomorphism
\begin{equation}\label{eq:H0-Sh}
\iota H^{0}(\Sh_{\ol{E}},\lql)\simeq \bigoplus_\pi \pi^{\infty},
\end{equation}
where the sum runs over one-dimensional automorphic representations $\pi$ such that $\pi_\infty$ is trivial when restricted to $G(\R)_+$. Indeed, at each prime $p$, we have $\dim \pi_p=1$ since $\pi_p$ factors through $G(\Q_p)\ra G(\Q_p)^{\tu{ab}}=G(\Q_p)/\varrho(G_{\textup{sc}}(\Q_p))$, cf.~Corollary \ref{cor:1-dim reps}. Since one-dimensional automorphic representations have automorphic multiplicity one, there is no multiplicity factor in \eqref{eq:H0-Sh}.

Now fix a prime $p\neq \ell$ and an open compact subgroup $K_p\subset G(\Q_p)$. By taking limit of \eqref{eq:pi_0-generic} over neat open compact subgroups $K^p\subset G(\A^{\infty,p})$, writing $\Sh_{K_p}:=\varprojlim_{K^p} \Sh_{K_p K^p}$,
 \begin{equation}\label{eq:pi_0-generic-away-from-p}
  \pi_0(\Sh_{K_p,\ol{E}}) \isom G(\Q)_+^- \backslash G(\A^\infty)/K_p.
\end{equation}
We have a $G(\A^{\infty,p})$-module\footnote{See \cite[\href{https://stacks.math.columbia.edu/tag/03Q4}{Tag 03Q4}]{stacks-project} for the canonical isomorphism, which is $G(\A^{\infty,p})$-equivariant by a routine check. Alternatively, it is harmless to think of the identity as a definition for the left hand side.}  
$$
H^{i}(\Sh_{K_p,\ol{E}},\lql)=\varinjlim_{K^p} H^{i}(\Sh_{K_pK^p,\ol{E}},\lql),\qquad i\ge 0,
$$
where $K^p$ runs over sufficiently small open compact subgroups of $G(\A^{\infty,p})$. 

\begin{lemma}\label{lem:H_c(Sh)-top}
There is a $G(\A^{\infty,p})$-module isomorphism
$$
\iota H^{0}(\Sh_{K_p,\ol{E}},\lql)\simeq \bigoplus_\pi \pi^{\infty,p}, 
$$
where the sum runs over discrete automorphic representations $\pi$ of $G(\A)$ such that (i) $\dim \pi = 1$, (ii) $\pi_p$ is trivial on $K_p$, and (iii) $\pi_\infty$ is trivial on $G(\R)_+$.
\end{lemma}

\begin{proof}
This is clear from \eqref{eq:H0-Sh} by taking $K_p$-invariants. 
\end{proof}

\subsection{Integral canonical models}\label{sub:integral-canonical-models}

Let $(G,X)$ be a Shimura datum of \emph{Hodge type}. This means that there exists an embedding into the Siegel Shimura datum
$$
i_{V,\psi}:(G,X)\hra (\GSp(V,\psi),S_{V,\psi}^{\pm}),
$$
where $(V,\psi)$ is a symplectic space over $\Q$, and $S_{V,\psi}^{\pm}$ denotes the associated Siegel half spaces. For simplicity we write $\GSp=\GSp(V,\psi)$ and $S^\pm=S_{V,\psi}^{\pm})$.
\begin{definition}\label{def:unramified-Shimura-datum}
An \textbf{unramified Shimura datum} is a quadruple $(G,X,p,\cG)$, where $(G,X)$ is a Shimura datum, $p$ is a prime, and $\cG$ is a reductive model for $G$ over $\Z_{(p)}$. (In particular $G$ is unramified over $\Q_p$.) Write $\mathcal{SD}^{\tu{ur}}_{\tu{Hodge}}$ for the collection of unramified Shimura data whose underlying Shimura data are of Hodge type.
\end{definition}
For the rest of this paper, we fix $(G,X,p,\cG)\in \mathcal{SD}^{\tu{ur}}_{\tu{Hodge}}$ and $i_{V,\psi}$, thus also a hyperspecial subgroup $K_p:=\cG(\Z_p)$ of $G(\Q_p)$. Since $G$ is unramified over $\Q_p$, the prime $p$ is unramified in the reflex field $E$. We fix an isomorphism $\iota_p:\C\simeq \lqp$, which induces an embedding $E\hra \lqp$ as well as a $p$-adic place $\fkp$ of $E$. Thereby we identify $\ol E_\fkp\simeq \lqp$. The integer ring $\cO_E$ localized at $\fkp$ is denoted by $\cO_{E,(\fkp)}$, and its residue field by $k(\fkp)$. Identify the residue field of $\lqp$ with $\Fpbar$, thus fixing an embedding $k(\fkp)\hra \Fpbar$.

We follow \cite[(1.3.3)]{KisinPoints} to review integral canonical models for $\Sh=\Sh(G,X)$ over $\cO_{E,(\fkp)}$, leaving the details to \emph{loc.~cit}. 
We may assume that $i_{V,\psi}$ is induced by an embedding $\cG \hra \GL(V_{\Z_{(p)}})$ for a $\Z_{(p)}$-lattice $V_{\Z_{(p)}}\subset V$ and that $\psi$ induces a perfect pairing on $V_{\Z_{(p)}}$. There exists a finite set of tensors $(s_\alpha)\subset V_{\Z_{(p)}}^\otimes$ such that $\cG$ is the scheme-theoretic stabilizer of $(s_\alpha)$ in $\GL(V_{\Z_{(p)}})$. We may assume that one of the tensors is given by $\psi\otimes \psi^\vee\in (V^\vee_{\Z_{(p)}})^{\otimes 2} \otimes V^{\otimes 2}_{\Z_{(p)}}$, whose stabilizer is $\GSp(V_{\Z_{(p)}},\psi)$.\footnote{This way the weak polarization in the sense of \cite{KisinPoints} is remembered by $(s_\alpha)$. So we need not keep track of polarizations on abelian varieties separately.} We fix the set $(s_\alpha)$. There is a hyperspecial subgroup $K'_p\subset \GSp(V,\psi)(\Q_p)$ extending $K_p$ (\ie~$K'_p\cap G(\Q_p)=K_p$) such that $i_{V,\psi}$ induces an $E$-embedding of Shimura varieties 
\begin{equation}\label{eq:Shimura-morphism-generic}
\Sh_{K_p}(G,X)\hra \Sh_{K'_p}(\GSp,S^{\pm})\otimes_\Q E.
\end{equation}

Kisin \cite[Thm.~2.3.8]{KisinModels} (for $p>2$) and Kim--Madapusi Pera \cite[Thm.~4.11]{KimMadapusiPera} (for $p=2$) constructed integral canonical models, as a projective system of smooth quasi-projective schemes $\mS_{K_pK^p}$ over $\cO_{E,(\fkp)}$ for all sufficiently small open compact subgroups $K^p\subset G(\A^{\infty,p})$ with finite \'etale transition maps $\mS_{K_pK^{p,\prime}}\ra \mS_{K_pK^p}$ for $K^{p,\prime}\subset K^p$. The projective system is equipped with an action of $G(\A^{\infty,p})$, given by the isomorphism
$$
\mS_{K_p K^p}\isom  \mS_{K_p g^{-1} K^p g},\qquad g\in G(\A^{\infty,p}),\quad K^p\subset G(\A^{\infty,p}),
$$
extending the isomorphism $\Sh_{K_p K^p}\isom  \Sh_{K_p g^{-1} K^p g}$ giving the action of $g$ on the generic fiber. The inverse limit $\mS_{K_p}:= \underleftarrow{\lim}_{K^p}\mS_{K_p K^p}$ is a scheme over $\cO_{E,(\fkp)}$ with a $G(\A^{\infty,p})$-action, uniquely characterized  by an extension property \cite[Thm.~2.3.8]{KisinModels}. The construction yields a map of $\cO_{E,(\fkp)}$-schemes
\begin{equation}\label{eq:Shimura-morphism-integral}
\mS_{K_p}\ra  \mS_{K'_p}(\GSp,S^{\pm})\otimes_{\Z_{(p)}} \cO_{E,(\fkp)},
\end{equation}
whose base change to $E$ is identified with \eqref{eq:Shimura-morphism-generic},
where $ \mS_{K'_p}(\GSp,S^{\pm})$ is the integral model over $\Z_{(p)}$ for $\Sh(\GSp(V,\psi),S_{V,\psi}^{\pm})$ parametrizing polarized abelian schemes up to prime-to-$p$ isogenies with prime-to-$p$ level structure, as in \cite[(2.3.3)]{KisinModels}. Moreover we have universal polarized abelian schemes $h:\mathcal A_{K_p K^p}\ra \mS_{K_p K^p}$ compatible with the transition maps in the projective system.

Let $\mS_{K_pK^p,k(\fkp)}:=\mS_{K_p K^p}\otimes_{\cO_{E,(\fkp)}} k(\fkp)$ denote the special fiber. Write $\Sh_{K_p}$ (resp.~$\mS_{K_p,k(\fkp)}$) for the  inverse limit of $\Sh_{K_pK^p}$ (resp.~ $\mS_{K_pK^p,k(\fkp)}$) over $K^p$. By base change to $\ol{E}_\fkp$, $\cO_{\ol{E}_\fkp}$, and $\ol {k(\fkp)}$, respectively, we obtain $\Sh_{K_p,\ol{E}_\fkp}$, $\mS_{K_p,\cO_{\ol{E}_\fkp}}$, and $\mS_{K_p,\ol{k(\fkp)}}$ from $\Sh_{K_p}$, $\mS_{K_p}$, and $\mS_{K_p,k(\fkp)}$. There are canonical $G(\A^{\infty,p})$-equivariant embeddings of generic and special fibers 
$$
\Sh_{K_p,\ol{E}_\fkp} \hookrightarrow \mS_{K_p,\cO_{\ol{E}_\fkp}}\hookleftarrow \mS_{K_p,\ol{k(\fkp)}}.
$$
These embeddings induce $G(\A^{\infty,p})$-equivariant bijections by means of arithmetic compactification as shown in  \cite[Cor.~4.1.11]{MadapusiPeraToroidal}:
$$
\pi_0(\Sh_{K_p,\ol{E}_\fkp})\isom \pi_0(\mS_{K_p,\cO_{\ol{E}_\fkp}}) \lisom \pi_0(\mS_{K_p,\ol{k(\fkp)}}).
$$

\begin{lemma}\label{lem:transitive-on-Sh}
The $G(\A^{\infty,p})$-action is transitive on $\pi_0(\Sh_{K_p,\ol{E}_\fkp})$ and $\pi_0(\mS_{K_p,\ol{k(\fkp)}})$.
\end{lemma}

\begin{remark}
Oki \cite{Oki} showed that the analogous transitivity is false if $G_{\Q_p}$ is ramified.
\end{remark}

\begin{proof}
By the preceding proposition, it is enough to check the transitivity on $\pi_0(\Sh_{K_p,\ol{E}_\fkp})$, which is \cite[Lem.~2.2.5]{KisinModels} (applicable since $K_p$ is hyperspecial). Alternatively, this also follows from weak approximation, which tells us that the diagonal embedding $G(\Q)\hra G(\Q_p)\times G(\R)$ has dense image. For this, apply \cite[Thm.~7.7]{PlatonovRapinchuk} and notice that the set $S_0$ of the theorem can be taken away from $p$ and $\infty$ from the discussion in \S7.3 of \emph{loc.~cit.} since $G$ is unramified at $p$.
\end{proof}

Let $T$ be a $k(\fkp)$-scheme. At each point $x\in \mS_{K_pK^p}(T)$ we have an abelian variety $\cA_x$ over $T$ (up to a prime-to-$p$ isogeny) pulled back from $\cA_{K_pK^p}$. As in \cite[(1.3.6)]{KisinPoints}, we have $(\mathbf{s}_{\alpha,\ell})\subset (R^1 h_{\tu{\'et}*}\Q_\ell)^\otimes$ for each prime $\ell\neq p$. By pullback, we equip the prime-to-$p$ rational Tate module $V^p(\mathcal A_x)$ of $\cA_x$ with $(\mathbf{s}_{\alpha,\ell,x})_{\ell\neq p}$.

When $T=\Spec k$ with $k/k(\fkp)$ an extension in $\li{k(\fkp)}$, write $\D(\cA_x[p^\infty])$ for the (integral) Dieudonn\'e module of $\cA_x[p^\infty]$, and $\Phi_x$ for the Frobenius operator acting on it. Following \cite[(1.3.10)]{KisinPoints} we have crystalline Tate tensors $(s_{\alpha,0,x})\subset \D(\cA_x[p^\infty])^\otimes$ coming from $(s_\alpha)$. Lovering \cite{Lovering}, and also Hamacher \cite[\S2.2]{Ham19}, have globalized $(s_{\alpha,0,x})$. Namely there exist crystalline Tate tensors $(\mathbf{s}_{\alpha,0})$ on the Dieudonn\'e crystal $\D(\cA_{K_pK^p}[p^\infty])$ associated with $\cA_{K_pK^p}[p^\infty]$ over $\mS_{K_pK^p,\ol{k(\fkp)}}$ such that $(\mathbf{s}_{\alpha,0})$  specializes to $(s_{\alpha,0,x})$ at every $x\in \mS_{K_pK^p}(\ol{k(\fkp)})$. 

\subsection{Central leaves}\label{sub:central-leaves}

Continuing from \S\ref{sub:integral-canonical-models}, we review central leaves in the special fiber of a Shimura variety of Hodge type. Let $B(G_{\Q_p})$ denote the set of $(G(\breve \Q_p),\sigma)$-conjugacy classes in $G(\breve\Q_p)$. Fix a Borel subgroup $B\subset \cG_{\Z_p}$ and a maximal torus $T\subset B$ over $\Z_p$. We have the set of dominant coweights $X_*(T_{\ol \Q_p})^+$ and $X_*(T_{\ol \Q_p})^+_\Q$. Via the fixed isomorphism $\iota_p:\lqp\simeq \C$, we obtain $T_{\C}\subset B_{\C}\subset G_{\C}$ as well as $X_*(T_{\C})^+$ and $X_*(T_{\C})^+_\Q$. Since the conjugacy class $\{\mu_X\}$ is defined over $E$ and since $G_{\Q_p}$ is quasi-split, we have a cocharacter 
$$\mu_p\in X_*(T_{\lqp})^+ \quad \tu{defined over}~E_\fkp$$
in the conjugacy class $\{\iota_p\mu_X\}$. When there is no danger of confusion, we omit the subscripts $\ol\Q_p$ and $\C$. Write $\rho\in X^*(T)_{\Q}$ for the half sum of all positive roots, and $\langle\cdot,\cdot \rangle$ for the canonical pairing $X^*(T)_{\Q}\times X_*(T)_{\Q}\ra \Q$ or its extension to $\C$-coefficients.

Each $b\in G(\breve\Q_p)$ gives rise to a Newton cocharacter $\nu_b\colon \D\ra G_{\breve \Q_p}$ (so it is a ``fractional'' cocharacter of $G_{\breve \Q_p}$) and a connected reductive group $J_b$ over $\Q_p$ given by
\begin{equation}\label{eq:def-Jb}
J_b(R):=\{ g\in G(R\otimes_{\Q_p} \breve \Q_p): g^{-1} b \sigma(g)=b\},\quad R: \Q_p\textup{-algebra}.
\end{equation}

\begin{lemma}\label{lem:nu-central}
 The Newton cocharacter $\nu_b$ factors through the center of $J_b$. The induced cocharacter $\D\ra A_{J_b}$ is $\Q_p$-rational.
\end{lemma}
\begin{proof}
The centrality follows from \cite[(4.4.2)]{KottwitzIsocrystal1}. The cocharacter $\D\ra A_{J_b}$ is $\sigma$-invariant by the definition of $J_b$, thus $\Q_p$-rational.
\end{proof}
We define an open compact subgroup of $J_b(\Q_p)$ (where ``int'' stands for integral): 
$$
J_b^{\tu{int}} :=J_b(\Q_p)\cap \cG( \breve \Z_p)= \{ g\in \cG( \breve \Z_p): g^{-1} b \sigma(g)=b\}.
$$
Given $b\in G(\breve \Q_p)$, we denote its $(G(\breve \Q_p),\sigma)$-conjugacy class by $[b]$ and $(\cG(\breve \Z_p),\sigma)$-conjugacy class by $[[b]]$. Recall that $b\in G(\breve\Q_p)$, or $[b]\in B(G_{\Q_p})$, is \textbf{basic} if $\nu_b:\D\ra G_{\breve \Q_p}$ has image in $Z(G_{\breve \Q_p})$, or equivalently if $J_b$ is an inner form of $G$ \cite[Prop.~1.12]{RR96}. The following condition will appear in our irreducibility results later. The definition makes a difference only when $G^{\tu{ad}}$ is not $\Q$-simple. See Lemma \ref{lem:Q-non-basic} below for a relation to \S\ref{sub:one-dim}.

\begin{definition}\label{def:Q-non-basic}
Let $G^{\tu{ad}}=\prod_{i\in I} G^{\tu{ad}}_i$ be a decomposition into $\Q$-simple factors. An element $b\in G(\breve \Q_p)$, or $[b]\in B(G_{\Q_p})$, is said to be \textbf{$\Q$-non-basic} if its image in $B(G_{i,\Q_p})$ via the natural composite map $G\ra G^{\tu{ad}}\ra G_i$ is non-basic for every $i\in I$.
\end{definition}

\begin{remark}
The definition is not purely local in that it depends on not only $G_{\Q_p}$ but also $G$. Compare $G=\GL_2\times \GL_2$ with $G=\Res_{F/\Q}\GL_2$, where $F$ is a real quadratic field in which $p$ splits.
\end{remark}

Let $x\in \mS_{K_pK^p,k(\fkp)}$. 
 Write $k$ for the residue field at $x$, $W:=W(k)$, and $L:=\Frac\,W$. We will still write $\sigma$ for the canonical Frobenius on $W$. There exists a $W$-linear isomorphism
\begin{equation}\label{eq:fix-p-adic-lattice}
 V^*_{\Z_{(p)}}\otimes _{\Z_{(p)}} W \simeq \D(\cA_{\ol x}[p^\infty])
\end{equation}
carrying $(s_\alpha)$ to $(s_{\alpha,0,x})$. The Frobenius operator $\Phi_{x}$ on the right hand side is transported to $b_{ x}(1\otimes \sigma)$ on the left hand side for a unqiue $b_{x}\in G(L)$. Then $[[b_{ x}]]$ (thus also $[b_{ x}]$) is independent of the choice of isomorphism. 
Now let $\ol x:\Spec k \hra \mS_{K_pK^p,k(\fkp)}$ be a geometric point supported at $x\in \mS_{K_pK^p,k(\fkp)}$, with $k$ an algebraically closed field over $k(\fkp)$. Then we can define $b_{\ol x}$, $[b_{\ol x}]$, and $[[b_{\ol x}]]$ similarly.

As a set, the central leaf associated with $b$ is defined as
$$
C_{b,K^p}:=\{x\in   \mS_{K_pK^p,k(\fkp)} \,:\, \exists~\textrm{isom.~of~crystals}~V^*_{\Z_{(p)}}\otimes _{\Z_{(p)}} W \simeq \D(\cA_{\ol x}[p^\infty])~\textrm{s.t.}~ (s_\alpha)\mapsto(\mathbf{s}_{\alpha,0,\ol{x}})  \},
$$
where the isomorphism of crystals means a $W$-linear isomorphism through which the $\sigma$-linear map $b(1\otimes \sigma)$ is carried to $\Phi_{\ol x}$. The definition of $C_{b,K^p}$ depends only on $[[b]]$.

By \cite[Cor.~4.12]{HamacherKim}, $C_{b,K^p}$ is a locally closed subset of $\mS_{K_pK^p,k(\fkp)}$. (The result is stated for $\Fpbar$-points there, but the same proof applies to the underlying topological spaces.) We promote $C_{b,K^p}$ to a locally closed $k(\fkp)$-subscheme of $\mS_{K_pK^p,k(\fkp)}$ equipped with reduced subscheme structure. We still write $C_{b,K^p}$ for the scheme and call it the \textbf{central leaf} associated with $b$. As $K^p$ varies, the transition maps for $\mS_{K_pK^p,k(\fkp)}$, which are finite \'etale (\!\!\cite[Thm.~2.3.8]{KisinModels}), induce finite \'etale transition maps between $C_{b,K^p}$. Put 
$C_b:=\varprojlim_{K^p} C_{b,K^p}.$

\begin{proposition}\label{prop:smoothness-and-dim-of-leaves}
The $k(\fkp)$-scheme $C_{b,K^p}$ is smooth. If nonempty, its dimension is $\langle 2\rho,\nu_b\rg$.
\end{proposition}

\begin{proof}
These properties can be checked after extending base to $\ol{k(\fkp)}$. Since $C_{b,K^p}$ is reduced, it is still reduced over $\ol k(\fkp)$. Thus the proposition follows from \cite[Prop.~ 2.6]{Ham19}.
\end{proof}

A finite subset $B(G_{\Q_p},\mu^{-1}_p)\subset B(G_{\Q_p})$ is defined in \cite[\S6]{KottwitzIsocrystal2} by a group-theoretic generalization of Mazur's inequality. The set $B(G_{\Q_p},\mu^{-1}_p)$ contains exactly one basic element, but may contain several elements that are not $\Q$-non-basic. 

\begin{proposition}\label{prop:C-nonempty}
The central leaf $C_{b,K^p}$ is nonempty if and only if $b\in \cG(\Z_p^{\tu{ur}}) \sigma\mu_p^{-1}(p)\cG(\Z_p^{\tu{ur}})$.
\end{proposition}

\begin{proof}
By \cite[Prop.~1.3.9]{KMPS}, the Newton stratum for $b$ is nonempty if and only if $[b]\in B(G_{\Q_p},\mu^{-1}_p)$. On the other hand, if $b\in \cG(\Z_p^{\tu{ur}}) \sigma\mu_p^{-1}(p)\cG(\Z_p^{\tu{ur}})$ then $[b]\in B(G_{\Q_p},\mu^{-1}_p)$ by \cite[\S4]{RapoportZink}. 

 To prove the ``only if'' part of the proposition, we assume $C_{b,K^p}\neq \emptyset$. Then $[b]\in B(G_{\Q_p},\mu^{-1}_p)$ by the preceding paragraph. Since $C_{b,K^p}$ is of finite type over $k(\fkp)$, there exists a point $x\in C_{b,K^p}$ with finite residue field, say $\F_{p^r}$. Then $[[b_x]]=[[b]]$, and $b_x\in \cG(\Z_{p^r})  \sigma\mu^{-1}_p(p) \cG(\Z_{p^r})$ by \cite[1.4.1]{KisinPoints}.

In the ``if'' direction, the condition on $b$ implies that $[b]\in B(G_{\Q_p},\mu^{-1}_p)$, so $N_{b,K^p}\neq \emptyset$ for neat subgroups $K^p$. Pick a closed point $x\in N_{b,K^p}$. Then $b_x\in \cG(\Z_{p^r})  \sigma\mu^{-1}_p(p) \cG(\Z_{p^r})$ by \cite[1.4.1]{KisinPoints}, and $[b_x]=[b]$. Writing $b=g^{-1} b_x \sigma(g)$ for some $g\in G(\breve \Q_p)$, we see that $g$ lies in the affine Deligne--Lusztig variety for $(b_x,\sigma\mu_p^{-1})$. Using $x$ as a base point, we can apply the $p$-power isogeny corresponding to $g$ to find a closed point $y\in N_{b,K^p}$, thanks to \cite[Prop.~1.4.4]{KisinPoints}. By construction $[[b_y]]=[[b]]$, so $C_{b,K^p}$ is nonempty as desired. 
\end{proof}

Henceforth we always assume that $b\in \cG(\Z_p^{\tu{ur}}) \sigma\mu_p^{-1}(p)\cG(\Z_p^{\tu{ur}})$ and choose a sufficiently divisible $r$ such that 
\begin{enumerate}
\item[(br1)] $b\in \cG(\Z_{p^r})\sigma\mu^{-1}_p(p) \cG(\Z_{p^r})$,
\item[(br2)]  $r$ divides $[E_{\fkp}:\Q_p]$ (equivalently $\F_{p^r}\supset k(\fkp)$),
\item[(br3)] $r \nu_{b_x}:\G_m\ra G_{\breve \Q_p}$ is a cocharacter (not just a fractional cocharacter).
\end{enumerate}
These conditions imply the following.
\begin{itemize}
\item[(br1)'] $[b]\in B(G_{\Q_p},\mu_p^{-1})$ and $\nu_b$ is defined over $\Q_{p^r}$, by (br1).
\item[(br2)'] $\mu_p$ is defined over $\Q_{p^r}$, by (br2).
\end{itemize}
Since $\mu_p$ is defined over $E_{\fkp}$, which is unramified over $\Q_p$, (br2)' is easy to see. In (br1)', $[b]\in B(G_{\Q_p},\mu_p^{-1})$ comes from \cite[Thm.~4.2]{RR96}.  we already explained above that $\nu_b$ is defined over $\Q_{p^r}$ if $b\in G(\Q_{p^r})$.  Since $b_x\in G(\Q_{p^r})$, \cite[(4.4.1)]{KottwitzIsocrystal1} tells us that $\nu_{b_x}$ is defined over $\Q_{p^r}$. 

Since the $G(\Q_{p^r})$-conjugacy class of $r\nu_b$ is defined over $\Q_p$ \cite[(4.4.3)]{KottwitzIsocrystal1}, and since $G_{\Q_p}$ is quasi-split, there exists $h\in G(\Q_{p^r})$ such that $h^{-1}(r\nu_{b})h$ is defined over $\Q_p$. Multiplying $h$ on the right by an element of $G(\Q_p)$, we can ensure that $h^{-1}(r\nu_{b})h$ factors through $\G_m\ra T$ and is $B$-dominant, namely $h^{-1}(r\nu_{b})h\in X_*(T)^+$. Fix such a $h$ and put $b^\circ := h^{-1}b\sigma(h)$ so that $\nu_{b^\circ}=h^{-1}(\nu_{b})h$ from \cite[(4.4.2)]{KottwitzIsocrystal1}. We also have a $\Q_p$-isomorphism
$$J_b \stackrel{\sim}{\ra} J_{b^\circ},\qquad  g\mapsto h^{-1} g h$$
determined by $h$, which carries $r\nu_b$ to $r\nu_{b^\circ}$.

Starting from $\nu_{b^\circ}\in X_*(T)^+_{\Q}$ defined over $\Q_p$ as above,
we put $P_{b^\circ}:=P_{\nu_{b^\circ}}$ in the notation of \S\ref{sub:nu-Jacquet}, and similarly define $P^{\textup{op}}_{b^\circ}$, $N_{b^\circ}$, $N_{b^\circ}^{\textup{op}}$, and $M_{b^\circ}$. In particular $P^{\tu{op}}_{b^\circ}$ (resp. $M_{b^\circ}$) is a standard $\Q_p$-rational parabolic (resp. Levi) subgroup of $G_{\Q_p}$, and $M_{b^\circ}$ is the centralizer of $\nu_{b^\circ}$ in $G_{\Q_p}$. There is an inner twist \cite[Cor.~1.14]{RapoportZink} 
\begin{equation}\label{eq:Jb-Mb-isom}
J_{b^\circ}\otimes_{\Q_p} \Q_{p^r}\simeq M_{b^\circ}\otimes_{\Q_p} \Q_{p^r}
\end{equation}
given by the cocycle $\Gal(\Q_{p^n}/\Q_p)\ra M_{b^\circ}(\Q_{p^r})$, $\sigma\mapsto b^\circ$. Thus $M_{b^\circ}$ is also an inner twist of $J_b$ over $\Q_p$ (which is independent of the choice of $b^\circ$ up to isomorphism of inner twists by routine check). Under the canonical $\Q_p$-isomorphisms $Z(M_{b^\circ})\simeq Z(J_b)$ and $A_{M_{b^\circ}}=A_{J_b}$, it is readily checked that $\nu_{b^\circ}$ is carried to $\nu_b$.

\begin{example}
We have the following for the ordinary strata of modular curves, when
$G_{\Q_p}=\GL_2$. Take $B$ and $T$ to the subgroup of upper triangular (resp.~diagonal) matrices. Then $\mu$ is the cocharacter $z\mapsto \tu{diag}(z,1)$ up to conjugacy. We can take $b=b^\circ$ such that $\nu_{b}(z)=\tu{diag}(1,z^{-1})$, which is visibly $B$-dominant. Then $P^{\tu{op}}_{b}=B=P_{-\nu_b}$, $M_{b}=T$, and $\delta_{P_{b}}(\nu_b(p))= |p^{-1}|=p$.
\end{example}

\begin{lemma}\label{lem:Q-non-basic}
The element $b\in G(\breve\Q_p)$ as above is $\Q$-non-basic if and only if \tu{($\Q$-nb($P_b$))} of \S\ref{sub:one-dim} holds.
\end{lemma}

\begin{proof}
Write $G^{\tu{ad}}=\prod_{i\in I} G^{\tu{ad}}_i$ as in Definition \ref{def:Q-non-basic} and $b_i\in G^{\tu{ad}}_i(\breve\Q_p)$ for the image of $b$. By functoriality of Newton cocharacters, the composition of $\nu_b$ with the natural map $G\ra G^{\tu{ad}}_i$ is $\nu_{b_i}$, which is $\Q_p$-rational since $\nu_b$ is. This implies that the image of $P_b$ in $G^{\tu{ad}}_i$ is $P_{b_i}$, where $P_{b_i}\subset G^{\tu{ad}}_i$ is defined analogously as $P_b$ in $G$ over $\Q_p$. Each $b_i\in G^{\tu{ad}}_i(\breve\Q_p)$ is basic if and only if $\nu_{b_i}$ is central in $G^{\tu{ad}}_i$ (i.e. trivial) if and only if $P_{b_i}=G^{\tu{ad}}_i$. Therefore \tu{($\Q$-nb($P_b$))} holds if and only if $b_i$ is non-basic for every $i\in I$, and the latter is the definition for $b$ to be $\Q$-non-basic.
\end{proof}

Let $1\ra Z_1\ra G_1\ra G\ra 1$ be a $z$-extension over $\Q_p$ that is unramified over $\Q_p$. Let $\mu_{p,1}:\G_m\ra G_{1,\Q_{p^r}}$ be a cocharacter lifting $\mu_p:\G_m\ra G_{\Q_{p^r}}$. (Such a $\mu_{p,1}$ always exists since $Z_1$ is connected, but we will make a choice of $\mu_{p,1}$ coming from a lift of Shimura data, cf.~\S\ref{sub:Igusa-basic-setup} below.)

\begin{lemma}\label{lem:lift-b,mu}
Assume that $b$ and $r$ satisfy (br1)--(br3). Then there exists an element $b_1\in G_1(\Q_{p^r})$ lifting $b$, as well as an element $b^\circ_1\in G_1(\Q_{p^r})$ in the $\sigma$-conjugacy class of $b_1$, such that
\begin{itemize}
\item the analogues of (br1), (br1)', and (br2)'  hold true with $G,\mu_p,b$ replaced by $G_1,\mu_{p,1},b_1$,
\item $\nu_{b^\circ}$ is defined over $\Q_p$ and lifts $\nu^\circ_b$.
\end{itemize}
Moreover, we can make $r$ more divisible (without changing $\mu_p$, $b$, $\mu_{p,1}$, $b_1$, $b^\circ_1$) such that $r\nu_{b_1}$ is a cocharacter of $G_1$.\footnote{A priori we only know that $r\nu_{b_1}$ is a fractional cocharacter, even though $r\nu_b$ is an (integral) cocharacter of $G$.} (So the analogues of (br1)--(br3) and (br1)'--(br2)' are satisfied for $G_1$ and the new $r$.)
\end{lemma}

\begin{proof}
Since $G_1(\Z_{p^r})\ra G(\Z_{p^r})$ is onto (by the surjectivity on $\F_{p^r}$-points and the smoothness of $G_1\ra G$), the map $G_1(\Q_{p^r})\ra G(\Q_{p^r})$ induces a surjection
$$
G_1(\Z_{p^r}) \sigma\mu^{-1}_{p,1}(p) G_1(\Z_{p^r}) \twoheadrightarrow \cG(\Z_{p^r}) \sigma\mu^{-1}_p(p) \cG(\Z_{p^r}).
$$
Take $b_1\in G_1(\Q_{p^r})$ to be any preimage of $b$ under this map. This takes care of the first bullet point. As for the second point, since $G_1$ is quasi-split over $\Q_p$, there exists $b^\circ_1 \in  G_1(\Q_{p^r})$ $\sigma$-conjugate to $b_1$ such that $\nu_{b^\circ_1}$ is defined over $\Q_p$, and also such that $\nu_{b^\circ_1}$ factors through $T_1\subset G_1$, where $T_1$ is the preimage of $T$. Then the composite of $\nu_{b^\circ_1}$ with $G_1\twoheadrightarrow G$ is conjugate to $\nu_{b^\circ}$ in $G$, so differs from $\nu_{b^\circ}$ by an element of the $\Q_p$-rational Weyl group of $G$ \cite[Lem.~1.1.3 (a)]{KottwitzTwistedOrbital}. Identifying the latter with the $\Q_p$-rational Weyl group of $G_1$, we can use the same element to modify $\nu_{b^\circ_1}$ so that $\nu_{b^\circ_1}$ maps to $\nu_{b^\circ}$ under $G_1\twoheadrightarrow G$. Finally, the last point on $r$ in the lemma is obvious.
\end{proof}

In the settting of the lemma, we introduce $\Q_p$-algebraic groups $J_{b_1}$, $J_{b_1^\circ}$, $P_{b_1^\circ}$, $M_{b_1^\circ}$, etc.~for $G_1$ by mimicking the definition for $G$. Let $T_1,B_1$ denote the preimages of $T,B$ in $G_1$. Since $\nu_{b^\circ_1}$ maps to $\nu_{b^\circ}$, it is clear that $\nu^\circ_{b_1}\in X_*(T_1)^+$, where $+$ means $B_1$-dominance, and that $P_{b_1^\circ}$, $M_{b_1^\circ}$ map to $P_{b^\circ}$, $M_{b^\circ}$. As before, we can identify $Z(M_{b_1^\circ})= Z(J_{b_1^\circ})$, which carries $\nu_{b^\circ_1}$ to $\nu_{b_1}$. With this understanding, we will abuse notation to write $M_b,P_b,M_{b_1},P_{b_1}$ etc.~for $M_{b^\circ},P_{b^\circ},M_{b^\circ_1},P_{b^\circ_1}$ etc.~to simplify notation, and  write $\nu_b,\nu_{b_1}$ for $\nu_{b^\circ},\nu_{b_1^\circ}$ if there is little danger of confusion.

\section{Igusa varieties}\label{sec:Igusa}

Here we state the main theorem on $H^0$ of Igusa varieties and carry out the initial reduction to the completely slope divisible case, where we have a tower of finite-type Igusa varieties over a fixed finite field. This prepares us to apply a fixed-point formula in the next section. 

\subsection{Infinite-level Igusa varieties}\label{sub:inf-level-Igusa} 
We continue in the setting of \S\ref{sub:central-leaves}, with $b\in G(\breve \Q_p)$ satisfying (br1)--(br3). Let $b'\in \GSp(\breve \Q_p)$ denote the image of $b$. By Dieudonn\'e theory, we have a polarized $p$-divisible group $\Sigma_{b'}$ over $\F_{p^r}$ such that $\D(\Sigma_{b'})=V^*_{\Z_{(p)}}\otimes_{\Z_{(p)}} \Z_{p^r}$ with Frobenius operator $b'(1\otimes \sigma)$. By $\Sigma_b$ we mean the $p$-divisible group $\Sigma_{b'}$ equipped with crystalline Tate tensors $(t_\alpha)$ on $\D(\Sigma_{b'})$ corresponding to $(s_{\alpha})$ on $V_{\Z_{(p)}}$.
When there is no danger of confusion, we still write $\Sigma_b$ and $\Sigma_{b'}$ for their base changes to $\Fpbar$.

Applying the construction of \S\ref{sub:central-leaves} to $\mS_{K'_p K^{\prime,p}}(\GSp,S^\pm)$ and $b'$, we obtain a central leaf $C_{b', K^{\prime,p}}\subset \mS_{K'_p K^{\prime,p}}(\GSp,S^\pm)$. Let $R$ be an $\Fpbar$-algebra. Following \cite[Sect.~4.3]{CS17} we have the Igusa variety $\IG_{b',K^{\prime,p}}\ra C_{b',K^{\prime,p},\Fpbar}$ whose $R$-points parametrize isomorphisms
\begin{equation}\label{eq:Igusa-level}
\Sigma_{b'}\times_{\Fpbar} R \simeq \cA_R[p^\infty]
\end{equation}
compatible with polarizations up to $\Z_{(p)}^\times$-multiples, where $\cA_R$ denotes the pullback of the universal abelian scheme via $\Spec R\ra C_{b',K^{\prime,p}\Fpbar}$. Then $\IG_{b',K^{\prime,p}}$ is a perfect scheme, which is an $\uAut(\Sigma_{b'})$-torsor over $C_{b',K^{\prime,p}\Fpbar}$ by \cite[Cor.~ 2.3.2]{CS19}, where $\uAut(\Sigma_{b'})$ denotes the group scheme of automorphisms of $\Sigma_{b'}$ (preserving the polarization up to $\Z_p^\times$-multiples).

The map $\mS_{K_pK^p,\Fpbar}\ra \mS_{K'_pK^{\prime,p},\Fpbar}$ clearly induces a map $C_{b, K^p,\Fpbar} \ra C_{b', K^{\prime,p},\Fpbar}$. We define the subscheme
\begin{equation}\label{eq:IG-def}
\IG_{b,K^p}\subset (\IG_{b',K^{\prime,p}}\times_{C_{b',K^{\prime,p},\Fpbar}} C_{b,K^p,\Fpbar})^{\tu{perf}}= \IG_{b',K^{\prime,p}}\times_{C^{\tu{perf}}_{b',K^{\prime,p}},\Fpbar} C^{\tu{perf}}_{b,K^p,\Fpbar}
\end{equation}
to be the locus where \eqref{eq:Igusa-level} carries $(s_{\alpha})$ to $(\mathbf{s}_{\alpha,0})$ on the Dieudonn\'e modules. Composing with the projection maps, we have $\Fpbar$-morphisms $\IG_{b,K^p} \ra \IG_{b',K^{\prime,p}}$ and $\IG_{b,K^p}\ra C_{b,K^p,\Fpbar}^{\tu{perf}}$. The latter gives rise to the composite map
$$
\IG_{b,K^p}\ra C_{b,K^p,\Fpbar}^{\tu{perf}}\ra C_{b,K^p,\Fpbar}\ra \mS_{K_pK^p,\Fpbar}.
$$
As $K^p$ varies, the Hecke action of $G(\A^{\infty,p})$ on $\mS_{K_pK^p,\Fpbar}$ restricts to an action on $C_{b,K^p,\Fpbar}$ and extends to an action on $\IG_{b,K^p}$ by \cite[Lem.~6.4]{HamacherKim}. (The point is that the central leaves and Igusa varieties are defined in terms of $p$-adic invariants, which are preserved under the prime-to-$p$ Hecke action.)

\begin{lemma}\label{lem:Igusa-basic} The following are true.
\begin{enumerate}
\item The $\Fpbar$-scheme $\IG_{b,K^p}$ is perfect and a pro-\'etale $J_b^{\tu{int}}$-torsor over $C_{b,K^p,\Fpbar}^{\tu{perf}}$.
 \footnote{It can be shown that $\IG_{b,K^p}\ra C_{b,K^p}$ is an $\uAut(\Sigma_b)$-torsor by~\cite[Cor.~ 2.3.2]{CS19} and adapting the argument there, but we do not need it.}
\item The map $\IG_{b,K^p} \ra \IG_{b',K^{\prime,p}}$ is a closed embedding, under which the $J_{b'}(\Q_p)$-action on $\IG_{b',K^{\prime,p}}$ restricts to an action of $J_b(\Q_p)$ on $\IG_{b,K^p}$ (via the embedding $J_b(\Q_p)\hra J_{b'}(\Q_p)$).
\end{enumerate}
\end{lemma}

\begin{proof}
This follows from \cite[Prop.~4.1, 4.10]{Ham19}, noting that our $\IG_{b,K^p}$ is his $\mathscr{J}_\infty^{(p^{-\infty})}$ (the perfection of his $\mathscr{J}_\infty$) and that our $J_b^{\tu{int}}$ is his $\Gamma_b$. Two points require some further explanation. Firstly, we see that $J_b^{\tu{int}}=\Gamma_b$ as follows. Observe that $J_b^{\tu{int}}\subset J_b(\Q_p)\subset J_{b'}(\Q_p)$ and $J_b^{\tu{int}}\subset G(\breve \Z_p)\subset \GSp(\breve\Z_p)$. Thus $J_b^{\tu{int}}$ consists of automorphisms of $\Sigma_{b'}$ which are exactly the stabilizers of $(t_{\alpha})$ via Dieudonn\'e theory. Secondly, \cite[Prop.~4.1]{Ham19} tells us that $\mathscr{J}_\infty\ra C_{b,K^p,\Fpbar}$ is a pro-\'etale $J_b^{\tu{int}}$-torsor. Since every perfection map (as a limit of absolute Frobenius) is a universal homeomorphism, which preserves the pro-\'etale topology \cite[Lem.~5.4.2]{BhattScholzeProEtale}, it follows that the perfection $\mathscr{J}_\infty^{(p^{-\infty})}\ra C_{b,K^p,\Fpbar}^{\tu{perf}}$ is also a pro-\'etale $J_b^{\tu{int}}$-torsor.
\end{proof}

\begin{lemma}\label{lem:Igusa-points}
Let $R$ be a perfect $\Fpbar$-algebra. Then $\IG_{b,K^p}(R)$ is identified with the set of equivalence classes of $(x,j)$, where
\begin{itemize}
\item $x\in \mS_{K_pK^p}(R)$ is an abelian scheme over $\Spec R$ and
\item $j: \Sigma_b \times_{\Fpbar} R \ra \cA_x[p^\infty]$ is a quasi-isogeny carrying $(s_\alpha)$ to $(\mathbf{s}_{\alpha,0,x})$,
\end{itemize}
and $\cA_x$ denotes the pullback of the universal abelian scheme along $x$. Here $(x,j)$ and $(x',j')$ are considered equivalent if, in the notation of \S\ref{sub:integral-canonical-models}, there exists a $p$-power isogeny $i:\cA_x\ra \cA_{x'}$ carrying $(\mathbf{s}_{\alpha,\ell,x})_{\ell\neq p}$ to $(\mathbf{s}_{\alpha,\ell,x'})_{\ell\neq p}$ and $(\mathbf{s}_{\alpha,0,x})$ to $(\mathbf{s}_{\alpha,0,x'})$ such that $i\circ j=j'$. Each $\rho\in J_b(\Q_p)$ acts on the $R$-points of $\IG_{b,K^p}$ by sending $j$ to $j\circ \rho$.\footnote{We make a right action of $J_b(\Q_p)$ on $\mathfrak{Ig}_{b,K^p}$ so that it becomes a left action on the cohomology. In \cite[\S4.3]{CS17}, their arrow $j$ is reverse to ours, from $A[p^\infty]$ to $\Sigma_b \times_{\Fpbar}R$. The two conventions are identified via taking the inverse of $j$ (with the understanding that the authors of \emph{loc.~cit.}~are also using the right action of $J_b(\Q_p)$, though this does not appear there explicitly).}
\end{lemma}

\begin{proof}
This is the Hodge-type analogue of \cite[Lem.~4.3.4]{CS17} proven in the PEL case. By \emph{loc.~cit.}, $\IG_{b',K^{\prime,p}}(R)$ is the set of $p$-power isogeny classes of $(A,j)$ with $A\in \mS_{K'_pK^{\prime,p}}(R)$ and $j: \Sigma_b \times_{\Fpbar} R \ra A[p^\infty]$ a quasi-isogeny compatible with polarizations up to $\Z_p^\times$. Now we have a commutative diagram from the construction of central leaves and Igusa varieties: 
$$
\xymatrix{
 \IG_{b,K^p} \ar[r] \ar@{^(->}[d]_-{\mathrm{closed}} & C_{b,K^p,\Fpbar} \ar@{^(->}[rr]^-{\mathrm{loc.~closed}} \ar[d] && \mS_{K_pK^p,\Fpbar} \ar[d]\\
 \IG_{b',K^{\prime,p}} \ar[r]  & C_{b',K^{\prime,p},\Fpbar} \ar@{^(->}[rr]^-{\mathrm{loc.~closed}}  && \mS_{K'_pK^{\prime,p},\Fpbar}
 }
$$
Now we prove the first assertion by constructing the maps in both directions, which are easily seen to be inverses of each other. Given $y\in  \IG_{b,K^p}(R)$, its image gives $x\in  \mS_{K_pK^p}(R)$. The $j$ comes from the image of $y$ in $\IG_{b',K^{\prime,p}}(R)$. The compatibility of $j$ with crystalline Tate tensors follows from the very definition of $\IG_{b,K^p}$. Conversely, let $(x,j)$ be as in the lemma. Modifying by a quasi-isogeny, we may assume that $j$ is an isomorphism. Then $(A,j)$ comes from a point $y'\in \IG_{b',K^{\prime,p}}(R)$ as observed above. Since $\Spec R$ and $C_{b,K^p}$ are reduced, $x\in \mS_{K_pK^{p}}(R)$ comes from a point in $x\in C_{b,K^p}(R)$. Then $y'$ and $x$ have the same image in $C_{b',K^{\prime,p},\Fpbar}(R)$, so determine a point 
$$
y\in   \big(\IG_{b',K^{\prime,p}} \times_{C_{b',K^{\prime,p},\Fpbar}} C_{b,K^p,\Fpbar} \big)^{\perf}(R) = \big(\IG_{b',K^{\prime,p}} \times_{C_{b',K^{\prime,p},\Fpbar}} C_{b,K^p,\Fpbar} \big)(R) . 
$$
The compatibility of $j$ with crystalline Tate tensors exactly tells us that $y\in \IG_{b,K^p}(R)$.
  
It remains to show the last assertion. In light of Lemma \ref{lem:Igusa-basic} (2), the assertion on the $J_b(\Q_p)$-action follows from the analogue description for $J_{b'}(\Q_p)$-action on $\IG_{b',K^{\prime,p}}$ as in \cite[Lem.~4.3.4, Cor.~4.3.5]{CS17}.
\end{proof}

The $J_b(\Q_p)$-action on $\IG_{b,K^p}$ commutes with the Hecke action of $G(\A^{\infty,p})$ (as $K^p$ varies) as it is clear on the moduli description. Now we would like to understand the $G(\A^{\infty,p})\times J_b(\Q_p)$-representation 
$$
H^i(\IG_{b},\lql):= \varinjlim_{K^p} H^i(\IG_{b,K^p},\lql),\qquad i\ge 0,
$$
where the limit is over sufficiently small open compact subgroups of $G(\A^{\infty,p})$.

From \S\ref{sub:transfer-one-dim} we obtain the following commutative diagram. Indeed, all maps and the commutativity are obvious except possibly the map $J_b(\Q_p)^{\ab}\twoheadrightarrow M_b(\Q_p)^{\ab}$, which comes from the proof of Corollary \ref{cor:1-dim reps}. (The latter also tells us that $M_b(\Q_p)^{\ab}=M_b(\Q_p)^\flat$ and $G(\Q_p)^{\ab}=G(\Q_p)^{\flat}$.)
\begin{equation}\label{eq:JbMbG}
\xymatrix{
J_b(\Q_p) \ar@{->>}[d] & M_b(\Q_p) \ar@{->>}[d] \ar[r] & G(\Q_p) \ar@{->>}[d]\\
J_b(\Q_p)^{\ab} \ar@{->>}[r] & M_b(\Q_p)^{\ab} \ar[r] & G(\Q_p)^{\ab}
}
\end{equation}
The diagram yields the composite maps
\begin{equation}\label{eq:JbG}
\zeta_b:J_b(\Q_p)\ra G(\Q_p)^{\ab} \quad\mbox{and}\quad \zeta^*_b:M_b(\Q_p)\ra G(\Q_p)^{\ab}
\end{equation}
Thus every one-dimensional smooth representation $\pi_p$ of $G(\Q_p)$ (necessarily factoring through $G(\Q_p)^{\ab}$) can be pulled back to one-dimensional representations of $J_b(\Q_p)$ and $M_b(\Q_p)$, to be denoted $\pi_p\circ \zeta_b$ and $\pi_p\circ \zeta^*_b$.
We are ready to state the main theorem in this paper.

\begin{theorem}[Main Theorem]\label{thm:H0(Ig)}
Let $(G,X,p,\cG)$ be a Shimura datum of Hodge type. Assume that $b\in G(\breve \Q_p)$ and $r\in \Z_{\ge1}$ satisfy (br1)--(br3) of \S\ref{sub:central-leaves} such that $b$ is $\Q$-non-basic. Then
 there is a $G(\A^{\infty,p})\times J_b(\Q_p)$-module isomorphism
$$
\iota H^0(\IG_{b},\lql)\simeq \bigoplus_{\pi\in \cA_{\mathbf 1}(G)}  \pi^{\infty,p}\otimes  (\pi_p\circ \zeta_b), 
$$
where $\cA_{\mathbf 1}(G)$ stands for the set of one-dimensional automorphic representations $\pi$ of $G(\A)$ such that $\pi_\infty$ is trivial on $G(\R)_+$.
\end{theorem}

\begin{proof}
After reduction to the completely slope divisible case by Corollary \ref{cor:red-to-completely-slope-div}, the theorem will be established in \S\ref{sec:cohomology-Igusa} below.
\end{proof}

\begin{remark}\label{rem:Jacquet-of-1dim}
Since $\dim \pi_p=1$, we have $(\pi_p\circ \zeta_b)\otimes \delta_{P_b}=J_{P_b^{\mathrm{op}}}(\pi_p)\otimes\delta^{1/2}_{P_b}$ as $J_b(\Q_p)$-representations. (The point is that the unipotent radical $N_b^{\mathrm{op}}$ acts trivially on $\pi_p$.) This is closely related to Lemma \ref{lem:O-tr-nu-ascent}. It is also worth comparing with \cite[Thm.~V.5.4]{HarrisTaylor} and \cite[Thm.~6.7]{ShinRZ}, where a similar expression appears in the description of cohomology of Igusa varieties.
\end{remark}

\subsection{Finite-level Igusa varieties in the completely slope divisible case}\label{sub:completely-slope-divisible}

 We recall the definition of finite-level Igusa varieties following \cite{Man05,CS17,Ham19}. From \S\ref{sub:central-leaves} we have $r\in \Z_{\ge1}$ such that conditions (br1)--(br3) hold. In this subsection, we further assume that $b$ is completely slope divisible in the sense of \cite[Def.~2.4.1]{KimLeaves}. In particular, 
the decency equation holds:
\begin{equation}\label{eq:decency}
b  \sigma(b) \cdots \sigma^{r-1}(b)=r\nu_b(p).
\end{equation}
A priori, \eqref{eq:decency} holds for \emph{some} $r\in\Z_{\ge 1}$ but then it is still true for all multiples of $r$. So we may and will assume that \eqref{eq:decency} holds for the same $r$ as in \S\ref{sub:central-leaves} by making $r$ more divisible.

We start from the Siegel case. Write $\Ig_{b',m,K^{\prime,p}}\ra C_{b',K^{\prime,p}}$ for Igusa varieties of level $m\in \Z_{\ge1}$ as in \cite[\S4]{Man05} or \cite[\S3.1]{Ham19} (the definition works over $\F_{p^r}$ not just over $\Fpbar$), defined to parametrize liftable isomorphisms on the $p^m$-torsion subgroup of each slope component. As shown in \emph{loc.~cit.}, $\Ig_{b',m,K^{\prime,p}}\ra C_{b',K^{\prime,p}}$ is a finite \'etale morphism, forming a projective system over varying $m$ via the obvious projection maps. Write $\Ig_{b',K^{\prime,p}}$ for the projective limit of $\Ig_{b',m,K^{\prime,p}}$ over $m$. There are maps $\IG_{b',K^{\prime,p}}\ra \Ig_{b',m,K^{\prime,p},\Fpbar}$ for $m\ge 1$ compatible with each other, since the isomorphism \eqref{eq:Igusa-level} induces isomorphisms on isoclinic components. This induces an isomorphism
 $\IG_{b',K^{\prime,p}}\ra \Ig_{b',K^{\prime,p},\Fpbar}^{\perf}$. See \cite[Prop.~4.3.8]{CS17} and the preceding paragraph for details.

Following \cite[\S4.1]{Ham19} (but working over $\F_{p^r}$ rather than $\Fpbar$), define the $\F_{p^r}$-subscheme
$$
\tilde{\Ig}_{b,K^p}\subset \big(\Ig_{b',K^{\prime,p}}\times_{C_{b',K^{\prime,p}}} C_{b,K^p} \big)^{\tu{perf}}
$$
to be the locus given by the same condition as in \eqref{eq:IG-def}.
Define $\Ig_{b,m,K^p}$ as the image of the composite map
$$
\tilde{\Ig}_{b,K^p} \ra \Ig_{b',K^{\prime,p}} \times_{C_{b',K^{\prime,p}}} C_{b,K^p} \ra \Ig_{b',m,K^{\prime,p}} \times_{C_{b',K^{\prime,p}}} C_{b,K^p}.
$$
The projection onto the second component gives an $\F_{p^r}$-morphism $\Ig_{b,m,K^p}\ra C_{b,K^p}$, which is  finite \'etale by \cite[Prop.~4.1]{Ham19}. Via the canonical projection $\Ig_{b,m+1,K^p}\ra \Ig_{b,m,K^p}$ commuting with the maps to $C_{b,K^p}$, we take the projective limit and denote it by $\Ig_{b,K^p}$.

Besides the Hecke action of $G(\A^{\infty,p})$ on the tower of $\Ig_{b,K^p}$, we also have an action on $\Ig_{b,K^p,\Fpbar}$ by a submonoid $S_b\subset J_b(\Q_p)$ defined in \cite[p.586]{Man05}. (The latter action is defined only over $\Fpbar$ in general since self quasi-isogenies of $\Sigma_b$ are not always defined over finite fields.)  The precise definition is unimportant, but it suffices to know two facts. Firstly, $S_b$ generates $J_b(\Q_p)$ as a group. Secondly, $S_b$ contains $p^{-1}$ (the inverse of the multiplication by $p$ map on $\Sigma_b$) and \footnote{Here is a note on the sign. On slope $0\le \lambda\le 1$ component, the action of $\tu{fr}^r$ is $p^{\lambda}$, but $\nu_b$ records slope $-\lambda$ since we use the covariant Dieudonn\'e theory.} 
$$
\tu{fr}^{-r}:=r\nu_b(p)\in J_b(\Q_p).
$$
By Lemma \ref{lem:nu-central}, $\tu{fr}^{-r}\in A_{J_b}(\Q_p)$. Let $\tu{Fr}$ denote the absolute Frobenius morphism on an $\F_p$-scheme.

\begin{lemma}\label{lem:finite-level-Igusa} The following hold true.
\begin{enumerate}
\item $\tu{fr}^{-r}\in A_{P^{\tu{op}}_b}^{--}\subset A_{M_b}(\Q_p)=A_{J_b}(\Q_p)$. As an element of $M_b(\Q_p)$, we have $\tu{fr}^{-r}\in A_{P^{\tu{op}}_b}^{--}$. 
(Recall that $A_{P^{\tu{op}}_b}^{--}$ was defined in \S\ref{sub:Jacquet-estimate}. For $A_{M_b}(\Q_p)=A_{J_b}(\Q_p)$, see \S\ref{sub:central-leaves}.)
\item The action of $\tu{Fr}^{r}\times 1$ on $\Ig_{b,K^p}\times_{\F_{p^r}} \Fpbar$ induces the same action on $\IG_{b,K^p}$ as
the action of $\tu{fr}^{-r}\in J_b(\Q_p)$.
\item There is a canonical isomorphism $\IG_{b,K^p}\simeq \Ig_{b,K^p,\Fpbar}^{\perf}$ over $C_{b,K^p,\Fpbar}$, compatible with the $G(\A^{\infty,p})\times S_b$-actions as $K^p$ varies.
\end{enumerate}
\end{lemma}

\begin{proof}
(1) We already know $\tu{fr}^{-r}\in A_{J_b}(\Q_p)=A_{M_b}(\Q_p)$. Since $r\nu_b$ is $B$-dominant (\S\ref{sub:central-leaves}), we have $r\nu_b(p)\in A^-_{P^{\tu{op}}_b}$. Moreover $r\nu_b(p)\in A^{--}_{P^{\tu{op}}_b}$ as the centralizer of $r\nu_b(p)$ in $G$ is exactly $M_b$.

(2) Write $\tu{Fr}_{\Sigma}$ for the absolute Frobenius action on $\Sigma_b/\F_{p^r}$. In view of \eqref{eq:decency}, $\tu{fr}^{-r}=r\nu_b(p)$ acts on $\Sigma_b/\F_{p^r}$ as $(\tu{Fr}_{\Sigma})^r$. Thus  $\tu{fr}^{-r}$ sends $(x,j)$ to $(x,j\circ \tu{Fr}^r_{\Sigma})$ in the description of $R$-points in Lemma \ref{lem:Igusa-points}. On the other hand, $\tu{Fr}^{r}\times 1$ on $\IG_{b,K^p}$ sends $(x,j)$ to $(x^{(r)},j^{(r)})$, where $x^{(r)}$ corresponds to the $p^r$-th power Frobenius twist of $x$ (so that $\cA_{x^{(r)}}=(\cA_x)^{(r)}$), and $j^{(r)}$ is the $p^r$-th power twist of $j$. Finally we observe that $(x^{(r)},j^{(r)})$ is equivalent to $(x,j\circ \tu{Fr}^r_{\Sigma})$ via the $p^r$-power relative Frobenius $\cA_x\ra \cA_{x^{(r)}}$.

(3) We have the map $\IG_{b,K^p}\ra \Ig_{b,K^p,\Fpbar}$ over $C_{b,K^p,\Fpbar}$ from the definition, which factors through $\IG_{b,K^p}\ra \Ig_{b,K^p,\Fpbar}^{\perf}$ since $\IG_{b,K^p}$ is perfect. This is shown to be an isomorphism exactly as in the proof of \cite[Prop.~4.3.8]{CS17}, the point being a canonical splitting of the slope decomposition over the perfect scheme $\Ig_{b,K^p,\Fpbar}^{\perf}$.
\end{proof}

Now we compare Igusa varieties arising from two central leaves in the same Newton stratum. Let $b,b_0\in G(\breve \Q_p)$. Assume that $b$ and $b_0$ satisfy the conditions at the end of \S\ref{sub:central-leaves}.
We have an isomorphism $J_b(\Q_p)\simeq J_{b_0}(\Q_p)$ (induced by a conjugation in the ambient group $G(\breve\Q_p)$), canonical up to $J_b(\Q_p)$-conjugacy.

\begin{proposition}\label{prop:Ig_b-Ig_b_0}
There exists a $G(\A^{\infty,p})$-equivariant isomorphism
$$
\IG_b \isom \IG_{b_0},
$$
which is also equivariant for the actions of $J_b(\Q_p)$ and $J_{b_0}(\Q_p)$ through a suitable isomorphism $J_b(\Q_p)\simeq J_{b_0}(\Q_p)$ in its canonical $J_b(\Q_p)$-conjugacy class.
\end{proposition}

\begin{proof}
Since $[b_x]=[b_{x_0}]$, there exists a quasi-isogeny $f:\Sigma_{b_0}\ra \Sigma_{b}$ compatible with $G$-structures. Using the description of Lemma \ref{lem:Igusa-points}, we can give an isomorphism $\IG_b \isom \IG_{b_0}$ on $R$-points by $(A,j)\mapsto (A,j\circ f)$. The equivariance property is evident.
\end{proof}
 
\begin{corollary}\label{cor:red-to-completely-slope-div}
 In the setting of Theorem \ref{thm:H0(Ig)}, if the theorem is true for every $b$ which satisfies \eqref{eq:decency} for some $r$, then the theorem is true in general.
\end{corollary}

\begin{proof}
Let $b_0$ be arbitrary in the setting of Theorem \ref{thm:H0(Ig)}. By \cite[Lem.~5.3]{HamacherKim} and \cite[Lem.~2.4.3]{KimLeaves} (alternatively by the argument of \cite[Lem.~4.2.8]{ZhangC}), there exists $b\in G(\breve \Q_p)$ which is $\sigma$-conjugate to $b_0$ such that
\begin{itemize}
\item $b\in G(\breve\Z_p) \sigma\mu_p(p)^{-1} G(\breve\Z_p)$,
\item $b$ is completely slope divisible and satisfies \eqref{eq:decency} for some $r\in \Z_{\ge 1}$.
\end{itemize}
It follows that we have (br1) for $b$, namely $b\in G(\Z_{p^r}) \sigma\mu_p(p)^{-1} G(\Z_{p^r})$ for some $r$. By making $r$ more divisible (note that \eqref{eq:decency} still holds for the new $r$), we can ensure (br2) and (br3) for $b$. We are assuming that Theorem \ref{thm:H0(Ig)} is true for this $b$. On the other hand, fixing an isomorphism $J_b(\Q_p)\simeq J_{b_0}(\Q_p)$ as in there, we see from Proposition \ref{prop:Ig_b-Ig_b_0} that
$$
H^0(\IG_{b},\lql)\simeq H^0(\IG_{b_0},\lql)\quad \textrm{as}~G(\A^{\infty,p})\times J_b(\Q_p)\textrm{-modules}.
$$
Therefore Theorem \ref{thm:H0(Ig)} for $b$ implies that the same theorem holds for $b_0$. (Note that the transfer of one-dimensional representations via $J_b(\Q_p)\simeq J_{b_0}(\Q_p)$ is canonical.)
\end{proof}

\section{Cohomology of Igusa varieties}\label{sec:cohomology-Igusa}

The main purpose of this section is to prove Theorem \ref{thm:H0(Ig)}. Throughout we are in the setting of \S\ref{sub:completely-slope-divisible}, namely we are assuming \eqref{eq:decency} on $b$ and $r$ in addition to (br1)--(br3) of \S\ref{sub:central-leaves}, since this is sufficient in light of Corollary \ref{cor:red-to-completely-slope-div}. We will swtich to compactly supported cohomology via Poincar\'e duality and apply Mack-Crane's Langlands--Kottwitz style formula to bring in techniques from the trace formula and harmonic analysis. All ingredients will be combined together in \S\ref{sub:completion-of-proof} to identify the leading term in the Lang--Weil estimate.

\subsection{Compactly supported cohomology in top degree}\label{sub:Hc}

In \S\ref{sub:completely-slope-divisible} we constructed $\Ig_{b,K^p}$ such that $\IG_{b,K^p}$ is isomorphic to the perfection of $\Ig_{b,K^p,\Fpbar}$ (compatibly with the transition maps as $K^p$ varies). Recall that $\dim \Ig_b=\langle 2\rho,\nu_b\rangle$. Define for $i\in \Z_{\ge 0}$,
$$
H^i_c(\Ig_{b,m,\Fpbar},\lql):=\varinjlim_{K^p} H^i_c(\Ig_{b,m,K^p,\Fpbar},\lql),\qquad H^i_c(\Ig_{b,\Fpbar},\lql):=\varinjlim_{m\ge 0}H^i_c(\Ig_{b,m,\Fpbar},\lql).
$$
As for $H^i_c(\Ig_{b,\Fpbar},\lql)$, we have a $G(\A^{\infty,p})\times J_b(\Q_p)$-module structure on $H^i_c(\Ig_{b,\Fpbar},\lql)$. This is an admissible $G(\A^{\infty,p})\times J_b(\Q_p)$-module as the cohomology is finite-dimensional at each finite level. It is convenient to prove the following dual version of Theorem \ref{thm:H0(Ig)}.

\begin{theorem}\label{thm:Hc(Ig)}
Assume that $b$ is $\Q$-non-basic, and that $\Sigma_b$ is completely slope divisible. Then there is a $G(\A^{\infty,p})\times J_b(\Q_p)$-module isomorphism
$$
\iota H^{\langle 4\rho,\nu_b\rangle}_c(\Ig_{b,\Fpbar},\lql)\simeq \bigoplus_{\pi\in \cA_{\mathbf 1}(G)} \pi^{\infty,p}\otimes \left( (\pi_p\circ \zeta_b) \otimes \delta_{P_b} \right).
$$
\end{theorem}

\begin{proof}
The proof will be carried out in \S\ref{sub:completion-of-proof} after recalling a stabilized trace formula (Theorem \ref{thm:stable-Igusa}), by employing the estimates in \S\ref{sec:asymptotic-analysis}.
\end{proof}

\begin{proof}[Theorem \ref{thm:Hc(Ig)} implies Theorem \ref{thm:H0(Ig)}]
We may put ourselves in the completely slope divisible case by Corollary \ref{cor:red-to-completely-slope-div}. Write $d:=\langle 2\rho,\nu_b\rangle$. Applying Poincar\'e duality to finite-level Igusa varieties $\Ig_{b,m,K^p}$ and taking direct limit over $m$ and $K^p$, we obtain a pairing
$$
H^0(\Ig_{b,\Fpbar}, \lql)\times H^{2d}_c(\Ig_{b,\Fpbar}, \lql(d)) \ra \lql,
$$
where $\lql(d)$ denotes the $d$-th power Tate twist. The construction of duality (Exp.XVIII,\S3 in \cite{SGA4-3}) goes through a family of canonical isomorphisms $Rf^!_{m,K^p}\lql\simeq \lql(d)[-2d]$ (concentrated in degree $2d$) over $m$ and $K^p$, where $f_{m,K^p}:\Ig_{b,m,K^p}\ra \Spec\F_{p^r}$ denotes the structure map. Thus the action of $G(\A^{\infty,p})\times J_b(\Q_p)$ on $\Ig_b=\{\Ig_{b,m,K^p}\}$ induces an action on $\lql(d)[-2d]$, through a character $\varsigma:G(\A^{\infty,p})\times J_b(\Q_p)\ra \lql^\times$. (As in \S\ref{lem:finite-level-Igusa}, the action of $J_b(\Q_p)$ is defined a priori on a submonoid $S_b$ and then extended to $J_b(\Q_p)$. Alternatively, this action can be defined directly after perfectifying $\Ig_b$.) Together with the $G(\A^{\infty,p})\times J_b(\Q_p)$-action on $\Ig_b$, this yields an action of $G(\A^{\infty,p})\times J_b(\Q_p)$ on $H^{2d}_c(\Ig_{b,\Fpbar},\lql(d))$ and $H^0(\Ig_{b,\Fpbar},\lql)$, respectively. It follows from the functoriality of Poincar\'e duality that the above pairing is $G(\A^{\infty,p})\times J_b(\Q_p)$-equivariant. Thus $H^0(\Ig_{b,\Fpbar},\lql)$ is isomorphic to the (smooth) contragredient of $H^{2d}_c(\Ig_{b,\Fpbar},\lql(d))$, which is isomorphic to $H^{2d}_c(\Ig_{b,\Fpbar},\lql)\otimes \varsigma$. Therefore Theorem \ref{thm:Hc(Ig)} implies that
\begin{equation}\label{eq:H2d-implies-H0}
\iota H^0(\Ig_{b,\Fpbar},\lql)\simeq \bigoplus_{\pi\in \cA_{\mathbf 1}(G)} \left((\pi^{\infty,p})\otimes (  (\pi_p\circ \zeta_b) \otimes \delta_{P_b}) \right)^\vee \otimes \varsigma^{-1}.
\end{equation}

On the other hand, $H^0(\Ig_{b,\Fpbar},\lql)$ is the space of smooth $\lql$-valued functions on $\pi_0(\Ig_{b,\Fpbar})$, on which $G(\A^{\infty,p})\times J_b(\Q_p)$ acts through right translation. (Here smoothness means invariance under an open compact subgroup of $G(\A^{\infty,p})\times J_b(\Q_p)$.) In particular, the trivial representation appears in $H^0(\Ig_{b,\Fpbar},\lql)$  as the subspace of constant functions on $\pi_0(\Ig_{b,\Fpbar})$. Hence $\varsigma^{-1}=(\pi_0^{\infty,p})\otimes (  (\pi_{0,p}\circ \zeta_b)\otimes \delta_{P_b})$ for some $\pi_0\in \cA_{\mathbf 1}(G)$. Plugging this formula into \eqref{eq:H2d-implies-H0} and using the fact that $\cA_{\mathbf 1}(G)$ is invariant under taking dual and twisting by $\pi_0$, we can rewrite \eqref{eq:H2d-implies-H0} as
$$
\iota H^0(\Ig_{b,\Fpbar},\lql) \simeq \bigoplus_{\pi\in \cA_{\mathbf 1}(G)} \pi^{\infty,p}\otimes  (\pi_p\circ \zeta_b) .
$$
Finally, the same holds with $\IG_b$ in place of $\Ig_{b,\Fpbar}$ thanks to Lemma \ref{lem:finite-level-Igusa} (3).
\end{proof}

\begin{remark}
It may be possible to compute the character $\varsigma$ in the proof, but we have got around it. As we know the Frobenius action on $\lql(d)[-2d]$, Lemma \ref{lem:finite-level-Igusa} (2) tells us that $\tu{fr}^{-r}\in J_b(\Q_p)$ acts by $p^{rd}$. We guess that $\varsigma$ is trivial on $G(\A^{\infty,p})$ and equal to $\delta_{P_b}^{-1}$ on $J_b(\Q_p)$.
\end{remark}

\subsection{Preparations in harmonic analysis}\label{sub:Igusa-basic-setup}
 
Let $\phi^{\infty,p}=\otimes'_{v\neq \infty,p} \phi_v\in \cH(G(\A^{\infty,p}))$ and $\phi_p\in \cH(J_b(\Q_p))$. With a view towards Theorem \ref{thm:Hc(Ig)}, we want to compute
$$
\Tr\left(\phi^{\infty,p}\phi^{(j)}_p \left|\iota H_c(\Ig_{b,\Fpbar},\lql) \right.\right)
:= \sum_{i\ge 0} (-1)^i \Tr\left(\phi^{\infty,p}\phi^{(j)}_p \left|\iota H^i_c(\Ig_{b,\Fpbar},\lql) \right.\right).
$$

We keep $T$, $B$, $b\in G(\breve \Q_p)$, and $r \in \Z_{\ge 1}$ as before, so that $r\nu_b\in X_*(T)_+$. Recall that $r\nu_b(p)\in A_{J_b}$. Given $\phi_p\in \cH(J_b(\Q_p))$, define
$$
\phi_p^{(j)}\in \cH(J_b(\Q_p))\quad\mbox{by}\quad \phi_p^{(j)}(\delta):=\phi_p(j\nu_b(p)^{-1}\delta),\qquad j\in r \Z_{\ge 1}.
$$
This coincides with the analogous definition of $\phi_p^{(k)}$ in \S\ref{sub:nu-Jacquet}, namely $\phi_p^{(j)}=\phi_p^{(k)}$ via $k=j/r$ and $\nu=r\nu_b$. (The difference is that $\nu$ is a cocharacter but $\nu_b$ is only a fractional cocharacter.)

An element $\delta\in J_b(\lqp)$ is \textbf{acceptable} if its image in $M_b(\lqp)$ is acceptable (Definition \ref{def:acceptable}) under the isomorphism $J_b(\lqp)\simeq M_b(\lqp)$ induced by some (thus any) inner twist at the end of \S\ref{sub:central-leaves}. As in \S\ref{sub:nu-Jacquet}, let $\cH_{\tu{acc}}(J_b(\Q_p))\subset \cH(J_b(\Q_p))$ denote the subspace of functions supported on acceptable elements. Choose $j_0\in \Z_{\ge 0}$ such that
$$
\phi_p^{(j)}\in \cH_{\tu{acc}}(J_b(\Q_p)),\qquad j\in r \Z,~~j\ge j_0.
$$
Such a $j_0$ exists by the argument of Lemma \ref{lem:nu-twist-support}.
By Lemma \ref{lem:finite-level-Igusa} and the definition of $\phi_p^{(j)}$, 
\begin{equation}\label{eq:trace-(j)=trace-Frob}
  \Tr\left(\phi^{\infty,p}\phi^{(j)}_p \left|\iota H_c(\Ig_{b,\Fpbar},\lql) \right.\right) = \Tr\left(\phi^{\infty,p}\phi_p \times (\tu{Fr}^{j}\times 1) \left|\iota H_c(\Ig_{b,\Fpbar},\lql) \right.\right),
\end{equation}
where $\tu{Fr}^j$ is the $j/r$-th power of the relative Frobenius of $\Ig_b$ over $\F_{p^r}$. Since the action of $\tu{Fr}^j$ is the same as the action of a central element of $J_b(\Q_p)$, it commutes with the action of $\phi^{\infty,p}\phi_p$. Thus \eqref{eq:trace-(j)=trace-Frob} and the Lang--Weil bound tell us that the top degree compactly supported cohomology in Theorem \ref{thm:Hc(Ig)} is captured by the leading term as $j\ra\infty$. This will be the basic idea underlying the proof of the theorem in \S\ref{sub:completion-of-proof} below.

We fix the global central character datum $(\fkX,\chi_0)=(A_{G,\infty},\textbf{1})$ for $G$, which can also be viewed as a central character datum for $G^*$ via $Z(G)=Z(G^*)$. (Since we compute the cohomology with constant coefficients, we do not need to consider nontrivial $\chi_0$.)

We also fix a $z$-extension $1\ra Z_1\ra G_1\ra G\ra 1$ over $\Q$ once and for all, which is unramified over $\Q_p$. As explained in \cite[\S7.3.3]{KSZ}, we can promote $G_1$ to a Shimura datum $(G_1,X_1)$ lifting $(G,X)$ together with the conjugacy class $\{ \mu_{X_1}\}$ of cocharacters of $G_{1,\C}$ lifting $\{ \mu_X\}$. Therefore, in $\iota_p\{\mu_{X_1}\}$, we can find a cochacter 
$$
\mu_{p,1}:\G_m\ra G_1
$$
which lifts $\mu_p$ and is defined over an unramified extension of $\Q_p$ (since the reflex field of $(G_1,X_1)$ is unramified at $p$). Making $r$ more divisible, we arrange that $\mu_{p,1}$ is defined over $\Q_{p^r}$. We apply Lemma \ref{lem:lift-b,mu} to find a lift $b_1\in G_1(\Q_{p^r})$ of $b$; we also ensure that $r\nu_{b_1}$ is a cocharacter as in the last assertion of that lemma. We mention that $\mu_{p,1}$ and $b_1$ are going to enter the construction of test functions at $\infty$ and $p$, respectively.

For each $\fke\in \cE_{\el}(G)$, fix a $\Q$-rational minimal parabolic subgroup of $G^\fke$ and its Levi component as at the start of \S\ref{sub:main-estiamte} (with $G^\fke$ in place of $G$ there). Call them $P_0^*$ and $M_0^*$ in the case $\fke=\fke^*$. On the other hand, we have $G_{\Q_p}\supset B_{\Q_p}\supset T_{\Q_p}$ from \S\ref{sub:central-leaves}. Since $G_{\Q_p}$ is quasi-split, there is a canonical $G^{\ad}(\Q_p)$-conjugacy class of isomorphisms $G_{\Q_p}\simeq G^*_{\Q_p}$. We fix one such isomorphism such that $B_{\Q_p}$ (resp. $T_{\Q_p}$) is carried into $P_{0,\Q_p}^*$ (resp. $M_{0,\Q_p}^*$). The images of $B_{\Q_p}$ and $T_{\Q_p}$ in $G^*_{\Q_p}$ will play the roles of $B$ and $T$ in \S\ref{sub:main-estiamte}.

For each $\fke\in \cE^<_{\el}(G)$, we have a central extension $1\ra Z_1\ra G_1^\fke \ra G^\fke\ra 1$ over $\Q$, determining an endoscopic datum $\fke_1$ for $G_1$ and a central character datum $(\fkX^{\fke}_1,\chi^\fke_1)$ for $G^\fke_1$ as in \S\ref{sub:endoscopy-z-ext}. 

\subsection{The test functions away from $p$}

For each $\fke\in \cE_{\el}(G)$, Let us introduce the test functions to enter the statement of Theorem \ref{thm:stable-Igusa} below. Here we consider the places away from $p$. The place $p$ will be treated in the next subsection.

The first case is away from $p$ and $\infty$. When $\fke=\fke^*$, we have $(f^{\Ig,*})^{\infty,p}=\otimes'_{v\neq \infty,p} f^*_v$, where $f^*_v\in \cH(G^*(\Q_v))$ is a transfer of $\phi_v$ as in \S\ref{sub:one-dim}. In case $\fke\in \cE^<_{\el}(G)$, at each $v\neq \infty,p$, the function $\phi_v$ admits a transfer $ f^\fke_{1,v} \in \cH(G^{\fke}_1(\Q_v),(\chi^{\fke}_{1,v})^{-1})$. Then we take 
$$
(f^{\Ig,\fke})^{\infty,p}_1=\otimes'_{v\neq \infty,p} f^\fke_{1,v} \in \cH(G^{\fke}_1(\A^{\infty,p}),(\chi^{\fke,\infty,p}_1)^{-1}).
$$

The next case is the real place. We construct the test function $f^{\Ig,\fke}_{1,\infty}\in \cH(G^{\fke}_1(\R),(\chi^{\fke}_{1,\infty})^{-1})$ by adapting \cite[\S7]{KottwitzAnnArbor} to the case with central characters. In the easier case of $\fke=\fke^*=(G^*,{}^L G^*,1,\tu{id})$, we take $f^{\Ig,*}_\infty:=e(G_\infty) f_{\mathbf 1}$ in the notation of \S\ref{sub:Lefschetz}. Now let $\fke\in \cE^<_{\el}(G)$. In the notation of \S\ref{sub:Lefschetz}, both $\xi$ and $\zeta$ are trivial in the current setup (since we are focusing on the constant coefficient case). Write $\xi_1$ and $\zeta_1$ for the pullbacks of $\xi$ and $\zeta$ from $G$ to $G_1$; they are again trivial. We obtain a discrete $L$-packet $\Pi(\xi_1,\zeta_1)$ for $G_1(\R)$ along with an $L$-parameter $\phi_{\xi_1,\zeta_1}:W_{\R}\ra {}^L G_1$ as in \S\ref{sub:Lefschetz}. Let $\Phi_2(G_{1,\R}^{\fke},\phi_{\xi_1,\zeta_1})$ denote the set of discrete $L$-parameters $\phi'\in \Phi(G_1^{\fke}(\R))$ such that $\eta_1^\fke\phi'\simeq \phi_{\xi_1,\zeta_1}$. Then define (cf.~\cite[p.186]{KottwitzAnnArbor})
$$
f^{\Ig,\fke}_{1,\infty}:=(-1)^{q(G_1)} \langle \mu_1,s_1^\fke\rangle \sum_{\phi'} \det(\omega_*(\phi')) f_{\phi'},
$$
where $f_{\phi'}$ is the averaged Lefschetz function for the $L$-packet of $\phi'$ defined in \S\ref{sub:Lefschetz}, and the sum runs over $\phi'\in \Phi(G_{1,\R}^{\fke},\phi_{\xi_1,\zeta_1})$. As in \cite[\S8.2.5]{KSZ} we check that $f^{\Ig,\fke}_{1,\infty}$ is $(\chi^{\fke}_{1,\infty})^{-1}$-equivariant and compactly supported modulo $\fkX_{1,\infty}^\fke$. 

\subsection{The test functions at $p$}\label{sub:test-function-at-p}

We apply the contents of \S\ref{sec:Jacquet-endoscopy} to the cocharacter $\nu:=r\nu_b$ over $F=\Q_p$ with uniformizer $\varpi=p$. In particular, we have $P_b:=P_\nu$ whose Levi factor is $M_b=M_\nu$.

Consider the case $\fke=\fke^*$. Each function $\phi_p\in \cH(J_b(\Q_p))$ admits a  transfer $\phi^*_p\in \cH(M_b(\Q_p))$ as explained in \S\ref{sub:transfer-one-dim}. When $\phi_p\in \cH_{\tu{acc}}(J_b(\Q_p))$, we can arrange that $\phi^*_p\in \cH_{\tu{acc}}(M_b(\Q_p))$ by multiplying the indicator function on the set of acceptable elements in $M_b(\Q_p)$. (This is possible as the subset of acceptable elements is nonempty, open, and stable under $J_b(\Q_p)$-conjugacy.) The image of $\phi^*_p$ in $\cS(M_b)$ depends only on $\phi_p$ (as an element of $\cS(J_b)$). In the notation of \S\ref{sub:nu-Jacquet}, define
$$
f_p^{\Ig,*,(j)}:=\mathscr{J}_\nu\left(\delta_{P_\nu}^{1/2}\cdot \phi_p^{*,(j)}\right)\in \cS(G),\qquad j\in \Z_{\ge 0},
$$
As before, we still write $f_p^{\Ig,*,(j)}$ for a representative in $\cH(G(\Q_p))$. Lemma \ref{lem:O-tr-nu-ascent} implies that 
$$
\Tr \left(f_p^{\Ig,*,(j)}|\pi_p \right)=\Tr \left( \phi_p^{*,(j)} | J_{P_\nu^{\tu{op}}}(\pi_p)\otimes \delta_{P_\nu}^{1/2} \right),\qquad \forall\pi_p\in \Irr(G(\Q_p)). 
$$

\begin{remark}\label{rem:Igusa-Kottwitz-sign}
In the definition of $f_p^{\Ig,*,(j)}$, we have not multiplied the constant $c_{M_H}$ appearing in \cite[\S6]{ShinIgusaStab} (with $H,M_H$ there corresponding to $G^*,M_b$ here). In the sign convention of Remark 6.4 therein, the transfer factor between $J_b$ and $M_b$ equals $e(J_b)$, resulting in $c_{M_H}=e(J_b)$. In contrast, we have taken the transfer factor between inner forms to be 1 (cf.~Remark \ref{rem:transfer-factor-sign}), so $c_{M_H}=1$ in our convention.
\end{remark}

Now let $\fke\in \cE^<_{\el}(G)$. Recall from \eqref{sub:Igusa-basic-setup} that $b_1\in G_1(\Q_{p^r})$ was chosen. Take $\nu_1:=r\nu_{b_1}$. By pulling back the $z$-extension $1\ra Z_1\ra G_1\ra G\ra 1$ via $M_b\hra G$, and using the definition of $J_b$ and $J_{b_1}$, we obtain $z$-extensions over $\Q_p$ as follows:
$$
1\ra Z_1\ra M_{b_1}\ra M_b\ra 1,\qquad 1\ra Z_1\ra J_{b_1}\ra J_b\ra 1.
$$
(For $J_b$, the point is that the $\sigma$-stabilizer subgroup of $\Res_{\breve \Q_p/\Q_p}\G_m$ is simply $\G_m$.) We pull back $\phi^{(j)}_p\in \cH(J_b(\Q_p))$ to obtain $\phi^{(j)}_{1,p}\in  \cH(J_{b_1}(\Q_p),\chi_{1,p})$. (Recall that $\chi_1=\prod_v \chi_{1,v}$ is the trivial character on $\fkX_1=Z_1(\A)$.) Write  $\phi^*_p\in \cH(M_b(\Q_p))$ for a transfer of $\phi_p$, and $\phi^*_{1,p}\in  \cH(M_{b_1}(\Q_p),\chi_{1,p})$ for the pullback of $\phi^*_p$. Then $\phi^{*,(j)}_p$ (defined in \S\ref{sub:nu-Jacquet}) is a transfer of $\phi^{(j)}_p$ (namely $\phi^{*,(j)}=(\phi^{(j)}_p)^*$ in $\cS(J_b)$), and $\phi^{*,(j)}_{1,p}$ is a transfer of $\phi^{(j)}_{1,p}$, for all $j\in \Z$.

The desired test function $f^{\Ig,\fke}_{1,p}$ is described by the process in \cite[\S6]{ShinIgusaStab} (which is applicable since $G_1$ has simply connected derived subgroup), with $J_{b_1},G^\fke_1,G_1$ in place of $J_b,H,G$ therein, followed by averaging over $\fkX_1=\fkX^\fke_1$. We summarize the construction as follows:
\begin{eqnarray}\label{eq:f-Ig-p-1}
f^{\Ig,\fke,(j)}_{1,p} &:=& \sum_{\omega\in \Omega_{\fke_1,\nu_1}}   f^{\Ig,\fke,(j)}_{1,p,\omega}, \qquad\mbox{where}\\
f^{\Ig,\fke,(j)}_{1,p,\omega} & := &  c_\omega\cdot \mathscr J_{\nu_{1,\omega}}\big(\tu{LS}^{\fke_1,\omega}(\delta_{P_{\nu_1}}^{1/2}\cdot\phi^{*,(j)}_{1,p})\big) \in \cH(G^\fke_1(\Q_p),\chi^{\fke,-1}_{1,p}), \nonumber
\end{eqnarray}
Here $c_\omega\in \C$ are constants (possibly zero) independent of $\phi_p$. Note that $\mathscr J_{\nu_{1,\omega}}$ and $\tu{LS}^{\fke_1,\omega}$ denote the maps in the setup with fixed central character
as in \S\ref{sub:fixed-central-character}. We observe the following about the right hand side of \eqref{eq:f-Ig-p-1}.
\begin{equation}\label{eq:f-Ig-p-2}
\delta_{P_{\nu_1}}^{1/2}\cdot\phi^{*,(j)}_{1,p} = \delta_{P_{\nu_1}}^{1/2}(\nu_1(p))  \big( \delta_{P_{\nu_1}}^{1/2}\cdot\phi^{*}_{1,p}\big)^{(j)}
= p^{\langle \rho,\nu\rangle} \big( \delta_{P_{\nu_1}}^{1/2}\cdot\phi^{*}_{1,p}\big)^{(j)}.
\end{equation}

\subsection{The stable trace formula for Igusa varieties}\label{sub:Igusa-STF}

Continuing from the preceding subsections, we freely use the notation from \S\ref{sub:STF} . The following stabilized formula for Igusa varieties is of key important to us.

\begin{theorem}\label{thm:stable-Igusa}
Given $\phi^{\infty,p}\in \cH(G(\A^{\infty,p}))$ and $\phi_p\in \cH(J_b(\Q_p))$, there exists $j_0=j_0(\phi^{\infty,p},\phi_p)\in \Z_{\ge 1}$ such that $\phi_p^{(j)}\in \cH_{\tu{acc}}(J_b(\Q_p))$ and
\begin{equation}\label{eq:stable-Igusa}
\Tr\left(\phi^{\infty,p}\phi^{(j)}_p \left|\iota H_c(\Ig_b,\lql) \right.\right) = \ST^{G^*}_{\el,\chi_0}(f^{\Ig,*,(j)})+ \sum_{\fke\in \cE^<_{\el}(G)}\iota(G,G^\fke) \ST^{G^\fke_1}_{\el,\chi_1^{\fke}} \left(f^{\Ig,\fke,(j)}_1\right).
\end{equation}
for every integer $j\ge j_0$ divisible by $r$.
\end{theorem}

\begin{proof} 
The point is to stabilize the main result of \cite{MackCrane}, which obtains the following expansion for
the left hand side of \eqref{eq:stable-Igusa}:
\begin{equation}\label{eq:IgusaTF}
\sum_{\gamma_0\in \Sigma_{\R\tu{-ell}}(G)}  \frac{c_2(\gamma_0)  \Tr \xi (\gamma_0)}{\bar \iota_G (\gamma_0)}  \sum_{\mathfrak c=(\gamma_0,a,[b_0])} c_1 (\mathfrak c) O_{\gamma}^{G(\A^{\infty,p})}(\phi^{\infty,p})O^{J_b(\Q_p)}_{\delta}(\phi^{(j)}_p),
\end{equation}
where the inner sum is over the set of acceptable $b$-admissible Kottwitz parameters $\mathfrak c$, and the conjugacy class $(\gamma,\delta)$ in $G(\A^{\infty,p})\times J_b(\Q_p)$ is determined by $\mathfrak c$ as in \emph{loc.~cit.} We do not recall the notation further, as it suffices for our purpose to observe that \eqref{eq:IgusaTF} resembles the analogous expansion for Shimura varieties \cite[(1.8.8.1)]{KSZ} except that different terms appear at $p$. Indeed, the stabilization of \eqref{eq:IgusaTF} into \eqref{eq:stable-Igusa} can be carried out by following the stabilization for Shimura varieties in \cite[\S8]{KSZ}, if we argue differently at $p$ following \cite[\S6]{ShinIgusaStab}.\footnote{The stabilization in \cite{KSZ} is written for arbitrary Shimura varieties (assuming the Shimura variety analogue of \eqref{eq:IgusaTF} when $(G,X)$ is not of abelian type). In \cite{ShinIgusaStab}, the main assumptions are that $(G,X)$ is of certain PEL type and that $G_{\der}$ is simply connected. The former is irrelevant for stabilization, and the latter is removed by working with $z$-extensions as in \cite{KSZ}.} The details are going to appear in \cite{BMS} with a precise normalization of transfer factors.
\end{proof}

\begin{remark}
In the argument above, we do not need a precise local normalization of transfer factors and Haar measures over all places of $\Q$ when $\fke\neq \fke^*$, since it will only affect error terms by constant factors in the estimate. For instance, we need not know what normalization of transfer factors should be taken at $p$ and $\infty$ (as well as away from $p$ and $\infty$) to satisfy the product formula for transfer factors, cf.~Remark \ref{rem:transfer-factors}. We intend to work out precise normalizations in a future paper, thus removing ambiguity in the coefficients $c_\omega$ in \eqref{eq:f-Ig-p-1}.
\end{remark}

\subsection{Completion of the proof of Theorem \ref{thm:Hc(Ig)}}\label{sub:completion-of-proof}

The main term in the right hand side of Theorem \ref{thm:stable-Igusa} will turn out to be the following. Recall $\cA_{\mathbf 1}(G)$ from Theorem \ref{thm:H0(Ig)}.

\begin{proposition}\label{prop:spectral-main-term}
Fix $\phi^{\infty,p} \phi_p\in \cH(G(\A^{\infty,p})\times J_b(\Q_p))$, from which $f^{*,(j)}\in \cH(G^*(\A))$ is given as in \S\ref{sub:test-function-at-p} for every $j\in \Z_{\ge 1}$ such that $j\ge j_0=j_0(\phi^{\infty,p} \phi_p)$.
As $j\ge j_0$ varies over positive integers divisible by $r$, we have the estimate
$$
T^{G^*}_{\disc,\chi_0}\big(f^{\Ig,*,(j)}\big) = \sum_{\pi\in \cA_{\mathbf 1}(G)} \Tr \left (\phi^{\infty,p} | \pi^{\infty,p} \right) \cdot \Tr \left( \phi_p^{(j)} |  (\pi_p\circ \zeta_b)\otimes \delta_{P_b} \right) \,+\, o \left( p^{j \langle 2\rho,\nu_b\rangle} \right).
$$
\end{proposition}

\begin{proof} 
We have
\begin{equation}\label{eq:spectral-main}
T^{G^*}_{\disc,\chi_0}\left(f^{\Ig,*,(j)}\right)= \sum_{\pi^*} m(\pi^*) \Tr \left (f^{\Ig,*,p} | \pi^{*,p} \right)\Tr \left(f^{\Ig,*,(j)}_p|\pi^*_p \right).
\end{equation}
Let $JH(J_{P_b^{\tu{op}}}(\pi^*_p))$ denote the multi-set of irreducible subquotients of $J_{P_b^{\tu{op}}}(\pi^*_p)$ (up to isomorphism). The central character of $\tau \in JH(J_{P_b^{\tu{op}}}(\pi^*_p))$ is denoted $\omega_\tau$. We see from Lemma \ref{lem:O-tr-nu-ascent} (ii) that
\begin{eqnarray}
\Tr \left(f^{\Ig,*,(j)}_p|\pi^*_p\right)&=& \Tr \left( \delta_{P_b}^{1/2} \phi^{*,(j)}_p | J_{P_b^{\tu{op}}}(\pi^*_p) \right)=\Tr \left(\phi^{*,(j)}_p | J_{P_b^{\tu{op}}}(\pi^*_p)\otimes \delta_{P_b}^{1/2} \right) \nonumber
\\ &=&
\sum_{\tau\in JH(J_{P_b^{\tu{op}}}(\pi^*_p))} \omega_\tau(j\nu_b(p)) \delta_{P_b}^{1/2}(j\nu_b(p)) \, \Tr (\phi^*_p| \tau). \label{eq:spectral-main-term-at-p-pre}
\end{eqnarray}
We have $j\nu_b(p)\in A_{P_b^{\tu{op}}}^{--}$, cf.~\S\ref{sub:nu-Jacquet}. Since our running assumption that $b$ is $\Q$-non-basic implies ($\Q$-nb($P^{\tu{op}}_b$)) by Lemma \ref{lem:Q-non-basic}, it follows from Corollary \ref{cor:BoundExp} that the largest growth of $\omega_\tau(j\nu_b(p))$ as a function in $j$ is achieved exactly when $\dim \pi^*=1$. In that case, we have $m(\pi^*)=1$ and $\pi^*_p$ is a unitary character. Via Lemma \ref{lem:one-dim-auto-rep}, $\pi^*$ corresponds to a unique one-dimensional automorphic representation $\pi$ of $G(\A)$. We have $\pi^*_p\simeq \pi_p$ via $G^*(\Q_p)\simeq G(\Q_p)$. Thus
\begin{align}\label{eq:spectral-main-term-at-p}
\Tr (\phi^{*,(j)}_p | J_{P_b^{\tu{op}}}(\pi^*_p) \otimes \delta_{P_b}^{1/2}) ~\stackrel{\textup{Rem.}\,\ref{rem:Jacquet-of-1dim}}{=\joinrel=\joinrel=\joinrel=}~
\Tr (\phi^{*,(j)}_p |  (\pi^*_p\circ \zeta^*_b)\otimes \delta_{P_b}  )
= \Tr (\phi^{(j)}_p |  (\pi_p\circ \zeta_b)\otimes \delta_{P_b}  )\nonumber\\
=  \delta_{P_b}(j\nu_b(p)) \Tr (\phi_p|  (\pi_p\circ \zeta_b)\otimes \delta_{P_b} )
=  p^{j \langle 2\rho,\nu_b\rangle} \Tr (\phi_p|  (\pi_p\circ \zeta_b)\otimes \delta_{P_b}).
\end{align}
We used Lemma \ref{lem:transfer-1-dim} for the second equality above. Indeed, $(\pi^*_p\circ \zeta^*_b)\otimes \delta_{P_b}$ as a character of $M_b(\Q_p)$ and $ (\pi_p\circ \zeta_b)\otimes \delta_{P_b}$ as a character of $J_b(\Q_p)$ correspond to each other via the diagram \eqref{eq:JbMbG}. 

Let $f_{\textbf{1}}$ denote the averaged Lefschetz function on $G(\R)$ as in \S\ref{sub:Lefschetz} with $\xi=\mathbf 1$ and $\zeta=\mathbf 1$. Write $e(G^{\infty,p}):=\prod_{v\neq\infty,p} e(G_v)$ for the product of Kottwitz signs. We rewrite \eqref{eq:spectral-main} as
$$
T^{G^*}_{\disc,\chi_0}\big(f^{\Ig,*,(j)}\big)~  =
\sum_{\pi^*\atop\dim\pi^*=1}\Tr \left(f^{\Ig,*}_{\infty} | \pi^*_\infty \right) \Tr \big(f^{\Ig,*,\infty,p} | \pi^{*,\infty,p}\big)\Tr \big(f^{\Ig,*,(j)}_p|\pi^*_p\big)
+ o \big( p^{j \langle 2\rho,\nu_b\rangle} \big)
$$
$$
= \sum_{\pi\atop\dim\pi=1} \Tr (f_{\textbf{1}} | \pi_\infty ) \Tr (\phi^{\infty,p} | \pi^{\infty,p} ) \Tr\big(\phi^{(j)}_p| (\pi_p\circ \zeta_b)\otimes \delta_{P_b} \big)  + o \big( p^{j \langle 2\rho,\nu_b\rangle}\big),
$$
where the last equality was obtained from \eqref{eq:spectral-main-term-at-p-pre} at $p$, Lemma \ref{lem:transfer-of-Lefschetz} at $\infty$, and Lemma \ref{lem:transfer-1-dim} at the places away from $p$. To conclude, we invoke Lemma \ref{lem:tr-of-char-at-infty} to see that $\Tr (f_{\textbf{1}}|\pi_\infty)=1$ if $\pi_\infty|_{G(\R)_+}=\mathbf 1$ and $\Tr (f_{\textbf{1}}|\pi_\infty)=0$ otherwise.
\end{proof}

Finally we complete the proof of Theorem \ref{thm:Hc(Ig)} employing the main estimates of \S\ref{sec:asymptotic-analysis}.

\begin{corollary}\label{cor:final-corollary}
Theorem \ref{thm:Hc(Ig)} is true.
\end{corollary}

\begin{proof}
Let $q\neq p$ be an auxiliary prime such that $G_{\Q_q}$ is split. Fix $\phi^{\infty,p,q}\phi_p\in \cH(G(\A^{\infty,p,q})\times J_b(\Q_p))$. Let $\fke \in \cE^<_{\el}(G)$. There exists a constant $C_{\fke}=C_{\fke}(\phi^{\infty,p,q},\phi_p)>0$ such that for each $\phi_q\in \cH(G(\Q_q))_{C_{\fke}\tu{-reg}}$, we have the following bound on endoscopic terms in the stabilization of Theorem \ref{thm:stable-Igusa} by applying the last bound in Corollary \ref{cor:main-estimate} to $k=j/r$, $\nu=r\nu_{1,\omega}$, $\phi_p^{(k)}=c_{\omega}(\delta_{P_{\nu_1}}^{1/2} \phi_{1,p}^*)^{(k)}$, and $\chi=\chi_0$ for each $\omega\in \Omega_{\fke_1,\nu_1}$. Notice that $f^{\Ig,\fke,(j)}_{1,p,\omega}$ is $p^{\langle \rho^\fke,\nu_{1,\omega}\rangle}$ times $f_p^{(k)}$ of that corollary, in light of \eqref{eq:f-Ig-p-1} and \eqref{eq:f-Ig-p-2}. We have 
$$
\ST^{G^\fke_1}_{\el,\chi_1^{\fke}} \left((f^{\Ig,\fke,p}_1) f^{\Ig,\fke,(j)}_{1,p,\omega}\right) = O\left( p^{k ( \langle 2\rho^\fke,r\nu_{1,\omega}\rangle+ \langle \chi_1^{\fke},r\ol{\nu}_{1,\omega} \rangle)}\right)
 = O\left( p^{j( \langle 2\rho^\fke,\nu_{1,\omega}\rangle+ \langle \chi_1^{\fke},\ol{\nu}_{1,\omega} \rangle)}\right).
$$
To turn this into a more manageable bound, we use (a) and (b) from the proof of Corollary \ref{cor:main-estimate} and the fact that $\langle \chi_{0,\infty},\ol \nu\rangle=0$ since $\chi_0$ (which plays the role of $\chi$ there) is trivial. Thereby we see that the right hand side is $o \left( p^{j \langle 2\rho,\nu_b\rangle} \right)$. Taking the sum over $\omega\in \Omega_{\fke_1,\nu_1}$, we obtain 
\begin{eqnarray}\label{eq:final-corollary}
\ST^{G^\fke_1}_{\el,\chi_1^{\fke}} \left(f^{\Ig,\fke,(j)}_1\right) =  o \left( p^{j \langle 2\rho,\nu_b\rangle} \right),\qquad \fke\in \cE^<_{\el}(G).
\end{eqnarray}

By Lemma \ref{lem:finiteness}, there are only finitely many $\fke$ contributing to the sum in Theorem \ref{thm:stable-Igusa} for a fixed choice of $\phi^{\infty,p,q}\phi_p$. Thus the coefficients $\iota(G,G^\fke)$ are bounded by a uniform constant (depending on $\phi^{\infty,p,q}\phi_p$). We deduce the following by applying Theorem \ref{thm:stable-Igusa}, \eqref{eq:final-corollary}, Corollary \ref{cor:main-estimate} (the first estimate therein), and Proposition \ref{prop:spectral-main-term} in the order: there exists a constant $C=C(\phi^{\infty,p},\phi_p)>0$ (e.g., the maximum of $C_\fke$ over the set of finitely many $\fke$ which contribute) such that for every $\phi_q\in \cH(G(\Q_q))_{C\tu{-reg}}$, we have
$$
\Tr\left(\phi^{\infty,p}\phi^{(j)}_p |\iota H_c(\Ig_b,\lql ) \right) = \ST^{G^*}_{\el,\chi} \big(f^{\Ig,*,(j)}\big) + o \left( p^{j \langle 2\rho,\nu_b\rangle} \right)
 = T^{G^*}_{\disc,\chi} \big(f^{\Ig,*,(j)}\big) + o \left( p^{j \langle 2\rho,\nu_b\rangle} \right)
$$
$$=  \sum_{\pi\in \cA_{\mathbf 1}(G)} \Tr \left(\phi^{\infty,p} | \pi^{\infty,p} \right) \cdot \Tr \left( \phi_p^{*,(j)} |  (\pi_p\circ \zeta_b)\otimes \delta_{P_b} \right)
 \,+\, o \left( p^{j \langle 2\rho,\nu_b\rangle} \right).
$$
We have seen in \eqref{eq:spectral-main-term-at-p} that $\Tr \left( \phi_p^{(j)} |  (\pi_p\circ \zeta_b)\otimes \delta_{P_b} \right)$ is either 0 or a nonzero multiple of $p^{j \langle 2\rho,\nu_b\rangle}$ as $j$ varies over multiples of $r$. Since $\dim\Ig_b=\langle 2\rho,\nu_b\rangle$, it is implied by \eqref{eq:trace-(j)=trace-Frob} and the Lang--Weil bound that the leading term should be of order $p^{j \langle 2\rho,\nu_b\rangle}$. \footnote{In fact, the Lang--Weil bound proves that $\dim\Ig_b=\langle 2\rho,\nu_b\rangle$ even if we did not know it a priori. This gives an alternative proof of the dimension formula in Proposition \ref{prop:smoothness-and-dim-of-leaves}.} Therefore 
$$
\Tr\left(\phi^{\infty,p}\phi^{(j)}_p |\iota H^{\langle 4\rho,\nu_b\rangle}_c(\Ig_b,\lql ) \right) =  \sum_{\pi\in \cA_{\mathbf 1}(G)} \Tr \left (\phi^{\infty,p} | \pi^{\infty,p} \right) \cdot \Tr \left( \phi_p^{(j)} | (\pi_p\circ \zeta_b)\otimes \delta_{P_b} \right).
$$
Let $B_{q}$ be a Borel subgroup of $G_{\Q_q}$ over $\Q_q$ with a Levi component $T_q$. By Lemma \ref{lem:C-regular-spectral-implication}, we have an isomorphism of $G(\A^{\infty,p,q})\times J_b(\Q_p)\times T_q(\Q_q)$-representations
$$
J_{B_{q}}\left( H^{\langle 4\rho,\nu_b\rangle}_c(\Ig_b,\lql )\right) \simeq  \sum_{\pi\in \cA_{\mathbf 1}(G)}  \pi^{\infty,p,q} \otimes ((\pi_p\circ \zeta_b)\otimes \delta_{P_b})\otimes J_{B_{q}}(\pi_q).
$$
(A priori the isomorphism exists up to semisimplification, but distinct one-dimensional representations have no extensions with each other.) Repeating the same argument for another prime $q'\notin \{p,q\}$ such that $G(\Q_{q'})$ is split, the above isomorphism exists with $q'$ in place of $q$. Comparing the two consequences, we deduce that as $G(\A^{\infty,p})\times J_b(\Q_p)$-modules,
$$
\iota H^{\langle 4\rho,\nu_b\rangle}_c(\Ig_b,\lql )\simeq \bigoplus_{\pi\in \cA_{\mathbf 1}(G)}  \pi^{\infty,p} \otimes ((\pi_p\circ \zeta_b)\otimes \delta_{P_b}).
$$
\end{proof}

\section{Applications to geometry}\label{sec:apps-to-geometry}

This section is devoted to working out geometric consequences of Theorem \ref{thm:H0(Ig)}, continuing in the setting of Hodge-type Shimura varieties with hyperspecial level at $p$.

\subsection{Irreducibility of Igusa varieties} \label{sub:main-to-irreducibility}

In \S\ref{sub:intro-irreducibility}, we reviewed earlier results on irreducibility of Igusa towers over the $\mu$-ordinary Newton strata of certain PEL-type Shimura varieties. Now we explain that our main theorem implies a generalization thereof to Hodge-type Shimura varieties and to non-$\mu$-ordinary strata.

Let $(G,X,p,\cG)\in \mathcal{SD}^{\tu{ur}}_{\tu{Hodge}}$ and $b\in G(\breve \Q_p)$ be as in \S\ref{sub:inf-level-Igusa}. Assume that $b$ is $\Q$-non-basic. Define $J(\Q_p)':=\ker(\zeta_b:J_b(\Q_p)\ra G(\Q_p)^{\ab})$, cf.~\eqref{eq:JbG}. Recall that $\tu{pr}:\mathfrak{Ig}_b\ra C_{b,\Fpbar}^{\tu{perf}}$ is a pro-\'etale $J_b^{\tu{int}}$-torsor. From Theorem \ref{thm:H0(Ig)}, we deduce that Igusa varieties are ``as irreducible as possible''.

\begin{theorem}[Irreducibility of Igusa varieties]\label{thm:irreducibility-Igusa-towers}
In the setting of \S\ref{sub:inf-level-Igusa}, the stabilizer of each connected component of $\mathfrak{Ig}_b$ under the $J_b(\Q_p)$-action is equal to $J_b(\Q_p)'$.
\end{theorem}

\begin{proof}
Fix a component $I\subset \mathfrak{Ig}_b$ and write $\tu{Stab}(I)$ for the stabilizer of $I$ in $J_b(\Q_p)$. Since the $J_b(\Q_p)$-action on every $\pi_p$ in Theorem \ref{thm:H0(Ig)} factors through $J_b(\Q_p)/J_b(\Q_p)'$, we see that $\tu{Stab}(I)\supset J_b(\Q_p)'$. To prove the reverse inclusion, we show that every $\delta\in J_b(\Q_p)\backslash J_b(\Q_p)'$ acts nontrivially on $H^0(\mathfrak{Ig}_b,\lql)$. Write $\delta^{\ab}:=\zeta_b(\delta)\in G(\Q_p)^{\ab}$. Then $\delta^{\ab}\neq 1$ by assumption. It suffices to show that some $\pi$ in the summation of Theorem \ref{thm:H0(Ig)} has the property that $\pi_p(\delta^{\ab})\neq 1$. This follows from Lemma \ref{lem:existence-1-dim}.
\end{proof}

\begin{corollary}\label{cor:irreducibility}
Let $S$ be a connected component of $C_{b,\Fpbar}^{\tu{perf}}$. Then the set $\pi_0(\tu{pr}^{-1}(S))\subset \pi_0(\mathfrak{Ig}_b)$ is a torsor under the group $J_b^{\tu{int}}/(J_b^{\tu{int}}\cap J_b(\Q_p)')$. Every component of $\tu{pr}^{-1}(S)$ is a pro-\'etale torsor under $J_b^{\tu{int}}\cap J_b(\Q_p)'$, and conversely, if $I\subset \mathfrak{Ig}_b$ is an open subscheme such that $I\ra S$ is a pro-\'etale $J_b^{\tu{int}}\cap J_b(\Q_p)'$-torsor via $\tu{pr}$, then $I$ is irreducible.
\end{corollary}

\begin{proof}
As $\tu{pr}$ is a $J_b^{\tu{int}}$-torsor, $J_b^{\tu{int}}$ acts transitively on $\pi_0(\tu{pr}^{-1}(S))$. Theorem \ref{thm:irreducibility-Igusa-towers} tells us that the action factors through a simply transitive action of $J_b^{\tu{int}}/J_b^{\tu{int}}\cap J_b(\Q_p)'$, implying the first assertion. The second assertion also follows from the same theorem and the fact that $\tu{pr}$ is a $J_b^{\tu{int}}$-torsor.
\end{proof}

For the remainder of this subsection, we compare with similar irreducibility results in the $\mu$-ordinary case. Thus we specialize to the case when $[b]\in B(G,\mu_p^{-1})$ is \textbf{$\mu$-ordinary}, meaning either of the following equivalent conditions  \cite[Rem.~5.7 (2)]{Wortmann}:
\begin{itemize}
\item  $[b]=[\mu_p^{-1}(p)]$ in $B(G)$ (which implies $[b]\in B(G,\mu_p^{-1})$).
\item $[b]$ is the unique minimal element in $B(G,\mu_p^{-1})$ for the partial order $\preceq$ therein.
\end{itemize}
In this case, we may and will take $b=b^\circ=\mu_p^{-1}(p)$. Indeed, we can change $b$ within its $\sigma$-conjugacy class thanks to Proposition \ref{prop:Ig_b-Ig_b_0}. Put $r:=[k(\fkp):\F_p]$. By the convention of \S\ref{sub:central-leaves}, $\mu_p$ is defined over $\Q_{p^r}$. Then we have $\nu_b=\frac{1}{r}\sum_{i=0}^{r-1}\sigma^i\mu_p^{-1}$ (this follows from (4.3.1)--(4.3.3) of \cite{KottwitzIsocrystal1} with $n=r$ and $c=1$), which is defined over $\Q_p$, and conditions (br2) and (br3) are satisfied.

We define the $\mu$-ordinary Newton stratum $N_{b,K^p}$ as in \cite{Wortmann}, that is, by changing the definition of $C_{b,K^p}$ (\S\ref{sub:central-leaves}) to require the existence of an isomorphism only after inverting $p$. Then $C_{b,K^p}\subset N_{b,K^p}$ is closed by \cite[\S2.3, Prop.~2]{Ham-almostproduct}. It is worth verifying that $C_{b,K^p}= N_{b,K^p}$, so that $\mathfrak{Ig}_b$ is a pro-\'etale torsor over $N_{b,K^p}^{\tu{perf}}$ (not just $C_{b,K^p}^{\tu{perf}}$).

\begin{lemma}\label{lem:central-leaf=Newton}
In the $\mu$-ordinary setup above, $C_{b,K^p}=N_{b,K^p}$.
\end{lemma}

\begin{proof}
This is a consequence of two facts: that the $\mu$-ordinary Newton stratum is an Ekedahl-Oort stratum \cite[Thm.~6.10]{Wortmann}, and that every Newton stratum contains an Ekedahl-Oort stratum that is a central leaf \cite[Thm.~D]{ShenZhang}. (We thank Pol van Hoften for communicating this proof to us.)
\end{proof}

We explain that Corollary \ref{cor:irreducibility} gives another proof for the irreducibility of Igusa towers in the $\mu$-ordinary case, for unitary similitude PEL-type Shimura varieties as in \cite{CEFMV,EischenMantovan}, cf.~\cite[\S2,\S3]{HidaIrreducibility2}. Analogous arguments can be made in the elliptic/Hilbert/Siegel modular cases.

Write $\big(\Ig^{\mu\tu{-ord}}_{m,K^p}\big)_{m\ge 1}$ for the Igusa tower $(\Ig_\mu)_{m,1}$, $m\ge 1$, over the $\mu$-ordinary stratum $N_{b,K^p}$ in \cite[\S3.2]{EischenMantovan} (relative to the same $K^p$) with finite \'etale transition maps. The scheme $\Ig^{\mu\tu{-ord}}_{K^p}=\varprojlim_m \Ig^{\mu\tu{-ord}}_{m,K^p}$ is a pro-\'etale $J_b^{\tu{int}}$-torsor over $N_{b,K^p}$. Then $\mathfrak{Ig}_{b,K^p} \simeq (\Ig^{\mu\tu{-ord}}_{K^p})^{\tu{perf}}$ compatibly with the actions of $G(\A^{\infty,p})\times S_b$ (see Prop.~4.3.8 and the paragraph above Cor.~4.3.9 in \cite{CS17}; see also \cite[Rem.~2.3.7]{CS19}), and we have a $J_b(\Q_p)$-equivariant bijection
$$
\pi_0(\mathfrak{Ig}_{b,K^p}) \simeq \pi_0((\Ig^{\mu\tu{-ord}}_{K^p})^{\tu{perf}}) \simeq  \pi_0(\Ig^{\mu\tu{-ord}}_{K^p}).
$$
Therefore each connected component of $\Ig^{\mu\tu{-ord}}_{K^p}$ has stabilizer $J_b(\Q_p)'$ in $J_b(\Q_p)$.

The $\Z_p$-group $J_\mu$ in \cite[Rem.~2.9.3]{EischenMantovan} has the property that $J_b^{\tu{int}}=J_\mu(\Z_p)$. Let $I\subset \Ig^{\mu\tu{-ord}}_{K^p}$ denote the open subscheme $\Ig^{SU}_\mu$ over a fixed component $S$ of $N_{b,K^p}$ as defined in \cite[\S3.3]{EischenMantovan} (more precisely, we mean the special fiber of $\Ig^{SU}_\mu$ over $\Fpbar$). Then $I$ is a pro-\'etale $J_b^{\tu{int}}\cap J(\Q_p)'$-torsor over $S$ by construction. (The determinant map of \cite[\S3.3]{EischenMantovan} goes from a $J_b^{\tu{int}}$-torsor to a torsor under $\cG^{\ab}(\Z_p)$, so the fiber is a torsor under $\ker(J_b^{\tu{int}}\ra \cG^{\ab}(\Z_p))$.)\footnote{In fact we have not understood the definition of the determinant map in \cite[\S3.3]{EischenMantovan} unless $B$ is a field, so we should restrict our comparison with \emph{loc.~cit.}~to this setup.}
Hence $I$ is irreducible by the preceding paragraph, cf.~the proof of Corollary \ref{cor:irreducibility}.

If $[b]$ is moreover \textbf{ordinary}, namely if $[b]$ is $\mu$-ordinary and $\nu_b$ is conjugate to $\mu_p^{-1}$, then $\mu_p$ is defined over $\Q_p$ (since the conjugacy class of $\nu_b$ is always defined over $\Q_p$). Also $r=[E_{\fkp}:\Q_p]=1$. By our choice $b=\mu_p(p)^{-1}$, we have $\nu_b=\mu_p^{-1}$ (not just conjugate) in this case. The following lemma is handy when comparing with results in the ordinary case such as \cite{HidaIrreducibility1,HidaIrreducibility2}. Note that trivially $\varrho(G_{\tu{sc}}(\Q_p))=G_{\der}(\Q_p)$ when $G_{\der}=G_{\tu{sc}}$.

\begin{lemma}
If $\mu$ is ordinary, then $J_b=M_b$, $J_b^{\tu{int}}=M_b(\Z_p)$, and $J_b(\Q_p)'=M_b(\Q_p)\cap \varrho(G_{\tu{sc}}(\Q_p))$.
\end{lemma}

\begin{proof}
By definition, $M_b$ is the centralizer of $\nu_b=\mu_p^{-1}$ in $G$. From the definition \eqref{eq:def-Jb} with $b=\mu^{-1}_p(p)$, we see that $M_b$ is a closed $\Q_p$-subgroup of $J_b$. On the other hand, $M_b$ is an inner form of $J_b$, so we conclude $M_b=J_b$. Then $J_b^{\tu{int}}=J_b(\Q_p)\cap \cG(\breve\Z_p)=M_b(\Q_p)\cap  \cG(\breve\Z_p) = M_b(\Z_p)$. The description of $J_b(\Q_p)'$ is obvious from $J_b=M_b$.
\end{proof}

\subsection{The discrete Hecke orbit conjecture}
\label{sub:main-to-discHO}

We state the Hecke orbit conjecture for Shimura varieties of Hodge type with hyperspecial level at $p$. We prove the discrete part of the Hecke orbit conjecture, and find purely local criteria for the irreduciblity of central leaves.

Fix $(G,X,p,\cG)\in \mathcal{SD}^{\tu{ur}}_{\tu{Hodge}}$. Let $x\in \mS_{K^pK_p,k(\fkp)}(\Fpbar)$. Denote by $\tilde x\subset |\mS_{K_p,k(\fkp)}|$ the preimage of $x$ in the topological space $|\mS_{K_p,k(\fkp)}|$ via the projection map $\mS_{K_p,k(\fkp)}\ra \mS_{K^p K_p,k(\fkp)}$. Define the prime-to-$p$ \textbf{Hecke orbit} as a set:
$$
H(x):=\tilde x\cdot G(\A^{\infty,p})~\subset~ |\mS_{K_p,k(\fkp)}|.
$$
Write $H_{K^p}(x)$ for the image of $H(x)$ in $|\mS_{K^p K_p,k(\fkp)}|$. By $C_{K^p}(x)$ we mean the central leaf through~$x$, namely $C_{b_x,K^p}$. Since the action of $G(\A^{\infty,p})$ does not change the $(G(\breve\Z_p),\sigma)$-conjugacy class $[[b_x]]$, we see that
$$
H_{K^p}(x)\subset |C_{K^p}(x)|.
$$
Following Chai and Oort, cf.~\cite{Chai-HO,ChaiICM}, we formulate the Hecke Orbit Problem as follows.

\begin{question}[Hecke Orbit Problem]\label{conj:HO}
Let $x\in \mS_{K^pK_p,k(\fkp)}(\Fpbar)$ such that $[b_x]$ is $\Q$-non-basic.
Does the subset $H_{K^p}(x)$ have the following properties?
\begin{itemize}
\item[\textup{(HO)}] $H_{K^p}(x)$ is Zariski dense in the central leaf $C_{K^p}(x)$.
\item[\textup{(HO$_{\textup{cont}}$)}] The Zariski closure of $H_{K^p}(x)$ in $C_{K^p}(x)$ is a union of irreducible components of $C_{K^p}(x)$.
\item[\textup{(HO$_{\textup{disc}}$)}] $H_{K^p}(x)$ meets every irreducible component of $C_{K^p}(x)$.
\item[\textup{(HO$^+_{\textup{disc}}$)}]
 For every $x\in \mS_{K^pK_p}(\ol{k(\fkp)})$ such that $[b_x]$ is $\Q$-non-basic, the immersion $C_{K^p}(x)\hra \mS_{K^pK_p,k(\fkp)}$ induces a bijection
 $\pi_0(C_{K^p}(x)) \isom \pi_0(\mS_{K_p K^p,\ol{k(\fkp)}}).$
\end{itemize}
\end{question}

\begin{remark}
The hypothesis on $[b_x]$ cannot be weakened to only requiring that $[b_x]$ be non-basic. For example, for Shimura varieties arising from $(G\times\cdots\times G,X\times \cdots \times X)$, with $(G,X)$ a Shimura datum, we see the necessity to assume $[b_x]$ to be basic in every copy of $G$ (which is a $\Q$-factor).
\end{remark}

Note that (HO) is the analogue of the Hecke orbit conjecture for Hodge-type Shimura varieties, which is divided into discrete and continuous parts in the sense that (HO$_{\textup{disc}}$) and (HO$_{\textup{cont}}$) combined is obviously equivalent to (HO). 
We usually refer to \textup{(HO$^+_{\textup{disc}}$)} as ``irreducibility of central leaves'', as it states that the central leaf through $x$ is irreducible in each connected component of the ambient Shimura variety. 
Regarding (HO$_{\textup{disc}}$) and (HO$^+_{\textup{disc}}$), we have the following relationship and representation-theoretic interpretations. 

\begin{lemma}\label{lem:HO-disc}
Let $b$ be $\Q$-non-basic with $b\in \cG(\Z_{p}^{\tu{ur}})\sigma\mu_p(p)^{-1} \cG(\Z_{p}^{\tu{ur}})$. Between the following statements, there are logical implications (1) $\Leftrightarrow$ (2) $\Rightarrow$ (3) $\Leftrightarrow$ (4).
\begin{enumerate}
\item $\textup{(HO}^+_{\textup{disc}})$ holds true for all neat $K^p\subset G(\A^{\infty,p})$ and all $x\in C_{b,K^p}(\Fpbar)$.
\item $H^0 (C_{b},\lql)\simeq H^0(\Sh_{K_p},\lql)$ as $G(\A^{\infty,p})$-modules. (This asserts the existence of an isomorphism, which need not be induced by the natural map $C_b\ra \Sh_{K_p}$.)
\item $\textup{(HO}_{\textup{disc}})$ holds true for all neat $K^p\subset G(\A^{\infty,p})$ and all $x\in C_{b,K^p}(\Fpbar)$.
\item $\dim_{G(\A^{\infty,p})} (\mathbf{1},H^0 (C_{b},\lql))=1$.
\end{enumerate}
\end{lemma}

\begin{proof}
As $K^p$ varies, the immersion $C_{b,K^p}\hra \mS_{K^pK_p,k(\fkp)}$ induces a $G(\A^{\infty,p})$-equivariant map
\begin{equation}\label{eq:pi0-C-S}
\pi_0(C_{b}) \ra \pi_0(\mS_{K_p,\ol{k(\fkp)}}),
\end{equation}
which is surjective since $G(\A^{\infty,p})$ acts transitively on $\pi_0(\mS_{K_p,\ol{k(\fkp)}})$ by Lemma \ref{lem:transitive-on-Sh}. 
Condition (1) is equivalent to the condition that the above map is an isomorphism, and (3) is equivalent to the condition that $G(\A^{\infty,p})$ acts transitively on $\pi_0(C_{b})$. From this, it is clear that (1) $\Rightarrow$ (2) and that (1) $\Rightarrow$ (3) $\Leftrightarrow$ (4).

Now suppose that (2) holds. Then the $G(\A^{\infty,p})$-equivariant injection $H^0(\Sh_{K_p},\lql)\hra H^0 (C_{b},\lql)$ \emph{induced by} \eqref{eq:pi0-C-S} must be an isomorphism by (2), since each $\pi^{\infty,p}$ appears with finite multiplicity in $\iota H^0(\Sh_{K_p},\lql)$ (Lemma \ref{lem:H_c(Sh)-top}). Hence \eqref{eq:pi0-C-S} is a bijection, and (1) follows.
\end{proof}

Now we allow $b\in G(\breve \Q_p)$ which need not be $\Q$-non-basic for the moment.
The map $\zeta_b: J_b(\Q_p)\ra G(\Q_p)^{\ab}$ from \eqref{eq:JbG} is an open map as it is the composite of open maps. Write $\cG^{\ab}$ for the abelianization of $\cG$ as an algebraic group over $\Z_p$. Then $\cG^{\ab}$ is a torus over $\Z_p$, and $\cG^{\ab}(\Z_p)$ is a unique maximal subgroup of $G^{\ab}(\Q_p)$. 
On the level of points, denote by $\cG(\Z_p)^{\ab}$ the image of $\cG(\Z_p)$ under the projection $G(\Q_p) \ra G(\Q_p)^{\ab}$. 
When $G_{\tu{der}}=G_{\tu{sc}}$, then $\cG(\Z_p)^{\ab}=\cG^{\ab}(\Z_p)$ as subgroups of $G(\Q_p)^{\ab}=G^{\ab}(\Q_p)$.

We define the affine Deligne--Lusztig set
$$X_\upsilon(b):=\{ g\in G(\breve\Q_p)/\cG(\breve\Z_p): g^{-1} b \sigma(g)\in \cG(\breve\Z_p)\upsilon(p) \cG(\breve\Z_p)\},$$
equipped with the left muliplication action by $J_b(\Q_p)$. In fact $X_\upsilon(b)$ is the set of closed points of a perfect variety over $\Fpbar$ \cite{ZhuGr,BhattScholzeGr}.

\begin{lemma}\label{lem:image-of-Jb-int}
The subgroup $\zeta_b(J_b^{\tu{int}})\subset G(\Q_p)^{\ab}$ is open, compact, and contained in $\cG(\Z_p)^{\ab}$.  Furthermore, there exists $b_0\in \cG(\Z_{p}^{\tu{ur}})\sigma\mu_p(p)^{-1} \cG(\Z_{p}^{\tu{ur}})$ which is $\sigma$-conjugate to $b$ in $G(\breve \Q_p)$ such that $\zeta_{b_0}(J_{b_0}^{\tu{int}})=\cG(\Z_p)^{\ab}$.
\end{lemma}

\begin{proof}
Since $\zeta_b$ is an open map, it carries the open subgroup $J_b^{\tu{int}}$ of $J_b(\Q_p)$ onto an open subgroup of $G(\Q_p)^{\ab}$. Since $\zeta_b(J_b^{\tu{int}})$ is contained in both $G(\Q_p)^{\ab}$ and the image of $\cG(\breve \Z_p)$ under the abelianization map, it is  contained in $\cG(\Z_p)^{\ab}$. This proves the first assertion. 

As for the second assertion, we start by claiming that there exists $b_0\in \cG(\Z_{p}^{\tu{ur}})\sigma\mu_p(p)^{-1} \cG(\Z_{p}^{\tu{ur}})$ which is $\sigma$-conjugate to $b$ such that $J_{b_0}^{\tu{int}}$ contains an Iwahori subgroup of $J_{b_0}(\Q_p)$. This follows from the proof of \cite[Prop.~3.1.4]{ZhouZhuADLV}, which is based on results of He  \cite{HeADLV}. (The claim amounts to the existence of a point on $X_{\sigma\mu_p^{-1}}(b)$ whose stabilizer in $J_b(\Q_p)$ contains an Iwahori subgroup. It is enough to check this on the level of Iwahori affine Deligne--Lusztig varieties. Moreover, the assertion is invariant under $J_b(\Q_p)$-equivariant bijections bewteen Iwahori affine Deligne--Lusztig varieties. With this in mind, take $w$ and $x$ as in the first two paragraphs of the proof of \cite[Prop.~3.1.4]{ZhouZhuADLV}. Then the claim follows from the fact that $J_{\dot x}(\Q_p)\cap \mathcal{P}(\breve \Z_p)$ is a parahoric subgroup of $J_{\dot x}(\Q_p)$, cf.~p.168, line 14 in \emph{loc.~cit.}, where $\mathcal{P}$ is a parahoric subgroup of $G_{\breve \Q_p}$ defined therein.)

By the last claim, it suffices to show that 
$$\zeta_{b_0}(\mathrm{Iw})=\cG(\Z_p)^{\ab}$$
 for just one Iwahori subgroup $\mathrm{Iw}$ of $J_{b_0}(\Q_p)$ since all Iwahori subgroups are $J_{b_0}(\Q_p)$-conjugate. As the statement is now only about $b_0$, we drop the subscript 0 to simplify notation.

By using an unramified $z$-extension $1\ra Z_1 \ra G_1 \ra G \ra 1$ over $\Q_p$ (which gives rise to a smooth map of reductive models $\cG_1\ra \cG$ with connected kernel over $\Z_p$; the induced map $\cG_1(\Z_p)\ra \cG(\Z_p)$ is thus surjective) and choosing $b_1\in G_1(\breve \Q_p)$ as in Lemma \ref{lem:lift-b,mu}, we reduce to the case when $G_{\tu{der}}=G_{\tu{sc}}$. So $M_b$ and $J_b$ also have simply connected derived subgroups. In particular, $G(\Q_p)^{\ab}=G^{\ab}(\Q_p)$ and likewise for $M_b$ and $J_b$.

Since $M_b$ splits over $\Q_p^{\tu{ur}}$, we see from \eqref{eq:Jb-Mb-isom} that $J_b$ also splits over $\Q_p^{\tu{ur}}$. By \cite[\S2.4]{DeBackerUnramifiedTori}, $J_b$ contains an unramified elliptic maximal torus $T_b$ over $\Q_p$. Write $\mathcal{T}_b$ for the torus over $\Z_p$ extending $T_b$. Then $\mathcal{T}_b$ is contained in some Iwahori subgroup of $J_b(\Q_p)$ (associated with the chamber whose closure contains the facet $F$ in \cite[\S2.4]{DeBackerUnramifiedTori}). In view of \eqref{eq:JbMbG}, we can think of $\zeta_b$ as the map on the set of $\Q_p$-points arising from the composite $\Q_p$-morphism
$$J_b \ra J_b^{\ab} \simeq M_b^{\ab} \ra G^{\ab}.$$ 
Composing with $T_b\hra J_b$, we obtain a $\Q_p$-morphism $T_b \ra G^{\ab}$ of unramified tori. This uniquely extends to a $\Z_p$-morphism $\mathcal{T}_b \ra \cG^{\ab}$, inducing the map 
$$\mathcal{T}_b(\Z_p) \ra \cG^{\ab}(\Z_p).$$
We will be done if this map is surjective. By smoothness and Lang's theorem over finite fields, it is enough to check that $\ker(\mathcal{T}_b \ra \cG^{\ab})$ is connected. To see this, observe that $T'_b:=\ker(T_b \ra G^{\ab})$ is connected, since it becomes a maximal torus of $G_{\der}$ after base change from $\Q_p$ to $\lqp$.
Thus we have a short exact sequence $1\ra T'_b \ra T_b \ra G^{\ab} \ra 1$ of unramified tori over $\Q_p$. It follows that $\ker(\mathcal{T}_b \ra \cG^{\ab})$ is the torus over $\Z_p$ extending $T'_b$, hence connected.
\end{proof}

\begin{remark}
In an earlier version of this paper, we incorrectly asserted that $\zeta_b(J_b^{\tu{int}}) = \zeta_{b_0}(J_{b_0}^{\tu{int}})$. This led us to mistakenly claim that $\textup{(HO}^+_{\textup{disc}})$ was true for non-$\Q$-basic $b$. As illustrated by Example \ref{ex:HO+false} below, $\textup{(HO}^+_{\textup{disc}})$ is false in general.
\end{remark}

Write $U_p(b)$ for the preimage of $\zeta_b(J_b(\Q_p))$ under the projection $G(\Q_p)\ra G^{\ab}(\Q_p)$. Then $U_p(b)$ is an open subgroup of $G(\Q_p)$ by Lemma \ref{lem:image-of-Jb-int}. 

\begin{theorem}\label{thm:HO-disc}
Let $b$ be as in Lemma \ref{lem:HO-disc}. As a $G(\A^{\infty,p})$-module,
$$
\iota H^0(C_{b},\lql)\simeq \bigoplus_{\pi\in \cA_{\mathbf 1}(G)} \dim  \pi_p^{U_p(b)} \cdot  \pi^{\infty,p}.
$$
Moreover, $\textup{(HO}_{\textup{disc}})$ holds true for all neat $K^p\subset G(\A^{\infty,p})$ and all $x\in C_{b,K^p}(\Fpbar)$.
\end{theorem}

\begin{proof}
The first assertion is a consequence of Theorem \ref{thm:H0(Ig)} and Lemma \ref{lem:Igusa-basic} (1), noting that perfection does not change cohomology. For the second assertion, let $\pi\in \cA_{\mathbf 1}(G)$ with $\pi^{\infty,p}=\mathbf{1}$. It suffices to check that if $\pi_p|_{U_p(b)}=\mathbf{1}$, then $\pi$ is trivial.
We have $\pi_\infty|_{G(\R)_+}=\mathbf{1}$ from $\pi\in \cA_{\mathbf 1}(G)$. Thus $\pi$ as a continuous character $G(\Q)\backslash G(\A)\ra \C^\times$ is trivial on $G(\A^{\infty,p})U_p(b) G(\R)_+$. Since $G(\Q)\hra G(\Q_p)\times G(\R)$ has dense image (cf.~proof of Lemma \ref{lem:transitive-on-Sh}), $\pi$ is trivial. 
\end{proof}

 In light of (1) $\Leftrightarrow$ (2) in Lemma \ref{lem:HO-disc}, the following theorem gives criteria for $\textup{(HO}^+_{\textup{disc}})$. Observe that (c2) and (c3) are purely local conditions at $p$, depending only on the data pertaining to $G_{\Q_p}$. 

\begin{theorem}\label{thm:HO-criterion} Let $b\in \cG(\Z_{p}^{\tu{ur}})\sigma\mu_p(p)^{-1} \cG(\Z_{p}^{\tu{ur}})$. Assume that $b$ is $\Q$-non-basic. The following are equivalent.
\begin{enumerate}
\item [(c0)] $\textup{(HO}^+_{\textup{disc}})$ holds true for all neat $K^p\subset G(\A^{\infty,p})$ and all $x\in C_{b,K^p}(\Fpbar)$.
\item[(c1)] $H^0 (C_{b},\lql)\simeq H^0(\Sh_{K_p},\lql)$ as $G(\A^{\infty,p})$-modules.
\item[(c2)] $\zeta_b(J_b^{\tu{int}})=\cG^{\ab}(\Z_p)$.
\item[(c3)] The stabilizer in $J_b(\Q_p)$ of $1\in X_{\sigma \mu_p^{-1}}(b)$ maps onto $\cG^{\ab}(\Z_p)$ under $\zeta_b$.
\end{enumerate}
\end{theorem}

\begin{proof}
(c0) $\Leftrightarrow$ (c1). Already shown in Lemma \ref{lem:HO-disc}.

(c2) $\Rightarrow$ (c1).
For each $\pi\in \cA_{\mathbf 1}(G)$, it is enough to check the claim that $\pi_p^{U_p(b)}\neq 0$ if and only if $\pi_p^{\cG(\Z_p)}\neq 0$. As a character $\pi_p$ factors through $G(\Q_p)\ra G^{\ab}(\Q_p)$, and (c2) ensures that the images of $U_p(b)$ and $\cG(\Z_p)$ in $G^{\ab}(\Q_p)$ are both equal to $\cG^{\ab}(\Z_p)$. The claim follows.

(c1) $\Rightarrow$ (c2). Assuming $\zeta_b(J_b^{\tu{int}})\subsetneq \cG^{\ab}(\Z_p)$, it is enough to find $\pi\in \cA(G)_{\mathbf 1}$ such that $\pi_p|_{\zeta_b(J_b^{\tu{int}})}=\textbf 1$ but $\pi_p|_{ \cG^{\ab}(\Z_p)} \neq \textbf 1$. (Here $\pi_p$ is viewed as a character of $G^{\ab}(\Q_p)$.) This is proved in the same way as Lemma \ref{lem:existence-1-dim}. (When reducing to the torus case, use a $z$-extension which is unramified at $p$.)

(c2) $\Leftrightarrow$ (c3). Clear since the stabilizer of  in $J_b(\Q_p)$ of $1\in X_{\sigma \mu_p^{-1}}(b)$ is nothing but $J_b^{\tu{int}}$.
\end{proof}

\begin{corollary}
$\tu{(HO}^+_{\tu{disc}}\tu{)}$ is true on the $\mu$-ordinary Newton stratum.
\end{corollary}

\begin{proof}
Let $b=b_x\in G(\breve \Q_p)$ for $x$ in the $\mu$-ordinary stratum. Choose $b_0$ as in Lemma \ref{lem:image-of-Jb-int} so that $\zeta_{b_0}(J_{b_0}^{\tu{int}})=\cG(\Z_p)^{\ab}$. Then (\tu{HO}$^+_{\tu{disc}}$) holds true for $C_{b_0,K^p}$ for every $K^p$, through criterion (c2) of Theorem \ref{thm:HO-criterion}. Since $C_{b_0,K^p}$ is the entire $\mu$-ordinary Newton stratum by Lemma \ref{lem:central-leaf=Newton}, the proof is finished.
\end{proof}

We also have a partial analogue of Theorem \ref{thm:HO-criterion}  that is isogeny-invariant, \ie depending on $b$ only through $[b]$.

\begin{corollary}\label{cor:HO+criterion}
Let $b\in G(\breve \Q_p)$. Assume that $[b]\in B(G_{\Q_p},\mu_p^{-1})$ and that $b$ is $\Q$-non-basic.  The following are equivalent.
\begin{enumerate}
\item[(C0)]  $\tu{(HO}^+_{\tu{disc}}\tu{)}$ is true on the Newton stratum $N_{b,K^p}$, for all neat open compact $K^p\subset G(\A^{\infty,p})$.
\item[(C3)] The stabilizer in $J_b(\Q_p)$ of each closed point of $X_{\sigma\mu_p^{-1}}(b)$ maps onto $\cG^{\ab}(\Z_p)$ under $\zeta_b$.
\end{enumerate}
\end{corollary}

\begin{proof}
(C3) $\Rightarrow$ (C0). Let $x\in N_{b,K^p}(\Fpbar)$. Then $b_x=g^{-1} b \sigma(g)$ for some $g\in G(\breve\Q_p)$. The map $x\mapsto x g$ induces an isomorphism $X_{\sigma\mu_p^{-1}}(b_x)\simeq X_{\sigma\mu_p^{-1}}(b)$ equivariantly with respect to the actions by the $\Q_p$-groups $J_{b_x}(\Q_p)\simeq J_b(\Q_p)$. The latter isomorphism comes from the conjugation by $g$ on $G(\breve \Q_p)$, and it commutes with the maps $\zeta_{b_x}$ and $\zeta_b$ to $G(\Q_p)^{\ab}$. Hence (c3) of Theorem \ref{thm:HO-criterion} for $b_x$ is implied by (C3) of this corollary. 
We deduce (C0) from the same theorem.

(C0) $\Rightarrow$ (C3). Fix a neat subgroup $K^p\subset G(\A^{\infty,p})$. Let $x\in N_{b,K^p}(\Fpbar)$. Since $[b_x]=[b]$, we may assume  $b=b_x$. 
Now, for each $b'\in X_{\sigma \mu_p^{-1}}(b)$, 
there exists  $y\in N_{b,K^p}(\Fpbar)$ such that $b'=b_y$ by \cite[Prop.~1.4.4]{KisinModels}. As in the proof of (2) $\Rightarrow$ (1), the stabilizer of $b'$ in $J_b(\Q_p)$ maps onto $\cG^{\ab}(\Z_p)$ under $\zeta_b$ if and only if (c3) of Theorem \ref{thm:HO-criterion} holds for $b_y$ (in place of $b$). The latter condition holds as we are assuming (C0), again via the same theorem. Hence (C3) holds.
\end{proof}

\begin{example}\label{ex:HO+false}
Condition (C3) of the last corollary makes it convenient to generate a counterexample to (\tu{HO}$^+_{\tu{disc}}$) by utilizing facts about affine Deligne--Lusztig varieties. 
We learned such an example from Rong Zhou in an email correspondence together with Pol van Hoften, via a Shimura datum $(G,X)$ of PEL type A such that $G^{\ad}$ is $\Q$-simple but (C3) is violated by some non-basic element $b\in B(G_{\Q_p},\mu_p^{-1})$.
It comes down to an explicit affine Deligne--Lusztig variety associated with $\GL_2$ which is a union of irreducible components isomorphic to projective lines. In this case, it can be shown that some component contains a closed point whose stabilizer is too small to satisfy (C3).
See \cite[\S6.3]{vanHoftenXiao} for details. 
\end{example}

\bibliographystyle{amsalpha}
 
\newcommand{\etalchar}[1]{$^{#1}$}
\providecommand{\bysame}{\leavevmode\hbox to3em{\hrulefill}\thinspace}
\providecommand{\MR}{\relax\ifhmode\unskip\space\fi MR }
 
\providecommand{\MRhref}[2]{ 
  \href{http://www.ams.org/mathscinet-getitem?mr=#1}{#2}
}
\providecommand{\href}[2]{#2}

\end{document}